\documentclass[12pt]{report}
\usepackage[utf8]{inputenc}
\usepackage{amssymb,amsthm, amsmath, mathabx}
\usepackage{color, enumerate}
\usepackage{subcaption}
\usepackage{longtable}
\usepackage{comment}
\usepackage{hyperref}
\usepackage[dvipsnames, usenames]{xcolor}
\usepackage{tikz}
\usetikzlibrary{arrows,calc}
\usepackage{graphicx}
\usepackage{epstopdf}
\usepackage[percent]{overpic}
\usepackage{accents}
\usepackage[font=small]{caption}
\usepackage{amssymb,mythesis,uwthesis}
\begin{document}
%% Comment out items by inserting a percent % character
\title{LIMIT SHAPE AND FLUCTUATIONS FOR EXACTLY SOLVABLE INHOMOGENEOUS CORNER GROWTH MODELS}
\author{Elnur Emrah}
\oraldate{July 7, 2016}
%% The names and titles below are examples. Change to suit your situation
\profA{David Anderson, Associate Professor, Mathematics}
\profB{Andreas Seeger, Professor, Mathematics}
\profC{Timo Sepp\"{a}l\"{a}inen, Professor, Mathematics}
\profD{Benedek Valk\'{o}, Professor, Mathematics}
\profE{Jun Yin, Assistant Professor, Mathematics}

%\inputpicturetrue  % By Jeff McGough. See uuguide and private thesis.sty
%\inputpicturefalse % To NOT produce pictures, uncomment this line
%% You may want to change these if you are in a different department
%% or getting a different degree or writing a thesis instead of a dissertation
\degree{Doctor of Philosophy}
\dept{Mathematics}
\thesistype{dissertation}
%% Starts page numbering as i, ii, etc.
\beforepreface
\prefacesection{Abstract}
\noindent {\bf Limit shape and fluctuations for exactly solvable inhomogeneous corner growth models}

We study a class of corner growth models in which the weights are either all exponentially or all geometrically distributed. The parameter of the distribution at site $(i, j)$ is $a_i+b_j$ in the exponential case and $a_i b_j$ in the geometric case, where $(a_i)_{i \ge 1}$ and $(b_j)_{j \ge 1}$ are themselves drawn randomly at the outset from ergodic distributions.   
These models are inhomogeneous generalizations of the much studied exactly solvable models in which the parameters are the same for all sites. Our motivation is to understand how inhomogeneity influences the limit shape and the corresponding limit fluctuations. We obtain a simple variational formula for the shape function and prove that it is strictly concave inside a cone (possibly the entire quadrant) but is linear outside. This is in contrast with the situation in the models with i.i.d. weights in which the shape function is expected to be strictly concave under mild assumptions. For the directions inside the cone, we show that the limit fluctuations are governed by the Tracy-Widom GUE distribution and derive bounds for the deviations of the last-passage times above the shape function. To obtain the shape result, we couple the model with an explicit family of stationary versions of it. For the fluctuation and large deviation results, we perform steepest-descent analysis on an available Fredholm determinant formula for the one-point distribution of the last-passage time. We also develop a detailed appendix on the steepest-descent curves of harmonic functions of two real variables and approximate the contour integral of an arbitrary meromorphic function along such curves. This material can facilitate the steepest-descent arguments in the treatments of other related models as well.

%The required abstract and a copy of your title page is sent to UMI for
%publication in Dissertation Abstracts International.  The abstract must be
%in English, double-spaced, must not exceed 350 words, and must be signed by
%your adviser.  (Abstracts exceeding 350 words will be returned by UMI.  The
%title is not included in the word count.) The dissertation title on the
%abstract must be identical to the title on the dissertation title page.  The
%original signature of your adviser must be at the end of the abstract text.
%This abstract is not part of your dissertation.  Follow the form in the
%samples section.  You may also view a sample with Mosaic on the Internet.  

%\noindent {\bf abstract within dissertation (optional)}
%
%Your department may require an abstract to be part of the dissertation.  If
%so, it is in addition to the abstract mentioned above.  Please follow your
%department's style requirements, and number all of these pages as part of
%the preliminary material (use lower case Roman numerals).  This abstract
%must appear in the table of contents.  

\prefacesection{Acknowledgements}
\noindent {\bf Acknowledgement}

I am deeply grateful to my advisor, Timo Sepp\"{a}l\"{a}inen, for pointing me towards an exciting direction of research, for charting a general course to follow but also granting me a comfortable degree of independence in my pursuits,  for making himself available in need, for his patience when the progress was slow, for his readiness to impart his wisdom and knowledge to me, for his encouragement that made me feel appreciated for my research and keep going forward, for his valuable comments to improve the quality of my work, for his guidance, counseling and countless advice through my graduate study, and so many other things I will not be able to adequately acknowledge here. I truly feel privileged to have been his student. 

I am thankful to the Department of Mathematics at University of Wisconsin--Madison for providing a supportive environment for research and offering a rich collection of courses to explore many subjects. To my surprise, the problems treated in the thesis have brought me to revisit and utilize tools I learned over the years in the graduate program from various professors and fellow graduate students. All are gratefully acknowledged here. 

I am very happy to have been a member of the probability research group at UW--Madison. I would like thank all professors and graduate students also involved in this group. Interacting with them through courses, seminars and more informal gatherings has contributed to my understanding of the field of probability and the probabilistic mode of thinking. 

I would like to extend special thanks to Benedek Valk\'{o}, who made the very useful suggestion of controlling the steepest-descent curves through their ODEs at a talk a few years ago when I presented a much earlier version of Chapter \ref{Ch4}. His remark led me to develop Appendix \ref{flucAp1}, which significantly clarified and shortened the treatment in the main text. 

I benefited from a graduate fellowship program between UW--Madison and University of Bonn, which supported my stay in Bonn in January and February, 2015. I am thankful to the Mathematics Department at University of Bonn for their hospitality, and to Patrik Ferrari for kindly agreeing to supervise me and sharing his insights on my research problem. 

Another special thanks go to my coauthor Chris Janjigian. I have enjoyed, learned a lot and received ideas for future work from our discussions. 

I would like thank many friends who helped keep my spirit up through the graduate school and have made living in Madison a very pleasant and enjoyable experience. 

Finally, I would like send many, many, many thanks to my parents, Ho\c{s}kadem and \c{S}ahin, and my sister, Nermin. Their love and support gave the strength to continue on the path that has led me to where I am now.

\listoffigures
\listoftables
% Optional front page, made from source "notation.tex".
% If you don't need it, then don't use it.
%
%\optionalfront{Notation and Symbols}{\input{notation.tex}}

%% Start the thesis itself, pages numbered 1, 2, etc.
% Start normal page numbering. Parts and chapters follow.%
\afterpreface
%%%
%%% This is the beginning of the actual thesis.  If you don't know latex
%%% then start with the LaTeX manual by Lamport and another easy
%%% reference, like the paperback by Jane Hahn, LaTeX for Everyone, PTI,
%%% 1991. See also $TEX/latex/sample.tex and $TEX/doc/story.tex, where
%%% $TEX==/usr/local/lib/tex
%%%
% Start this dissertation....
%
\chapter{Introduction}

\section{The corner growth model}

Stochastic models of planar growth have a long history in probabilistic research, dating at least as far back as the \emph{Eden model} \cite{Eden61} and the more general \emph{Richardson models} \cite{Richardson73}. One can interpret these models as describing the spread of an infection in a tissue of cells over time. Each cell becomes susceptible to the infection under certain deterministic conditions such as the presence of at least one infected neighbor. 
Individual cells vary in resilience against the infection; therefore, the time period between the beginning of the susceptible state and the infected state for each cell is assumed random. 

Despite the randomness in the dynamics, the set of infected cells can look like a deterministic limit shape after a long time. The first rigorous result of this sort is perhaps due to D. Richardson \cite{Richardson73}. It has been of interest to understand when a limit shape exists and identify its features. Proceeding further, one can also inquire about the fluctuations around the limit shape. Satisfactory progress on these problems has been limited to specific rules of growth and the probability distributions governing the randomness.  

The topic of this work is one of the most studied stochastic growth models known as the \emph{corner growth model}, see \cite{Martin06}, \cite{Seppalainen4}, \cite{Seppalainen09}. The model is intimately connected to various other important models including the directed random polymer, directed last-passage percolation, M/M/1 queues in series and the totally asymmetric simple exclusion process (TASEP). Certain special cases of the model are known to be \emph{exactly solvable} in the sense that rather precise analysis can be carried out. For these cases (the classical examples of which will be described below), it is known that the model is a member of the conjectural KPZ (Kardar-Parisi-Zhang) universality class, a large collection of statistical models that are expected to share certain universal behavior. For example, it is predicted that the scaling exponent of the fluctuations of the interesting observables in the KPZ-models is $1/3$ and the limit fluctuations are given by a Tracy-Widom distribution, see the survey \cite{Corwin} and the references therein.

In this introductory section, we define the general version of the corner growth model and allude to its connections to various other models. We also state the well-known shape and fluctuation results that our work extends. A brief guide to the organization of the material in the remainder of the text follows. Section \ref{intS2} defines precisely the class of corner growth models treated here. Section \ref{intS3} includes a discussion of our main results. Section \ref{intS4} is a short survey of related works from the literature. The proof of the limit shape result is given in Chapter \ref{sha}, which is taken from \cite{Emrah16}. Chapter \ref{Ch3} is an exposition of the exact one-point distribution formula for the last-passage and appeared earlier in \cite{Emrah15}. Chapter \ref{Ch4} proves the fluctuation and large deviation results. 

We now define the corner growth model. Represent a set of \emph{sites} with $\bbN^2$, where $\bbN = \{1, 2, 3, \dotsc\}$. At the outset, each site $(i, j) \in \bbN$ is assigned a \emph{weight} $W(i, j)$, a randomly chosen nonnegative real number. In the general version of the model, the joint distribution $\bfP$ of the weights $\{W(i, j): i, j \in \bbN\}$ is an arbitrary (Borel) probability measure on $\bbR_+^{\bbN^2}$. The dynamics begins with all sites initially colored white at time $t = 0$. Each site $(i, j)$ changes color to red permanently after $W(i, j)$ amount of time has elapsed since the first time both of the following conditions hold. 
\begin{itemize}
\item If $i > 1$ then $(i-1, j)$ is red. 
\item If $j > 1$ then $(i, j-1)$ is red. 
\end{itemize} 
To relate to the story of an epidemic, imagine the sites as cells, white cells as healthy, red cells as infected and the weights as the resilience of the cells. A healthy cell is susceptible to the infection if its left and bottom neighbors (if they exist) are already infected. Hence, the epidemic starts out from the corner $(1, 1)$ and spreads to the entire quadrant $\bbN^2$ over time.  

For each $(i, j) \in \bbN^2$, let $G(i, j)$ denote the time when the site $(i, j)$ becomes red. Let $\sS_t$ denote the set of red sites at time $t \ge 0$ i.e. 
\begin{align}
\sS_t = \{(i, j) \in \bbN^2: G(i, j) \le t\} \quad \text{ for } t \ge 0. 
\end{align}
To understand the limit shape of red sites, it is helpful to study the random variables $\{G(i, j): i, j \in \bbN\}$. The growth rule above translates to the recursion 
\begin{align}
\label{shae2}
G(i, j) &= G(i-1, j) \vee G(i, j-1) + W(i, j) \\ 
G(i, 0) &= G(0, j) = 0
\end{align}
for $i, j \in \bbN$. One can also picture the preceding recursion as a sequence of servers labeled with $i \in \bbN$ serving a sequence of customers labeled with $j \in \bbN$ \cite{GlynnWhitt91}. Each server delivers service to the customers one by one in the order the customers arrive. Each customer arrives the system by joining the queue at server $1$. For each $i, j \in \bbN$, the amount of time server $i$ needs for the service of customer $j$ is $W(i, j)$. When this service is completed, customer $j$ immediately departs and joins the queue at the server $i+1$, and the server $i$ becomes available for the next customer $j+1$. We will not take advantage of this viewpoint here and refer the interested reader to \cite{Martin06} and \cite{Seppalainen09}. 

Another equivalent description of the last-passage times is in terms of directed paths. A directed path $\pi$ is a finite sequence $\pi = ((i_k, j_k))_{1 \le k \le l}$ in $\bbZ^2$ such that 
\begin{align*}(i_{k+1}-i_k, j_{k+1}-j_k) \in \{(1, 0), (0, 1)\} \quad \text{ for } 1 \le k < l. \end{align*} 
We say that $\pi$ is from $(u, v) \in \bbZ^2$ to $(u', v') \in \bbZ^2$ if $(i_1, j_1) = (u, v)$ and $(i_l, j_l) = (u', v')$. See Figure \ref{Fi1}. We have 
\begin{equation}\label{shaeq51}G(m, n) = \max_{\pi \in \Pi_{1, 1, m, n}} \sum \limits_{(i, j) \in \pi} W(i, j) \quad \text{ for } m, n \in \bbN, \end{equation}
where $\Pi_{u, v, u', v'}$ denotes the set of all directed paths from $(u, v)$ to $(u', v')$. Hence, $\{G(i, j): i, j \in \bbN\}$ are exactly the \emph{last-passage time} variables of the directed last-passage percolation. This connection seems to have been first observed in \cite{Muth}. 

\begin{figure}[h]
\centering
\begin{tikzpicture}[scale = 1.2]
\draw[gray,thin] (0,0) grid (5,5);
\draw (0, 0)node[below]{$(1, 1)$};
\draw[<->](0, 5.5)node[above]{$j$}--(0, 0)--(5.5, 0)node[right]{$i$};
\draw (5.5, 5)node[above]{$(m, n)$};
\foreach \i in {0, ..., 5}
{
	\foreach \j in {0, ..., 5}
	{ \fill[blue] (\i, \j) circle (0.1cm); }
}
\draw[red, ultra thick] (0, 0) -- (2, 0) -- (2, 1) -- (3, 1) -- (3, 4) -- (4, 4) -- (4, 5) -- (5, 5); 
\end{tikzpicture}
\caption[A directed path]{A directed path (red) from $(1, 1)$ to $(m, n) = (6, 6)$. Blue dots indicate sites. } \label{Fi1}
\end{figure}
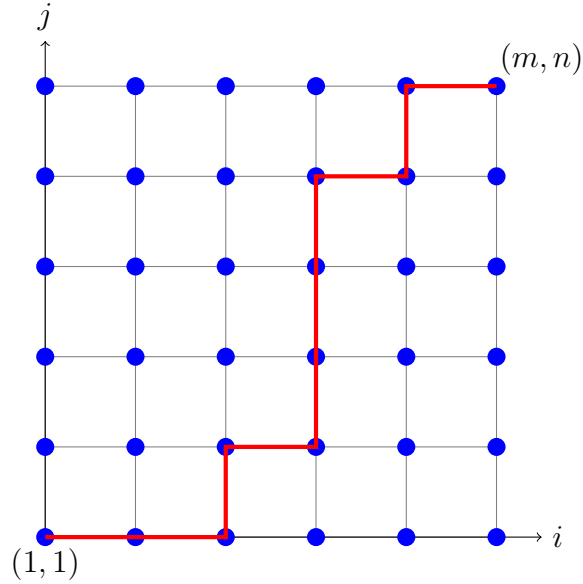

The first limit shape result for the corner growth model appeared in the pioneering work of H. Rost \cite{Rost81}, which 
studied the totally asymmetric simple exclusion process (TASEP). This is a fundamental interacting particle system that can be viewed as a toy model for single-lane traffic. In the version relevant to the present discussion, we consider particles labeled with $i \in \bbN$ residing at the sites of $\bbZ$. Let $\sigma(i, t) \in \bbZ$ denote the position of particle $i \in \bbN$ at time $t \ge 0$. Impose the jam (step) initial condition i.e. $\sigma(i, 0) = -i$ for $i \in \bbN$. At the outset particles are equipped with independent Poisson clocks with rate $\lambda > 0$. The positions of the particles are updated as follows. A particle currently at site $j \in \bbZ$  jumps to $j+1$ when its clock rings provided that there is no particle at $j+1$. This exclusion constraint, which disallows presence of two particles at the same site, is the interaction between otherwise independently moving particles. As explained, for example, in \cite[p~5]{Seppalainen09}, the distribution
of $\{(i, j) \in \bbN^2: \sigma(i, t) \le i + j\}$ is the same as the distribution of $\sS_t$ in the \emph{exponential model}, the corner growth model with i.i.d. weights such that  
\begin{align}
\label{E94}
\bfP(W(i, j) \ge x) = e^{-\lambda x} \text{ for } i, j \in \bbN \text{ and } x \ge 0. 
\end{align}
As a corollary of his results for TASEP, Rost obtained that 
\begin{align}
\label{shaE1}
\lim \limits_{n \rightarrow \infty} \frac{G(\lf ns \rf, \lf nt \rf)}{n} = \frac{(\sqrt{s}+\sqrt{t})^2}{\lambda} \quad \text{ for } s, t > 0 \quad \bfP\text{-a.s.} 
\end{align}
He also showed that the rescaled set of red sites 
\begin{align}
\frac{1}{t} \sS_t = \{(i, j) \in \bbN^2: G(\lf ti \rf, \lf tj \rf) \le t\}
\end{align}
converges $\bfP$-a.s. as $t \rightarrow \infty$ to the parabolic region $\{s,t \in \bbR_+: \sqrt{s}+\sqrt{t} \le \sqrt{\lambda}\}$ in a certain sense. 

%To see this, consider particles labeled with $i \in \bbN$ on $\bbZ$. For each $i \in \bbN$,  define the position $\sigma(i, t)$ of particlEe $i$ at time $t \ge 0$ by
%\begin{align}
%\sigma(i, t) = \max \{j \in \bbZ_+: G(i, j) \le t\}-i+1. 
%\end{align}

There is also a discrete-time version of TASEP in which time variable $t \in \bbZ_+$. Particles now carry independent coins with tails probability $q \in (0, 1)$. For $i \in \bbN$ and $t \in \bbZ_+$, the position of particle $i$ at time $t+1$ is determined from the particle configuration at time $t$ as follows. Particle $i$ flips its coin at time $t$. If heads comes up and the position $\sigma(i, t)+1$ is empty then $\sigma(i, t+1) = \sigma(i, t)+1$; otherwise, $\sigma(i, t+1) = \sigma(i, t)$.  
The discrete-time TASEP with step initial condition is equivalent to the corner growth model in which the weights are i.i.d. and 
\begin{align}
\label{intE2}
P(W(i, j) \ge k) = q^k \quad \text{ for } i, j \in \bbN \text{ and } k \in \bbZ_+. 
\end{align}
For this \emph{geometric model}, the limit shape has also been computed. 
\begin{align}
\label{shaE2}
\lim \limits_{n \rightarrow \infty} \frac{G(\lf ns \rf, \lf nt \rf)}{n} = \frac{q}{1-q}(s+t) + \frac{2\sqrt{q}}{1-q} \sqrt{st} \quad \text{ for } s , t > 0 \quad \bfP\text{-a.s., }
\end{align}
\cite{CohnElkiesPropp}, \cite{JockuschProppShor}, \cite{Seppalainen2}. 

The limit fluctuations corresponding to (\ref{shaE1}) and (\ref{intE2}) were derived in the breakthrough work of K. Johansson \cite[Theorem~1.2]{Johansson00}. For the geometric model, for $r > 0$,  
\begin{equation}
\label{E1}
\lim \limits_{n \rightarrow \infty} \bfP(G(\lf nr \rf, n) \le n \gamma(r) + n^{1/3}\sigma(r) s) = F_{\GUE}(s) \quad \text{ for } s \in \bbR, \end{equation}
where 
\begin{align}
\gamma(r) &= \frac{q(1+r)+2\sqrt{qr}}{1-q} \label{E2}, \\
\sigma(r) &= \frac{1}{1-q} \left(\frac{q}{r}\right)^{1/6}(\sqrt{q} + \sqrt{r})^{2/3}(1+\sqrt{qr})^{2/3} \label{E3} 
\end{align}
and $F_{\GUE}$ denotes the c.d.f. of the Tracy-Widom GUE distribution defined in Section \ref{flucSe2}. 
For the exponential model, (\ref{E1}) is also true with different explicit constants $\gamma(r)$ and $\sigma(r)$, \cite[Theorem~1.6]{Johansson00}. 

Precise description of the limit shape and limit fluctuations such as (\ref{shaE2}) and (\ref{E1}) for the corner growth model is presently only possible for the exponential and geometric models, and their derivatives (an example of which is treated in this work). It is expected that the model exhibit certain features that are universal, which are there irrespective of the underlying weight distribution. For example, at least when $\bfP$ is i.i.d. and satisfies mild conditions, the shape function is expected to be strictly concave and differentiable. The limit fluctuations should be governed by the Tracy-Widom GUE distribution as in (\ref{E1}).  

%KPZ universality suggests that (\ref{E1}) should be true more generally, for instance, if $\bfP$ is i.i.d. and the weights have enough moments {\color{red} cite}. 
%\begin{align*}
%g(s, t) = \lim_{n \rightarrow \infty} \frac{G(\lf ns \rf, \lf nt \rf)}{n} \quad \text{ for } s, t > 0 \quad P\text{-a.s.}
%\end{align*} 

\section{Exponential and geometric models with inhomogeneous parameters}
\label{intS2}

The class of models we study are \emph{inhomogeneous} generalizations of the exponential and geometric models. The parameters $\lambda$ and $q$ themselves will now be randomly chosen for each site at the outset and then kept fixed throughout the dynamics. It turns out that, for certain choices of the parameter distributions, the resulting models are still amenable to precise analysis. Taking advantage of this situation, we find the analogues of H.Rost's limit shape result (\ref{shaE1}) and K.Johansson's fluctuation result (\ref{E1}) for these inhomogeneous models. The next section discusses in detail the contribution of the present work. Here, we provide a precise description of the inhomogeneous exponential and geometric models and allude to their aspects that lead to exact solvability.  

For concreteness, let us construct the sample space of the weights as $\bbR_+^{\bbN^2}$ equipped with the Borel $\sigma$-algebra. Define $W(i, j)$ as the projection map $\bbR_+^{\bbN^2} \rightarrow \bbR_+$ onto coordinate $(i, j)$ for $i, j \in \bbN$. Let $\bfa = (a_n)_{n \in \bbN}$ and $\bfb = (b_n)_{n \in \bbN}$ be stationary random sequences with terms in $(0, \infty)$ for the exponential model and in $(0, 1)$ for the geometric model. For the shape results, we will assume that the pair $(\bfa, \bfb)$ is totally ergodic with respect to the shifts $(a_n, b_n) \mapsto (a_{n+k}, b_{n+l})$ for $k, l \in \bbN$, see Section \ref{shaS1}. As an example, take $\bfa$ and $\bfb$ as independent i.i.d. sequences. For the fluctuation results, the total ergodicity assumption is weakened to ergodicity of $\bfa$ and $\bfb$ (i.e. these sequences are each ergodic and no assumption is made on their joint distribution). Suppose that, given $(\bfa, \bfb)$, the weights $\{W(i, j): i, j \in \bbN\}$ are conditionally independent and the joint conditional distribution $\bfP_{\bfa, \bfb}$ of the weights are given by 
\begin{align}
\bfP_{\bfa, \bfb}(W(i, j) \ge x) = e^{-(a_i+b_j)x} \quad \text{ for } i, j \in \bbN \text{ and } x \ge 0 \label{shaE4}
\end{align} 
for the exponential model and 
\begin{align}
\bfP_{\bfa, \bfb}(W(i, j) \ge k) = a_i^k b_j^k \quad \text{ for } i, j \in \bbN \text{ and } k \in \bbZ_+ \label{shaE5}
\end{align} for the geometric model. 

Let $\bbP$ denote the (unconditional) joint distribution of the weights. Thus, $\bbP(B) = \int \bfP_{\bfa, \bfb}(B) \mu(\dd\bfa, \dd\bfb)$ for any Borel set $B \subset \bbR_+^{\bbN^2}$, where $\mu$ is the joint distribution of $(\bfa, \bfb)$. By stationarity of $\bfa$ and $\bfb$, the weights are identically distributed under $\bbP$. However, there are nontrivial correlations among the weights since $W(i, j)$ and $W(i', j')$ are, in general, not independent under $\bbP$ if $i = i'$ or $j = j'$. 

%TASEP interpretation. particles and wholes have disorders.  

We mention the key features that render these inhomogeneous models well-suited for analysis.  For the limit shape result, we exploit an explicit one-parameter family of product probability measures $\bfP_{\bfa, \bfb}^z$ on the extended sample space $\bbR_+^{\bbZ_+^2}$ that project to $\bfP_{\bfa, \bfb}$ on $\bbR_+^{\bbN^2}$, see (\ref{shaeq78}) and (\ref{shaeq79}) below. These measures satisfy a useful stationarity property (Proposition \ref{shap5}) sometimes referred to as the \emph{Burke property} that enables exact computations. This approach was introduced in \cite{Seppalainen2}.  
For the fluctuation results, our starting point is that the probabilities $\bfP_{\bfa, \bfb}(G(m, n) \le k)$ can be represented as a Fredholm determinant with a tractable kernel \cite{Johansson2}. Derivation of this formula is discussed at length in Chapter \ref{Ch3}. 

\section{Main results}
\label{intS3}

Our first result is the identification of the shape function in terms of a one-dimensional variational problem. Some notation is needed to state it. Let $\alpha$ and $\beta$ denote the distributions of $a_1$ and $b_1$, respectively. (By stationarity of $\bfa$ and $\bfb$, $a_n$ and $b_n$ are distributed as $\alpha$ and $\beta$, respectively, for $n \in \bbN$). For a Borel measure $\eta$ on $\bbR$, let 
$\ubar{\eta}$ and $\bar{\eta}$ denote the left and right endpoints of the support of $\eta$ (the complement of the largest open $\eta$-null set). Extend $\alpha$ and $\beta$ as Borel probability measures to $\bbR$. Then the support of $\alpha$ and $\beta$ are contained in $[0, \infty)$ with $\alpha(\{0\}) = \beta(\{0\}) = 0$ for the exponential model, and the supports are contained in $[0, 1]$ with $\alpha(\{0, 1\}) = \beta(\{0, 1\}) = 0$ for the geometric model. Also, $\ubar{\alpha}$ and $\bar{\alpha}$ equal, respectively, the essential infimum and the essential supremum of $a_1$. Similarly for $\beta$. 

The shape function is given by 
\begin{align}
g(s, t) = \inf \limits_{z \in [-\ubar{\alpha}, \ubar{\beta}]} \left\{s \int_0^\infty \frac{\alpha(\dd a)}{a+z} + t \int_0^\infty \frac{\beta(\dd b)}{b-z}\right\} \quad \text{ for } s, t > 0 \label{shaE6}
\end{align}
for the exponential model. This will be restated as Theorems \ref{shaT1} and \ref{shaT2}. The infimum above can be computed explicitly when $\alpha$ and $\beta$ are uniform measures. 
For example, if $\alpha$ and $\beta$ are uniform on the interval $[1/2, 3/2]$ then 
\begin{align}
g(s, t) = s \log\bigg(\frac{t+7s+\sqrt{(s+t)^2+12st}}{4s}\bigg) + t \log\bigg(\frac{s+7t+\sqrt{(s+t)^2+12st}}{4t}\bigg).
\label{shaE7}
\end{align}
for $s, t > 0$. For this choice of $\alpha$ and $\beta$, Figure \ref{Fi3} below illustrates the limit shape result 
when $\bfa$ and $\bfb$ are i.i.d. sequences. If $\ubar{\alpha}=\ubar{\beta} = 0$ then (\ref{shaE6}) reduces to 
\begin{align}
g(s, t) = s \int_0^\infty \frac{\alpha(\dd a)}{a} + t \int_0^\infty \frac{\beta(\dd b)}{b} \quad \text{ for } s, t > 0.  
\end{align}
In particular, $g$ is a linear function if the above integrals are both finite, and is identically infinite otherwise. For $\ubar{\alpha}+\ubar{\beta} > 0$, the following properties of $g$ can be derived from (\ref{shaE6}), see Corollary \ref{shac3}. There are critical values $0 \le c_1 < c_2 \le \infty$ determined by $\alpha$ and $\beta$ such that, as a function of $(s, t)$, $g$ is strictly concave for $c_1 < s/t < c_2$ but is linear for $s/t \le c_1$ or $s/t \ge c_2$. We have $c_1 = 0$ if and only if $\int_0^\infty (a-\ubar{\alpha})^{-2}\alpha(\dd a) = \infty$, and $c_2 = \infty$ if and only if $\int_0^\infty (b-\ubar{\beta})^{-2}\beta(\dd b) = \infty$. Hence, the linear regions of $g$ can be empty as is the case in (\ref{shaE7}). Also, $g$ is $\sC^1$. Consequently, the graph of $g$ does not have a sharp corner along the critical lines $s/t = c_1$ and $s/t = c_2$. These features are depicted on the level curve $g(s, t) = 1$ for a particular choice of $\alpha$ and $\beta$ in Figure \ref{Fi4} below. 

\begin{figure}[h!]
\centering
\includegraphics[scale=0.8]{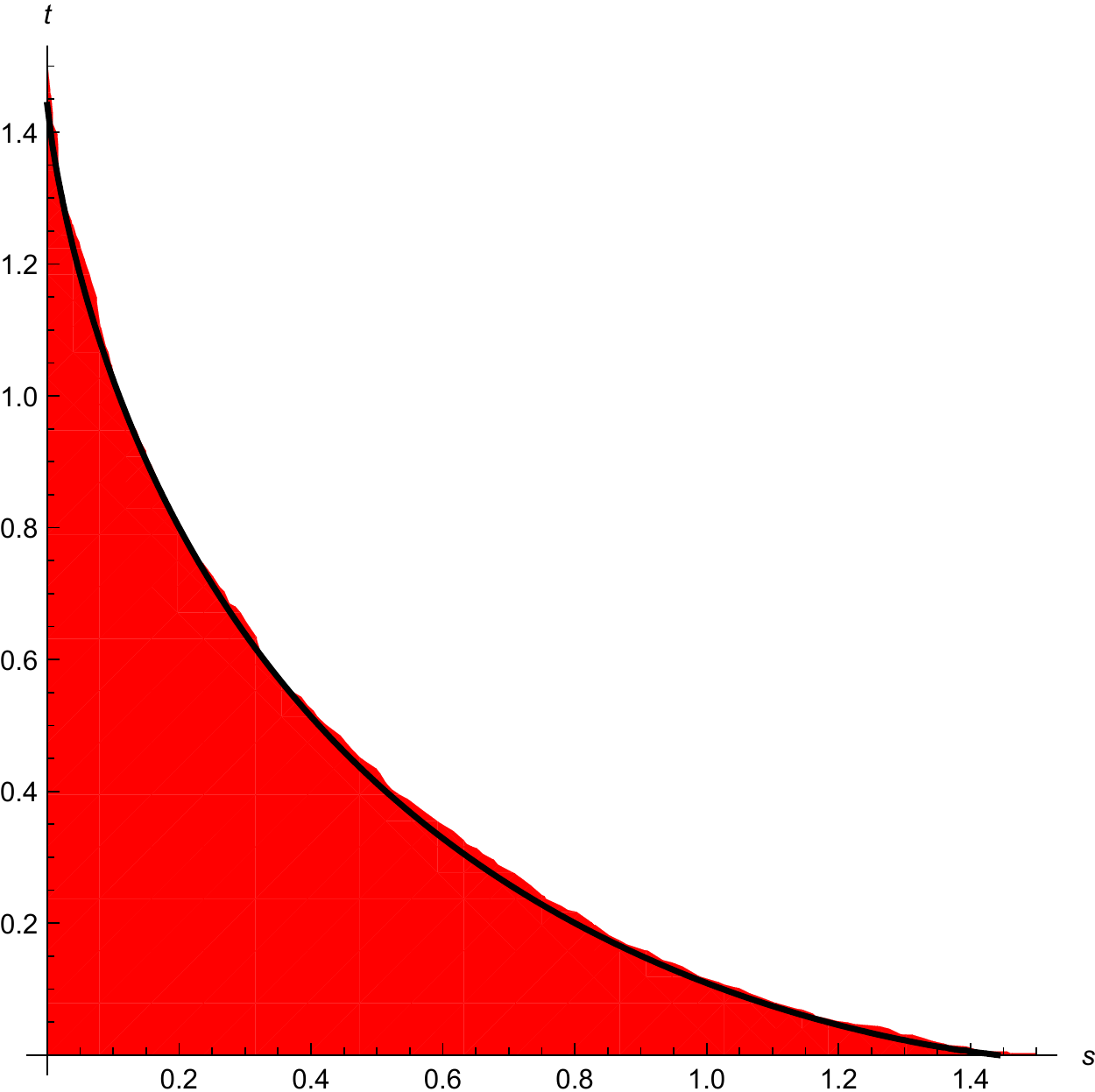}
\caption[A simulation of the inhomogeneous exponential model]{The output of a simulation of the inhomogeneous exponential model up to time $t = 2000$. The sequences $\bfa$ and $\bfb$ are i.i.d, and $\alpha$ and $\beta$ are uniform on $[1/2, 3/2]$. The red region is the rescaled set of red sites $t^{-1}\sS_t$ when $t = 2000$. The black curve is the level set $g(s, t) = 1$ for (\ref{shaE7}).}
\label{Fi3}  
\end{figure}

\begin{figure}[h!]
\centering
\includegraphics[scale=0.8]{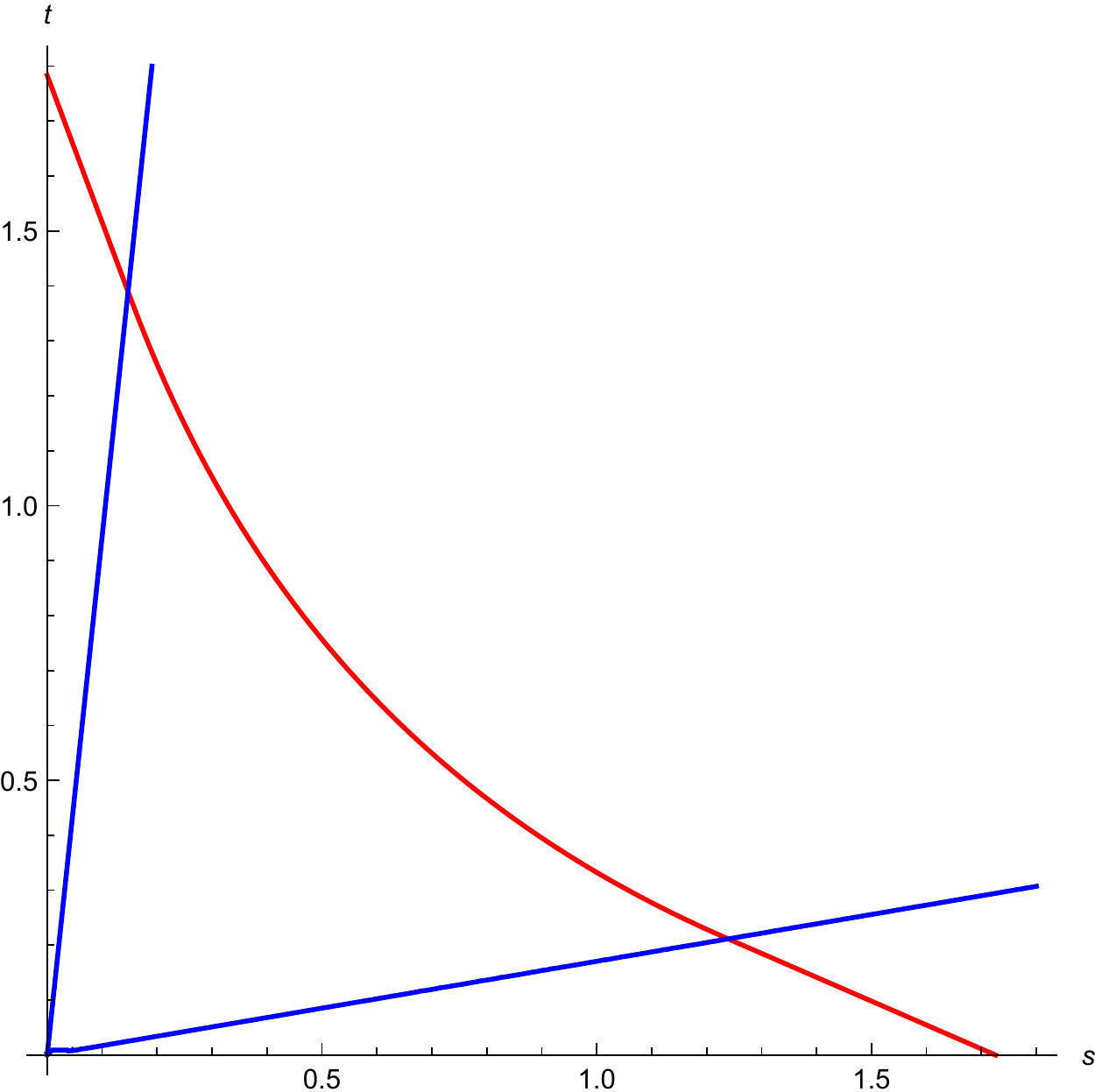}
\caption[A level curve of the shape function]{The level curve $g(s, t) = 1$ and the rays $s/t = c_1 = (-8+12\log2)/3 \approx 0.105922$ and $s/t = c_2 = 4/(9-12\log2) \approx 5.863092$ for (\ref{shaE6}) when $\alpha(\dd a) = \one_{\{0 \le a \le 1\}}3a^2\dd a$ and $\beta(\dd b) = \one_{\{1 \le b \le 2\}} 4(b-1)^3\dd b$.}
\label{Fi4}  
\end{figure}

For the geometric model, we have similar results for the limit shape. The analogue of (\ref{shaE6}) is stated as Theorems \ref{shaT3} and \ref{shaT4}, and a closed form formula similar to (\ref{shaE7}) can be derived when $\alpha$ and $\beta$ have densities proportional to $x \mapsto 1/x$. If $\bar{\alpha} = \bar{\beta} = 1$ then $g$ is linear or infinite. If $\bar{\alpha}\bar{\beta} < 1$ then $g$ is strictly concave inside the region $c_1 < s/t < c_2$ and linear outside for some constants $0 \le c_1 < c_2 \le \infty$ as exemplified in Figure \ref{Fi4}. Now $c_1 = 0$ if and only if $\int_0^1 (\bar{\alpha}-a)^{-2} \alpha(\dd a) = \infty$, and $c_2 = \infty$ if and only if $\int_{0}^1 (\bar{\beta}-b)^{-2} \beta(\dd b) = \infty$. 

The next set of results concerns the fluctuations around the limit shape with respect to $\bfP_{\bfa, \bfb}$ for a.e.\space realization of $\bfa$ and $\bfb$. At present, we only study the fluctuations in the strictly concave region $c_1 < s/t < c_2$ of the geometric model. We do not expect difficulty in repeating the analysis here for the strictly concave region of the exponential model. However, the treatment of the linear regions (in both models) seems to require significant further work. 

Our first result bounds the right tail deviations of the last-passage times i.e. deviations above the shape function. Fix $r \in (c_1, c_2)$ and set $\gamma(r) = g(r, 1)$. Then, for a.e. $(\bfa, \bfb)$, there are deterministic constants $C, c > 0$, $n_0 = n_0^{\bfa, \bfb} \in \bbN$  and a sequence $\gamma_n = \gamma_n^{\bfa, \bfb}$ converging a.s.\ to $\gamma$ such that 
\begin{align}
\label{intE1}
\bfP_{\bfa, \bfb}(G(\lf nr \rf, n) \ge n\gamma_n + n s) &\le  Ce^{-cn(s^{3/2} \wedge s)}
\end{align}
for $s \ge 0$ and $n \ge n_0$. In particular, up to a constant, the decay rate of the exponential bound is $s^{3/2}$ and $s$ for small and large values of $s$, respectively. In (\ref{intE1}), the transition occurs at $s = 1$ although this point can be moved to any positive number by altering the constants $C, c$. We state (\ref{intE1}) in a slightly more general form as Theorem \ref{flucT6}. When $\alpha$ and $\beta$ are delta masses then $\bfa$ and $\bfb$ are constant sequences a.s. In this case, we also have $\gamma_n = \gamma$ for $n \in \bbN$, and (\ref{intE1}) then becomes 
\begin{align}
\label{intE3}
\bfP(G(\lf nr \rf, n) \ge n\gamma + ns) \le Ce^{-cn(s^{3/2} \wedge s)}
\end{align}
for $s \ge 0$ and $n \ge n_0$ for some constants $C, c > 0$ and $n_0 \in \bbN$, where $\bfP$ is the i.i.d. measure given by (\ref{intE2}) with $q = a_1 b_1$. We point out that (\ref{intE3}) is not new. Although not recorded explicitly, it can be deduced from \cite[Theorem 2.2, (2.17), (2.21)-(2.23), Corollary 2.4]{Johansson00}. For the exponential model, (\ref{intE3}) (for small values of $s$) was previously pointed out in \cite[p.~622]{Seppalainen98}. Our proof of (\ref{intE3}) is quite different; it comes from a steepest-descent analysis and does not involve computation of the right tail large deviation rate function. 

We also mention that (\ref{intE3}) improves the following right tail bound that recently appeared in \cite[Lemma 2.2]{CorwinLiuWang}. There exists $M, \delta, c > 0$ such that 
\begin{align}
\label{intE4}
\bfP(G(\lf nr \rf, n) \ge n\gamma + n^{1/3}x) \le e^{-cx}
\end{align}
for $M \le x \le \delta n^{1/3}$. In contrast, setting $s = n^{-2/3}x$ in (\ref{intE3}) yields 
\begin{align}
\bfP(G(\lf nr \rf, n) \ge n\gamma + n^{1/3}x) \le Ce^{-c(x^{3/2} \wedge (n^{1/3}x))}
\end{align}
for $x \ge 0$ and $n \ge n_0$. In the regime $M \le x \le \delta n^{1/3}$ (where $M$ is a large constant), we have $x \le x^{3/2} \le n^{1/3}x$. Thus, 
\begin{align}
\bfP(G(\lf nr \rf, n) \ge n\gamma + n^{1/3}x) \le Ce^{-cx^{3/2}} \le e^{-c'x^{3/2}}
\end{align}
for another constant $c' > 0$, and (\ref{intE4}) follows. 

Another result we wish to highlight here is that the quenched limit fluctuations of the last-passage times in the stricly concave region is governed by the Tracy-Widom GUE distribution denoted $F_{\GUE}$ below (see Section \ref{flucSe2} for its definition). More precisely, 
\begin{align}
\label{intE5}
\lim_{n \rightarrow \infty} \bfP_{\bfa, \bfb}(G(\lf nr \rf, n) \le n\gamma_n + n^{1/3} \sigma_n s) = F_{\GUE}(s) \quad \text{ for } s \in \bbR
\end{align}
for a.e. $(\bfa, \bfb)$, where $\sigma_n$ is a sequence (depending on $\bfa, \bfb$) converging to a deterministic quantity $\sigma > 0$. This result is contained Theorem \ref{fluceq23}. Note that due to continuity of $F_{\GUE}$ and the monotonicity of the left-hand side in $s \in \bbR$, (\ref{intE5}) also holds if $\sigma_n$ is replaced with the limit value $\sigma$. We confirm via (\ref{intE5}) that the inhomogeneous geometric model behaves as a KPZ class model in the strictly concave region. This was predicted in \cite{EmrahJanjigian15} on account of the expansion of the right tail rate functions and by analogy with a similar result of J.Gravner, C.Tracy and H.Widom \cite[Theorem~3]{GravnerTracyWidom02a} for a related model known as Johansson-Sepp\"{a}l\"{a}inen model, which we elaborate on in the next section. 

Since $\lim_{n \rightarrow \infty} \gamma_n = \gamma$, (\ref{intE5}) also implies $n^{-1}G(\lf nr \rf, n) = \gamma$ in $\bfP_{\bfa, \bfb}$-probability a.s under the weaker hypotheses that $\bfa$ and $\bfb$ are each ergodic rather than jointly totally ergodic. We expect that a.s convergence also holds but at present we only have $\limsup_{n \rightarrow \infty} n^{-1}G(\lf nr \rf, n) \le \gamma(r)$ a.s., which follows from the Borel-Cantelli lemma using (\ref{intE1}). An exponential bound for the left tail deviations in the i.i.d. setting is known \cite{BaikDeiftMcLaughlinMillerZhou}. Obtaining the corresponding bound in the inhomogeneous setting is left for future work. 

The main strategy to prove the fluctuation and large deviation bound comes from \cite{GravnerTracyWidom02a}. The distributions of the last-passage times admit a Fredholm determinant representation with an explicit kernel depending on finitely many terms from $\bfa$ and $\bfb$ \cite{Johansson2}. This naturally leads to the consideration of the 
empirical distributions $\alpha_{n}$ and $\beta_{n}$ approximating $\alpha$ and $\beta$, respectively. The sequence $\gamma_n(r)$ is, in fact, precisely the shape function (\ref{shaE6}) computed with $(s, t) = (r, 1)$ and using $\alpha_{\lf nr \rf}$ and $\beta_{n}$ in place of $\alpha$ and $\beta$, respectively. The analysis of the kernel involves use of the steepest-descent curves through the minimizer of the variational formula for $\gamma_n$. Since this minimizer is a saddle point of order $2$, the Airy kernel arises in the asymptotics. Because we work with infinitely many steepest-descent curves, sufficiently strong uniform control of these curves is needed for various steps in the argument. In our case, such a control is afforded by the ergodicity of $\bfa$ and $\bfb$. 

This work also attempts to go beyond the purpose of deriving the Tracy-Widom limit for the inhomogeneous geometric model with various future applications in mind. In Section \ref{flucAp1}, we provide a short and largely self-contained treatment of the steepest-descent curves of harmonic functions on $\bbC$. These lemmas offer quantative bounds and help describe rigorously the local and global nature of the steepest-descent curves. In Section \ref{flucAp2}, we approximate the contour integral of an arbitrary meromorphic function along the steepest-descent curve of its (approximate) logarithm. We present an application of this material in our derivation of the fluctuation results. Another immediate use in a work in progress is to compute the right tail rate function. 
%
%{\color{red}
%As usual, one needs to find suitable deformations of the contours to make analysis possible. The appealing side of the approach taken is that one does not need explicit parametrization of the contours; rather, some useful properties of the contours are observed from general considerations. This enables us to use the same approach to study the inhomogeneous corner growth model described above. }

\section{Literature review}
\label{intS4}

We give a brief overview of the literature in relation to the present work. By (\ref{shaE1}) and (\ref{shaE2}), the shape functions of the exponential and geometric models satisfy 
\begin{align}
\label{shaE3}
g(s, t) = m(s+t)+2 \sqrt{\sigma^2 st} \quad \text{ for } s, t > 0, 
\end{align}   
where $m$ and $\sigma^2$ denote the common mean and the variance of the weights. Such explicit description of the shape function for all directions $(s, t)$ were not previously available for other weight distributions. Our variational characterization of the shape function furnishes new explicit formulas such as (\ref{shaE7}). J.Martin proved that, for i.i.d $\bfP$ and assuming the weights have sufficiently light tail, (\ref{shaE3}) also holds up to an error of order $o(\sqrt{t})$ as $t \downarrow 0$. In particular, $g$ cannot be linear near the axes in contrast with the inhomogeneous setting where this is a real possibility, see Figure \ref{Fi4}. For i.i.d. $\bfP$, the linear regions can appear away from the axes. 
For example, if the weights are bounded and attain their maximum with probability larger than the critical probability for the oriented site percolation then $g(s, t) = s+t$ in a nontrivial cone \cite{AuffingerDamron}, \cite{DurrettLiggett}, \cite{Marchand}. Variational formulas characterizing $g$ have recently been derived for general i.i.d. weights with finite $2+\epsilon$ moments for some $\epsilon > 0$ \cite{geor-rass-sepp-lppbuse}.

For the exponential model with i.i.d. $\bfa$ and constant $\bfb$, \cite{KrugSepp} obtained a variational description of $g$, which (\ref{shaE6}) includes as a special case. Asymptotics of $g$ near the axes are determined for more general $P$ in \cite{Lin}. Their Theorem 2.4 can be deduced from (\ref{shaE6}) as well.  Another direction of generalizing the exponential and geometric models is to choose the parameters at site $(i, j)$ as $\lambda = \Lambda(i/n, j/n)$ and $q = Q(i/n, j/n)$ for some deterministic functions $\Lambda$ and $Q$ that encode inhomogeneity. Then, under suitable conditions, $g$ can be characterized as the unique monotone viscosity solution of a certain Hamilton-Jacobi equation \cite{Calder}.

The fluctuation exponents for the exponential model were identified as KPZ exponents in \cite{Balasz}. 
The limit distribution for the rescaled last-passage times for the geometric and exponential models were proved to be the Tracy-Widom GUE distribution by K. Johansson in \cite{Johansson00}, see (\ref{E1}). \cite{Johansson03} proved that suitably rescaled last-passage time process along the antidiagonal through $(n, n)$ converges to the Airy process, \cite{Johansson03}. The measures $\bfP_{\bfa, \bfb}$ defined at (\ref{shaE4}) and (\ref{shaE5}) appeared earlier in \cite{BorodinPeche} and \cite{Johansson2}, respectively, and is closely related to the Schur measures introduced in \cite{Okounkov}. This connection leads to representations of the last-passage distributions in terms of Fredholm determinants with explicit kernels, see the exposition in Chapter \ref{Ch3} and the references therein. \cite{BorodinPeche} considered $\bfP_{\bfa, \bfb}$ in (\ref{shaE4}) such that $(a_i)_{i > k}$ and $(b_j)_{j > l}$ are constant for fixed $k, l \in \bbZ_+$ and identified the limit (in the sense of finite dimensional distributions) of the rescaled last-passage process (the parameters $a_i, b_j$ for $i \in [k]$ and $j \in [l]$ are also rescaled suitably) along a certain line through $(nr, n)$ as a generalization of the extended Airy process. For similar setting with (\ref{shaE5}), \cite{CorwinLiuWang} also determined the limit distribution of the rescaled last-passage times.  

In \cite{GravnerTracyWidom02a}, J.Gravner, C.Tracy and H.Widom studied a similarly inhomogeneous version of a variant of corner growth model known as oriented digital boiling or Johansson-Sepp\"{a}l\"{a}inen model introduced in \cite{Seppalainen3}. The recursion (\ref{shaE2}) is now replaced with   
\begin{align}
\label{E95}
G(i, j) &= G(i-1, j) \vee (G(i, j-1) + W(i, j))\quad \text{ for } i, j \in \bbN.  
\end{align}
%There is also a path description of the last-passage times akin to (\ref{shaeq51}) but replacing $\Pi_{1, 1, m_n}$ with the set of weak/strict paths from $(1, 1)$ to $(m, n)$ as shown in Figure \ref{Fi2}. 
The weights $\{W(i, j): i, j \in \bbN\}$ are independent and each $W(i, j)$ is Bernoulli-distributed with parameter $p_j$ for some ergodic sequence $(p_j)_{j \in \bbN}$. The shape function in this model has a constant, linear and strictly concave regions. A result from \cite{GravnerTracyWidom02a} is that, conditioned on the parameters $p_j$, suitably rescaled last-passage converge in distribution to the Tracy-Widom GUE distribution. (\ref{intE5}) is the analogue of this result for the inhomogeneous geometric model and is derived through similar techniques. 

%\begin{figure}[h]
%\centering
%\begin{tikzpicture}[scale = 1.2]
%\draw[gray,thin] (0,0) grid (5,5);
%\draw (0, 0)node[below]{$(1, 1)$};
%\draw[<->](0, 5.5)node[above]{$j$}--(0, 0)--(5.5, 0)node[right]{$i$};
%\draw (3, 5)node[above]{$(m, n)$};
%\foreach \i in {0, ..., 5}
%{
%	\foreach \j in {0, ..., 5}
%	{ \fill[blue] (\i, \j) circle (0.1cm); }
%}
%\draw[red, very thick] (0, 0) -- (1, 1) -- (1, 2) -- (2, 4) -- (3, 5); 
%\end{tikzpicture}
%\caption[A weak/strict path]{A weak/strict path from $(1, 1)$ to $(m, n) = (4, 5)$.} \label{Fi2}
%\end{figure}

\section{Notation and conventions} \label{intS5}
Some standard notation that appears in this note are listed below.  
\begin{longtable}{c c}
\caption{A guide to the notation} \label{intTa1}\\
\hline
\textbf{Notation} & \textbf{Definition} \\
\hline
\endfirsthead
\multicolumn{2}{c}%
{\tablename\ \thetable\ -- \textit{Continued from previous page}} \\
\hline
\textbf{Notation} & \textbf{Definition}\\
\hline
\endhead
\hline \multicolumn{2}{c}{\textit{Continued on next page}} \\
\endfoot
\hline
\endlastfoot
$\bbN$& the set of natural numbers $\{1, 2, 3, \dotsc\}$ \\
$\bbZ_+$& the set of nonnegative integers $\{0, 1, 2, \dotsc\}$ \\
$\bbR_+$& the set of nonnegative real numbers \\
$\bbH$& the set of $z \in \bbC$ with $\Im z > 0$.\\
$\ii$& the imaginary unit \\
$[n]$& the set $\{1, \dotsc, n\}$ for $n \in \bbN$ \\
%$\{x\}$& the fractional part $x-\lf x \rf$ of $x \in \bbR$ \\ 
$a \vee b$& the maximum of $a, b \in \bbR$ \\
$a \wedge b$& the minimum of $a, b \in \bbR$ \\
$\hash S$& the number of elements in the set $S$ \\
%$\cl{S}$& the closure of a subset $S \subset \bbC$ \\
$\conj{z}$& the complex conjugate of $z \in \bbC$ \\
$\lf x \rf$& the largest integer less than or equal to $x \in \bbR$ \\
$\lc x \rc$& the least integer greater than or equal to $x \in \bbR$ \\
$|x|$& the Euclidean norm of $x \in \bbR^d$ \\
$\Disc(z, r)$& the (open) disk of radius $r$ centered at $z \in \bbC$ \\
$\Disc'(z, r)$& the punctured disk $\Disc(z, r) \smallsetminus \{z\}$ \\
$\Circle(z, r)$& the circle of radius $r$ centered at $z \in \bbC$ \\
$\Ann(z, r_1, r_2)$& the annulus $\{w \in \bbC: r_1 < |w-z| < r_2\}$ \\
$[z, w]$& the oriented line segment from $z \in \bbC$ to $w \in \bbC$ \\
$\delta_{i, j}$& the Kronecker delta function \\
$\bar{\eta}$& the right endpoint of the support of a Borel measure $\eta$ on $\bbR$ \\
$\ubar{\eta}$ & the left endpoint of the support of a Borel measure $\eta$ on $\bbR$ \\
$x_+$ & the positive part $\max \{x, 0\}$ of $x \in \bbR$.\\
$\sgn(x)$ & the sign of $x \in \bbR$. Equals $0$ if $x = 0$ and $x/|x|$ otherwise. 
\end{longtable}
The \emph{direction} of $z \in \bbC \smallsetminus \{0\}$ is defined as $z/|z|$. Adjectives \emph{increasing} and \emph{decreasing} are used in the strict sense. For convenience, we set $0^0 = 1$, $1/0 = \infty$ and $1/\infty = 0$. 

In several computations, we will benefit from viewing $\bbR^2$ as $\bbC$ via the bijection $(x, y) \mapsto x + \ii y$. Under this identification, we have a \emph{dot product} on $\bbC$ defined by $z \cdot w = \Re \{\conj{z}w\}$ for $z, w \in \bbC$, which corresponds to the usual dot product on $\bbR^2$.

\chapter{Limit Shape} \label{sha}

\section{Introduction} \label{shaS1}

Refer to Section \ref{intS2} for the descriptions of the exponential and geometric model. Let $\tau_k$ denote the shift map $(c_n)_{n \in \bbN} \mapsto (c_{n+k})_{n \in \bbN}$ for $k \in \bbZ_+$. We assume that the joint distribution $\mu$ of $(\bfa, \bfb)$ is totally ergodic with respect to the shifts $\tau_k \times \tau_l$ for $k, l \in \bbN$. This means $\mu$ is separately ergodic under the map $\tau_k \times \tau_l$ for each $k, l \in \bbN$. 

Our results are formally stated in Section \ref{shaS3}. 
We sketch the existence and the basic properties of $g$ in Section \ref{shaS2}. We discuss stationary versions of the exponential and geometric models in Section \ref{shaS4}. We prove (\ref{shaE6}) in Section \ref{shaS5}. 

\section{Results}\label{shaS3}

Let $\E$ denote the expectation under $\mu$ (the distribution of $(\bfa, \bfb)$). Recall that $a = a_1$ and $b = b_1$. It is convenient to break (\ref{shaE6}) into the next two theorems. 

\begin{thm}
\label{shaT1}
Suppose that $\ubar{\alpha} + \ubar{\beta} > 0$ in the exponential model. Then
\begin{equation}\label{shaeq3}
g(s, t) = \inf \limits_{z \in (-\ubar{\alpha}, \ubar{\beta})}\  \left\{s\E\left[\frac{1}{a+z}\right] + t\E\left[\frac{1}{b-z}\right]\right\} \quad \text{ for } s, t > 0. \end{equation}
\end{thm}
Hence, $g$ depends on $(\bfa, \bfb)$ only through the marginal distributions $\alpha$ and $\beta$. Let us write $g^{\alpha, \beta}$ to indicate this. Replacing $z$ with $-z$ in (\ref{shaeq3}) reveals that $g^{\alpha, \beta}(s, t) = g^{\beta, \alpha}(t, s)$ for $s, t > 0$, which is expected due to the symmetric roles of $\bfa$ and $\bfb$ in the model. In particular, if $\alpha$ and $\beta$ are the same then $g(s, t) = g(t, s)$ for $s, t > 0$. Also, (by dominated convergence) the infimum can be taken over $[-\ubar{\alpha}, \ubar{\beta}]$ in (\ref{shaeq3}). When $\ubar{\alpha} = \ubar{\beta} = 0$, this interval degenerates to $\{0\}$ and we expect that $g(s, t) = s\E[1/a]+t\E[1/b]$ for $s, t > 0$. Indeed, this is true.  
\begin{thm}
\label{shaT2}
Suppose that $\ubar{\alpha} = \ubar{\beta} = 0$ in the exponential model. Then 
\begin{equation*}
g(s, t) = s\E\left[\frac{1}{a}\right] + t\E\left[\frac{1}{b}\right] \text{ for } s, t > 0. 
\end{equation*}
\end{thm}

We turn to the concavity and differentiability properties of $g$. In the case $\ubar{\alpha}+\ubar{\beta} > 0$, define the critical values
$c_1 = \dfrac{\E[(b+\ubar{\alpha})^{-2}]}{\E[(a-\ubar{\alpha})^{-2}]}$ and $c_2 = \dfrac{\E[(b-\ubar{\beta})^{-2}]}{\E[(a+\ubar{\beta})^{-2}]}$.  
Note that $0 \le c_1 < c_2 \le \infty$. Also, $c_1 = 0$ if and only if $\E[(a-\ubar{\alpha})^{-2}] = \infty$, and $c_2 = \infty$ if and only if $\E[(b-\ubar{\beta})^{-2}] = \infty$.
\begin{cor}
\label{shac3}
Suppose that $\ubar{\alpha} + \ubar{\beta} > 0$ in the exponential model. Then 
\begin{enumerate}[(a)] 
\item $g(s, t) = s\E[(a-\ubar{\alpha})^{-1}] + t\E[(b+\ubar{\alpha})^{-1}]$ for $s/t \le c_1$. 
\item $g(s, t) = s\E[(a+\ubar{\beta})^{-1}] + t\E[(b-\ubar{\beta})^{-1}]$ for $s/t \ge c_2$. 
\item $g(cs_1+(1-c)s_2, ct_1+(1-c)t_2) > cg(s_1, t_1) + (1-c)g(s_2, t_2)$ for $c \in (0, 1)$ and $s_1, s_2, t_1, t_2 > 0$ such that $c_1 < s_1/t_1, s_2/t_2 < c_2$ and $(s_1, t_1) \neq k (s_2, t_2)$ for any $k \in \bbR$. 
\item $g$ is continuously differentiable.
\end{enumerate}
\end{cor}
By Schwarz inequality, if $c_1 > 0$ then $\E[(a-\ubar{\alpha})^{-1}] < \infty$ and if $c_2 < \infty$ then $\E[(b-\ubar{\beta})^{-1}] < \infty$. Hence, $g$ is finite and linear in $(s, t)$ in the regions $s/t \le c_1$ and $s/t \ge c_2$.  
\begin{proof}[Proof of Corollary \ref{shac3}]
Let $A(z) = \E[(a+z)^{-1}]$ for $z > -\ubar{\alpha}$ and $B(z) = \E[(b-z)^{-1}]$ for $z < \ubar{\beta}$. Using dominated convergence, $A$ and $B$ can be differentiated under the expectation. Thus, 
$A'(z) = -\E[(a+z)^{-2}],\ B'(z) = \E[(b-z)^{-2}],\ A''(z) = 2\E[(a+z)^{-3}],\ B''(z) = 2\E[(b-z)^{-3}]$, etc. Also, define $A, B$ and their derivatives at the endpoints by substituting $-\ubar{\alpha}$ and $\ubar{\beta}$ for $z$ in the preceding formulas. Then, by monotone convergence, the values at the endpoints match the appropriate one-sided limits, that is, $A(-\ubar{\alpha}) = \lim_{z \downarrow -\ubar{\alpha}} A(z) = \E[(a-\ubar{\alpha})^{-1}]$, $B(\ubar{\beta}) = \lim_{z \uparrow \ubar{\beta}} B(z) = \E[(b-\ubar{\beta})^{-1}]$, and similarly for the derivatives. 

Since $A'$ and $B'$ are increasing and continuous on $(-\ubar{\alpha}, \ubar{\beta})$, the derivative $z \mapsto sA'(z) + tB'(z)$ is positive if $s/t \le c_1$, is negative if $s/t \ge c_2$ and has a unique zero if $c_1 < s/t < c_2$. Hence, (a) and (b) follow, and if $c_1 < s/t < c_2$ then $g(s, t) = sA(z) + tB(z)$, where 
$z \in (-\ubar{\alpha}, \ubar{\beta})$ is the unique solution of the equation
\begin{equation}\label{shaeq8}-\frac{B'(z)}{A'(z)} = \frac{s}{t}.\end{equation}
Since $ -B'/A'$ is increasing and continuous, it has an increasing inverse $\zeta$ defined on $(c_1, c_2)$. Let $s_1, t_1, s_2, t_2$ be as in (c). 
%Suppose that $s_1, t_1, s_2, t_2 > 0$ such that $c_1 < s_1/t_1, s_2/t_2 < c_2$ and $(s_1, t_1) \neq k(s_2, t_2)$ for any $k \in \bbR$. 
Then $\zeta(s_1/t_1) \neq \zeta(s_2/t_2)$, which implies the strict inequality  
\begin{equation}
\label{shaeq100}
\begin{aligned}
(s_1+s_2)A(z) + (t_1+t_2)B(z) &> s_1A(\zeta(s_1/t_1)) + t_1B(\zeta(s_1/t_1))\\ 
&+s_2A(\zeta(s_2/t_2)) + t_2B(\zeta(s_2/t_2)) \\
&= g(s_1, t_1) + g(s_2, t_2) \\
\end{aligned}
\end{equation}   
for any $z \in (-\ubar{\alpha}, \ubar{\beta})$. Note that $c_1 < (s_1+s_2)/(t_1+t_2) < c_2$. Setting
$z = \zeta((s_1+s_2)/(t_1+t_2))$ in (\ref{shaeq100}) yields  $g(s_1+s_2, t_1+t_2) > g(s_1, t_1) + g(s_2, t_2)$, and (c) comes from this and homogeneity. Since $-B'/A'$ is continuously differentiable with positive derivative
(as $A'', B', B > 0$ and $A' < 0$ on $(-\ubar{\alpha}, \ubar{\beta})$), by the inverse function theorem, $\zeta$ is continuously differentiable as well. Using (\ref{shaeq8}), we compute the gradient of $g$ for $c_1 < s/t < c_2$ as 
$
\nabla g(s, t) =  (A(\zeta(s/t)), B(\zeta(s/t))), 
$
which tends to $(A(-\ubar{\alpha}), B(-\ubar{\alpha}))$ as $s/t \rightarrow c_1$ and to $(A(\ubar{\beta}), B(\ubar{\beta}))$ as $s/t \rightarrow c_2$. Hence, (d). 
\end{proof}

When $\alpha$ and $\beta$ are uniform distributions, we can compute the infimum in (\ref{shaeq3}) explicitly.  
\begin{cor}
\label{shaC1}
Let $\lambda, l, m > 0$. Suppose that $\alpha$ and $\beta$ are uniform distributions on $[\lambda/2, \lambda/2+l]$ and $[\lambda/2, \lambda/2+m]$, respectively. 
Then, for $s, t > 0$,  
\begin{align*}
g(s, t) &= \frac{s}{l}\log\left(1+\frac{l}{\lambda} + \frac{l}{\lambda} \cdot \frac{lt-ms + \sqrt{(lt-ms)^2+4st(\lambda+l)(\lambda+m)}}{2s(\lambda+m)}\right)\\
&+ \frac{t}{m}\log\left(1+\frac{m}{\lambda} + \frac{m}{\lambda} \cdot \frac{ms-lt + \sqrt{(lt-ms)^2+4st(\lambda+l)(\lambda+m)}}{2t(\lambda+l)}\right). 
\end{align*}
\end{cor}
\begin{proof}
Since $\alpha$ and $\beta$ are uniform distributions,
\begin{align*}
A(z) = \E\left[\frac{1}{a+z}\right] = \frac{1}{l}\log\bigg(1+\frac{l}{z+\lambda/2}\bigg) \quad B(z) = \E\left[\frac{1}{b-z}\right] = \frac{1}{m}\log\bigg(1+\frac{m}{-z+\lambda/2}\bigg)
\end{align*} 
for $z \in (-\lambda/2, \lambda/2)$. We compute the derivatives as 
\begin{align*}
A'(z) = -\frac{1}{(z+\lambda/2)(z+\lambda/2+l)} \quad B'(z) = \frac{1}{(-z+\lambda/2)(-z+\lambda/2+m)}.
\end{align*}
Because $A'(-\lambda/2) = -\infty$ and $B'(\lambda/2) = \infty$, we have $c_1 = 0$ and $c_2 = \infty$. Also, (\ref{shaeq8}) leads to 
\begin{align*}
(s-t)z^2-(s(\lambda+m)+t(\lambda+l))z +s\lambda(\lambda+2m)/4-t\lambda(\lambda+2l)/4 = 0. 
\end{align*}
It follows from the discriminant formula that the solution in the interval $(-\lambda/2, \lambda/2)$ is 
\begin{align*}
z = \frac{\lambda}{2} \frac{s(\lambda+2m)-t(\lambda+2l)}{s(\lambda+m)+t(\lambda+l)+ \sqrt{(sm+tl)^2+4st\lambda(\lambda+m+l})}. 
\end{align*}
Inserting this into $g(s, t) = sA(z)+tB(z)$ and some elementary algebra yield the result.  
\end{proof}
The preceding argument can be repeated when $l = 0$ or $m = 0$. In these cases, $\alpha$ and $\beta$ are understood as point masses at $\lambda/2$. For instance, when $l = 0$ and $m > 0$, we obtain 
\begin{align*}
g(s, t) &= \frac{2s\lambda+ms+\sqrt{(ms)^2+4st\lambda(\lambda+m)}}{2\lambda(\lambda+m)}\\
&+\frac{t}{m}\log\left(1+\frac{m}{\lambda} + \frac{m}{\lambda} \cdot \frac{ms + \sqrt{(ms)^2+4st\lambda(\lambda+m)}}{2t\lambda}\right)
\end{align*}
When $l = 0$ and $m = 0$, we recover (\ref{shaE1}). 

We can also determine $g$ along the diagonal when $\alpha$ and $\beta$ are the same. 
\begin{cor}
Suppose that $\alpha = \beta$. Then $g(s, s) = 2s\E\left[\dfrac{1}{a}\right]$ for $s > 0$. 
\end{cor}
\begin{proof}
We have 
$
(a+z)^{-1}+(a-z)^{-1} \ge 2a^{-1}
$
for $|z| \le \ubar{\alpha}$ with equality if only if $z = 0$. Therefore, 
\begin{align*}g(s, s) = s \inf_{z \in (-\ubar{\alpha}, \ubar{\alpha})} \E\left[\frac{1}{a+z}+\frac{1}{a-z}\right] = 2s \E\left[\frac{1}{a}\right]. \qquad \qedhere
\end{align*}
\end{proof}

We only report the analogous results for the geometric model. 
\begin{thm}
\label{shaT3}
Suppose that $\bar{\alpha}\bar{\beta} < 1$ in the geometric model. Then
\begin{equation*}
g(s, t) = \inf \limits_{z \in (\bar{\alpha}, 1/\bar{\beta})}\  \left\{s\E\left[\frac{a/z}{1-a/z}\right] + t\E\left[\frac{bz}{1-bz}\right]\right\} \quad \text{ for } s, t > 0. \end{equation*}
\end{thm}

\begin{thm}
\label{shaT4}
Suppose that $\bar{\alpha} = \bar{\beta} = 1$ in the geometric model. Then 
\begin{equation*}
g(s, t) = s\E\left[\frac{a}{1-a}\right] + t\E\left[\frac{b}{1-b}\right] \quad \text{ for } s, t > 0. 
\end{equation*}
\end{thm}

\begin{cor}
\label{shaC2}
Let $q \in (0, 1)$ and $0 < l, m < \sqrt{q}$. Choose $\alpha$ and $\beta$ as the distributions with densities proportional to $x \mapsto 1/x$ on the intervals $[\sqrt{q}-l, \sqrt{q}]$ and $[\sqrt{q}-m, \sqrt{q}]$, respectively. 
Then 
\begin{align*}
g(s, t) &= \frac{s}{L}\log\left(1 + \frac{l\sqrt{q}}{1-q}+\frac{l}{1-q}\frac{ly-mx+\sqrt{\Delta}}{2x(1+m\sqrt{q}-q)}\right)\\
&+ \frac{t}{M}\log\left(1 + \frac{m\sqrt{q}}{1-q}+\frac{m}{1-q}\frac{mx-ly+\sqrt{\Delta}}{2y(1+l\sqrt{q}-q)}\right)
\end{align*}
for $s, t > 0$, where $x = slM$, $y = tmL$, $L = \log\left(\dfrac{\sqrt{q}}{\sqrt{q}-l}\right)$, $M = \log\left(\dfrac{\sqrt{q}}{\sqrt{q}-m}\right)$ and 
\begin{align*}
\Delta = (ly-mx)^2+4xy(1+m\sqrt{q}-q)(1+l\sqrt{q}-q)
\end{align*}
\end{cor}
%Can we deduce this heuristically from the exponential model? 

\begin{cor}
Suppose that $\alpha = \beta$. Then $g(s, s) = 2s \E\left[\dfrac{a}{1-a}\right]$ for $s > 0$. 
\end{cor}

\section{The existence of the shape function} \label{shaS2}
\begin{lem}
\label{shaL1}
There exists a deterministic function $g:(0, \infty)^2 \rightarrow [0, \infty]$ such that  
\begin{equation*}
\lim \limits_{n \rightarrow \infty} \frac{G(\lf ns \rf, \lf nt \rf)}{n} = g(s, t) \quad \text{ for } s, t > 0 \quad \bbP\text{-a.s.}
\end{equation*}
Furthermore, $g$ is nondecreasing, homogeneous and concave. 
\end{lem}
Here, nondecreasing means that $g(s', t') \le g(s, t)$ for $0 < s' \le s$ and $0 < t' \le t$, and homogeneity means that $g(cs, ct) = cg(s, t)$ for $s, t, c > 0$. In the exponential model, $g$ is finite if $\ubar{\alpha}+\ubar{\beta} > 0$. This is by the standard properties of the stochastic order \cite[Theorem 1.A3]{Shaked}. Briefly, the i.i.d. measure $P$ on $\bbR_+^{\bbN^2}$ under which each $W(i, j)$ is exponentially distributed with rate $\ubar{\alpha}+\ubar{\beta}$ stochastically dominates $\bfP_{\bfa, \bfb}$ and $G$ is a nondecreasing function of the weights. Thus, $g(s, t)$ does not exceed the right-hand side of (\ref{shaE1}) with $\lambda = \ubar{\alpha}+\ubar{\beta}$. Similarly, $g$ is finite in the geometric model if $\bar{\alpha}\bar{\beta} < 1$. Extend $g$ to $\bbR_+^2$ by setting $g(0, 0) = 0$, $g(s, 0) = \lim_{t \downarrow 0} g(s, t)$ and $g(0, t) = \lim_{s \downarrow 0} g(s, t)$ for $s, t > 0$. 

Lemma \ref{shaL1} can be proved using the ergodicity properties of $\bbP$ and superadditivity of the last-passage times. As this is quite standard, we will leave out many details. For $k, l \in \bbZ_+$, let $\theta_{k, l}: \bbR^{\bbN^2} \rightarrow \bbR^{\bbN^2}$ be given by $\theta_{k, l}(\omega)(i, j) = \omega(i+k, j+l)$ for $i, j \in \bbN$ and $\omega \in \bbR^{\bbN^2}$. Note that $\bbP$ is stationary with respect to $\theta_{k, l}$ because $\bbP(\theta_{k, l}^{-1}(B)) = \E\bfP_{\bfa, \bfb}(\theta_{k, l}^{-1}(B)) = \E\bfP_{\tau_k (\bfa), \tau_l (\bfb)}(B) = \bbP(B)$ for any Borel set $B \subset \bbR_+^{\bbN^2}$.
\begin{lem}
\label{shaL2}
$\bbP$ is ergodic with respect to $\theta_{k, l}$ for any $k, l \in \bbN$. 
\end{lem} 
\begin{proof}
Suppose $\theta_{k, l}^{-1}(B) = B$ for some Borel set $B \subset \bbR_+^{\bbN^2}$. For $n \ge 1$, let $\sT_n$ denote the $\sigma$-algebra generated by $A_n$, the collection of $W(i, j)$ with $i > k(n-1)$ and $j > l(n-1)$. Then $B$ is in $\sT = \bigcap_{n \in \bbN} \sT_n$. Also, $\sT$ is the tail $\sigma$-algebra of the $\sigma$-algebras generated by $A_{n}\smallsetminus A_{n+1}$. Because $\bfP_{\bfa, \bfb}$ is a product measure, by Kolmogorov's $0$--$1$ law,
$\bfP_{\bfa, \bfb}(B) \in \{0, 1\}$. Therefore, 
$\bbP(B) = \mu(\bfP_{\bfa, \bfb}(B) = 1)$. 
On the other hand, 
\[\begin{aligned}
(\tau_k \times \tau_l)^{-1}\{\bfP_{\bfa, \bfb}(B) = 1\} &= \{\bfP_{\tau_k \bfa, \tau_l \bfb}(B) = 1\} = \{\bfP_{\bfa, \bfb}(\theta_{k, l}^{-1}(B)) = 1\} = \{\bfP_{\bfa, \bfb}(B) = 1\}. 
\end{aligned}\]
Since $\mu$ is ergodic under $\tau_k \times \tau_l$, we conclude that $\bbP(B) \in \{0, 1\}$.  
\end{proof}

\begin{proof}[Proof of Lemma \ref{shaL1}]
Fix $s, t \in \bbN$ and define, for integers $0 \le m < n$,  
\begin{align*}Z(m, n) = -G((n-m)s, (n-m)t) \circ \theta_{ms, mt} = \max \limits_{\pi \in \Pi_{ms+1, mt+1, ns, nt}} \sum \limits_{(i, j) \in \pi} W(i, j).\end{align*}
Using the definition and Lemma \ref{shaL2}, we observe that $\{Z(m, n): 0 \le m < n\}$ is a subadditive process that satisfies the hypotheses of Liggett's subadditive ergodic theorem \cite{Liggett}.  
Hence, $Z(0, n)/n = G(ns, nt)/n$ converges $\bbP$-a.s. to a deterministic limit, $g(s, t)$. The existence of the 
limit for all $s, t > 0$ $\bbP$-a.s. and the claimed properties of $g$ follow as in the case of i.i.d. weights  \cite[Theorem~2.1]{Seppalainen09}. 
\end{proof}

\section{Stationary distributions of the last-passage increments}  
\label{shaS4}

%In this section, we consider recursion (\ref{shae2}) with suitably chosen distributions for the boundary values $\{G(i, 0), G(0, j): i, j \in \bbN\}$ such that the distributions of the last-passage increment processes $(G(i, k)-G(i-1, k))_{i \in \bbN}$ and $(G(k, j)-G(k, j-1))_{j \in \bbN}$ are product measures and do not depend on $k \in \bbZ_+$. Using these features, we then compute the corresponding shape functions $g_z$, where $z$ is a convenient real parameter indexing 
%various choices for the boundary values. We also obtain (\ref{shaE8}) and, thereby, connect $g_z$ with $g$.  The idea of using stationary processes to compute limit shapes was introduced in \cite{Seppalainen4}. Here, we adapt from \cite{Seppalainen} which relied on the above scheme to prove (\ref{shaE2}). 

Let us extend the sample space to $\bbR_+^{\bbZ_+^2}$. Now $W(i, j)$ denotes the projection onto coordinate $(i, j)$ for $i, j \in \bbZ_+$. Define the last-passage time $\widehat{G}(i, j)$ through recursion (\ref{shae2}) but with the boundary values 
$\widehat{G}(i, 0) = \sum_{k=1}^i W(k, 0)$ and $\widehat{G}(0, j) = \sum_{k=1}^j W(0, k)$ for $i, j \in \bbN$. 
We then have
\begin{equation}
\label{shaeq42}
\widehat{G}(m, n) = \max \limits_{\pi \in \Pi_{0, 0, m, n}} \sum \limits_{(i, j) \in \pi} W(i, j) \quad \text{ for } m, n \in \bbZ_+. 
\end{equation}

In the exponential model, 
%thus, $\bfa = (a_n)_{n \in \bbN}$ and $\bfb = (b_n)_{n \in \bbN}$ are random sequences in $(0, \infty)$ and $\bfP_{\bfa, \bfb}$ satisfies (\ref{shaE4}). 
for each value of $(\bfa, \bfb)$ such that $a_n \ge \ubar{\alpha}$ and $b_n \ge \ubar{\beta}$ for $n \in \bbN$ (which holds $\mu$-a.s.) and parameter $z \in (-\ubar{\alpha}, \ubar{\beta})$, define $\bfP_{\bfa, \bfb}^z$ as the product measure on $\bbR_+^{\bbZ_+^2}$ by 
\begin{equation}
\label{shaeq78}
\begin{aligned}
&\bfP_{\bfa, \bfb}^z(W(i, j) \ge x) = \exp(-(a_i+b_j)x) \quad &&\bfP_{\bfa, \bfb}^z(W(0, 0) = 0) = 1 \\
&\bfP_{\bfa, \bfb}^z(W(i, 0) \ge x) = \exp(-(a_i+z)x) \quad &&\bfP_{\bfa, \bfb}^z(W(0, j) \ge x) = \exp(-(b_j-z)x)
\end{aligned}
\end{equation}
for $x \ge 0$ and $i, j \in \bbN$. When $\ubar{\alpha} = \ubar{\beta} = 0$, we make definition (\ref{shaeq78}) for $z = 0$. Note that the projection of $\bfP_{\bfa, \bfb}^z$ onto coordinates $\bbN^2$ is $\bfP_{\bfa, \bfb}$. %Also, define the probability measure $\bbP^z$ by $\bbP^z(B) = \E[\bfP_{\bfa, \bfb}^z(B)]$ for any Borel set $B \subset \bbR_+^{\bbZ_+^2}$. 
For the geometric model, the construction is similar. %Now $\bfa$ and $\bfb$ are random sequences in $(0, 1)$ and $\bfP_{\bfa, \bfb}$ satisfies (\ref{shaE5}). 
For $z \in (\bar{\alpha}, 1/\bar{\beta})$ and each value of $(\bfa, \bfb)$ such that $a_n \le \bar{\alpha}$ and $b_n \le \bar{\beta}$ for $n \in \bbN$, the measure $\bfP_{\bfa, \bfb}^z$ is given by  
\begin{equation}
\label{shaeq79}
\begin{aligned}
&\bfP_{\bfa, \bfb}^z(W(i, j) \ge k) = a_i^k b_j^k \quad &&\bfP_{\bfa, \bfb}^z(W(0, 0) = 0) = 1 \\
&\bfP_{\bfa, \bfb}^z(W(i, 0) \ge k) = a_i^k/z^k \quad &&\bfP_{\bfa, \bfb}^z(W(0, j) \ge k) = b_j^kz^k
\end{aligned}
\end{equation}
for $k \in \bbZ_+$ and $i, j \in \bbN$. When $\bar{\alpha} = \bar{\beta} = 1$, definition (\ref{shaeq79}) makes sense for $z = 1$. 
 
Introduce the increment variables as $I(m, n) = \widehat{G}(m, n)-\widehat{G}(m-1, n)$ for $m \ge 1$ and $n \ge 0$, and $J(m, n) = \widehat{G}(m, n)-\widehat{G}(m, n-1)$ for $m \ge 0$ and $n \ge 1$. We capture the stationarity of the increments in the following proposition. 
\begin{prop}
\label{shap5}
Let $k, l \in \bbZ_+$. Under $\bfP_{\bfa, \bfb}^z$,   
\begin{enumerate}[(a)]
\item $I(i, l)$ has the same distribution as $W(i, 0)$ for $i \in \bbN$. 
\item $J(k, j)$ has the same distribution as $W(0, j)$ for $j \in \bbN$. 
\item The random variables $\{I(i, l): i > k\} \cup \{J(k, j): j > l\}$ are (jointly) independent. 
\end{enumerate}
\end{prop}

(\ref{shae2}) leads to the recursion \cite[(2.21)]{Seppalainen09}
\begin{equation}
\label{shaeq81}
\begin{aligned}
I(m, n) &= I(m, n-1) - I(m, n-1) \wedge J(m-1, n) + W(m, n) \\
J(m, n) &= J(m-1, n) - I(n, n-1) \wedge J(m-1, n) + W(m, n) 
\end{aligned}
\end{equation}
for $m, n \in \bbN$. Proposition \ref{shap5} can be proved via induction using (\ref{shaeq81}) and Lemma \ref{shal6} below. We will omit the induction argument as it is the same as in \cite[Theorem~2.4]{Seppalainen09}.  
 \begin{lem}
\label{shal6}
 Let $F: \bbR^3 \rightarrow \bbR^3$ denote the map 
$(x, y, z) \mapsto (x-x \wedge y + z, y - x \wedge y + z, x \wedge y)$. 
Let $P$ be a product measure on $\bbR^3$ with marginals $P_1, P_2, P_3$. Suppose that one of the following holds. 
\begin{enumerate}[(\romannumeral1)]
\item $P_1, P_2$ and $P_3$ are exponential distributions with rates $a, b$ and $a+b$, for some $a, b \in (0, \infty)$. 
\item $P_1, P_2$ and $P_3$ are geometric distributions with parameters $a, b$ and $ab$, for some $a, b \in (0, 1)$. 
\end{enumerate}
Then $P( F^{-1}(B)) = P(B)$ for any Borel set $B \subset \bbR^3$. 
 \end{lem}
In earlier work \cite[Lemma~2.3]{Seppalainen09} and \cite[Lemma~4.1]{Balasz}, Lemma \ref{shal6} was proved by comparing the Laplace transforms of the measures $P$ and $P(F^{-1}(\cdot))$. 
We include another proof below. 
 \begin{proof}[Proof of Lemma \ref{shal6}]
We prove (\romannumeral1) only as the proof of (\romannumeral2) is the discrete version of the same argument and is simpler. Observe that $F$ is a bijection on $\bbR^3$ with $F^{-1} = F$. 
It suffices to verify the claim for any open set $B$ in $\bbR^3$.  By continuity, $F^{-1}(B)$ is also open. Furthermore, $F$ is differentiable on the open set $\{(x, y, z): x > y \text{ or } x < y\}$ and its Jacobian equals $1$ in absolute value. Hence, by the change of variables \cite[Theorem 7.26]{Rudin}, 
\[
\begin{aligned}
P(F^{-1}(B)) &= ab(a+b)\int \limits_{F^{-1}(B)} e^{-ax-by-(a+b)z} dx dy dz \\
&= ab(a+b)\int \limits_{F^{-1}(B)} e^{-a(x-x \wedge y + z)-b(y-x \wedge y + z) -(a+b)(x \wedge y)} dx dy dz \\ 
&= ab(a+b)\int \limits_{B} e^{-au-bv-(a+b)w} du dv dw = P(B). \qedhere
\end{aligned}
\]

%The proof of (\romannumeral2) is a simpler, discrete version of the preceding argument.  Now, it suffices to verify (\ref{shaeq80})
%with $B \subset \bbZ_+^3$.  Note that $F^{-1}(B) = F(B) \subset \bbZ_+^3$ as well. Hence,  
%\[
% \begin{aligned}
%P(F^{-1}(B)) &= (1-a)(1-b)(1-ab) \sum \limits_{(i, j, k) \in F^{-1}(B)}  a^i b^j (ab)^k \\
%&= (1-a)(1-b)(1-ab) \sum \limits_{(i, j, k) \in F^{-1}(B)}  a^{i-i \wedge j+k} b^{j-i \wedge j + k} (ab)^{i \wedge j}  \\ 
%&=  (1-a)(1-b)(1-ab) \sum \limits_{(u, v, w) \in B} a^u b^v (ab)^w  \\
%&= P(B). \qedhere\\
%\end{aligned}
%\]
 \end{proof}

In the exponential and geometric models, respectively, define 
\begin{align*}g_z(s, t) &= s\E\left[\frac{1}{a+z}\right]+t\E\left[\frac{1}{b-z}\right] \quad \text{ for } s, t \ge 0 \text{ and } z \in [-\ubar{\alpha}, \ubar{\beta}]\\
g_z(s, t) &= s\E\left[\frac{a/z}{1-a/z}\right]+t\E\left[\frac{bz}{1-bz}\right] \quad \text{ for } s, t \ge 0 \text{ and } z \in [\bar{\alpha}, 1/\bar{\beta}]. 
\end{align*} 

\begin{lem}
\label{shaL4}
In the exponential model, let $z \in (-\ubar{\alpha}, \ubar{\beta})$ if $\ubar{\alpha}+\ubar{\beta} > 0$, and let $z = 0$ and assume that $\E[1/a+1/b] < \infty$ if $\ubar{\alpha} = \ubar{\beta} = 0$. In the geometric model, let $z \in (\bar{\alpha}, 1/\bar{\beta})$ if $\bar{\alpha}\bar{\beta} < 1$, and let $z = 1$ and assume that $\E[a/(1-a)+b/(1-b)] < \infty$ if $\bar{\alpha} = \bar{\beta} = 1$. Then 
\begin{align}\lim \limits_{n \rightarrow \infty} \frac{\widehat{G}(\lf ns \rf, \lf nt \rf)}{n} = g_z(s, t) \quad \text{ for } s, t \ge 0 \text{ in } \bfP_{\bfa, \bfb}^z\text{-probability for } \mu\text{-a.e. } (\bfa, \bfb). \label{shaE11}
\end{align} 
\end{lem}
In fact, the convergence in (\ref{shaE11}) is $\bfP_{\bfa, \bfb}^z$-a.s. for $\mu$-a.e $(\bfa, \bfb)$ provided that $\ubar{\alpha}+\ubar{\beta} > 0$ in the exponential model and $\bar{\alpha}\bar{\beta} < 1$ in the geometric model \cite[Theorem~4.3]{Emrah}. 
%This is claimed in the earlier version of this paper \cite[Theorem~4.3]{Emrah} erroneously for all $\alpha$ and $\beta$; however, the proof there does not work if $\ubar{\alpha} = \ubar{\beta} = 0$ and $\bar{\alpha} = \bar{\beta} = 1$ in the exponential and geometric models, respectively. This leads to a gap in the proof of \cite[Theorem~2.7]{Emrah}. We fix the gap using the weaker convergence in Lemma (\ref{shaL4}), which will also be sufficient for other purposes in the paper. 
By (\ref{shaeq51}), (\ref{shaeq42}) and nonnegativity of weights, $G(m, n) \le \widehat{G}(m, n)$ for $m, n \in \bbN$. Then Lemma \ref{shaL4} implies that $g(s, t) \le g_z(s, t)$ for any $s, t \ge 0$. The main result of this paper is that $g(s, t) = \inf_z g_z(s, t)$. 

\begin{proof}[Proof of Lemma \ref{shaL4}]
We will consider the exponential model only, the geometric model is treated similarly.   
Note that $\widehat{G}(\lf ns \rf, \lf nt \rf) = \sum_{i=1}^{\lf ns \rf} I(i, 0) + \sum_{j = 1}^{\lf nt \rf} J(\lf ns \rf, j)$ for $s, t \ge 0$ and $n \in \bbN$. By Proposition \ref{shap5}, $\{J(\lf ns \rf, j): j \in \bbN\}$ has the same distribution as $\{J(0, j): j \in \bbN\}$ under $\bfP_{\bfa, \bfb}^z$. Hence, it suffices to show that 
\begin{align*}
\lim_{n \rightarrow \infty} \frac{1}{n} \sum \limits_{i=1}^{n} I(i, 0) = \E\left[\frac{1}{a+z}\right] \text{ and }  \lim_{n \rightarrow \infty} \frac{1}{n} \sum \limits_{j = 1}^{n} J(0, j) = \E\left[\frac{1}{b-z}\right] %\label{shaee12}
\end{align*}
in $\bfP_{\bfa, \bfb}^z$ for $\mu$-a.s. We will only derive the first limit above, for which we will show that, for $z > -\ubar{\alpha}$ and for $z = -\ubar{\alpha}$ when $\E[(a-\ubar{\alpha})^{-1}] < \infty$,  
\begin{align}
\lim_{n \rightarrow \infty} \frac{1}{n} \sum \limits_{i=1}^{n} I(i, 0) = \E\left[\frac{1}{a+z}\right] \quad \text{ in } \bfQ_{\bfa}^z \ \mu\text{-a.s.}, \label{shaE16}
\end{align}
where $\bfQ_{\bfa}^z$ is the product measure on the coordinates $\bbN \times \{0\}$ given by $\bfQ_{\bfa}^z(W(i, 0) \ge x) = e^{-(a_i+z)x}$ for $i \in \bbN$ and $x \ge 0$.
It suffices to prove the convergence in distribution under $\bfQ_{\bfa}^z$ $\mu$-a.s. because the limit is deterministic.  

The characteristic function of
$n^{-1}\sum_{i=1}^n I(i, 0)$ under $\bfQ_{\bfa}^z$ is given by 
\begin{align*}\prod \limits_{i=1}^n \left(1-\frac{\bfi x}{n(a_i+z)}\right)^{-1} = \exp\left(-\sum \limits_{i=1}^n \log\left(1-\frac{\bfi x}{n(a_i+z)}\right)\right) \quad \text{ for } x \in \bbR, \end{align*}
where the complex logarithm denotes the principal branch. Hence, (\ref{shaE16}) follows if we prove   
\begin{align*}
\lim_{n \rightarrow \infty} -\sum \limits_{i=1}^n \log\left(1-\frac{\bfi x}{n(a_i+z)}\right) = \bfi x \E\left[\frac{1}{a+z}\right]\quad \text{ for } x \in \bbR \quad \mu\text{-a.s.,}
\end{align*}
Using the bound $|\log(1+\bfi x)| \le |x|$ for $x \in \bbR$ 
%\begin{align*}|\log(1+\bfi x)| \le \left|\int_{1}^{1+\bfi x} \frac{dw}{w}\right| = \left|\int_0^x \frac{\bfi du}{1+\bfi u}\right| \le |x| \quad \text{ for } x \in \bbR, \end{align*} 
%the triangle inequality 
and the ergodicity of $\bfa$, we obtain 
\begin{align*}\limsup_{n \rightarrow \infty} \left|\sum \limits_{i=1}^n \log\left(1-\frac{\bfi x}{n(a_i+z)}\right)\right| \le \lim_{n \rightarrow \infty} \frac{|x|}{n} \sum \limits_{i=1}^n \frac{1}{a_i+z} = |x| \E\left[\frac{1}{a+z}\right] \quad \text{ for } x \in \bbR \quad \mu\text{-a.s.}\end{align*} 
Therefore, it suffices to prove the following for $x \in \bbR$ $\mu$-a.s. 
\begin{align}\lim_{n \rightarrow \infty}-\sum \limits_{i=1}^n \arg\left(1-\frac{\bfi x}{n(a_i+z)}\right) - x\E\left[\frac{1}{a+z}\right] = \lim_{n \rightarrow \infty} \sum \limits_{i=1}^n \arctan\left(\frac{x}{n(a_i+z)}\right) - \frac{x}{n (a_i+z)} = 0.\label{shaE10}\end{align}
Since $\arctan x = \int_0^x (1+u^2)^{-1}du$, we can rewrite the second sum above as 
\begin{align*}
\sum \limits_{i=1}^n \int \limits_0^{xn^{-1}(a_i+z)^{-1}} \frac{du}{1+u^2}-\frac{x}{n (a_i+z)} &= -\sum \limits_{i=1}^n \int \limits_0^{xn^{-1}(a_i+z)^{-1}} \frac{u^2\ du}{1+u^2} 
= -\frac{x}{n} \sum \limits_{i=1}^n \int \limits_0^{(a_i+z)^{-1}} \frac{x^2v^2\ dv}{n^2+x^2v^2}, 
\end{align*}
where we changed the variables via $u = vx/n$. Pick $M > 0$. The limsup as $n \rightarrow \infty$ of the absolute value of the last sum is bounded $\mu$-a.s. by $|x|$ times  
\begin{align*}
\lim_{n \rightarrow \infty} \frac{1}{n} \sum \limits_{i=1}^n \int \limits_{0}^{(a_i+z)^{-1}} \frac{x^2v^2\ dv}{M^2+x^2v^2} =  \E\left[\int \limits_{0}^{(a+z)^{-1}} \frac{x^2 v^2 \ dv}{M^2+x^2v^2} \right],  
\end{align*}
where the a.s. convergence is due to the ergodicity of $\bfa$ and the integrability of 
\begin{align*}\int_0^{(a+z)^{-1}} \dfrac{x^2v^2 \ dv}{M^2+x^2v^2} \le \frac{1}{a+z}.\end{align*} 
The last integral is monotone in $x^2$ and vanishes as $M \rightarrow \infty$. Hence, (\ref{shaE10}) holds for $x \in \bbR$ $\mu$-a.s.
\end{proof}

The next proposition relates $g_z$ to $g$ through a variational formula. 
\begin{prop}
\begin{equation}\label{shaE8.1}g_z(1, 1) = \sup_{t \in [0, 1]} \max \{g_z(1-t, 0) + g(t, 1), g_z(0, 1-t) + g(1, t)\} \end{equation} For $z \in (-\ubar{\alpha}, \ubar{\beta})$ in the exponential model and for $z \in (\bar{\alpha}, 1/\bar{\beta})$ in the geometric model. 
\end{prop}
\begin{proof}
Fix $z \in (-\ubar{\alpha}, \ubar{\beta})$ in the exponential model. Since $g \le g_z$ and $g_z$ is linear, (\ref{shaE8.1}) with $\ge$ instead of $=$ is immediate. For the opposite inequality, we adapt the argument in \cite[Proposition~2.7]{Seppalainen09}. It follows from (\ref{shaeq51}) and (\ref{shaeq42}) that  
\begin{equation}
\label{shaeq23}
\widehat{G}(n, n) = \max \limits_{k \in [n]} \max \{\widehat{G}(k, 0) + G(n-k+1, n) \circ \theta_{k-1, 0}, \widehat{G}(0, k) + G(n, n-k+1) \circ \theta_{0, k-1}\}. 
\end{equation}
Let $L \in \bbN$ and consider $n > L$ large enough so that $\lc (i+1)n/L\rc > \lc in/L\rc$ for $0 \le i < L$. 
For any $k \in [n]$ there exists some $0 \le i < L$ such that $\lc in/L \rc < k \le \lc (i+1)n/L\rc$, and the weights are nonnegative. Therefore, (\ref{shaeq23}) implies that 
\begin{equation}
\label{shaeq24}
\begin{aligned}
\widehat{G}(n, n) \le \max \limits_{0 \le i < L} \max \{&\widehat{G}(\lc (i+1)n/L\rc, 0) + G(\lf (1-i/L)n \rf, n) \circ \theta_{\lc in/L \rc, 0},\\
&\widehat{G}(0, \lc (i+1)n/L\rc) + G(n, \lf (1-i/L)n \rf) \circ \theta_{0, \lc in/L \rc} \}.
\end{aligned}
\end{equation} 
By stationarity of $\bbP$, we have the following limits in $\bbP$-probability. 
\begin{equation}
\label{shaE15}
\begin{aligned}
\lim_{n \rightarrow \infty} \frac{G(\lf (1-i/L)n \rf, n) \circ \theta_{\lc in/L \rc, 0}}{n} &= g\left(1-i/L, 1\right) \\ \lim_{n \rightarrow \infty} \frac{G(n, \lf (1-i/L)n \rf) \circ \theta_{0, \lc in/L \rc}}{n} &= g(1, 1-i/L)
\end{aligned}
\end{equation}
Hence, these limits are $\bbP$-a.s. and, consequently, $\bfP_{\bfa, \bfb}$ a.s. $\mu$-a.s. if $n \rightarrow \infty$ along a suitable sequence $(n_k)_{k \in \bbN}$. Also, by Lemma \ref{shaL4}, there is a subsequence $(n'_k)_{k \in \bbN}$ in $\bbN$ $\mu$-a.s. such that $\bfP_{\bfa, \bfb}^z$ a.s. 
\begin{equation}
\label{shaE14}
\begin{aligned}
\lim_{k \rightarrow \infty} \frac{\widehat{G}(\lc (i+1)n'_k/L\rc, 0)}{n'_k} = g_z(i+1/L, 0) \quad \lim_{k \rightarrow \infty} \frac{\widehat{G}(0, \lc (i+1)n'_k/L\rc)}{n'_k} = g_z(0, i+1/L)
\end{aligned}
\end{equation}
Because $\bfP_{\bfa, \bfb}$ is a projection of $\bfP_{\bfa, \bfb}^z$, we can choose $(\bfa, \bfb)$ such that (\ref{shaE15}) and (\ref{shaE14}) hold $\bfP_{\bfa, \bfb}^z$-a.s. Hence, we obtain from (\ref{shaeq24}) that 
\[
\begin{aligned}
g_z(1, 1) &\le \max \limits_{0 \le i < L}\max\{g_z((i+1)/L, 0) + g(1-i/L, 1), \ g_z(0, (i+1)/L) + g(1, 1-i/L)\} \\
&\le \sup \limits_{0 \le t \le 1} \max \{g(t, 1) + g_z(1-t, 0), \ g(1, t) + g_z(0, 1-t)\} + \frac{\E[(a+z)^{-1}]+\E[(b-z)^{-1}]}{L}.
\end{aligned}
\]
Finally, let $L \rightarrow \infty$. The geometric model is treated similarly. 
\end{proof}

\section{Variational characterization of the shape function}
\label{shaS5}

We now prove Theorems \ref{shaT1} and \ref{shaT2}. The assumption $\ubar{\alpha}+\ubar{\beta} > 0$ is in force until the proof of Theorem \ref{shaT2}. We begin with computing $g$ on the boundary. Recall that $g$ is extended to the boundary of $\bbR_+^2$ through limits. By homogeneity, it suffices to determine $g(1, 0)$ and $g(0, 1)$. 
\begin{lem}
\label{shal7}
\[
\begin{aligned}
g(1, 0) = \E\left[\frac{1}{a+\ubar{\beta}}\right] \qquad g(0, 1) = \E\left[\frac{1}{b+\ubar{\alpha}}\right]. \\
\end{aligned}
\]
\end{lem}
\begin{proof}
We have $g(1, 0) \le g_z(1, 0) = \E[(a+z)^{-1}]$ for all $z \in (-\ubar{\alpha}, \ubar{\beta})$. Letting $z \uparrow \ubar{\beta}$ yields the upper bound $g(1, 0) \le \E[(a+\ubar{\beta})^{-1}]$. Now the lower bound. Let $\epsilon > 0$. By Lemma \ref{shaL1}, (\ref{shaE16}) and since $\mu(b_1 \le \ubar{\beta}+\epsilon) > 0$, there exists $(\bfa, \bfb)$ such that $b_1 \le \ubar{\beta}+\epsilon$ and 
\begin{align}
\lim_{n \rightarrow \infty} \frac{G(n, \lf n\epsilon \rf)}{n} &= g(1, \epsilon) \quad \bfP_{\bfa, \bfb}\text{-a.s.}\label{shaE17}\\
\lim_{n \rightarrow \infty} \frac{1}{n}\sum_{i=1}^{n} I(i, 0) &= \E\left[\frac{1}{a+\ubar{\beta}+\epsilon}\right]\quad \text{ in } \bfQ_{\bfa}^{\ubar{\beta}+\epsilon}\text{-probability.} \label{shaE18}
\end{align}
($\bfQ_{\bfa}^z$ is defined immediately after (\ref{shaE16})). 
The distribution of $\{W(i, 1): 1 \le i \le n\}$ under $\bfP_{\bfa, \bfb}$ stochastically dominates the distribution of $\{I(i, 0): 1 \le i \le n\}$ under $\bfQ_{\bfa}^{\ubar{\beta}+\epsilon}$ as these distributions have product forms and $i$th marginals are exponentials with rates $a_i+b_1 \le a_i+\ubar{\beta}+\epsilon$ for $i \in [n]$. Therefore, for $x \in \bbR$ and $n \ge 1/\epsilon$,  
\begin{align*}
\bfP_{\bfa, \bfb}(G(n, \lf n\epsilon \rf) \ge nx) \ge \bfP_{\bfa, \bfb}\left(\sum_{i=1}^n W(i, 1) \ge nx\right) \ge \bfQ_{\bfa}^{\ubar{\beta}+\epsilon}\left(\sum_{i=1}^n I(i, 0) \ge nx \right). 
\end{align*}
Set $x = \E[(a+\ubar{\beta}+\epsilon)^{-1}]-\epsilon$ and let $n \rightarrow \infty$. By (\ref{shaE17}) and (\ref{shaE18}), we obtain $g(1, \epsilon) \ge x$. Sending $\epsilon \downarrow 0$ gives $g(1, 0) \ge \E[(a+\ubar{\beta})^{-1}]$. Computation of $g(0, 1)$ is similar.  
\end{proof}

We now extract $g$ from (\ref{shaE8.1}). For this, we will only use the boundary values of $g$ provided in Lemma \ref{shal7}, and that $A(z) = \E[(a+z)^{-1}]$ and $B(z) = \E[(b-z)^{-1}]$ are continuous, stricly monotone functions on $(-\ubar{\alpha}, \ubar{\beta})$. 
\begin{lem} 
\label{shal8}
Let $r$ be a positive, continuous function on $[0, \pi/2]$. For $z \in (-\ubar{\alpha}, \ubar{\beta})$,  
\[\sup_{0 \le \theta \le \pi/2} \{g(x(\theta), y(\theta))-g_z(x(\theta), y(\theta))\} = 0, \]
where $(x(\theta), y(\theta)) = (r(\theta)  \cos \theta,  r(\theta) \sin \theta)$ for  $0 \le \theta \le \pi/2$. 
\end{lem}
\begin{proof}
We can rewrite (\ref{shaE8.1}) as 
\begin{align}
A(z) + B(z) &= \sup \limits_{\pi/4 \le \theta \le \pi/2} \{(1-\cot \theta) A(z) + g(\cot \theta, 1)\} \vee \sup \limits_{0 \le \theta \le \pi/4} \{(1-\tan \theta)B(z) + g(1, \tan \theta)\} \nonumber\\
&= \sup \limits_{\pi/4 \le \theta \le \pi/2} \left\{\left(1-\frac{x(\theta)}{y(\theta)}\right) A(z) + g\left(\frac{x(\theta)}{y(\theta)}, 1\right) \right\} \nonumber\\ 
&\ \vee \sup \limits_{0 \le \theta \le \pi/4} \left\{\left(1-\frac{y(\theta)}{x(\theta)}\right)B(z) + g\left(1, \frac{y(\theta)}{x(\theta)}\right)\right\}, \nonumber
\end{align}
where we use that $x$ and $y$ are nonzero, respectively, on the intervals $[0, \pi/4]$ and $[\pi/4, \pi/2]$. Collecting the terms on the right-hand side and using homogeneity, we obtain that  
\begin{equation}
\label{shaeq26}
\begin{aligned}
0 = \max \bigg \{&\sup \limits_{\pi/4 \le \theta \le \pi/2} \frac{1}{y(\theta)}\{-x(\theta)A(z) -y(\theta) B(z)+ g(x(\theta), y(\theta))\},\\ &\sup \limits_{0 \le \theta \le \pi/4} \frac{1}{x(\theta)}\{-x(\theta)A(z) -y(\theta) B(z)+ g(x(\theta), y(\theta))\} \bigg \}.
\end{aligned}
\end{equation}
The expressions inside the supremums in (\ref{shaeq26}) are continuous functions of $\theta$ over closed intervals. Hence, there exists $\theta_z \in [0, \pi/2]$ such that 
\begin{align*}
0 = -x(\theta_z)A(z) - y(\theta_z) B(z) + g(x(\theta_z), y(\theta_z)) = \sup \limits_{0 \le \theta \le \pi/2} \{-x(\theta)A(z) -y(\theta) B(z)+ g(x(\theta), y(\theta))\}, 
\end{align*}
where the second equality is due to $g \le g_z$.  
\end{proof}

\begin{cor} 
\label{shac2}
\begin{equation}
\label{shaeq89}
B(z) = \sup \limits_{0 \le s < \infty} \{-sA(z) + g(s, 1)\} \quad \text{ for } z \in (-\ubar{\alpha}, \ubar{\beta}).  
\end{equation}
\end{cor}
\begin{proof}
Let $S > 0$. The set $\{(s, 1): 0 \le s \le S\} \cup \{(S, t): 0 \le t \le 1\}$ is the image of a curve $\theta \mapsto (r(\theta)\cos\theta, r(\theta)\sin\theta)$ for $[0, \pi/2]$ with continuous and positive $r$. Hence, by Lemma \ref{shal8},  
\begin{equation}
\label{shaeq88}
0 = \max \{ \sup \limits_{0 \le s \le S} \{g(s, 1) - g_z(s, 1)\}, \sup \limits_{0 \le t \le 1} \{g(S, t) - g_z(S, t)\}\}.
\end{equation}
Using homogeneity and Lemma \ref{shal7}, we observe that  
\begin{equation}
\label{sha90}
\begin{aligned}
g(S, t)-g_z(S, t) &= g(S, t) - SA(z)-tB(z) \le S(g(1, 1/S) - A(z)) \rightarrow -\infty \quad \text{ as } S \rightarrow \infty.
\end{aligned}
\end{equation}
Hence, the second supremum in (\ref{shaeq88}) can be dropped provided that $S$ is sufficiently large, which results in 
$
0 = \sup_{0 \le s \le S} \{g(s, 1) - g_z(s, 1)\}.
$
This equality remains valid if $S$ is replaced with $\infty$ by (\ref{sha90}) with $t = 1$. Rearranging terms gives (\ref{shaeq89}). 
\end{proof}

\begin{proof}[Proof of Theorem \ref{shaT1}]
Define the function $\gamma:\bbR \rightarrow \bbR \cup \{\infty\}$ by $\gamma(s) = -g(s, 1)$ for $s \ge 0$ and $\gamma(s) = \infty$ for $s < 0$. By Proposition \ref{shaL1}, $\gamma$ is nonincreasing, continuous and convex on $[0, \infty)$ and completely determines $g$. Let $\gamma^*$ denote the convex conjugate of $\gamma$, that is,  
\begin{align}
\gamma^*(x) &= \sup \limits_{s \in \bbR} \{sx-\gamma(s)\} = \sup \limits_{s \ge 0} \{sx - \gamma(s)\} \quad \text{ for } x \in \bbR. \label{shaeq30}
\end{align}
Let $f$ be the function whose graph is the image of the curve $z \mapsto (-A(z), B(z))$. That is, $f$ is defined on the interval 
$(-A(-\ubar\alpha), -A(\ubar{\beta}))$ and is given by the formula $f(x) = B \circ A^{-1}(-x)$. %properties of f. 
By Corollary \ref{shac2}, 
\begin{equation}\label{shaeq29} f(x) = \sup \limits_{0 \le s < \infty} \{sx-\gamma(s)\} \quad \text{ for } x \in (-A(-\ubar{\alpha}), -A(\ubar{\beta}))\end{equation} Comparison of (\ref{shaeq30}) and (\ref{shaeq29}) shows that $\gamma^*$ coincides with $f$ on $(-A(-\ubar{\alpha}), -A(\ubar{\beta}))$. 
Since $\gamma$ is a lower semi-continuous, proper convex function on the real line, by the Fenchel-Moreau theorem, $\gamma$ equals the convex conjugate of $\gamma^*$, hence, 
\begin{equation}
\label{shaeq31}
\gamma(s) = \sup \limits_{x \in \bbR} \{sx - \gamma^*(x)\} \quad \text{ for } s \in \bbR
\end{equation} 

To prove the result, we need to show the supremum in (\ref{shaeq31}) can be taken over the interval $(-A(-\ubar{\alpha}), -A(\ubar{\beta}))$ instead of the real line. It is clear from (\ref{shaeq30}) that $\gamma^*$ is nondecreasing and is bounded below by $-\gamma(0) = g(0, 1) = B(-\ubar{\alpha})$. Since $\gamma^*$ agrees with $f$ on  $(-A(-\ubar{\alpha}), -A(\ubar{\beta}))$, 
\begin{equation}
\label{shaeq32}
\begin{aligned}
B(-\ubar{\alpha}) &\le \gamma^*(-A(-\ubar{\alpha})) \le \lim \limits_{x \downarrow -A(-\ubar{\alpha})} f(x) \\ 
&= \lim \limits_{x \downarrow -A(-\ubar{\alpha})} B \circ A^{-1}(-x) = \lim \limits_{z \rightarrow -\ubar{\alpha}} B(z)  = B(-\ubar{\alpha}), 
\end{aligned}
\end{equation}
where we used continuity of $A^{-1}$ and $B$. Hence, $\gamma^*(x) = B(-\ubar{\alpha})$ for $x \le -A(-\ubar{\alpha})$. On the other hand, if $x > -A(\ubar{\beta}) = -g(1, 0)$ then 
$\gamma^*(x) = \infty$ by (\ref{shaeq30}) because 
\[
\lim \limits_{s \rightarrow \infty} sx - \gamma(s) = \lim \limits_{s \rightarrow \infty} s(x + g(1, 1/s)) = \infty. 
\]
Finally, we compute $\gamma^*$ at $-A(\ubar{\beta})$. Being a convex conjugate, %lemma 
$\gamma^*$ is lower semi-continuous. Since $\gamma^*$ is also nondecreasing, 
$\lim_{y \uparrow x} \gamma^*(y) = \gamma^*(x)$ for any $x \in \bbR$. Then, proceeding as in (\ref{shaeq32}), 
\[\gamma^*(-A(\ubar{\beta})) = \lim \limits_{x \uparrow -A(\ubar{\beta})}f(x) = B(\ubar{\beta}).\] 
We conclude that the function $x \mapsto sx - \gamma^*(x)$ is increasing for $x \le -A(-\ubar{\alpha})$ and is $-\infty$ for $x > -A(\ubar{\beta})$. Moreover, the left- and right-hand limits agree with the value of the function at $-A(\ubar{\beta})$ and $-A(-\ubar{\alpha})$, respectively. Hence, by (\ref{shaeq31}), 
\[
\begin{aligned}
\gamma(s) &= \sup \limits_{s \in (-A(-\ubar{\alpha}), -A(\ubar{\beta}))} \{sx - \gamma^*(x)\} = \sup \limits_{z \in (-\ubar{\alpha}, \ubar{\beta})} \{-sA(z)-B(z)\} = -\inf \limits_{z \in (-\ubar{\alpha}, \ubar{\beta})} \{sA(z) + B(z)\}, 
\end{aligned}
\]
which implies (\ref{shaeq3}). 
\end{proof}

\begin{proof}[Proof of Theorem \ref{shaT2}]
Introduce $\delta > 0$ and let $\varphi: \bbR_+^{\bbN} \rightarrow \bbR_+^{\bbN}$ denote the map $(c_n)_{n \in \bbN} \mapsto (c_n \vee \delta)_{n \in \bbN}$. Because $\varphi$ commutes with the shift $\tau_1$, $\varphi(\bfa)$ and $\varphi(\bfb)$ are stationary sequences in $(0, \infty)$. Moreover, for each $k, l \in \bbN$, the distribution $\mu_\delta$ of $(\varphi(\bfa), \varphi(\bfb))$ is ergodic with respect to $\tau_k \times \tau_l$. To see this, suppose that $B = (\tau_k \times \tau_l)^{-1}(B)$ for some $k, l \in \bbN$ and Borel set $B \subset \bbR_+^{\bbN} \times \bbR_+^{\bbN}$. Then $(\varphi \times \varphi)^{-1}(B) = (\varphi \times \varphi)^{-1}((\tau_k \times \tau_l)^{-1}(B)) = (\tau_k \times \tau_l)^{-1}((\varphi \times \varphi)^{-1}(B))$. Hence, by the ergodicity of $\mu$, we get $\mu_\delta(B) = \mu((\varphi \times \varphi)^{-1}(B)) \in \{0, 1\}$. 
 
Let $\alpha_\delta$ and $\beta_\delta$ denote the marginal distributions of $\varphi(\bfa)$ and $\varphi(\bfb)$, respectively. Then $\ubar{{\alpha_\delta}} = \ubar{{\beta_\delta}} = \delta$. 
Applying Theorem \ref{shaT1} gives
\begin{align}g^{\alpha_\delta, \beta_\delta}(s, t) = \inf \limits_{z \in (-\delta, \delta)} \left\{s \E\left[\frac{1}{a \vee \delta + z}\right] + t \E\left[\frac{1}{b \vee \delta-z}\right]\right\}.\label{shaE9}\end{align}
Since $\bfP_{\bfa, \bfb}$ stochastically dominates $\bfP_{\varphi(\bfa), \varphi(\bfb)}$, we have $g^{\alpha_\delta, \beta_\delta}(s, t) \le g^{\alpha, \beta}(s, t)$ for $s, t \ge 0$. Using this and (\ref{shaE9}), we obtain 
\begin{align*}g^{\alpha, \beta}(s, t) \ge \inf \limits_{z \in (-\delta, \delta)} \left\{s \E\left[\frac{1}{a \vee \delta' + z}\right] + t \E\left[\frac{1}{b \vee \delta'-z}\right]\right\}, \end{align*}
where we fix $\delta' > \delta$. Because the expression inside the infimum is continuous in $z$, letting $\delta \downarrow 0$ yields $g^{\alpha, \beta}(s, t) \ge s\E[(a \vee \delta')^{-1}] + t\E[(b \vee \delta')^{-1}]$ for $s, t \ge 0$. 
Then, by monotone convergence, letting $\delta' \rightarrow 0$ results in
\begin{equation*}
g^{\alpha, \beta}(s, t) \ge s\E\left[\frac{1}{a}\right] + t\E\left[\frac{1}{b}\right] \quad \text{ for } s,t \ge 0. 
\end{equation*} 
The opposite inequality is noted after Lemma \ref{shaL4}.  
\end{proof}

\chapter{One-point Distribution of the Last-passage Time} \label{Ch3}

\section{Introduction}

In this expository chapter, we derive a Fredholm determinant representation for the probabilities $\bfP_{\bfa, \bfb}(G(m, n) \le k)$ based on the discussion in \cite{Johansson2};  
we refer the reader to \cite{Berestycki}, \cite{BorodinGorin}, \cite{Johansson3} for more detailed accounts. For a precise statement, we introduce some notation. Define 
\begin{align}
\label{flucE165}
F_{m, n, x}^{\bfa, \bfb}(z) &= \frac{\prod_{j = 1}^n (1-zb_j)}{\prod_{i = 1}^m (z-a_i)} \cdot z^{m+x} \quad \text{ for } m, n \in \bbN,\ x \in \bbZ_+ \text{ and } z \in \bbC \smallsetminus \{a_1, \dotsc, a_m\}   
\end{align} 
and the contour integral 
\begin{align}
\label{flucE3}
I_{m, n, x}^{\bfa, \bfb} = \frac{1}{2\pi\ii} \oint \limits_{|z| = 1} F_{m, n, x}^{\bfa, \bfb}(z) \dd z \quad \text{ for } m, n \in \bbN \text{ and } x \in \bbZ_+, 
\end{align}
where the circle of integration is oriented counter-clockwise. Define the kernel as 
\begin{align}
\label{flucE71}
\KK_{m, n}^{\bfa, \bfb}(x, y) = \sum \limits_{l=0}^\infty I_{m, n, x+l}^{\bfa, \bfb} I_{n, m, y+l}^{\bfb, \bfa} \quad \text{ for } m, n \in \bbN \text{ and } x, y \in \bbZ_+. 
\end{align}
The convergence of the series above will be verified at (\ref{flucEq5.1}). 
\begin{thm}
\label{flucT7}
Let $m, n \in \bbN$. Then, for $k \in \bbZ_+$,  
\begin{align}
\bfP_{\bfa, \bfb}(G(m, n) \le k) = 1 + \sum \limits_{l=1}^{n} \frac{(-1)^l}{l!} \sum \limits_{\substack{x_1, \dotsc, x_l \in \bbZ_+ \\ x_i \ge k}}\det\limits_{i, j \in [l]}[\KK_{m, n}^{\bfa, \bfb}(x_i, x_j)]. \label{flucE105}
\end{align}
\end{thm}

\section{One-point distribution of the last-passage time}\label{flucSe5}

One of the main tools utilized in the argument is the Robinson-Schensted-Knuth (RSK) correspondence. To state it, some definitions are in order. A \emph{weak composition} $\alpha = (\alpha_i)_{i \in \bbN}$ is a sequence in $\bbZ_+$ with finitely many nonzero terms. Define the \emph{length} and the \emph{size} of $\alpha$ as 
\begin{align*}l(\alpha) = \max \{i \in \bbN: \alpha_i > 0\} \qquad \text{ and } \qquad
|\alpha| = \sum_i \alpha_i, 
\end{align*} respectively. A \emph{partition} $\lambda = (\lambda_i)_{i \in \bbN}$ is a nonincreasing weak composition. Each $\lambda_i$ is called a \emph{part} of $\lambda$. To each partition $\lambda$, we associate a \emph{Young diagram} 
\begin{align*}Y(\lambda) = \{(i, j) \in \bbN^2: i \le \lambda_j \text{ and } j \le l(\lambda)\}.\end{align*} 
A \emph{semi-standard Young tableau (SSYT)} of \emph{shape} $\lambda$ is a map $P: Y(\lambda) \rightarrow \bbN$ such that $P = P(i, j)$ is nondecreasing in $i$ and (strictly) increasing in $j$. 
We write $\lambda = \shape(P)$. Also, define the \emph{type} of $P$ as the weak composition 
\begin{align*}\type(P) = (\hash\{(i, j): P(i, j) = k\})_{k \in \bbN}.\end{align*} 
See Figure \ref{flucFi3} for a visualization of a Young diagram and an SSYT. A \emph{generalized permutation} $\varsigma$ (of \emph{length} $l \in \bbN$) is a finite sequence $\varsigma = ((i_k, j_k))_{k \in [l]}$ in $\bbN^2$ that is nondecreasing with respect to the lexicographic order ($i_k \le i_{k+1}$ and if $i_k = i_{k+1}$ then $j_k \le j_{k+1}$ for all $1 \le k < l$). We write $L(\varsigma)$ for the maximal length of a nondecreasing subsequence of $(j_k)_{k \in [l]}$.
Let $\sP$ denote the set of all generalized permutations and $\sT$ denote the set of all pairs of SSYTs $(P, Q)$ such that $\shape(P) = \shape(Q)$. 

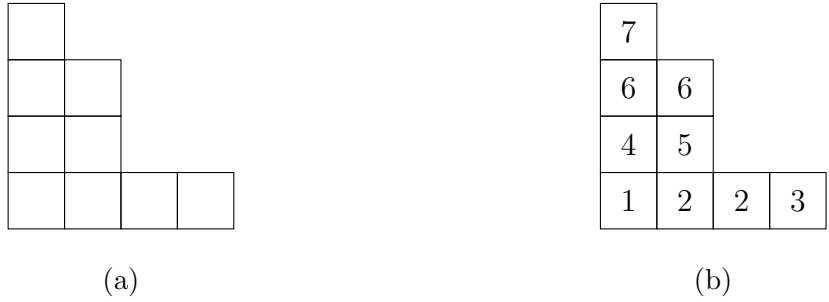
\begin{figure}[h!]
\centering
\begin{subfigure}{.5\textwidth}
  \centering
  \begin{tikzpicture}[scale = 1.5]
\draw (0,0) rectangle (0.5,0.5);
\draw (0.5, 0) rectangle (1,0.5);
\draw (1,0) rectangle (1.5,0.5);
\draw (1.5,0) rectangle (2,0.5);
\draw (0,0.5) rectangle (0.5,1);
\draw (0.5,0.5) rectangle (1,1);
\draw (0,1) rectangle (0.5,1.5);
\draw (0.5,1) rectangle (1,1.5);
\draw (0,1.5) rectangle (0.5,2);
 \end{tikzpicture}
 \caption{}
\end{subfigure}%
\begin{subfigure}{.5\textwidth}
  \centering
  \begin{tikzpicture}[scale = 1.5]
\draw (0,0) rectangle (0.5,0.5);
\node at (0.25,0.25) {1};
\draw (0.5, 0) rectangle (1,0.5);
\node at (0.75,0.25) {2};
\draw (1,0) rectangle (1.5,0.5);
\node at (1.25,0.25) {2};
\draw (1.5,0) rectangle (2,0.5);
\node at (1.75,0.25) {3};
\draw (0,0.5) rectangle (0.5,1);
\node at (0.25,0.75) {4};
\draw (0.5,0.5) rectangle (1,1);
\node at (0.75,0.75) {5};
\draw (0,1) rectangle (0.5,1.5);
\node at (0.25,1.25) {6};
\draw (0.5,1) rectangle (1,1.5);
\node at (0.75,1.25) {6};
\draw (0,1.5) rectangle (0.5,2); 
\node at (0.25,1.75) {7};
\end{tikzpicture}
  \caption{}
 \end{subfigure}
\caption[A Young diagram and a semi-standard Young Tableau]{\small{
(a) The Young diagram $Y(\lambda)$ for $\lambda = (4, 2, 2, 1, 0, \dotsc)$ viewed as the set of unit squares with upper-right corners at $(i, j) \in Y(\lambda)$. 
(b) An SSYT $P$ of shape $\lambda = (4, 2, 2, 1, 0, \dotsc)$. The value $P(i, j)$ is 
written inside the corresponding unit square. For example, $P(1, 2) = 4$ and $P(2, 3) = 6$. Note that the numbers are 
nondecreasing along rows and increasing along columns. Also, $\type(P) = (1, 2, 1, 1, 1, 2, 1, 0, \dotsc)$. 
}}
\label{flucFi3}
\end{figure}

\begin{thm}[RSK correspondence]
\label{flucTh4}
There exists a bijection $\RSK: \sP \rightarrow \sT$ with the following property: If $\varsigma = ((i_k, j_k))_{k \in [l]} \in \sP$ and $(P, Q) = \RSK(\varsigma)$ then 
\[\type(P) = (\hash \{k: j_k = n\})_{n \in \bbN}\]
\[\type(Q) = (\hash \{k: i_k = n\})_{n \in \bbN}\]
Furthermore, if $\lambda = (\lambda_i)_{i \in \bbN} = \shape(P) = \shape(Q)$ then $\lambda_1 = L(\varsigma)$.  
\end{thm}  
For a proof, see \cite{Fulton}, \cite{Stanley}. We will use the corollary below noted in \cite{Johansson00}. For $m, n \in \bbN$, define 
\begin{align*}\sP_{m, n} &= \{((i_k, j_k))_{k \in [l]} \in \sP: i_k \in [m] \text{ and } j_k \in [n] \text{ for } k \in [l]\} \\ 
\sT_{m, n} &= \{(P, Q) \in \sT: l(\type(P)) \le n \text{ and } l(\type(Q)) \le m\}.
\end{align*} 
\begin{cor}
\label{flucCo1}
Let $m, n \in \bbN$. There exists a bijection $f_{m, n}: \bbZ_+^{[m] \times [n]} \rightarrow \sT_{m, n}$ such that if $A \in \bbZ_+^{[m] \times [n]}$ and $(P, Q) = f_{m, n}(A)$ then 
\begin{enumerate}[(a)]
\item $|\lambda| = \sum_{i = 1}^m \sum_{j=1}^n A(i, j)$, where $\lambda = \shape(P) = \shape(Q)$. 
\item $\type(P)_j  = \sum_{i=1}^m A(i, j)$ for $j \in [n]$. 
\item $\type(Q)_i = \sum_{j=1}^n A(i, j)$ for $i \in [m]$. 
\item $\lambda_1 = \max_{\pi \in \Pi(m, n)} \sum_{(i, j) \in \pi} A(i, j)$, where $\lambda_1$ is the largest part of $\lambda$ in (a). 
\end{enumerate} 
\end{cor}
\begin{proof}
For each $A \in \bbZ_+^{[m] \times [n]}$, define $g_{m, n}(A)$ as the unique $\varsigma = ((i_k, j_k))_{k \in [l]} \in \sP_{m, n}$ such that each $(i, j) \in [m] \times [n]$ is repeated exactly $A(i, j)$ times in $\varsigma$. Note that $g_{m, n}$ is a bijection and $l = \sum_{i=1}^m \sum_{j=1}^n A(i, j)$. Moreover, the lengths of the maximal nondecreasing subsequences of $(j_k)_{k \in [l]}$ are given by $\sum_{(i, j) \in \pi} A(i, j)$ for various  $\pi \in \Pi_{m, n}$.  Hence, $L(\varsigma)$ equals the last-passage time $\max_{\pi \in \Pi(m, n)} \sum_{(i, j) \in \pi} A(i, j)$. It follows from Theorem \ref{flucTh4} that the map $\RSK$ restricts to a bijection $\RSK_{m, n}$ between $\sP_{m, n}$ and $\sT_{m, n}$. Now, the composition $f_{m, n} = \RSK_{m, n} \circ g_{m, n}$ is a bijection between $\bbZ_+^{[m] \times [n]}$ and $\sT_{m, n}$ with properties (a)-(d). 
\end{proof}

We will also rely on the following generalization of the Cauchy-Binet identity \cite[Proposition~2.10]{Johansson3}. 
\begin{prop}
\label{flucPr1}
Let $(X, \mu)$ be a measure space, $n \in \bbN$ and $f_i, g_i: X \rightarrow \bbC$ be measurable functions for $i \in [n]$ such that $f_i g_j$ is integrable for any $i, j \in [n]$. Then 
\begin{align*}
\det_{i, j \in [n]} \left[ \int_{X} f_i(x) g_j(x) \mu(dx) \right] = \frac{1}{n!} \int_{X^n} \det_{i, j \in [n]}[f_i(x_j)] \det_{i, j \in [n]}[g_i(x_j)] \mu(\dd x_1) \dotsc \mu(\dd x_n). 
\end{align*}  
\end{prop}

We next obtain a Fredholm determinant representation for the distribution of $G(n, n)$ in the case of injective $(a_i)_{i \in [n]}$ and $(b_j)_{j \in [n]}$ (terms do not repeat). A more general version of the following proof can also be found in \cite{BorodinGorin} and \cite{Johansson3}.  
\begin{thm}
\label{flucTh3}
Let $n \in \bbN$. Suppose that $(a_i)_{i \in [n]}$ and $(b_j)_{j \in [n]}$ are injective sequences. Define 
\begin{align}
\label{flucEq67}K_n(x, y) = \sum_{i, j \in [n]} \frac{a_i^x b_j^y}{1-a_i b_j} \frac{\prod \limits_{k \in [n]} (1-a_ib_k) (1-a_kb_j)} {\prod \limits_{\substack{k \in [n] \\ k \neq i}} (a_k-a_i) \prod \limits_{\substack{k \in [n] \\ k \neq j}} (b_k-b_j)}
\quad \text{ for } x, y \in \bbZ_+.
\end{align}
Then 
\begin{align}
\label{flucEq99}
\bfP(G(n, n) \le k) = 1 + \sum_{l=1}^n \frac{(-1)^l}{l!} \sum_{x_1, \dotsc, x_l \ge k+n} \det_{i, j \in [l]}\left[K_{n}(x_i, x_j)\right] \quad \text{ for } k \in \bbZ_+.  
\end{align}
\end{thm}
\begin{proof}
Let $\Phi$ denote the map that sends $A \in \bbZ_+^{[n] \times [n]}$ to the common shape of the corresponding SSYT pair under the bijection in Corollary \ref{flucCo1}, and define $\Lambda = \Phi([W(i, j)]_{i, j \in [n]})$. Then $\bfP(G(n, n) \le k) = \bfP(\Lambda_1 \le k)$. Moreover, for any partition $\lambda$, we have 
\begin{align}
\bfP(\Lambda = \lambda) &= \sum_{A: \Phi(A) = \lambda} \bfP(W(i, j) = A(i, j) \text{ for } i, j \in [n]) \nonumber\\
&= \sum_{A: \Phi(A) = \lambda} \prod_{i, j \in [n]} (1-a_i b_j) a_i^{A(i, j)} b_j^{A(i, j)} \nonumber\\
&= \prod_{i, j \in [n]} (1-a_i b_j) \sum_{A: \Phi(A) = \lambda}  \prod_{i \in [n]} a_i^{\sum_{j \in [n]} A(i, j)} \prod_{j \in [n]} b_j^{\sum_{i \in [n]} A(i, j)} \nonumber\\
&= \prod_{i, j \in [n]} (1-a_i b_j) \sum_{\substack{P: \shape(P) = \lambda \\ l(\type(P)) \le n}}  \prod_{j \in [n]} b_j^{\type(P)_j} \sum_{\substack{Q: \shape(Q) = \lambda \\ l(\type(Q)) \le n}} \prod_{i \in [n]} a_i^{\type(Q)_i} \label{flucEq57}
\end{align} 
Note the inequality $l(\type(P)) \ge l(\shape(P))$ for any SSYT $P$; hence, (\ref{flucEq57}) is zero unless $l(\lambda) \le n$. 

We now use the polynomial identity 
\begin{align}
\label{flucEq56}
\sum_{\substack{P: \shape(P) = \lambda \\ l(\type(P)) \le n}}  \prod_{j \in [n]} X_j^{\type(P)_j} = \frac{\det \limits_{i, j \in [n]}[X_i^{\lambda_j-j+n}]}{\det \limits_{i, j \in [n]}[X_i^{-j+n}]}, 
\end{align}
either side of which is the Schur polynomial indexed by $\lambda$ in $n$ variables $X_1, \dotsc, X_n$. For a proof of (\ref{flucEq56}), see \cite[Chapter~7]{Stanley}. Since $(a_i)_{i \in [n]}$ and $(b_j)_{j \in [n]}$ are injective, the Vandermonde determinant 
\begin{align*}\det \limits_{i, j \in [n]}[X_i^{-j+n}] = \prod_{1 \le i < j \le n} (X_i-X_j)\end{align*} 
is nonzero when evaluated by setting $X_i = a_i$ for $i \in [n]$ or $X_i = b_i$ for $i \in [n]$.   
Hence, by (\ref{flucEq57}) and (\ref{flucEq56}), we obtain 
\begin{align}\label{flucEq66}\bfP(\Lambda = \lambda) = Z_n^{-1} \det \limits_{i, j \in [n]}[a_i^{\lambda_j-j+n}] \det \limits_{i, j \in [n]}[b_i^{\lambda_j-j+n}], \end{align} where the normalization constant is given by 
\begin{align}
\label{flucEq61}
Z_n = \frac{\prod_{1 \le i < j \le n} (a_i-a_j) (b_i-b_j)}{\prod_{i, j \in [n]}(1-a_i b_j)}. 
\end{align}
The probability $\bfP(\Lambda_1 \le k)$ can then be written as  
\begin{align}
\bfP(\Lambda_1 \le k) &= \frac{1}{Z_n}\sum_{\substack{\lambda: \lambda_1 \le k \\ l(\lambda) \le n}} \det \limits_{i, j \in [n]}[a_i^{\lambda_j-j+n}] \det \limits_{i, j \in [n]} [b_i^{\lambda_j-j+n}]\nonumber\\
&= \frac{1}{n!Z_n} \sum_{\substack{m_1, \dotsc, m_n \in \bbZ_+ \\ m_j \le k+n-1}} \det \limits_{i, j \in [n]}[a_i^{m_j}] \det \limits_{i, j \in [n]}[b_i^{m_j}]. \label{flucEq58}  
\end{align}
For the last equality, first change the summation index from $\lambda$ to the decreasing sequence $(m_j)_{j \in [n]} = (\lambda_j-j+n)_{j \in [n]}$ in $\bbZ_+$ with $m_1 \le k+n-1$. (Note that $\lambda$ is uniquely determined from $(m_j)_{j \in [n]}$ because $l(\lambda) \le n$). Then remove the ordering of $(m_j)_{j \in [n]}$, which introduces the factor $1/n!$ Finally, allow repeats  in $(m_j)_{j \in [n]}$, which does not alter the sum because the added terms are all zero. 

It remains to turn (\ref{flucEq58}) into the desired form. For any $l \in \bbZ_+$, by Proposition \ref{flucPr1}, 
\begin{align}
\sum_{\substack{m_1, \dotsc, m_n \in \bbZ_+ \\ m_i \le l}} \det \limits_{i, j \in [n]}[a_i^{m_j}] \det\limits_{i, j \in [n]}[b_i^{m_j}] = n! \det\limits_{i, j \in [n]} \left[\sum_{m=0}^{l} a_i^m b_j^m\right] = n! \det\limits_{i, j \in [n]}\left[\frac{1-a_i^{l+1}b_j^{l+1}}{1-a_i b_j}\right]. \label{flucEq59}   
\end{align}
Letting $k \rightarrow +\infty$ in (\ref{flucEq58}) and $l \rightarrow +\infty$ in (\ref{flucEq59}) yield 
\begin{align}
Z_n = \det\limits_{i, j \in [n]} \left[\sum_{m=0}^{\infty} a_i^m b_j^m\right] = \det\limits_{i, j \in [n]}\left[\dfrac{1}{1-a_i b_j}\right]. \label{flucEq62}
\end{align} 
Since $Z_n \neq 0$ by (\ref{flucEq61}), the matrix $C = \left[(1-a_i b_j)^{-1}\right]_{i, j \in [n]}$ is invertible, and  
\begin{align*}
\bfP(\Lambda_1 \le k) &= \frac{\det\limits_{i, j \in [n]}\left[C(i, j)-\sum_{m=k+n}^{\infty} a_i^m b_j^m\right]}{\det\limits_{i, j \in [n]} [C(i, j)]} \\
&= \det\limits_{i, j \in [n]} \left[\delta_{i, j}-\sum_{p=1}^n C^{-1}(i, p) \sum_{m=k+n}^{\infty} a_p^m b_j^m\right] \\
&= \det\limits_{i, j \in [n]}\left[\delta_{i, j}-\sum_{m=k+n}^{\infty} b_j^m h_i(m) \right], 
\end{align*}
where $h_i(m) = \sum_{p=1}^n C^{-1}(i, p) a_p^m$. Then, %applying Lemma \ref{flucL10} and 
applying Proposition \ref{flucPr1} twice, we obtain  
\begin{align}
\bfP(\Lambda_1 \le k) &= 1 + \sum_{q=1}^{n} \frac{(-1)^q}{q!} \sum_{\substack{i_1, \dotsc, i_q \in [n]}}\det\limits_{r, s \in [q]}\left[\sum_{m=k+n}^{\infty} b_{i_s}^m h_{i_r}(m)\right] \nonumber\\
&= 1 + \sum_{q=1}^{n} \frac{(-1)^q}{q!} \sum_{\substack{i_1, \dotsc, i_q=1}}^n \frac{1}{q!}\sum_{m_1, \dotsc, m_q \ge k+n} \det\limits_{r, s \in [q]}[b_{i_s}^{m_{r}}] \det\limits_{r, s \in [q]}[h_{i_s}(m_r)] \nonumber\\
&= 1 + \sum_{q=1}^{n} \frac{(-1)^q}{q!}  \sum_{m_1, \dotsc, m_q \ge k+n} \frac{1}{q!} \sum_{\substack{i_1, \dotsc, i_q=1}}^n  \det\limits_{r, s \in [q]}[b_{i_s}^{m_{r}}] \det\limits_{r, s \in [q]}[h_{i_s}(m_r)] \nonumber\\
&= 1 + \sum_{q=1}^{n} \frac{(-1)^q}{q!}  \sum_{m_1, \dotsc, m_q \ge k+n} \det\limits_{r, s \in [q]}\left[\sum_{i=1}^n h_i(m_r) b_i^{m_s}\right] \nonumber\\
&= 1 + \sum_{q=1}^{n} \frac{(-1)^q}{q!}  \sum_{m_1, \dotsc, m_q \ge k+n} \det\limits_{r, s \in [q]}\left[\sum_{i=1}^n \sum_{p=1}^n a_p^{m_r} C^{-1}(i, p) b_i^{m_s}\right].  \label{flucEq64}
\end{align}

Finally, we compute the inverse of $C$. Note that $C^{i, j}$, the $i, j$-minor of $C$, has the same structure (in terms of entries) as $C$; therefore, by (\ref{flucEq61}) and (\ref{flucEq62}), 
\begin{align*}
\det C^{i, j} &= \prod \limits_{\substack{k, l \in [n] \\ k \neq i \\ l \neq j}} \frac{1}{1-a_k b_l} \prod \limits_{\substack{k, l \in [n] \\ k < l \\ k, l \neq i}} (a_k-a_l) \prod \limits_{\substack{k, l \in [n] \\ k < l \\ k, l \neq j}} (b_k-b_l) \\
&= \frac{\det C}{1-a_i b_j} \prod \limits_{k \in [n]} (1-a_i b_k) (1-a_k b_j) \\
&\cdot (-1)^{i+j} \prod \limits_{\substack{k \in [n] \\ k \neq i}}(a_i-a_k)^{-1} \prod \limits_{\substack{k \in [n] \\ k \neq j}}(b_j-b_k)^{-1}
\end{align*}
Then, by Cramer's rule, 
\begin{align}\label{flucEq63}C^{-1}(i, j) = (-1)^{i+j}\frac{\det C^{j, i}}{\det C} = \frac{1}{1-a_jb_i}
\frac{\prod \limits_{k \in [n]} (1-a_jb_k)(1-a_kb_i)}{\prod \limits_{\substack{k \in [n] \\ k \neq i}} (b_k-b_i) \prod \limits_{\substack{k \in [n] \\ k \neq j}}(a_k-a_j)}.\end{align}
Inserting (\ref{flucEq63}) into (\ref{flucEq64}) completes the proof. 
\end{proof} 
Note that the assumption of independent weights distributed as (\ref{shaE5}) is crucial in the preceding proof to obtain (\ref{flucEq57}), the representation of the distribution of the last-passage times in terms of the Schur polynomials. 
The probability measure (\ref{flucEq66}) on the space of partitions is an example of a Schur measure introduced in \cite{Okounkov}. 

Another probability measure of interest derived from (\ref{flucEq66}) is the distribution of the random set $\sS = \{\Lambda_j-j+n: j \in [n]\}$, which is given by  
\begin{align}
\label{flucEq96}
\bfP(\sS = S) = \frac{1}{n! Z_n}\det\limits_{i, j \in [n]}[a_i^{m_j}] \det\limits_{i, j \in [n]}[b_i^{m_j}]
\end{align}
for any $S = \{m_1, \dotsc, m_n\} \subset \bbZ_+$. By a general fact from the theory of point processes, for any distinct $x_1, \dotsc, x_q \in \bbZ_+$,  
\begin{align}
\label{flucEq97}
\bfP(\{x_1, \dotsc, x_q\} \subset \sS\} = \det\limits_{r, s \in [q]}[K_n(x_r, x_s)]. 
\end{align}
In the language of the theory, $\sS$ can be viewed as a determinantal point process on $\bbZ_+$ with correlation kernel $K_n$. Since $G(n, n) = \Lambda_1 = \max \sS-n+1$, a restatement of (\ref{flucEq99}) is that 
\begin{align*}
\bfP(\max \sS \le k+n-1) = 1+\sum_{l=1}^n \frac{(-1)^l}{l!} \sum \limits_{x_1, \dotsc, x_l \ge k+n} \bfP(\{x_1, \dotsc, x_l\} \subset \sS) \quad \text{ for } k \in \bbZ_+, 
\end{align*} 
which is an application of the inclusion/exclusion principle. This furnishes a probabilistic interpretation of (\ref{flucEq99}). 
For a proof of (\ref{flucEq97}) and a detailed discussion of the notions in this paragraph, we refer the reader to \cite{Borodin2}, \cite{BorodinGorin},  and \cite{Johansson3}. 

A useful conclusion from (\ref{flucEq97}) is that 
\begin{align}\label{flucEq98}\det\limits_{r, s \in [q]} [K_n(x_r, x_s)] \ge 0 \quad \text{ for any } x_1, \dotsc, x_q \in \bbZ_+.\end{align} 
Moreover, this determinant equals $0$ if $q > n$. One can also make these observations more directly using Proposition \ref{flucPr1}; we have 
\begin{align}
\det\limits_{r, s \in [q]}[K_n(x_r, x_s)] &= \det\limits_{r, s \in [q]}\left[\sum_{i, j \in [n]} a_i^{x_r} C^{-1}(j, i) b_j^{x_s}\right] \nonumber\\
&= \frac{1}{(q!)^2} \sum_{i_1, \dotsc, i_q \in [n]} \sum_{j_1, \dotsc, j_q \in [n]} \det\limits_{r, s \in [q]}[a_{i_s}^{x_r}]\det\limits_{r, s \in [q]}[C^{-1}(j_s, i_r)]\det\limits_{r, s \in [q]}[b_{j_s}^{x_r}] \label{flucEq70} \\
&= \frac{1}{q!} \sum_{i_1, \dotsc, i_q \in [n]} \sum_{j_1, \dotsc, j_q \in [n]} \frac{\det \limits_{s, r \in [q]}[a_{i_s}^{x_r}]\det\limits_{s, r \in [q]}[b_{j_s}^{x_r}]}{q! \det\limits_{s, r \in [q]}[C(j_s, i_r)]},  \label{flucEq69} 
\end{align}
where the last equality requires $q \le n$; otherwise the determinants in (\ref{flucEq70}) are zero. For each choice of $i_1, \dotsc, i_q$ and $j_1, \dotsc, j_q$, the summand is a probability by (\ref{flucEq96}) and, hence, the sum is nonnegative.  

\section{Contour integral representation of the kernel}

For the purposes of asymptotics as well as to extend Theorem \ref{flucTh3} to the case of noninjective parameters, it is useful to express (\ref{flucEq67}) as a contour integral.  
Let $\max_{1 \le i \le m} a_i \vee \max_{1 \le j \le n} b_j < \rho < 1$. Then, deforming the contour in (\ref{flucE3}), we obtain
\begin{equation}
\label{flucEq8.1}
I_{m, n, x}^{\bfa, \bfb} = \frac{1}{2\pi \ii} \oint \limits_{|z| = \rho} F_{m, n, x}^{\bfa, \bfb}(z) \dd z.  
\end{equation}
Using the identity 
$(1-zw)^{-1} = \sum_{l = 0}^\infty z^l w^l$ and Fubini-Tonelli theorem, we can rearrange (\ref{flucE71}) as
%Since $\rho < 1$, using the identity 
%$(1-zw)^{-1} = \sum_{l = 0}^\infty z^l w^l$ and Fubini-Tonelli theorem, we can rearrange (\ref{flucE5.1}) as 
\begin{align}
K_{m, n}^{\bfa, \bfb}(x, y) &= \frac{1}{(2\pi \ii)^2} \oint \limits_{|w| = \rho} \oint \limits_{|z| = \rho} \sum \limits_{l=0}^\infty F_{m, n, x+l}^{\bfa, \bfb}(z) F_{n, m, y+l}^{\bfb, \bfa}(w) \ dz\ dw \nonumber\\
&= \frac{1}{(2\pi \ii)^2} \oint \limits_{|w| = \rho} \oint \limits_{|z| = \rho} \frac{F_{m, n, x}^{\bfa, \bfb}(z) F_{n, m, y}^{\bfb, \bfa}(w)} {1-zw}\ \dd z\ \dd w, \label{flucEq5.1} 
\end{align}
where the circles of integration are oriented counterclockwise. 

\begin{proof}[Proof of Theorem \ref{flucT7}]
If $(a_i)_{i \in [n]}$ and $(b_{i})_{i \in [n]}$ are injective, then for each $u \in \Disc(0, 1)$ the only singularities of the functions $z \mapsto \dfrac{F^{\bfa, \bfb}_{n, n, x}(z)}{1-zu}$ and $w \mapsto \dfrac{F^{\bfb, \bfa}_{n, n, y}(w)}{1-wu}$ inside $\Disc(0, 1)$ are simples poles at $a_1, \dotsc, a_n$ and $b_1, \dotsc, b_n$, respectively. Therefore, by Cauchy's residue formula and (\ref{flucEq67}),  
\begin{align*}
K_{n, n}^{\bfa, \bfb}(x, y) &= \sum_{i, j \in [n]} \frac{\Res_{a_i} F_{n, n, x}^{\bfa, \bfb} \Res_{b_j} F_{n, n, y}^{\bfb, \bfa}}{1-a_i b_j} \\
&= \sum_{i, j \in [n]} \frac{a_i^{x+n}b_j^{y+n}}{1-a_ib_j}\frac{\prod \limits_{k \in [n]} (1-a_ib_k) (1-a_kb_j)}{\prod \limits_{\substack{k \in [n] \\ k \neq i}} (a_k-a_i) \prod \limits_{\substack{k \in [n] \\ k \neq j}} (b_k-b_j)} \\ 
&= K_{n}(x+n, y+n)
\end{align*}
for $x, y \in \bbZ_+$ provided that $(a_i)_{i \in [n]}$ and $(b_{i})_{i \in [n]}$ are injective. Then, by Theorem \ref{flucTh3}, 
\begin{align}\label{flucEq68}\bfP(G(n, n) \le k) = 1 + \sum \limits_{l=1}^{n} \frac{(-1)^l}{l!} \sum \limits_{\substack{x_1, \dotsc, x_l \in \bbZ_+ \\ x_i \ge k}}\det\limits_{i, j \in [l]}[K_{n, n}^{\bfa, \bfb}(x_i, x_j)] \quad \text{ for } k \in \bbZ_+.\end{align}

We claim that both sides of (\ref{flucEq68}) are continuous in parameters $(a_i)_{i \in [n]}$ and $(b_i)_{i \in [n]}$. Then (\ref{flucEq68}) holds even if $(a_i)_{i \in [n]}$ or $(b_i)_{i \in [n]}$ has repeats. In particular, setting $a_i = 0$ for $m < i \le n$ and $b_j = 0$ for $n < j \le m$, we obtain the result. To prove the claim, note that $\bfP(G(n, n) \le k)$ is continuous because it can be written as the finite sum of probabilities 
\begin{align}\label{flucEq100}\bfP(W(i, j) = A(i, j) \text{ for } i, j \in [n]) = \prod_{i, j \in [n]} (1-a_ib_j) a_i^{A(i, j)}b_j^{A(i, j)}\end{align} 
over matrices $A \in \bbZ_+^{[n] \times [n]}$ for which $\max_{\pi \in \Pi(m, n)} \sum_{(i, j) \in \pi} A(i, j) \le k$, and (\ref{flucEq100}) is continuous. Pick $\delta > 0$ small so that 
\begin{align}\label{flucEq101}\max_{1 \le i \le n} a_i \vee \max_{1 \le j \le n} b_j < \rho-\delta, \end{align}
where $\rho$ is as in (\ref{flucEq5.1}). Then there exists $C > 0$ (which depends on $n$ and $\delta$) such that 
$K_{n, n}^{\bfa, \bfb}(x, y) \le C \rho^{x+y}$
for any $x, y \in \bbZ_+$, which leads to the bound 
\begin{align*}
\det\limits_{i, j \in [l]}[K_{n, n}^{\bfa, \bfb}(x_i, x_j)] &= \sum_{\sigma \in \sS_l} \sgn(\sigma) \prod_{i=1}^l K_{n, n}^{\bfa, \bfb}(x_i, x_{\sigma(i)}) \le l! C^l \rho^{2 \sum_{i=1}^l x_i},  
\end{align*} 
which is summable over $x_1, \dotsc, x_l \ge k$. Hence, the inner sum in (\ref{flucEq68}) converges uniformly on the set $[0, \rho-\delta)^{2n}$. This and continuity of $K_{n, n}^{\bfa, \bfb}$ imply that the right-hand side of (\ref{flucEq68}) is also continuous on $[0, \rho-\delta)^{2n}$. Since $\rho-\delta$ can be chosen arbitrarily close to $1$, the claim is proved. 
\end{proof}

\chapter{Fluctuations and Right Tail Deviations}\label{Ch4}

\section{Introduction}

Recall the description of the model from Subsection \ref{intS2}. Let $\Theta: (0, 1)^{\bbN} \rightarrow (0, 1)^{\bbN}$ denote the shift map $(c_n)_{n \in \bbN} \mapsto (c_{n+1})_{n \in \bbN}$. In Chapter \ref{sha}, we computed $\bfP$-a.s. limit of $n^{-1}G(\lf nr \rf, n)$ as $n \rightarrow \infty$ for $r > 0$ when $\bfa$ and $\bfb$ are random and satisfy this total ergodicity assumption. 
\begin{align}
\label{flucA1}
\text{The pair $(\bfa, \bfb)$ is ergodic with respect to $\Theta^k \times \Theta^l$ separately for each $(k, l) \in \bbN^2$.} 
\end{align} 
\noindent In particular, $\bfa$ and $\bfb$ are ergodic (with respect to $\Theta$). Let $\alpha$ and $\beta$ denote the common marginal distributions of each term in $\bfa$ and $\bfb$, respectively. 
To describe the a.s. limit, define  
\begin{align}
g(z, r) &= r \int \limits_{(0, 1)} \frac{a\alpha(\dd a)}{z-a}+\int \limits_{(0, 1)} \frac{bz\beta(\dd b)}{1-bz} \quad \text{ for } z \in \bbC \smallsetminus [0, \bar{\alpha}] \smallsetminus [1/\bar{\beta}, \infty).  \label{flucE11}
\end{align}
Use (\ref{flucE11}) also to define the values of $g$ for $z = \bar{\alpha}$ and $z = 1/\bar{\beta}$. (These values equal $\infty$ precisely when $\int_{(0, 1)} (\bar{\alpha}-a)^{-1} \alpha(\dd a) = \infty$ and $\int_{(0, 1)} (\bar{\beta}-b)^{-1} \beta(\dd b) = \infty$). Let 
\begin{align}
\gamma(r) &= \inf \limits_{z \in [\bar{\alpha}, 1/\bar{\beta}]} g(z, r). \label{flucE12}
\end{align} 
Then (see Theorems \ref{shaT3} and \ref{shaT4})
\begin{align}
\label{flucE122}
\lim_{n \rightarrow \infty} \frac{G(\lf nr \rf, n)}{n} = \gamma(r) \quad \text{ for } r > 0 \quad \bfP\text{-a.s.} \text{ for a.e. } (\bfa, \bfb).
\end{align} 
 
In the case $\bar{\alpha} \bar{\beta} < 1$, 
introduce the critical values  
\begin{align}
c_1(\alpha, \beta) = \frac{\int_0^1 b(1-b\bar{\alpha})^{-2}\beta(\dd b)}{\int_0^1 a(\bar{\alpha}-a)^{-2} \alpha(\dd a)}
\qquad  c_2(\alpha, \beta) = \frac{\int_0^1 b(\bar{\beta}-b)^{-2}\beta(\dd b)}{\int_0^1 a(1-a\bar{\beta})^{-2} \alpha(\dd a)}. \label{flucE20}
\end{align}
By the strict convexity of $g$ on $(\bar{\alpha}, 1/\bar{\beta})$, there exists a unique minimizer $\zeta(r) \in [\bar{\alpha}, 1/\bar{\beta}]$ in (\ref{flucE12}). Moreover, 
$\zeta = \bar{\alpha}$ if $r \le c_1$, $\zeta = 1/\bar{\beta}$ if $r \ge c_2$ and $\zeta \in (\bar{\alpha}, 1/\bar{\beta})$ otherwise. As a consequence, $\gamma$ is strictly concave for $c_1 < r < c_2$, and is linear for $r \le c_1$ or $r \ge c_2$, see Corollary \ref{shac3}. 

\section{Results} \label{flucS2}

In this section, we list the recurring assumptions imposed on the model throughout, introduce some notation and then give the formal statements of our main results. 

We replace (\ref{flucA1}) with the weaker assumption that 
\begin{align}\label{flucA2}\text{$\bfa$ and $\bfb$ are random and ergodic. }\end{align}
\noindent In particular, the joint distribution of $\bfa$ and $\bfb$ is not relevant in contrast with Chapter \ref{sha} and \cite{EmrahJanjigian15}. Various limit statements and bounds below are valid for a.e. realization of $\bfa$ and $\bfb$. %{\color{red} We also obtain some of the results only in the special case when $\bfa$ and $\bfb$ are i.i.d. sequences.} 
We next assume that  
\begin{align}\label{flucA3}\bar{\alpha} \bar{\beta} < 1. \end{align}
Hence, the critical values $c_1(\alpha, \beta) < c_2(\alpha, \beta)$ in (\ref{flucE20}) are defined. For $i, j \in \bbN$, we have $\bfE W(i, j) = a_i b_j(1-a_i b_j)^{-1}$ and, by (\ref{flucA2}), 
\begin{align}\limsup_{i \rightarrow \infty} a_i \stackrel{\text{a.s.}}{=} \bar{\alpha} \qquad \text{ and } \qquad \limsup_{j \rightarrow \infty} b_j \stackrel{\text{a.s.}}{=} \bar{\beta}.\label{flucE196}\end{align} 
Therefore, an assumption equivalent to (\ref{flucA3}) is $\sup_{i, j \in \bbN} \bfE W(i, j) < \infty$ a.s. We will restrict attention to the strictly concave region of the shape function i.e. assume that 
\begin{align} \label{flucA4}c_1(\alpha, \beta) < r <  c_2(\alpha, \beta). \end{align}
Hence, $\zeta(r) \in (\bar{\alpha}, 1/\bar{\beta})$. 
%Finally, assume that the parameter $x$ satisfies 
%\begin{align}
%x &< \min\{g(\bar{\alpha}, r), g(1/\bar{\beta}, r)\}-\gamma(r),  \label{flucA5}
%\end{align}
%which together with (\ref{flucA4}) implies that $\zeta^{\pm}(r, x) \in (\bar{\alpha}, 1/\bar{\beta})$ as well. Consequently, $\zeta^-(r, x)$ and $\zeta^+(r, x)$ are unique on the intervals $(\bar{\alpha}, \zeta(r)]$ and $(\zeta(r), 1/\bar{\beta})$ such that 
%\begin{align}
%g(\zeta^\pm(r, x), r) = \gamma(r)+x. 
%\end{align}
%In summary, assumptions (\ref{flucA3}), (\ref{flucA4}) and (\ref{flucA5}) correspond to the situation depicted in Figure \ref{flucF2}c. 
The analysis of this case is simpler mainly because $\zeta$ is bounded away from the zeros and poles of the integrands in (\ref{flucE3}). 

%Since $\partial_z g(z) < 0$ for $z \in (\bar{\alpha}, \zeta)$ and $\partial_z g(z) > 0$ for $z \in (\zeta, 1/\bar{\beta})$, we can define 
%\begin{align}
%\sigma^-(r, x) &= \{-\zeta^-(r, x) \partial_z g(\zeta^-(r, x), r)\}^{1/2} \nonumber\\
%&= (\zeta^-(r, x))^{1/2} \bigg(r \int \limits_{(0, 1)} \frac{a\alpha(\dd a)}{(\zeta^-(r, x)-a)^2} - \int \limits_{(0, 1)} \frac{b\beta(\dd b)}{(1-b\zeta^-(r, x))^2}\bigg)^{1/2} \label{flucE126}\\
%\sigma^+(r, x) &= \{\zeta^+(r, x)\partial_z g(\zeta^+(r, x), r)\}^{1/2} \nonumber \\
%&= (\zeta^+(r, x))^{1/2}\bigg(-r \int \limits_{(0, 1)} \frac{a\alpha(\dd a)}{(\zeta^+(r, x)-a)^2} + \int \limits_{(0, 1)} \frac{b\beta(\dd b)}{(1-b\zeta^+(r, x))^2}\bigg)^{1/2}. \label{flucE127}
%\end{align}
%for $x > 0$. In the case $x = 0$, the definitions of $\sigma^{\pm}(r, 0)$ are the same and given by 
Define
\begin{align}
\label{flucE128}
\sigma(r) &= \bigg(\frac{1}{2}\zeta(r)^{2} \partial_z^2 g(\zeta(r), r)\bigg)^{1/3} = \zeta(r)^{2/3} \left(r \int \limits_{(0, 1)} \frac{a\alpha(\dd a)}{(\zeta(r)-a)^3} + \int \limits_{(0, 1)} \frac{b^2\beta(\dd b)}{(1-b\zeta(r))^3}\right)^{1/3}. 
\end{align}

For $m, n \in \bbN$, introduce the empirical distributions 
\begin{align}
\alpha_m = \frac{1}{m} \sum_{i=1}^m \delta_{a_i} \qquad \beta_n = \frac{1}{n} \sum_{j=1}^n \delta_{b_j}.  
\end{align}
The statements of our results involve quantities %$\gamma, \zeta^\pm$ and $\sigma^\pm$ 
$\gamma, \zeta$ and $\sigma$ 
computed with $\alpha_m$ and $\beta_n$ in place of $\alpha$ and $\beta$, respectively, and $r = m/n$. In this case, we will put $m, n$ in the subscripts for distinction, and omit $r$. Now, (\ref{flucE11}) and (\ref{flucE12}) specialize to 
\begin{align}
g_{m, n}(z) &= \frac{1}{n}\sum \limits_{i=1}^m \frac{a_i}{z-a_i} + \frac{1}{n} \sum \limits_{j=1}^n \frac{b_j z}{1-b_j z} \quad \text{ for } z \in \bbC \smallsetminus \{a_1, \dotsc, a_m, 1/b_1, \dotsc, 1/b_n\} \label{flucE7} \\
\gamma_{m, n} &= \inf \limits_{\substack{a_i < z < 1/b_j \\ i \in [m], j \in [n]}} \left\{g_{m, n}(z)\right\}, \label{flucE5}
\end{align}
respectively. Note that (\ref{flucE7}) also extends the domain $\bbC \smallsetminus [0, \max_{i \in [m]} a_i] \smallsetminus [\min_{j \in [n]} 1/b_j, \infty)$ of $g_{m, n}$ in (\ref{flucE11}). We have (\ref{flucA3}) and, due to 
\begin{align}\int_{(0, 1)} (\bar{\alpha}_m-a)^{-1} \alpha_m(\dd a) = \int_{(0, 1)} (\bar{\beta}_n-b)^{-1} \beta_n(\dd b) = \infty,\end{align} (\ref{flucA4}) %and (\ref{flucA5}) 
as well for $\alpha_m$ and $\beta_n$. Therefore, %$\zeta_{m, n, x}^\pm \in (\max_{i \in [m]} a_i, \min_{j \in [n]} 1/b_j)$.
$\zeta_{m, n} \in (\max_{i \in [m]} a_i, \min_{j \in [n]} 1/b_j)$.
Finally, definition  
%(\ref{flucE126}), (\ref{flucE127}) and 
(\ref{flucE128}) takes the form 
\begin{align}
%\sigma^-_{m, n, x} &= (\zeta^-_{m, n, x})^{1/2} \bigg(\frac{1}{n} \sum \limits_{i=1}^{m} \frac{a_i}{(\zeta_{m, n, x}^--a_i)^2} -  \frac{1}{n} \sum \limits_{i=1}^{n} \frac{b_j}{(1-b_j\zeta^-_{m, n, x})^2}\bigg)^{1/2} \label{flucE129}\\
%\sigma^+_{m, n, x} &= (\zeta^+_{m, n, x})^{1/2}\bigg(-\frac{1}{n} \sum \limits_{i=1}^{m} \frac{a_i}{(\zeta^+_{m, n, x}-a_i)^2} + \frac{1}{n} \sum \limits_{i=1}^{m} \frac{b_j}{(1-b_j\zeta^+_{m, n, x})^2}\bigg)^{1/2} \label{flucE130} \\
\sigma_{m, n} &= (\zeta_{m, n})^{2/3} \left(\frac{1}{n} \sum \limits_{i=1}^{m} \frac{a_i}{(\zeta_{m, n}-a_i)^3} + \frac{1}{n} \sum \limits_{i=1}^{n} \frac{b_j^2}{(1-b_j\zeta_{m, n})^3}\right)^{1/3}. \label{flucE131}
\end{align}
%respectively. 

Let $m_k, n_k \in \bbN$ such that $\lim_{k \rightarrow \infty} m_k/n_k = r$ and $n_k \rightarrow \infty$. Also, for $k \in \bbN$ and $s \in \bbR$, define 
\begin{align}
p_k(s) &= \lf n_k \gamma_{m_k, n_k} + n_k^{1/3} \sigma_{m_k, n_k} s \rf \label{flucE17} %\\
%p_k'(x) &= p_k(x\sigma_{m_k, n_k}^{-1}n_k^{2/3}) = \lf n_k\gamma_{m_k, n_k} + n_k x\rf. \label{flucE178}
\end{align}
For brevity, we will sometimes only write $k$ in place of the pair $m_k, n_k$ in the subscripts. 

Our first main result is a right tail bound for the last-passage times. 
\begin{thm}
\label{flucT6}
There exist (deterministic) constants $C, c > 0$ such that, for a.e. $(\bfa, \bfb)$, there exists $k_0 = k_0^{\bfa, \bfb}$ such that 
\begin{align}
\bfP_{\bfa, \bfb}(G(m_k, n_k) \ge n_k\gamma_{m_k, n_k}^{\bfa, \bfb} + n_ks) &\le  Ce^{-cn_k(s^{3/2} \wedge s)}
\end{align}
for $s > 0$ and $k \ge k_0$. %We can also use \gamma_{m_k, n_k} for gamma. 
\end{thm}

Let us write $F_{\GUE}$ for the Tracy-Widom GUE distribution, see Section \ref{flucSe2}. 
\begin{thm}
\label{flucT1} 
%Let $r \in (c_1, c_2)$.  
\begin{equation}
\label{fluceq23}
\lim \limits_{n \rightarrow \infty} \bfP_{\bfa, \bfb}(G(m_k, n_k) \le n_k \gamma_{m_k, n_k}^{\bfa, \bfb} + n_k^{1/3}\sigma_{m_k, n_k}^{\bfa, \bfb}s) = F_{\GUE}(s) \quad \text{ for } s \in \bbR \text{ for a.e. } (\bfa, \bfb). \end{equation}
\end{thm}
%{\color{red} What if centered by $\gamma$}

Preceding results are based on suitable estimates and a scaling limit of the correlation kernel defined in (\ref{flucE71}). We present some of the results on the correlation kernel here as well. Let \begin{align}P_k(x) = x_+^{3/2} \wedge (x_+n_k^{1/3}).\label{flucE143}\end{align}

\begin{thm}
\label{flucT2}
Let $s_0 \ge 0$. There exist (deterministic) constants $C, c > 0$ such that, for a.e. $(\bfa, \bfb)$, there exists $k_0 = k_0^{\bfa, \bfb}$ such that 
\begin{align*}
\det \limits_{i, j \in [l]}[\KK_{m_k, n_k}^{\bfa, \bfb}(p_k^{\bfa, \bfb}(s_i), p_k^{\bfa, \bfb}(s_j))] \le C^l l^{l/2} n_k^{-l/3} e^{-c \sum_{i=1}^l P_k(s_i)}
\end{align*}
for $l \in \bbN$, $k \ge k_0$ and $s_1, \dotsc, s_l \ge -s_0$. 
\end{thm}
%k_0 or C, c may possibly be chosen nonrandom. 
\begin{thm}
\label{flucT5}
Let $s_0 \ge 0$.   
\begin{align*}
\lim \limits_{k \rightarrow \infty} \det \limits_{i, j \in [l]} [\KK_{m_k, n_k}^{\bfa, \bfb}(p_k^{\bfa, \bfb}(s_i), p_k^{\bfa, \bfb}(s_{j}))] (\sigma_{m_k, n_k}^{\bfa, \bfb})^l n_k^{l/3}  = \det \limits_{i, j \in [l]}[\AiK(s_i, s_j)]
\end{align*}
uniformly in $s_1, \dotsc, s_l \in [-s_0, s_0]$ and $l \in \bbN$ for a.e. $(\bfa, \bfb)$.
\end{thm}

\section{Steepest-descent analysis}  

A main step in our approach to obtain the results in Section \ref{flucS2} is to understand the behavior of the contour integrals (\ref{flucE3}) with $m = m_k, n = n_k$ and $x = p_k(s)$,
\begin{align}
\label{flucE192}
I_{m_k, n_k, p_k(s)}^{\bfa, \bfb} = \frac{1}{2\pi \ii}\oint \limits_{|z| = 1} F_{m_k, n_k, p_k(s)}^{\bfa, \bfb}(z) \dd z, 
\end{align}
as $k$ grows large. Following the standard ideas in the method of steepest-descent \cite{deBruijn}, we seek a suitable deformation of the contour $|z| = 1$ to simplify the analysis. More specifically, we would like the absolute value of the integrand $|F_{m_k, n_k, p_k(s)}|$ to have a unique maximizer on the new contour and the ratio of $|F_{m_k, n_k, p_k(s)}|$ to its maximum value to decay exponentially as we move away from the maximizer along the contour. Then the main contribution to the integral should come from a small neighborhood of the maximizer. The construction of a contour with such properties naturally leads to the consideration of the steepest-descent curves (see Definition \ref{flucD1}) of $\log |F_{m_k, n_k, p_k(s)}|$. 

The preceding discussion motivates us to define  
\begin{align}f_{m, n}(z) = -\frac{1}{n} \sum \limits_{i = 1}^m \log(z-a_i)  +  \frac{1}{n} \sum \limits_{j=1}^n \log(1-zb_j) + \bigg(\gamma_{m, n}+\frac{m}{n}\bigg) \log z\label{flucE25}\end{align}
for $z \in \bbC \smallsetminus (-\infty, \max_{i \in [m]} a_i] \smallsetminus [\min_{j \in [n]} 1/b_j, \infty)$, where the logarithms are principal branches. The real and the imaginary parts of $f_{m, n}$ are, respectively, given by 
\begin{align}
u_{m, n}(z) &=  -\frac{1}{n} \sum \limits_{i = 1}^m \log|z-a_i|  +  \frac{1}{n} \sum \limits_{j=1}^n \log|1-zb_j| + \bigg(\gamma_{m, n}+\frac{m}{n}\bigg) \log |z| \label{flucE4}\\
v_{m, n}(z) &= -\frac{1}{n} \sum \limits_{i = 1}^m \arg(z-a_i)  +  \frac{1}{n} \sum \limits_{j=1}^n \arg(1-zb_j) + \bigg(\gamma_{m, n}+\frac{m}{n}\bigg) \arg z. \label{flucE134}
\end{align}
We also compute the derivative
\begin{align}
f_{m, n}'(z) = -\frac{1}{n} \sum \limits_{i = 1}^m \frac{1}{z-a_i}  -  \frac{1}{n} \sum \limits_{j=1}^n \frac{b_j}{1-zb_j} + \bigg(\gamma_{m, n}+\frac{m}{n}\bigg) \frac{1}{z}. \label{flucE6}
\end{align}
Extend the domains of $u_{m, n}$ and $f_{m, n}'$ to $\bbC \smallsetminus \{0, a_1, \dotsc, a_m, 1/b_1, \dotsc, 1/b_j\}$ through the formulas in (\ref{flucE4}) and (\ref{flucE6}), respectively. It follows from the identity $\nabla (\log |z|) = \dfrac{1}{\conj{z}}$ for $z \in \bbC \smallsetminus \{0\}$ and the chain rule that 
\begin{align}
\label{flucE190}
\nabla u_{m, n}(z) = \conj{f_{m, n}'}(z) \quad \text{ for } z \in \bbC \smallsetminus \{0, a_1, \dotsc, a_m, 1/b_1, \dotsc, 1/b_j\}. 
\end{align}

Let us make the utility of (\ref{flucE25}) more clear. For $s \in \bbR$ and $k \in \bbN$, define 
\begin{align}
p_k'(s) = \frac{p_k(s)-n_k \gamma_{m_k, n_k}}{n_k^{1/3} \sigma_k} = \frac{\lf n_k \gamma_{m_k, n_k} + n_k^{1/3} \sigma_k s \rf - n_k \gamma_{m_k, n_k}}{n_k^{1/3} \sigma_k}, \label{flucEE4}
\end{align}
which is a bounded sequence because $|p_k'(s)| \le |s| + n_k^{-1/3}\sigma_k^{-1}$ and $\sigma_k$ has a positive limit (see Lemma \ref{flucL5} below). Hence, $f_{m_k, n_k}$ captures the leading order behavior of $F_{m_k, n_k, p_k(s)}$ since  
\begin{align}
F_{m_k, n_k, p_k(s)}(z) &= \exp\{n_k f_{m_k, n_k}(z) + n_k^{1/3}\sigma_k p_k'(s) \log z\} 
\end{align}
for $z \in \bbC \smallsetminus (-\infty, \max_{i \in [m]} a_i] \smallsetminus [\min_{j \in [n]} 1/b_j, \infty)$. We also have 
\begin{align}
|F_{m_k, n_k, p_k(s)}(z)| &= \exp\{n_k u_{m_k, n_k}(z) + n_k^{1/3}\sigma_k p_k'(s) \log |z|\} \label{flucE193}
\end{align}
for $z \in \bbC \smallsetminus \{0, a_1, \dotsc, a_m, 1/b_1, \dotsc, 1/b_n\}$. The last equality leads us to study the steepest-descent curves of $u_{m_k, n_k}$, which should presumably approximate well the steepest-descent curves of $\log |F_{m_k, n_k, p_k(s)}|$ for large $k$. 

Fix $m, n \in \bbN$ and the sequences $\bfa$ and $\bfb$. The nature of the steepest-descent curves of $u_{m, n}$ depends on the poles $\sP_{m, n} = \{0, a_1, \dotsc, a_m, 1/b_1, \dotsc, 1/b_n\}$ and the zeros $\sZ_{m, n}$ of $\nabla u_{m, n} = \conj{f'}_{m, n}$. The next lemma locates the zeros.  
\begin{lem}
\label{flucL24}
Let $(\check{a}_i)_{i \in [\check{m}]}$ and $(\check{b}_j)_{j \in [\check{n}]}$ denote the distinct terms of $(a_i)_{i \in [m]}$ and $(b_j)_{j \in [n]}$ in increasing order. $f_{m, n}'$ has a zero of order $1$ in each of the intervals $(\check{a}_i, \check{a}_{i+1})$ for $1 \le i < \check{m}$ and $(1/\check{b}_{j+1}, 1/\check{b}_j)$ for $1 \le j < \check{n}$. $f_{m, n}'$ has a zero of order $2$ at $\zeta_{m, n}$. There are no other zeros. 
\end{lem}
\begin{proof}
Note from (\ref{flucE7}) and (\ref{flucE6}) that 
\begin{align}zf_{m, n}'(z) = -g_{m, n}(z) + \gamma_{m, n}. \label{flucE77} \end{align}
Since $\zeta_{m, n}$ is the minimizer of (\ref{flucE5}), 
\begin{align}
f_{m, n}'(\zeta_{m, n}) = f_{m, n}''(\zeta_{m, n}) = 0. \qquad  \label{flucE18}
\end{align} 
Let $\check{m}_i$ and $\check{n}_j$ be the multiplicities of $\check{a}_i$ and $\check{b}_j$ in $(a_i)_{i \in [n]}$ and $(b_j)_{j \in [m]}$, respectively. We have  
\begin{align}
zf_{m, n}'(z) &= \gamma_{m, n} +1 - \frac{1}{n}\sum \limits_{i=1}^{\check{m}} \frac{\check{m}_i \check{a}_i}{z-\check{a}_i} - \frac{1}{n} \sum \limits_{j=1}^{\check{n}} \frac{\check{n}_j}{1-\check{b}_j z} \label{flucE194}\\
&= \frac{P(z)}{\prod \limits_{i=1}^{\check{m}} (z-\check{a}_i) \prod \limits_{j = 1}^{\check{n}} (1-\check{b}_j z)}, \label{flucE8} 
\end{align}
where $P$ is a polynomial of degree $\check{m}+\check{n}$ that is nonzero at $0$ (since (\ref{flucE194}) equals $\gamma_{m, n} + m/n > 0$ for $z = 0$), $\check{a}_i$ for $i \in [\check{m}]$ and $1/\check{b}_j$ for $j \in [\check{n}]$. Hence, the zeros of (\ref{flucE194}), $P$ and $f_{m, n}'$ are the same. There is at least one zero of $P$ in each of the intervals $(\check{a}_i, \check{a}_{i+1})$ and $(1/\check{b}_{j+1}, 1/\check{b}_{j})$. This is because (\ref{flucE194}) is continuous and real-valued on $\bbR \smallsetminus \sP_{m, n}$, and its limits at the endpoints of each of these intervals are infinities with opposite signs. Hence, we have counted (with multiplicity) at least $2+\check{m}-1+\check{n}-1 = \check{m}+\check{n} = \deg P$ zeros of $P$. Hence, these are the all zeros.   
\end{proof}

Note from (\ref{flucE4}) that $u_{m, n}$ is harmonic since it equals the linear combination of various functions obtained from the harmonic function $z \mapsto \log |z|$ via translations and dilations. In the remainder of this section, we will rely heavily on Subsections \ref{S2Sub1} and \ref{S2Sub2}, which develop precisely the notion of steepest-descent curves of harmonic functions that emanate from a point. 

Using (\ref{flucE77}) and (\ref{flucE18}), we compute
\begin{align}
f_{m, n}'''(\zeta_{m, n}) &= - \frac{g_{m, n}''(\zeta_{m, n})}{\zeta_{m, n}} = -\frac{2\sigma_{m, n}^3}{\zeta_{m, n}^3} < 0. \label{flucE137}
\end{align}
Take $u = u_{m, n}$, $f = f_{m, n}$, $a = \zeta_{m, n}$ and $\xi = e^{\bfi \pi /3}$ in the appendix, see the paragraph preceding (\ref{flucE112}). Let $\Phi_{m, n}: [0, T_{m, n}) \rightarrow (\bbC \smallsetminus (\sP_{m, n} \cup \sZ_{m, n})) \cup \{\zeta_{m, n}\}$ and $\Psi_{m, n}: [0, S_{m, n}) \rightarrow (\bbC \smallsetminus (\sP_{m, n} \cup \sZ_{m, n})) \cup \{\zeta_{m, n}\}$ denote, respectively, the curves $\varphi_3^+$ and $\varphi_2^+$ defined by (\ref{flucE110}) (with $n=3$).  
By Proposition \ref{flucL6} and the remarks surrounding it, $\Phi_{m, n}$ and $\Psi_{m, n}$ are, respectively, steepest-descent and -ascent curves of $u_{m, n}$ that emanate from $\zeta_{m, n}$, that is, 
\begin{align}
\Phi_{m, n}'(t) &= - \frac{\nabla u_{m, n}(\Phi_{m, n}(t))}{|\nabla u_{m, n}(\Phi_{m, n}(t))|} \quad \text{ for } t \in (0, T_{m, n})\label{flucE19}\\ 
\Psi_{m, n}'(t) &= \frac{\nabla u_{m, n}(\Psi_{m, n}(t))}{|\nabla u_{m, n}(\Psi_{m, n}(t))|} \qquad \text{ for } t \in (0, S_{m, n}). \label{flucE189}\\
\Phi_{m, n}(0) &= \Psi_{m, n}(0) = \zeta_{m, n} \label{flucE184}
\end{align}
The endpoints $T_{m, n}, S_{m, n} \in (0, \infty]$ are maximal in the sense that there are no proper extensions of $\Phi_{m, n}$ and $\Psi_{m, n}$ satisfying (\ref{flucE19}), (\ref{flucE189}) and (\ref{flucE184}). It follows from (\ref{flucE64}), $\Phi_{m, n}'$ and $\Psi_{m, n}'$ are continuous up to the endpoint $0$ and, by (\ref{appE4}), 
\begin{align}
\Phi_{m, n}'(0) &= e^{\bfi2\pi/3} \qquad \Psi_{m, n}'(0) = e^{\bfi \pi/3}. \label{flucE39}
\end{align}
Also, note that $|\Phi_{m, n}'(t)| = 1$ for $t \in [0, T_{m, n})$ and $|\Psi_{m, n}'(t)| = 1$ for $t \in [0, S_{m, n})$. In particular, the lengths of $\Phi_{m, n}$ and $\Psi_{m, n}$ are $T_{m, n}$ and $S_{m, n}$, respectively. 

We mention a few other steepest-descent and -ascent curves of $u_{m, n}$, which are illustrated in Figure \ref{flucF2} below along with $\Phi_{m, n}$ and $\Psi_{m, n}$. Observe from (\ref{flucE6}) that $\conj{f_{m, n}'}(z) = f_{m, n}'(\conj{z})$ for $z \in \bbC \smallsetminus \sP_{m, n}$. Hence, by (\ref{flucE190}) and that $\dd/\dd t$ commutes with complex conjugation, the curve $t \mapsto \conj{\Phi}_{m, n}(t)$ for $t \in [0, T_{m, n})$ also satisfies the ODE in (\ref{flucE19}), and $\conj{\Phi}_{m, n}(0) = \zeta_{m, n}$. Thus, $\conj{\Phi}_{m, n}$ is a steepest-descent curve of $u_{m, n}$ that emanates from $\zeta_{m, n}$. By (\ref{appE4}) and $\conj{\Phi_{m, n}'}(0) = e^{-\ii 2\pi/3}$, $\conj{\Phi}_{m, n}$ corresponds to $\varphi_2^-$ defined in (\ref{appE5}). Similarly, the curve $\conj{\Psi}_{m, n}$ defined on $[0, S_{m, n})$ is a steepest-ascent curve of $u_{m, n}$ that emanates from $\zeta_{m, n}$ and corresponds to $\varphi_3^-$. Since $\zeta_{m, n}$ is real and $\nabla u_{m, n} = \conj{f_{m, n}'}$ is real-valued on $\bbR \smallsetminus \sP_{m, n}$, we can apply Proposition \ref{flucL19}. Due to the sign in (\ref{flucE137}) being negative, the proposition states that the curves $t \mapsto t+\zeta_{m, n}$ for $t \in [0, \min_{j \in [n]} 1/b_j - \zeta_{m, n})$ and $t \mapsto -t+\zeta_{m, n}$ for $t \in [0, \zeta_{m, n}-\max_{i \in [m]} a_i]$ are, respectively, steepest-descent and -ascent curves of $u_{m, n}$ that emanate from $\zeta_{m, n}$. These curves correspond to $\varphi_1^+$ and $\varphi_1^-$, respectively. We also consider the situation at $x \in \sZ_{m, n} \smallsetminus \{\zeta_{m, n}\}$. If $x < \zeta_{m, n}$ then, by Lemma \ref{flucL24}, there exist unique $i, i' \in [m]$ such that $(a_{i}, a_{i'}) \cap (\sP_{m, n} \cup \sZ_{m, n}) = \{x\}$. Since the limits of $u_{m, n}(z)$ as $z \rightarrow a_i$ and $z \rightarrow a_{i'}$ are $+\infty$, the minimum of $u_{m, n}$ over $(a_i, a_{i'})$ occurs at $x$. This and $f_{m, n}''(x) \neq 0$ (by Lemma \ref{flucL24}) implies that $f_{m, n}''(x) < 0$. Then, by Propositions \ref{flucL25} and \ref{flucL19} (now $n = 2$), the curves $t \mapsto x+t$ for $t \in [0, a_{i'}-x)$ and $t \mapsto x-t$ for $t \in [0, x-a_i)$ are the all steepest-ascent curves of $u_{m, n}$ that emanate from $x$. Similarly, if $x > \zeta_{m, n}$ then $(1/b_j, 1/b_{j'}) \cap (\sP_{m, n} \cup \sZ_{m, n})$ for some $j, j' \in [n]$ and $f_{m, n}''(x) > 0$. Hence, the curves $t \mapsto x+t$ for $t \in [0, 1/b_{j'}-x)$ and $t \mapsto x-t$ for $t \in [0, x-1/b_j)$ are the all steepest-descent curves of $u_{m, n}$ that emanate from $x$.

We next justify through Lemma \ref{flucL4} below that $\Phi_{m, n}$ and $\Psi_{m, n}$ both stay in $\bbH$ for positive time values, $\Phi_{m, n}$ ends at $0$ in finite time (i.e. has finite length), and $\Psi_{m, n}$ goes off to infinity as depicted in Figure \ref{flucF2}. (That $\Phi_{m, n}((0, T_{m, n}))$ and $\Psi_{m, n}((0, S_{m, n}))$ do not intersect is because $u_{m, n} \circ \Phi_{m, n}$ is decreasing but $u_{m, n} \circ \Psi_{m, n}$ is increasing). We will only use the properties of $\Phi_{m, n}$ in the sequel. 

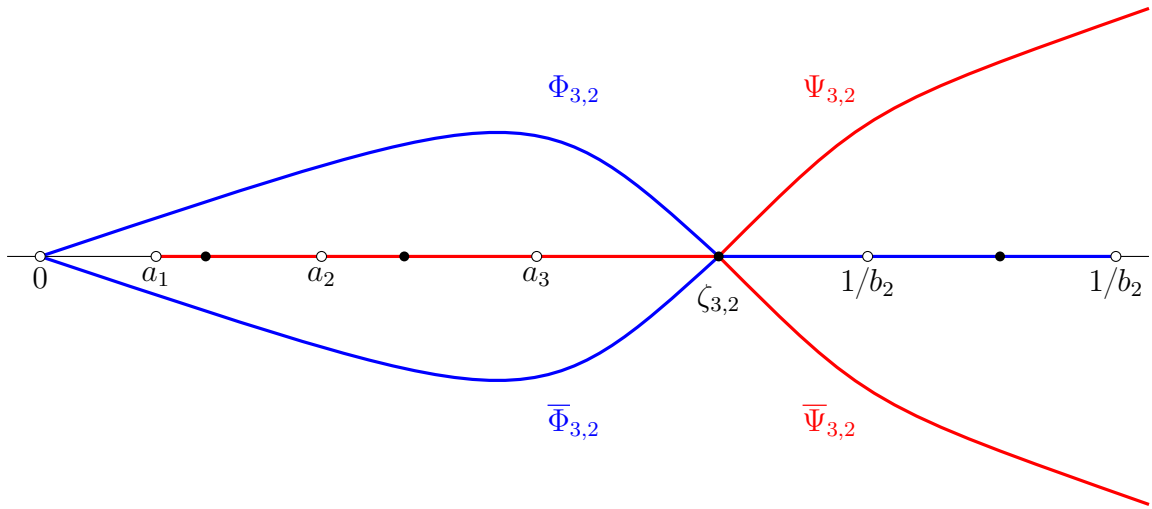
\begin{figure}[h]
\centering
\begin{tikzpicture}[scale = 2.2]
\draw[-](0.8, 0)--(1,0)node[below]{$0$}--(1.7, 0)node[below]{$a_1$}--(2.7, 0)node[below]{$a_2$}--(4, 0)node[below]{$a_3$}--(6, 0)node[below]{$1/b_2$}--(7.5, 0)node[below]{$1/b_2$}--(7.7, 0);
\node at (5.1, -0.1)[below]{$\zeta_{3, 2}$};
%\draw[dashed, thick](6, 0)..controls (5.7, 0.5)..(5.3, 0.83);
%\draw[dashed,  thick](6, 0)..controls (5.7, -0.5)..(5.3, -0.83);
%\draw[dashed,  thick] (5.3, 0.83)..controls (5, 1)..(1, 0);
%\draw[dashed, thick] (5.3, -0.83)..controls (5, -1)..(1, 0);
\draw[blue, very thick] (5.1, 0)..controls(4, 1)..(1, 0);
\draw[red, very thick] (5.1, 0)..controls(6, 0.9)..(7.7, 1.5);
\draw[blue, very thick] (5.1, 0)..controls(4, -1)..(1, 0);
\draw[red, very thick] (5.1, 0)..controls(6, -0.9)..(7.7, -1.5);

\draw[red, very thick](1.7, 0)--(2.7, 0);
\draw[red, very thick](4, 0)--(2.7, 0);
\draw[red, very thick](4, 0)--(5.1, 0);
\draw[blue, very thick](6, 0)--(5.1, 0);
\draw[blue, very thick](6, 0)--(7.5, 0);

\fill[white](1,0)circle(0.03);
\draw(1,0)circle(0.03);
\fill[white](1.7,0)circle(0.03);
\draw(1.7,0)circle(0.03);
\fill[white](2.7,0)circle(0.03);
\draw(2.7,0)circle(0.03);
\fill[white](4,0)circle(0.03);
\draw(4,0)circle(0.03);
\fill[white](6,0)circle(0.03);
\draw(6,0)circle(0.03);
\fill[white](7.5,0)circle(0.03);
\draw(7.5,0)circle(0.03);

\fill(5.1,0)circle(0.03);
\fill(2,0)circle(0.03);
\fill(3.2,0)circle(0.03);
\fill(6.8,0)circle(0.03);

%\fill(5.3, 0.83)circle(0.03);
%\node at (5.3, 0.90)[above]{$\Phi_k(\tau_k)$};
%\node at (5.3, -0.90)[below]{$\conj{\Phi_k}(\tau_k)$};
%\fill(5.3, -0.83)circle(0.03);
%\draw[-, red, thick](6, 0)--(5.3, 0.83);
%\draw[-, red, thick](6, 0)--(5.3, -0.83);
%\draw[blue] (6.2, 0) arc (0: 120: 0.2);
%\node at (6.2, 0.1)[above]{{\color{blue}$2\pi/3$}}; 
%\draw[red] (6.5, 0) arc (0: 130: 0.5);
%\node at (6.5, 0.4)[above]{{\color{red}$2\pi/3+\phi$}}; 
\node at (4, 1.0)[right]{{\color{blue}$\Phi_{3, 2}$}};
\node at (6, 1.0)[left]{{\color{red}$\Psi_{3, 2}$}};
\node at (4, -1.0)[right]{{\color{blue}$\conj{\Phi}_{3, 2}$}};
\node at (6, -1.0)[left]{{\color{red}$\conj{\Psi}_{3, 2}$}};
\end{tikzpicture}
\caption[An illustration of steepest-descent and -ascent curves]{\small{An illustration of some steepest-descent (blue) and -ascent (red) curves of $u_{m, n}$ that emanate from the zeros (black dots) of $\nabla u_{m, n}$ for $m = 3$, $n = 2$, $a_1 < a_2 < a_3$ and $b_1 < b_2$.}}\label{flucF2}
\end{figure}
\begin{lem}
\label{flucL4} \ 
\begin{enumerate}[(a)]
\item $T_{m, n} < \infty$, $\lim_{t \rightarrow T_{m, n}} \Phi_{m, n}(t) = 0$ and $\Phi_{m, n}((0, T_{m, n})) \subset \bbH$.  
\item $S_{m, n} = \infty$, $\lim_{t \rightarrow \infty} |\Psi_{m, n}(t)| = \infty$ and $\Psi_{m, n}((0, \infty)) \subset \bbH$. 
\end{enumerate}
\end{lem}
%See Figure \ref{flucF3}. 
Extend $\Phi_{m, n}$ to $[0, T_{m, n}]$ by setting $\Phi_{m, n}(T_{m, n}) = 0$. By Lemma \ref{flucL4}a, the curve $\Gamma_{m, n}: [0, 2T_{m, n}] \rightarrow \bbC$ given by 
\begin{align}
\Gamma_{m, n}(t) = 
\begin{cases}\Phi_{m, n}(t) \quad &\text{ for } t \in [0, T_{m, n}] \\
\conj{\Phi}_{m, n}(2T_{m, n}-t) \quad &\text{ for } t \in (T_{m, n}, 2T_{m, n}]
\end{cases} 
\end{align}
is simple, closed, oriented counter-clockwise and piecewise $\sC^1$, and encloses $a_i$ for $i \in [m]$. Therefore, $\Gamma_{m, n}$ can also serve as the contour of integration in (\ref{flucE192}) i.e. 
\begin{align}
\label{flucE89}
I_{m_k, n_k, p_k(s)} = \frac{1}{2\pi \ii} \oint \limits_{\Gamma_{m_k, n_k}} F_{m_k, n_k, p_k(s)}(z)\ \dd z. 
\end{align}
Since $\Phi_{m_k, n_k}$ and $\conj{\Phi}_{m_k, n_k}$ are both steepest-descent curves of $u_{m_k, n_k}$ that emanate from $\zeta_{m_k, n_k}$ and by (\ref{flucE193}), we expect that the analysis of (\ref{flucE89}) reduces to a small neighborhood of $\zeta_{m_k, n_k}$ for large $k$. In Section \ref{flucS6}, we will see that this is indeed the case. 

\begin{proof}[Proof of Lemma \ref{flucL4}]
Let $t_0 \in (0, T_{m, n})$ and put $z_0 = \Phi_{m, n}(t_0)$. The curve with the reverse orientation $\tilde{\Phi}:[0, t_0) \rightarrow (\bbC \smallsetminus (\sP_{m, n} \cup \sZ_{m, n})) \cup \{z_0\}$ defined by $\tilde{\Phi}(t) = \Phi_{m, n}(t_0-t)$ for $t \in [0, t_0)$ is a steepest-ascent curve of $u_{m, n}$ that emanate from $z_0$, see the remarks following Definition \ref{flucD1}. (Although we will not make use of it, note that the domain of $\tilde{\Phi}$ is maximal since $\lim_{t \rightarrow t_0} \tilde{\Phi}(t) = \zeta_{m, n} \in \sZ_{m, n}$). The image of $\tilde{\Phi}$ contains points from $\bbH$ due to (\ref{flucE39}) (also see Proposition \ref{flucL10}b). If $z_0 \in \bbR \smallsetminus (\sP_{m, n} \cup \sZ_{m, n})$ then $\tilde{\Phi}$ is the unique steepest-ascent curve that emanates from $z_0$ by Proposition \ref{flucL25} and, since $\nabla u$ is real-valued on $\bbR \smallsetminus \sP_{m, n}$, the image of $\tilde{\Phi}$ is contained in $\bbR$ by Proposition \ref{flucL19}. We conclude from this contradiction, the arbitrariness of $z_0 = \Phi_{m, n}(t_0)$ and $\Phi_{m, n}(t) \in \bbH$ for small values of $t$ that $\Phi_{m, n}((0, T_{m, n})) \subset \bbH$. Similar reasoning shows that $\Psi_{m, n}((0, S_{m, n})) \subset \bbH$ as well. 

Note from (\ref{flucE4}) that $u_{m, n}(z)-(1+\gamma_{m, n})\log |z|$ is bounded for $|z| \ge R$ for some $R > 0$. Then, since $u_{m, n} \circ \Phi_{m, n}$ is decreasing, $\Phi_{m, n}$ remains inside $\cl{\Disc}(0, R)$ provided that $R$ is sufficiently large. Hence, by Proposition \ref{flucP1}, $T_{m, n} < \infty$ and $\lim_{t \rightarrow T_{m, n}} \Phi_{m, n}(t) = x_0$ for some $x_0 \in (\sP_{m, n} \cup \sZ_{m, n}) \smallsetminus \{\zeta_{m, n}\}$. 

We argue that $x_0 \in \sZ_{m, n}$ with $x_0 < \zeta_{m, n}$ leads to a contradiction. Note that, since $T_{m, n} < \infty$ and $x_0$ is in the domain of $u_{m, n}$, the curve $\tilde{\Phi}$ can be defined with $t_0 = T_{m, n}$ and $z_0 = x_0$, and is a steepest-ascent curve of $u_{m, n}$ that emanate from $x_0$. However, as shown in Figure \ref{flucF2} and pointed out in the preceding discussion as a consequence of Propositions \ref{flucL25} and \ref{flucL19}, the images of both steepest-ascent curves that emanate from $x_0$ (there are two of them) are contained in $\bbH$. We then have a contradiction because $\tilde{\Phi}$ passes through nonreal points. By a similar reasoning, $\Psi_{m, n}$ cannot converge to a point in $\sZ_{m, n}$ greater than $\zeta_{m, n}$. 

Define the Jordan curve (simple, closed curve) $J: [0, T_{m, n}+|x_0-\zeta_{m, n}|] \rightarrow \bbC$ by $J(t) = \Phi_{m, n}(t)$ for $t \in [0, T_{m, n})$ and $J(t) = x_0 - \sgn(x_0-\zeta_{m, n})t$ for $t \in [T_{m, n}, T_{m, n}+|x_0-\zeta_{m, n}|]$. The image of $J$ consists of the image of $\Phi_{m, n}$ and the line segment from $\zeta_{m, n}$ to $x_0$. 
We now rule out that $x_0 > \zeta_{m, n}$. If this were the case then $\Psi_{m, n}$ would remain in the interior of $J$ because of (\ref{flucE39}) and that $\Psi_{m, n}((0, S_{m, n}))$ does not intersect the image of $J$. 
Then, since $u_{m, n} \circ \Psi_{m, n}$ is increasing and $\lim_{z \rightarrow 1/b_j} u_{m, n}(z) = -\infty$ for $j \in [n]$, Proposition \ref{flucP1} forces that $\lim_{t \rightarrow S_{m, n}} \Psi_{m, n}(t) = y_0$ for some $y_0 \in \sZ_{m, n}$ with $y_0 > \zeta_{m, n}$. However, this is not possible either as noted at the end of the preceding paragraph.  

The conclusion is that $x_0 = 0$. Now $\Psi_{m, n}$ is in the exterior of $J$ and in $\bbH$. Therefore, $\Psi_{m, n}$ remains a positive distance away from all points of $\sP_{m, n} \cup \sZ_{m, n}$ less than $\zeta_{m, n}$. Since $\Psi_{m, n}$ also does not converge to any point in $\sP_{m, n} \cup \sZ_{m, n}$ greater or equal to $\zeta_{m, n}$, the only remaining possibility is that $S_{m, n} = \infty$ and $\lim_{t \rightarrow S_{m, n}} |\Psi_{m, n}| = \infty$, by Proposition \ref{flucP1}. 
\end{proof}

\section{Uniform control of the steepest-descent curves}
\label{flucS5}

For brevity, let us write $k$ in place of $m_k, n_k$ in the subscripts e.g. $\Phi_{k}$ stands for $\Phi_{m_k, n_k}$. In this section, we first show that various quantities defined with the empirical distributions $\alpha_{m_k}$ and $\beta_{n_k}$ a.s.\ converge as $k \rightarrow \infty$ to their counterparts defined with $\alpha$ and $\beta$. We then obtain upper bounds that are uniform in $k$ for $|\Phi_k|$ and the length of $\Phi_k$. The lemmas developed here will enable us to 
estimate the integral (\ref{flucE89}) uniformly in $k$ (see Lemma \ref{flucL7}). 

\subsection{Various limits}
This subsection computes the limits of the sequences $g_k, \zeta_k, \gamma_k, \sigma_k$ and $f_k$, and introduces constants that will be used in the sequel. Some of these convergences do not require assumptions (\ref{flucA3}) or (\ref{flucA4}). We elaborate on this point at the end of the subsection.

\begin{lem}
\label{flucL31}
Let $K \subset \bbC$ be nonempty, compact and disjoint from $[0, \bar{\alpha}] \cup [1/\bar{\beta}, \infty)$. Then $\lim \limits_{k \rightarrow \infty} \sup \limits_{z \in K} |g_k(z)-g(z)| = 0$ for a.e. $(\bfa, \bfb)$.
\end{lem}
\begin{proof}
Put $\delta = \dist\{K, [0, \bar{\alpha}] \cup [1/\bar{\beta}, \infty)\}$ and $M = \sup_{z \in K} |z|$. Note that 
\begin{align}
\sup_{z \in K} \sup_{a \in [0, \bar{\alpha}]} \left|\partial_a \left(\frac{a}{z-a}\right)\right| &= \sup_{z \in K} \sup_{a \in [0, \bar{\alpha}]} \frac{|z|}{|z-a|^2} \le \frac{M}{\delta^2} < \infty \nonumber\\
\sup_{z \in K} \sup_{b \in [0, \bar{\beta}]} \left|\partial_b \left(\frac{bz}{1-bz}\right)\right| &= \sup_{z \in K} \sup_{b \in [0, \bar{\beta}]} \frac{|z|}{|1-bz|^2} \nonumber\\ 
&=  M \sup_{z \in K} \sup_{b \in [0, \bar{\beta}]} \max \bigg\{\frac{\one_{\{b \le (M+1)^{-1}\}}}{|1-bz|^2}, \frac{\one_{\{b > (M+1)^{-1}\}}}{b^2|1/b-z|^2}\bigg\} \nonumber \\
&\le M(M+1)^2\max\left\{1, \frac{1}{\delta^2}\right\} < \infty. \nonumber
\end{align}
Applying Lemma \ref{flucAL1} twice, we obtain 
\begin{align*}
\lim_{k \rightarrow \infty} \frac{1}{m_k} \sum_{i=1}^{m_k} \frac{a_i}{z-a_i} = \int \limits_{(0, 1)} \frac{a\alpha(\dd a)}{z-a} \qquad
\lim_{k \rightarrow \infty} \frac{1}{n_k} \sum_{j=1}^{n_k} \frac{b_jz}{1-b_jz} = \int \limits_{(0, 1)} \frac{bz\beta(\dd b)}{1-bz}
\end{align*}
uniformly for $z \in K$ a.s. Also because $\lim_{k \rightarrow \infty} m_k/n_k = r$, the claim follows.
\end{proof}

\begin{lem}
\label{flucL5} 
$\lim \limits_{k \rightarrow \infty} \zeta_k = \zeta$, $\lim \limits_{k \rightarrow \infty} \gamma_k = \gamma$ and $\lim \limits_{k \rightarrow \infty} \sigma_k = \sigma$ for a.e. $(\bfa, \bfb)$. 
%{\color{red}$\lim_{k \rightarrow \infty} \zeta_k^{\pm} = \zeta^{\pm}$, $\lim_{k \rightarrow \infty} \gamma_k = \gamma$ and $\lim_{k \rightarrow \infty} \sigma_k = \sigma$ for a.e. $(\bfa, \bfb)$. Furthermore, when $x > 0$, $\lim_{k \rightarrow \infty} \sigma_k^{-} = \sigma^-$ if (\ref{flucA5}) and $\lim_{k \rightarrow \infty} \sigma_k^{+} = \sigma^+$ if (\ref{flucA6}). }  
\end{lem}
\begin{proof}
Let $\delta > 0$ satisfy $(\zeta-\delta, \zeta+\delta) \subset (\bar{\alpha}, 1/\bar{\beta})$, and $S = \{k \in \bbN: \zeta_k \ge \zeta+\delta\}$. Note that $g_k(\zeta+\delta/2) > g_k(\zeta+\delta)$ for $k \in S$. If $S$ is infinite with positive probability then Lemma \ref{flucL31} implies that $g(\zeta+\delta/2) \ge g(\zeta+\delta)$, which contradicts that $g$ is (strictly) increasing on $[\zeta, 1/\bar{\beta})$. Hence, $S$ is a.s.\ finite. Since $\delta$ is arbitrary, we conclude that $\limsup_{k \rightarrow \infty} \zeta_k \stackrel{\text{a.s.}}{\le} \zeta$. Similarly, we obtain $\liminf_{k \rightarrow \infty} \zeta_k \stackrel{\text{a.s.}}{\ge} \zeta$. Hence, $\lim_{k \rightarrow \infty} \zeta_k \stackrel{\text{a.s.}}{=} \zeta$. Now, $\zeta_k \in [\zeta-\delta, \zeta+\delta]$ for sufficiently large $k$ a.s., and   
\begin{align*}
|\gamma - \gamma_{k}| &= |g(\zeta)-g_{k}(\zeta_{k})| \le |g(\zeta)-g(\zeta_{k})| + |g(\zeta_{k})-g_{k}(\zeta_{k})| \\
&\le |g(\zeta)-g(\zeta_{k})| + \sup \limits_{z \in [\zeta-\delta, \zeta+\delta]}|g(z)-g_{k}(z)|. 
\end{align*}
The last expression vanishes as $k \rightarrow 0$ a.s.\ by the continuity of $g$ at $\zeta$ and Lemma \ref{flucL31}. 
Therefore, $\lim_{k \rightarrow \infty} \gamma_{k} \stackrel{\text{a.s.}}{=} \gamma$. Since $g_k$ is holomorphic, Lemma \ref{flucL31} also implies that $\lim_{k \rightarrow \infty} \partial_z^i g_k = \partial_z^i g$ uniformly on compact subsets of $\bbC \smallsetminus [0, \bar{a}] \smallsetminus [1/\bar{\beta}, \infty)$ a.s.\ for $i \in \bbN$. Then $\lim_{k \rightarrow \infty} g_k''(\zeta_k) \stackrel{\text{a.s.}}{=} g''(\zeta)$ since, for sufficiently large $k$,  
\begin{align*}
|g''(\zeta)-g_k''(\zeta_k)| \stackrel{\text{a.s.}}{\le} |g''(\zeta)-g''(\zeta_{k})| + \sup \limits_{z \in [\zeta-\delta, \zeta+\delta]}|g''(z)-g_{k}''(z)|
\end{align*}
Using this, $\lim_{k \rightarrow \infty} \zeta_k \stackrel{\text{a.s.}}{=} \zeta$ and (\ref{flucE128}) gives
\begin{align*}
\lim_{k \rightarrow \infty} \sigma_k = \lim_{k \rightarrow \infty} \bigg(\frac{1}{2}\zeta_k^2 g_k''(\zeta_k)\bigg)^{1/3} \stackrel{\text{a.s.}}{=} \bigg(\frac{1}{2}\zeta^2 g''(\zeta)\bigg)^{1/3} = \sigma, 
\end{align*}
which completes the proof. 
\end{proof}

We next describe the limit of $f_{k}$. Using the principal branches for the logarithms, define 
\begin{align}\label{flucE24}f(z, r)&= - r \int \limits_{(0, 1)} \log(z-a) \alpha(\dd a) +  \int \limits_{(0, 1)} \log(1-zb) \beta(\dd b) + (\gamma(r) +r)\log z
\end{align}
for $z \in \bbC \smallsetminus (-\infty, \bar{\alpha}] \smallsetminus [1/\bar{\beta}, \infty)$. Note from (\ref{flucE25}) that (\ref{flucE24}) is precisely $f_k$ if we set $r = m_k/n_k$ and replace $\alpha$ and $\beta$ with $\alpha_{m_k}$ and $\beta_{n_k}$, respectively. Write $u$ and $v$ for the real and the imaginary parts of $f$. Hence, 
\begin{align}
u(z, r) &=  - r \int \limits_{(0, 1)} \log|z-a| \alpha(\dd a) +  \int \limits_{(0, 1)} \log|1-zb| \beta(\dd b) + (\gamma(r) +r)\log |z| \label{flucE141}\\
v(z, r) &= - r \int \limits_{(0, 1)} \arg(z-a) \alpha(\dd a) +  \int \limits_{(0, 1)} \arg(1-zb) \beta(\dd b) + (\gamma(r) +r)\arg z . \label{flucEE181}
\end{align}
Extend the domain of $u$ and the derivative 
\begin{align}
\label{flucE21}
f'(z) = - r \int \limits_{(0, 1)} \frac{\alpha(\dd a)}{z-a} - \int \limits_{(0, 1)} \frac{b\beta(\dd b)}{1-bz} + \frac{\gamma(r) +r}{z}. 
\end{align}
to $\bbC \smallsetminus [0, \bar{\alpha}] \smallsetminus [1/\bar{\beta}, \infty)$. By (\ref{flucE11}) and (\ref{flucE21}), we have 
\begin{align}zf'(z) = -g(z) + \gamma \quad \text{ for } z \in \bbC \smallsetminus [0, \bar{\alpha}] \smallsetminus [1/\bar{\beta}, \infty). \label{flucE140}\end{align} 
It follows from (\ref{flucE140}) that 
\begin{align}
\label{flucE90}
f'(\zeta) = f''(\zeta) = 0 \qquad f'''(\zeta) = -\frac{g''(\zeta)}{\zeta} = -2\bigg(\frac{\sigma}{\zeta}\bigg)^3, 
\end{align}
Another observation of later use is that $f'$ has no zero other than $\zeta$. 
\begin{lem}
\label{flucL32}
$f'(z) \neq 0$ for $z \in \bbC \smallsetminus [0, \bar{\alpha}] \smallsetminus [1/\bar{\beta}, \infty) \smallsetminus \{\zeta\}$. 
\end{lem}
\begin{proof}
By (\ref{flucE140}), the only zero of $f'$ on the interval $(\bar{\alpha}, 1/\bar{\beta})$ is $\zeta$. Lemma \ref{flucL24} shows that $f_k'$ has no zero in $\bbC \smallsetminus [0, \infty)$ and, by Lemmas \ref{flucL31} and \ref{flucL5}, $\lim_{k \rightarrow \infty}f_k' = f'$ uniformly over compact subsets of $\bbC \smallsetminus [0, \infty)$. Since $f_k'$ are holomorphic, it follows that either $f' = 0$ identically or $f'$ is nonzero on $\bbC \smallsetminus [0, \infty)$ \cite[Exercise~10.20]{Rudin}. The former is not the case because $\lim_{z \rightarrow x} zf'(z) = -g(x)+\gamma < 0$ for any $x \in (\bar{\alpha}, 1/\bar{\beta}) \smallsetminus \{\zeta\}$ since $\zeta$ is the unique minimizer of (\ref{flucE12}). 
\end{proof}

\begin{lem}
\label{flucL30}
Let $K \subset \bbC$ be nonempty and compact. If $K$ is disjoint from $[0, \bar{\alpha}] \cup [1/\bar{\beta}, \infty)$ then, for a.e.\ $(\bfa, \bfb)$,  
$\lim \limits_{k \rightarrow \infty} \sup \limits_{z \in K} |f_k'(z)-f'(z)| = 0$ and $\lim \limits_{k \rightarrow \infty} \sup \limits_{z \in K} |u_k(z)-u(z)| = 0$. If $K$ is disjoint from $(-\infty, \bar{\alpha}] \cup [1/\bar{\beta}, \infty)$ then $\lim \limits_{k \rightarrow \infty} \sup \limits_{z \in K} |v_k(z)-v(z)| = 0$ for a.e. $(\bfa, \bfb)$. 
\end{lem}
\begin{proof}
Suppose that $K$ is disjoint from $[0, \bar{\alpha}] \cup [1/\bar{\beta}, \infty)$. By Lemma \ref{flucL31}, $\lim_{k \rightarrow \infty} g_k(z) = g(z)$ uniformly in $z \in K$ a.s. Also, by Lemma \ref{flucL5}, $\lim_{k \rightarrow \infty} \gamma_k \stackrel{\text{a.s}}{=} \gamma$. Then, using identities (\ref{flucE77}) and (\ref{flucE140}), we obtain 
\begin{align}\label{flucEE1}\lim_{k \rightarrow \infty} f_k'(z) = \lim_{k \rightarrow \infty} \dfrac{-g_k(z)+\gamma_k}{z} = \dfrac{-g(z)+\gamma}{z} = f'(z)\end{align} 
uniformly in $z \in K$ a.s.

We next show that $\lim_{k \rightarrow \infty} u_k(z) = u(z)$ uniformly in $z \in K$ a.s. Since $K$ is compact, there exist $\delta > 0$ and finitely many points $z_1, \dotsc, z_n \in \bbC$ such that $\conj{D}(z_p, \delta)$ is disjoint from $[0, \bar{\alpha}] \cup [1/\bar{\beta}, \infty)$ and $K \subset \bigcup_{p \in [n]} \Disc(z_p, \delta)$. It suffices to show that $\lim_{k \rightarrow \infty} u_k(z) = u(z)$ uniformly in $z \in \conj{\Disc}(z_p, \delta)$ a.s.\ for any $p \in [n]$. For $z \in \Disc(z_p, \delta)$, the line segment $[z_p, z] \subset \Disc(z_p, \delta)$ and, thus, we have 
\begin{align}
|u_k(z)-u(z)| &= \bigg|u_k(z_p)-u(z_p)+\int_{0}^1 \nabla \{u_k-u\}(z_p+t(z-z_p)) \cdot (z-z_p) \dd t\bigg| \nonumber\\
&\le |u_k(z_p)-u(z_p)| + |z-z_p| \int_{0}^1 \bigg| \nabla \{u_k-u\}(z_p+t(z-z_p))\bigg| \dd t \nonumber\\
&\le |u_k(z_p)-u(z_p)| + \delta \max \limits_{\conj{\Disc}(z_p, \delta)} |\nabla u_k-\nabla u| \nonumber\\
&= |u_k(z_p)-u(z_p)| + \delta \max \limits_{\conj{\Disc}(z_p, \delta)} |f_k'-f'| \label{flucEE2}
\end{align}
We claim that the right-hand side tends to $0$ a.s. Note that (\ref{flucEE1}) (using $\conj{\Disc}(z_p, \delta)$ in place of $K$) implies that $\lim_{k \rightarrow \infty} \max \limits_{\conj{\Disc}(z_p, \delta)} |f_k'-f| \stackrel{\text{a.s.}}{=} 0$. The inequalities  
\begin{align*}
\log \delta &\le \log |z_p-a| \le \log (|z_p| + \bar{\alpha}) \quad \text{ for } a \in [0, \bar{\alpha}]
\end{align*}
show that the function $a \mapsto \log |z_p-a|$ is bounded on $[0, \bar{\alpha}]$. Then, 
\begin{align*}
\lim_{k \rightarrow \infty} \frac{1}{m_k} \sum_{i=1}^{m_k} \log |z_p-a_i| &\stackrel{\text{a.s.}}{=} \int \limits_{(0, 1)} \log|z_p-a| \alpha(\dd a), 
\end{align*}
by the ergodicity of $\bfa$. Similarly, for $b \in [0, \bar{\beta}]$,  
\begin{align*}
-\one_{\{b|z_p| \le 1/2\}} \log 2 + \one_{\{b|z_p| > 1/2\}} \log \bigg(\frac{\delta}{2|z_p|}\bigg) &\le \log |1-bz_p| \le \log (1 + \bar{\beta}|z_p|). 
\end{align*} 
(For the lower bound in the case $b|z_p| > 1/2$, use $\log|1-bz_p| = \log b + \log |1/b-z_p|$). Then, by the ergodicity  of $\bfb$,  
\begin{align*}
\lim_{k \rightarrow \infty} \frac{1}{n_k} \sum_{j=1}^{n_k} \log|1-b_j z_p| &\stackrel{\text{a.s.}}{=} \int \limits_{(0, 1)} \log|1-bz_p| \beta(\dd b). 
\end{align*}
Also, $\lim_{k \rightarrow \infty} m_k/n_k = r$ and $\lim_{k \rightarrow \infty} \gamma_k \stackrel{\text{a.s.}}{=} \gamma$. Then, recalling (\ref{flucE4}) and (\ref{flucE141}), we have $\lim_{k \rightarrow \infty} u_k(z_p) \stackrel{\text{a.s.}}{=} u(z_p)$. The claim follows.  

Suppose now that $K$ is disjoint from $(-\infty, \bar{\alpha}] \cup [1/\bar{\beta}, \infty)$. To prove $\lim_{k \rightarrow \infty} v_k(z) = v(z)$ uniformly in $z \in K$ a.s., we only indicate the modifications needed in the preceding paragraph. We now have $\conj{\Disc}(z_p, \delta)$ disjoint from $(-\infty, \bar{\alpha}] \cup [1/\bar{\beta}, \infty)$ for $p \in [n]$. Instead of (\ref{flucEE2}), we have 
\begin{align*}
|v_k(z)-v(z)| \le |v_k(z_p)-v_k(z)| + \delta \max \limits_{\conj{\Disc}(z_p, \delta)} |f_k'-f'|. 
\end{align*}
Since $\arg$ is bounded, 
\begin{align*}
\lim_{k \rightarrow \infty} \frac{1}{m_k} \sum_{i=1}^{m_k} \arg(z_p-a_i) &\stackrel{\text{a.s.}}{=} \int \limits_{(0, 1)} \arg(z_p-a) \alpha(\dd a) \\
\lim_{k \rightarrow \infty} \frac{1}{n_k} \sum_{j=1}^{n_k} \arg(1-b_j z_p) &\stackrel{\text{a.s.}}{=} \int \limits_{(0, 1)} \arg(1-bz_p) \beta(\dd b), 
\end{align*}
by the ergodicity of $\bfa$ and $\bfb$. This implies that $\lim_{k \rightarrow \infty} v_k(z_p) \stackrel{\text{a.s.}}{=} v(z_p)$.
\end{proof}

%Is it true for general r?
%\begin{lem}
%\label{flucL30}
%Let $K \subset \bbC$ be compact and disjoint from $(-\infty, \bar{\alpha}] \cup [1/\bar{\beta}, \infty)$. Then $\lim \limits_{k \rightarrow \infty} \sup \limits_{z \in K} |f_k(z)-f(z)| = 0$ for a.e. $(\bfa, \bfb)$.  
%\end{lem}
%\begin{proof}
%We argue as in the proof of Lemma \ref{flucL31}. Let $\delta = \dist(K, (-\infty, \bar{\alpha}] \cup [1/\bar{\beta}, \infty))$ and $M = \sup_{z \in K} |z|$. Then  
%\begin{align}
%\sup_{z \in K} \sup_{a \in [0, \bar{\alpha}]} \left|\partial_a \log(z-a) \right| &= \sup_{z \in K} \sup_{a \in [0, \bar{\alpha}]} \frac{1}{|z-a|} \le \frac{1}{\delta} < \infty \nonumber\\
%\sup_{z \in K} \sup_{b \in [0, \bar{\beta}]} \left|\partial_b \log(1-bz) \right| &= \sup_{z \in K} \sup_{b \in [0, \bar{\beta}]} \frac{|z|}{|1-bz|} =  M(M+1)\max\left\{1, \frac{1}{\delta}\right\} < \infty. \nonumber
%\end{align}
%Now, by Lemma \ref{flucAL1}, 
%\begin{align*}
%\lim_{k \rightarrow \infty} \frac{1}{m_k} \sum_{i=1}^{m_k} \log(z-a_i) &= \int \limits_{(0, 1)} \log(z-a) \alpha(\dd a) \\
%\lim_{k \rightarrow \infty} \frac{1}{n_k} \sum_{j=1}^{n_k} \log(1-b_j z) &= \int \limits_{(0, 1)} \log(1-bz) \beta(\dd b). 
%\end{align*}
%uniformly for $z \in K$ a.s. Also, $\lim_{k \rightarrow \infty} (\gamma_k + m_k/n_k) \log z = (\gamma + r) \log z$ uniformly for $z \in K$ a.s., by Lemma \ref{flucL5}. 
%\end{proof}

We now specialize the discussion in Appendix \ref{flucAp1} to the case $\varphi = \Phi_{k}$ and $a = \zeta_{k}$. Our purpose is to observe that the counterparts of the constants $K_0$ and $\epsilon_0$ defined by (\ref{flucE112}) are bounded uniformly in $k$. Choose $\rho > 0$ such that $(\zeta-\rho, \zeta+\rho) \subset (\bar{\alpha}, 1/\bar{\beta})$. 
By Lemma \ref{flucL5}, $\lim_{k \rightarrow \infty} \zeta_{k} = \zeta$ a.s. Therefore, we can a.s.\ $k_0 \in \bbN$ such that 
\begin{align}(\zeta_{k}-\rho, \zeta_{k}+\rho) \subset (\bar{\alpha}, 1/\bar{\beta}) \qquad \text{ for } k \ge k_0. \label{flucE146}\end{align} 
For $k \ge k_0$, define 
\begin{align}
K_{0, k} &= \frac{48}{\rho^{4}} \frac{\zeta_k^3}{\sigma_k^{3}} \sup \limits_{\conj{\Disc}(\zeta_k, \rho/2)} \{|f_k|+ \rho|f_{k}'|\}\label{flucE173}\\ 
\epsilon_{0, k} &= \frac{1}{16} \min \bigg\{\rho, \frac{1}{K_{0, k}}, 1\bigg\}. \label{flucE174}
\end{align}
It follows from Lemmas \ref{flucL5} and \ref{flucL30} that there exists a finite, deterministic (not dependent on $\bfa, \bfb$) constant $K_0$ such that 
\begin{align}
K_0 > \sup \limits_{k \ge k_0} K_{0, k}, \label{flucE157}
\end{align}
by choosing $k_0$ larger if necessary. Put
\begin{align}
\epsilon_0 &= \frac{1}{16} \min \bigg\{\rho, \frac{1}{K_{0}}, 1\bigg\} < \inf \limits_{k \ge k_0} \epsilon_{0, k}. \label{flucE158}
\end{align}
From now on, we will restrict to the full probability event of all realizations of $\bfa$ and $\bfb$ for which $k_0$ exists and the preceding inequalities hold. 

We conclude with some remarks on assumptions (\ref{flucA2}), (\ref{flucA3}) and (\ref{flucA4}). Note that the proof of Lemma \ref{flucL31} only uses the ergodicity of $\bfa$ and $\bfb$. We can also prove $\lim_{k \rightarrow \infty} \zeta_k \stackrel{\text{a.s.}}{=} \zeta$ under the ergodicity assumption only by adding a few sentences to the proof of Lemma \ref{flucL5}. By (\ref{flucE196}) and $\zeta_{k} \in (\max_{i \in [m_k]} a_i, \min_{j \in [n_k]}1/b_j)$,  
\begin{align}\bar{\alpha} \stackrel{\text{a.s.}}{\le} \liminf_{k \rightarrow \infty} \zeta_k \le \limsup_{k \rightarrow \infty} \zeta_k \stackrel{\text{a.s.}}{\le} 1/\bar{\beta}. \label{flucE195}\end{align} 
If  (\ref{flucA3}) fails then 
$\lim_{k \rightarrow \infty} \zeta_k \stackrel{\text{a.s.}}{=} \bar{\alpha} = 1/\bar{\beta} = \zeta = 1$. Now suppose that (\ref{flucA3}) holds but $\zeta = \bar{\alpha}$, which means that (\ref{flucA4}) fails. In the proof of Lemma \ref{flucL5}, the argument until the conclusion $\limsup_{k \rightarrow \infty} \zeta_k \stackrel{\text{a.s.}}{\le} \zeta = \bar{\alpha}$ goes through provided that $\delta > 0$ is now chosen to satisfy $\bar{\alpha}+\delta < 1/\bar{\beta}$. Also using the first inequality in (\ref{flucE195}), we conclude that $\lim_{k \rightarrow \infty} \zeta_k \stackrel{\text{a.s.}}{=} \bar{\alpha}$. The remaining case $\zeta = 1/\bar{\beta}$ is treated similarly.   
%
%We can also recover the lemmas of this section if we work with the following deterministic assumptions instead of Assumption \ref{flucA2}. Suppose that $\bfa$ and $\bfb$ are fixed (nonrandom) sequences. Since the measures $\alpha$ and $\beta$ are no longer available, we will take (\ref{flucE196}) as the definition of $\bar{\alpha}$ and $\bar{\beta}$ now. Suppose that 

\subsection{A uniform bound for the distances to the origin}

In this subsection, we show that $\Phi_{k}((0, T_{k})) \subset \Disc(0, \zeta_{k})$ for $k \in \bbN$, and there exists a deterministic constant $c \in (0, 1)$ such that $\Phi_k$, after leaving a small disk of fixed radius around $\zeta_k$, remains inside $\Disc(0, c\zeta_{k})$ for sufficiently large $k$ a.s. These facts are recorded in Lemma \ref{flucL8} below and will come in handy in obtaining the bound on $I_{m_k, n_k, p_k(s)}$ in Lemma \ref{flucL7}b without any upper bound on $s$. We predicted Lemma \ref{flucL8} based on a similar observation made in the proof of \cite[Lemma 2.2]{CorwinLiuWang} for the geometric model with i.i.d.\  weights.   

Let $\epsilon > 0$ and, for $k \in \bbN$, define 
\begin{align}\tau_{k}(\epsilon) = \inf \{t \in [0, T_{k}): |\Phi_{k}(t)-\zeta_{k}| = \epsilon\}. \label{flucE28}\end{align} 
We will suppress the dependence on $\epsilon$ from notation. It follows from Lemma \ref{flucL4}a that $\tau_k < \infty$ provided that $\epsilon < \zeta_k$.  
\begin{lem}\label{flucL2}
If $\epsilon < \epsilon_0/128$ then 
$\tau_{k} < 2\epsilon$ and $|\Phi_{k}(t)-\zeta_{k}| \ge \epsilon/2$ for $t \in [\tau_{k}, T_{k})$ and $k \ge k_0$.  
\end{lem} 
\begin{proof}    
Apply Proposition \ref{flucL10} with $\varphi = \Phi_{k}$ and $a = \zeta_{k}$. Then parts (a) and (d) of the proposition imply the claimed conclusions for $\epsilon < \epsilon_{0, k}/128$. Now the result comes from (\ref{flucE158}). 
\end{proof}

\begin{lem}
\label{flucL8}
$|\Phi_{k}(t)| < \zeta_{k}$ for $t \in (0, T_{k})$ and $k \in \bbN$. Moreover, there exists a deterministic $c \in (0, 1)$ such that, for a.e. $\bfa$ and $\bfb$, there exists $k_0 = k_0^{\bfa, \bfb} \in \bbN$ such that  
\begin{align}
|\Phi_k(t)| \le c\zeta_k \quad \text{ for } t \in [\tau_k, T_k) \text{ and } k \ge k_0. \label{flucE142}
\end{align}
\end{lem}
\begin{proof}
Introduce $x \in (\bar{\alpha}, 1/\bar{\beta})$. Recall from (\ref{flucE24}) that the domain of $v = \Im f$ is $\bbC \smallsetminus (-\infty, \bar{\alpha}] \smallsetminus [1/\bar{\beta}, \infty)$. We first show that the function $\theta \mapsto v(xe^{\ii \theta})$ for $\theta \in [0, \pi)$ is strictly convex. By the chain rule (see the dot product on $\bbC$ defined in Section \ref{intS5}),   
\begin{align}
\frac{\dd}{\dd\theta}v(x e^{\ii \theta}) = \nabla v (x e^{\ii \theta}) \cdot \{\ii x e^{\ii \theta}\} = \{\ii \conj{f'}(x e^{\ii \theta})\} \cdot \{\ii x e^{\ii \theta}\} = \Re \left\{f'(x e^{\ii \theta})x e^{\ii \theta}\right\}  \label{flucEE9}
\end{align}
for $\theta \in [0, \pi)$. Using (\ref{flucE11}), $\gamma = g(\zeta)$, (\ref{flucE140}) and $\alpha\{(\bar{\alpha}, 1)\} = \beta\{(\bar{\beta}, 1)\} = 0$, we note that 
\begin{align}
\Re(zf'(z)) &= \Re(-g(z)+g(\zeta)) \nonumber \\
&= r\int \limits_{(0, \bar{\alpha}]} \Re \{M_1(z)\} \frac{a}{\zeta-a} \alpha(\dd a) +\int \limits_{(0, \bar{\beta}]} \Re \{M_2(z)\} \frac{b}{1-b \zeta} \beta(\dd b), \label{flucE10}
\end{align}
where $M_1(z) = \dfrac{z-\zeta}{z-a}$ and $M_2(z) = \dfrac{\zeta-z}{1-bz}$. Now compute the derivatives 
\begin{align}
\frac{\dd}{\dd \theta} \Re \{M_1(xe^{\ii \theta})\} &= \nabla \{\Re M_1\}(xe^{\ii \theta}) \cdot \{x \ii e^{\ii \theta}\} = \{\conj{M_1'}(x e^{\ii \theta})\} \cdot \{x \ii e^{\ii \theta}\}\nonumber\\
&= -x \Im \{M_1'(xe^{\ii \theta})e^{\ii \theta}\} = \frac{x(x^2-a^2)(\zeta-a)\sin \theta}{|xe^{\ii \theta}-a|^2} \label{flucE182}\\
\frac{\dd}{\dd \theta} \Re \{M_2(xe^{\ii \theta})\} &= -x \Im \{M_2'(xe^{\ii \theta})e^{\ii \theta}\} = \frac{x(1-b^2x^2)(1-b\zeta)\sin \theta}{|1-bxe^{\ii \theta}|^2} \label{flucE183}
\end{align}
Since $x, \zeta \in (\bar{\alpha}, 1/\bar{\beta})$, the last expressions in (\ref{flucE182}) and (\ref{flucE183}) as well as $\dfrac{a}{\zeta-a}$ and $\dfrac{b}{1-b\zeta}$ are positive for $a \in (0, \bar{\alpha}]$, $b \in (0, \bar{\beta}]$ and $\theta \in (0, \pi)$. Hence, we conclude from (\ref{flucE10}) that (\ref{flucEE9}) is increasing in $\theta$, which implies the strict convexity of $\theta \mapsto v(xe^{\ii \theta})$ for $\theta \in [0, \pi)$. 

We have $v(x) = 0$ by (\ref{flucE181}) and $\arg(x-a) = \arg(1-bx) = \arg x = 0$ for $a \in [0, \bar{\alpha}]$ and $b \in [0, \bar{\beta}]$. In particular, $v(\zeta) = 0$. Note from (\ref{flucEE9}) and (\ref{flucE10}) that 
\begin{align}
\frac{\dd}{\dd \theta}v(\zeta e^{\ii \theta}) \bigg|_{\theta = 0} = 0. 
\end{align}
Then, by strict convexity, $\dfrac{\dd}{\dd \theta}v(\zeta e^{\ii \theta}) > 0$ for $\theta \in (0, \pi)$. Therefore, $v(\zeta e^{\ii \theta}) > 0$ for $\theta \in (0, \pi)$. Specializing now to the case $r = m_k/n_k$, $\alpha = \alpha_{m_k}$ and $\beta = \beta_{n_k}$ yields $v_k(\zeta_k e^{\ii \theta}) > 0$ for $\theta \in (0, \pi)$. Since $\Phi_k$ is a stationary curve of $v_k$ that emanate from $\zeta_k$, we also have $v_k(\Phi_k(t)) = v_k(\zeta) = 0$ for $t \in [0, T_k)$. Then we conclude from Lemma \ref{flucL4}a that $\Phi_k((0, T_k)) \subset \Disc(0, \zeta_k)$. This completes the half of the proof. 

To derive (\ref{flucE142}), we may assume that $\epsilon < \dfrac{\epsilon_0}{2048}$, where $\epsilon_0$ is given by (\ref{flucE158}). Apply Proposition \ref{flucL10}b with $\Phi_k$ and $\tau_k$ (recalling from (\ref{flucE184}) and (\ref{flucE39}) that $\xi = e^{\bfi \pi/3}$, $j = 3$ and $n=3$ in the present setting). We obtain 
\begin{align}
|\Phi_k(\tau_k)-\zeta_k-\epsilon e^{\ii 2\pi/3}| \le \frac{256}{\epsilon_0}\epsilon^2 \le \frac{\epsilon}{8}\quad \text{ for } k \in \bbN. 
\end{align}
Then, using $\cos(2\pi/3) = -1/2$, we have 
\begin{align}
\label{flucE185}
\Re \Phi_k(\tau_k) \le \zeta_k - \frac{3\epsilon}{8} \quad \text{ for } k \ge k_0. 
\end{align}
Let $z_k$ denote the unique element in $\Circle(0, \zeta_k-\epsilon/4) \cap \Circle(\zeta_k, \epsilon) \cap \bbH$ for $k \in \bbN$. By an elementary computation that we omit, 
\begin{align}
\Re z_k = \zeta_k - \frac{\epsilon}{4} - \frac{15\epsilon^2}{32\zeta_k} \quad \text{ for } k \in \bbN. \label{flucE186} 
\end{align}
Note from (\ref{flucE158}) that $\epsilon_0 \le \dfrac{\rho}{16} < \dfrac{\zeta_k}{16}$ for $k \ge k_0$. Combining this with (\ref{flucE185}) and (\ref{flucE186}) leads to 
\begin{align}
\Re z_k > \zeta_k - \frac{9\epsilon}{32} > \Re \Phi_k(\tau_k) \quad \text{ for } k \ge k_0. 
\end{align}
Also since $\Phi_k(\tau_k) \in \Circle(\zeta_k, \epsilon)$ (by (\ref{flucE28})), we conclude that $\Phi_k(\tau_k) \in \Disc(0, \zeta_k-\epsilon/4)$ for $k \ge k_0$, see Figure \ref{flucF3} below. Because $\Phi_k(0) = \zeta_k$, by continuity of $|\Phi_k|$, there exist $t_k \in (0, \tau_k)$ and $\theta_k \in (0, \pi)$ such that $\Phi_k(t_k) = (\zeta_k-\epsilon/4)e^{\ii \theta_k}$ for $k \ge k_0$. Hence, 
\begin{align}
v_k((\zeta_k-\epsilon/4)e^{\ii \theta_k}) = v_k(\Phi_k(t_k)) = 0 \quad \text{ for } k \ge k_0. 
\end{align}
Recall now that $\theta \mapsto v_k((\zeta_k-\epsilon/4)e^{\ii \theta})$ for $\theta \in [0, \pi)$ is strictly concave and $v_k((\zeta_k-\epsilon/4)) = 0$. Therefore, $v_k((\zeta_k-\epsilon/4)e^{\ii \theta}) \neq 0$ for $\theta \in (0, \pi) \smallsetminus \{\theta_k\}$ and $k \ge k_0$. Then, since $u_k \circ \Phi_k$ is decreasing, $\Phi_k(t) \in \Circle(0, \zeta_k-\epsilon/4)$ if and only if $t = t_k$ for $k \ge k_0$. Using this, $\tau_k \ge t_k$ and Lemma \ref{flucL4}a, we conclude that $\Phi_k([\tau_k, T_k)) \subset \Disc(0, \zeta_k-\epsilon/4)$ for $k \ge k_0$. Choosing $k_0$ larger if necessary, we also have $\zeta_k < 2\zeta$ for $k \ge k_0$. Hence, 
\begin{equation*}
\frac{\Phi_k(t)}{\zeta_k} < 1-\frac{\epsilon}{4\zeta_k} < 1 - \frac{\epsilon}{8\zeta} \quad \text{ for } t \in [\tau_k, T_k) \text{ and } k \ge k_0. \qedhere
\end{equation*} 
\end{proof}
\begin{figure}[h]
\centering
\begin{tikzpicture}[scale = 2.2]
\draw[-](0.8, 0)--(1,0)node[below]{$0$}--(7.7, 0);
\node at (5.1, -0.1)[below]{$\zeta_{k}$};
%\draw[dashed, thick](6, 0)..controls (5.7, 0.5)..(5.3, 0.83);
%\draw[dashed,  thick](6, 0)..controls (5.7, -0.5)..(5.3, -0.83);
%\draw[dashed,  thick] (5.3, 0.83)..controls (5, 1)..(1, 0);
%\draw[dashed, thick] (5.3, -0.83)..controls (5, -1)..(1, 0);
\draw[black, very thick] (5.1, 0)..controls(4, 1)..(1, 0);
\draw[red, very thick](5.1, 0)circle(1.5);
\draw[blue, very thick](4.725, 0)arc(0:25:4.725);
\draw[blue, very thick](4.725, 0)arc(0:-25:4.725);

\fill[white](1,0)circle(0.03);
\draw(1,0)circle(0.03);

\fill(5.1,0)circle(0.03);
\fill(4.723, 0.345)circle(0.03);
\fill(4.525, 1.375)circle(0.03);
\fill(3.79, 0.75)circle(0.03);
\fill(4.725, 0)circle(0.03);
\node at (3.62, 0.75)[above]{$\Phi_k(\tau_k)$};
\node at (5.85, 0)[below]{$\epsilon$};
\node at (4.2, 0)[below]{$3\epsilon/4$};
\node at (4.725, 0.345)[right]{$\Phi_k(t_k)$};
\node at (4.5, 1.4)[left]{$z_k$};

%\fill(5.3, 0.83)circle(0.03);
%\node at (5.3, 0.90)[above]{$\Phi_k(\tau_k)$};
%\node at (5.3, -0.90)[below]{$\conj{\Phi_k}(\tau_k)$};
%\fill(5.3, -0.83)circle(0.03);
%\draw[-, red, thick](6, 0)--(5.3, 0.83);
%\draw[-, red, thick](6, 0)--(5.3, -0.83);
%\draw[blue] (6.2, 0) arc (0: 120: 0.2);
%\node at (6.2, 0.1)[above]{{\color{blue}$2\pi/3$}}; 
%\draw[red] (6.5, 0) arc (0: 130: 0.5);
%\node at (6.5, 0.4)[above]{{\color{red}$2\pi/3+\phi$}}; 
\end{tikzpicture}
\caption[A figure for the proof of Lemma \ref{flucL8}]{\small{An illustration of $\Phi_k((0, T_k))$ (black), $\Circle(\zeta_k, \epsilon)$ (red) and an arc of $\Circle(0, \zeta_k-\epsilon/4)$ (blue) along with the points $\zeta_k, \Phi_k(\tau_k), z_k$ and $\Phi_k(t_k)$ as described in the proof of Lemma \ref{flucL8}.}}\label{flucF3}
\end{figure}
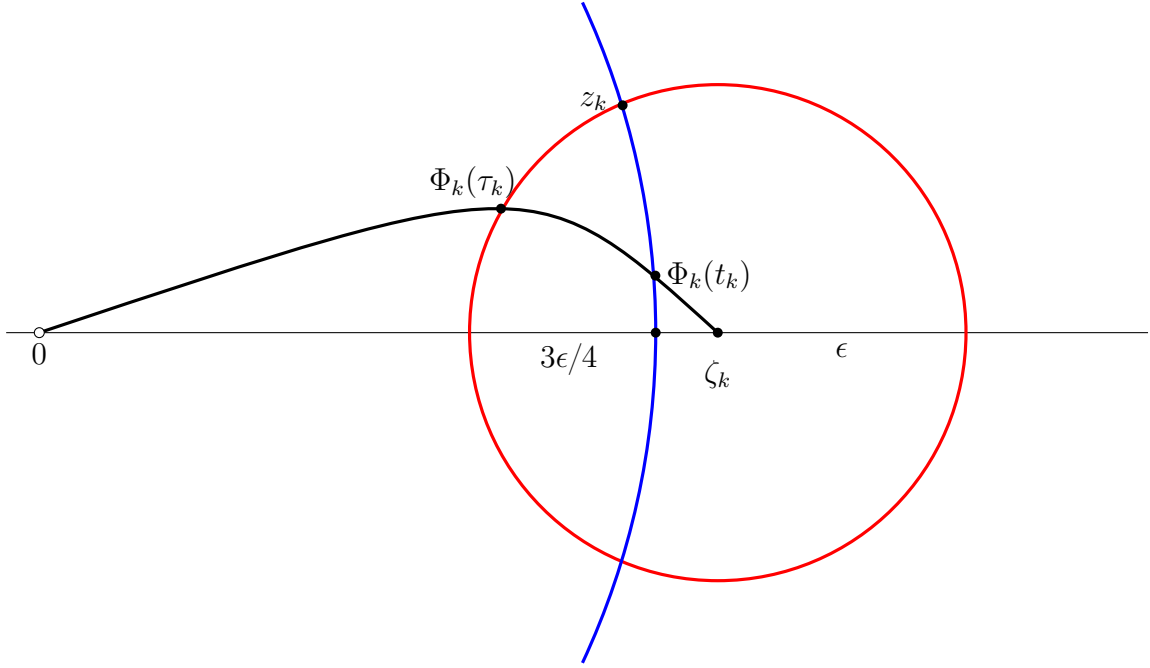

\subsection{A uniform bound for the lengths}
In this subsection, we show that there exists a deterministic constant $C > 0$ such that the length of $\Phi_k$ up to the first time $\Phi_k$ hits a fixed small disk around $0$ is bounded by $C$ for sufficiently large $k$ a.s., see Lemma \ref{flucL9} below. To obtain this result, we will use Proposition \ref{flucL14} together with an upper bound for $|u_k|$ and a lower bound for $|\nabla u_k|$ that are uniform in $k$. Hence, we first argue that $\Phi_k$, until it is within a small disk around $0$, remains a uniform positive distance away from $\sP_{k} \cup \sZ_{k}$. 

For $\epsilon > 0$ and $\theta \in (0, \pi/2)$, define
\begin{align}
E_1(\epsilon, \theta) &= \{0 < \arg(z) < \theta \text{ and } \pi-\theta < \arg(z-\bar{\alpha}+\epsilon/2) < \pi\} \label{flucE144}\\ 
E_2(\epsilon, \theta) &= \{0 < \arg(z-1/\bar{\beta}+\epsilon/2) < \theta\} \label{flucE145} \\
E(\epsilon, \theta) &= \Disc(\bar{\alpha}, \epsilon) \cup \Disc(1/\bar{\beta}, \epsilon) \cup E_1(\epsilon, \theta) \cup E_2(\epsilon, \theta). 
\end{align}
See Figure \ref{flucF1}. The next lemma is adapted from \cite[Lemma 3.4]{GravnerTracyWidom02a}.  
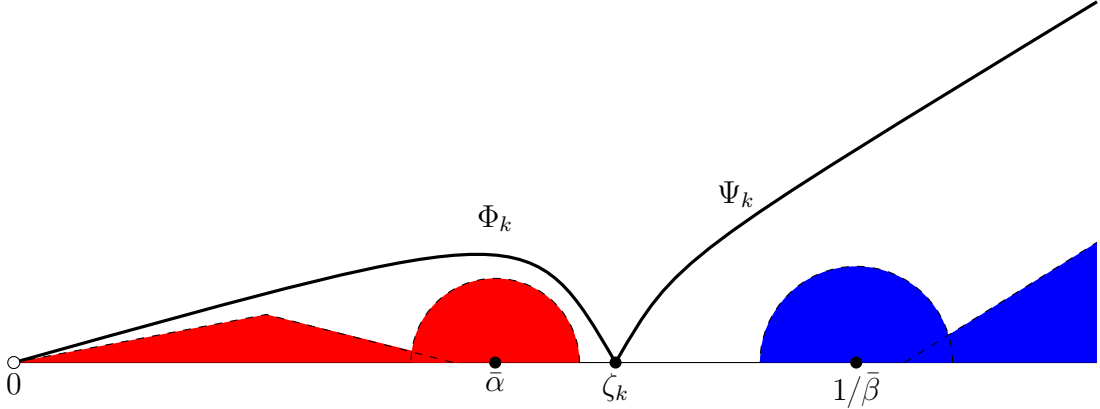
\begin{figure}[h]
\centering
\begin{tikzpicture}[scale = 1.6]
\draw[-](1,0)node[below]{$0$}--(5, 0)node[below]{$\bar{\alpha}$}--(6, 0)node[below]{$\zeta_k$}--(8, 0)node[below]{$1/\bar{\beta}$}--(10, 0);
\fill[red](5.7, 0) arc (0:180:0.7);
\fill[red](4.7, 0)--(3.1, 0.4)--(1, 0);
\draw[dashed](5.7, 0) arc (0:180:0.7);
\draw[dashed](4.65, 0)--(3.1, 0.4)--(1, 0);
\fill[blue](8.8, 0)arc(0:180:0.8);
\fill[blue](8.4, 0)--(10, 1)--(10, 0);
\draw[dashed](8.8, 0)arc(0:180:0.8);
\draw[dashed](8.4, 0)--(10, 1);
%\draw[red, very thick](6, 0)--(3, 0);
\draw[very thick](6, 0)..controls (6.5, 0.865)..(10, 3);
\draw[very thick](6, 0)..controls (5.3, 1.2)..(1, 0);
\node at (7,1.2)[above] {$\Psi_k$};
\node at (5,1)[above] {$\Phi_k$};
\fill[white](1,0)circle(0.05);
\draw(1, 0)circle(0.05);
\fill(5,0)circle(0.05);
\fill(6,0)circle(0.05);
\fill(8,0)circle(0.05);
\end{tikzpicture}
\caption[An illustration of the sets $E_1$ (red) and $E_2$ (blue).]{An illustration of the sets $E_1$ (red) and $E_2$ (blue). The images of the curves $\Phi_k$ and $\Psi_k$ do not intersect $E_1 \cup E_2$. }\label{flucF1}
\end{figure}

\begin{lem} 
\label{flucL3}
There exist $\epsilon_1 > 0$ and $\theta_0 \in (0, \pi/2]$ such that, for a.e. $(\bfa, \bfb)$, there exists $k_0 = k_0^{\bfa, \bfb} \in \bbN$ such that $\Phi_{k}([0, T_k)) \cup \Psi_{k}([0, \infty))$ is disjoint from $E(\epsilon, \theta)$ for $k \ge k_0$, $0 < \epsilon \le \epsilon_1$ and $\theta \in (0, \theta_0]$.
\end{lem}
\begin{proof}
Choose $\epsilon_1 > 0$ such that $\zeta \in (\bar{\alpha}+\epsilon_1, 1/\bar{\beta}-\epsilon_1)$. Since $\lim_{k \rightarrow \infty} \zeta_k \stackrel{\text{a.s.}}{=} \zeta$, for a.e. $(\bfa, \bfb)$, there exists $k_0 = k_0^{\bfa, \bfb} \in \bbN$ such that $\zeta_k$ is not contained in $\conj{\Disc}(\bar{\alpha}, \epsilon_1) \cup \conj{\Disc}(1/\bar{\beta}, \epsilon_1)$ for $k \ge k_0$. From now on we work with the a.s.\ event on which $k_0$ exists, and $a_i \le \bar{\alpha}$ for $i \in \bbN$ and $b_j \le \bar{\beta}$ for $j \in \bbN$.  Let $\epsilon \in (0, \epsilon_1]$. Minimizing each term in (\ref{flucE4}) separately over $z \in \conj{\Disc}(\bar{\alpha}, \epsilon) \smallsetminus \{a_i: i \in [m_k]\}$, we obtain 
\begin{align}
u_{k}(z) &\ge -\frac{1}{n_k} \sum \limits_{i = 1}^{m_k} \log|\bar{\alpha}+\epsilon-a_i|  +  \frac{1}{n_k} \sum \limits_{j=1}^{n_k} \log|1-(\bar{\alpha}+\epsilon)b_j| + \bigg(\gamma_{k}+\frac{m_k}{n_k}\bigg) \log |\bar{\alpha}-\epsilon| \nonumber\\
&= u_{k}(\bar{\alpha}+\epsilon) + \bigg(\gamma_k + \frac{m_k}{n_k}\bigg) \log \left|\frac{\bar{\alpha}-\epsilon}{\bar{\alpha}+\epsilon}\right|. \label{flucE132}
\end{align}
Similarly, maximizing each term in (\ref{flucE4}) separately over $z \in \conj{\Disc}(1/\bar{\beta}, \epsilon) \smallsetminus \{1/b_j: j \in [n_k]\}$ gives 
\begin{align}
u_{k}(z) &\le -\frac{1}{n_k} \sum \limits_{i = 1}^{m_k} \log|1/\bar{\beta}-\epsilon-a_i|  +  \frac{1}{n_k} \sum \limits_{j=1}^{n_k} \log|1-(1/\bar{\beta}-\epsilon)b_j| + \bigg(\gamma_{k}+\frac{m_k}{n_k}\bigg) \log |1/\bar{\beta}+\epsilon| \nonumber\\
&= u_{k}(1/\bar{\beta}-\epsilon) + \bigg(\gamma_k + \frac{m_k}{n_k}\bigg) \log \left|\frac{1/\bar{\beta}+\epsilon}{1/\bar{\beta}-\epsilon}\right|.\label{flucE13}
\end{align}
Note from $(\ref{flucEE181})$ that $v = 0$ on $(\bar{\alpha}, 1/\bar{\beta})$. Hence, by (\ref{flucE140}), $u'(x) = f'(x) < 0$ for $x \in (\bar{\alpha}, 1/\bar{\beta}) \smallsetminus \{\zeta\}$ and, therefore, $u$ is decreasing on $(\bar{\alpha}, 1/\bar{\beta})$. Then, choosing $\epsilon_1$ smaller if necessary,  
\begin{align}
u(\zeta) &< u(\bar{\alpha}+\epsilon) + (\gamma + r) \log \left|\frac{\bar{\alpha}-\epsilon}{\bar{\alpha}+\epsilon}\right|\\
u(\zeta) &> u(1/\bar{\beta}-\epsilon) + (\gamma + r) \log \left|\frac{1/\bar{\beta}+\epsilon}{1/\bar{\beta}-\epsilon}\right|. 
\end{align} 
We also have $\lim_{k \rightarrow \infty} (\gamma_k+m_k/n_k) \stackrel{\text{a.s.}}{=} \gamma+r$ and $\lim_{k \rightarrow \infty} u_k(\zeta_k) \stackrel{\text{a.s.}}{=} u(\zeta)$, by Lemmas \ref{flucL5} and \ref{flucL30}. To see the last convergence, also use $\zeta \in (\bar{\alpha}, 1/\bar{\beta})$ and the inequality 
\begin{align}
|u_k(\zeta_k)-u(\zeta)| \stackrel{\text{a.s.}}{\le} |u(\zeta_k)-u(\zeta)| + \sup_{z \in [\zeta-\delta, \zeta+\delta]} |u_k(z)-u(z)|, 
\end{align}
which holds for sufficiently large $k$ since $\lim_{k \rightarrow \infty} \zeta_k \stackrel{\text{a.s.}}{=} \zeta$. 
It follows that we can restrict further to an a.s.\ event and choose $k_0$ larger if necessary such that  
\begin{align}
u_k(\zeta_k) &< u_k(\bar{\alpha}+\epsilon) + \bigg(\gamma_k + \frac{m_k}{n_k}\bigg) \log \left|\frac{\bar{\alpha}-\epsilon}{\bar{\alpha}+\epsilon}\right|\\
u_k(\zeta_k) &> u_k(1/\bar{\beta}-\epsilon) + \bigg(\gamma_k + \frac{m_k}{n_k}\bigg) \log \left|\frac{1/\bar{\beta}+\epsilon}{1/\bar{\beta}-\epsilon}\right|. 
\end{align} 
Then, for $k \ge k_0$, by (\ref{flucE132}), $u_k(\zeta_k) < u_k(z)$ for $z \in \conj{\Disc}(\bar{\alpha}, \epsilon) \smallsetminus \{a_i: i \in [m_k]\}$ and, by (\ref{flucE13}), $u_k(\zeta_k) > u_k(z)$ for $z \in \conj{\Disc}(1/\bar{\beta}, \epsilon) \smallsetminus \{1/b_j: j \in [n_k]\}$. Hence, for $k \ge k_0$, $\Phi_{k}([0, T_{k}))$ is disjoint from $\conj{\Disc}(\bar{\alpha}, \epsilon)$ because $\Phi_{k}(0) = \zeta_{k}$ and $u_{k} \circ \Phi_{k}$ is decreasing, and $\Psi_{k}([0, \infty))$ is disjoint from $\conj{\Disc}(1/\bar{\beta}, \epsilon)$ because $\Psi_{k}(0) = \zeta_{k}$ and $u_{k} \circ \Psi_{k}$ is increasing.  

By Lemma \ref{flucL4}, the curve $J_k(t): [0, T_{k}+\zeta_{k}] \rightarrow \bbC$ defined by $J_k(t) = \Phi_{k}(t)$ for $t \in [0, T_{k}]$ and $J_k(t) = t-T_{k}$ for $t \in [T_{k}, T_{k}+\zeta_{k}]$ is a Jordan curve, and $\Psi_{k}((0, \infty))$ is in the exterior of $J_k$ because $\Psi_{k}((0, \infty))$ is disjoint from the image of $J_k$ and $\lim_{t \rightarrow \infty} |\Psi_{k}(t)| = \infty$). On the other hand, $\conj{\Disc}(\bar{\alpha}, \epsilon) \cap \bbH$ is in the interior of $J_k$ for $k \ge k_0$. Therefore, $\Psi_{k}([0, \infty))$ is disjoint from $\conj{\Disc}(\bar{\alpha}, \epsilon)$ for $k \ge k_0$. To see that $\Phi_{k}([0, T_k))$ is disjoint from $\conj{\Disc}(1/\bar{\beta}, \epsilon)$ for $k \ge k_0$, argue by contradiction. Suppose this is not the case for some $k \ge k_0$. Then there exists a minimal $\tilde{T} \in (0, T_{k})$ with $\Phi_{k}(\tilde{T}) \in \Circle(1/\bar{\beta}, \epsilon)$. Consider the Jordan curve $\tilde{J}$ the image of which consists of $\Phi_{k}([0, \tilde{T}])$, the arc of the circle $\Circle(1/\bar{\beta}, \epsilon)$ in $\conj{\bbH}$ from $\Phi_{k}(\tilde{T})$ to $1/\bar{\beta}-\epsilon$ and the line segment $[\zeta_{k}, 1/\bar{\beta}-\epsilon]$. Now, $\Psi_{k}((0, \infty))$ does not intersect the image of $\tilde{J}$, and contains points from the interior $\tilde{J}$ 
%because $\Psi_{k}(0) = \zeta_{k} \in (\zeta_{k}, 1/\bar{\beta}-\epsilon)$ and $\Im \Psi_{k}'(0) > 0$ (see (\ref{flucE78})) when $x > 0$ and 
because $\Psi_{k}(0) = \zeta_k = \Phi_k(0)$ and $\Re \Phi_k'(0) < \Re \Psi_k'(0)$ (see (\ref{flucE39}). %when $x = 0$. 
Hence, $\tilde{J}$ encloses $\Psi_{k}((0, \infty))$. However, this is not possible since $\lim_{t \rightarrow \infty} |\Psi_{k}(t)| = \infty$. In conclusion, $\Phi_{k}([0, T_{k}))$ does not intersect $\conj{\Disc}(1/\bar{\beta}, \epsilon)$ for $k \ge k_0$. 

For $z \in \bbH$, we have $\arg (z-a_i) \ge 0$ for $i \in \bbN$ and $\arg (1-b_j z) = \arg(1/b_j -z) \le 0$ for $j \in \bbN$. Furthermore, if $z \in E_1$ and $a_i > \bar{\alpha}-\epsilon/2$ then $\arg(z-a_i) > \pi - \theta$, and if $z \in E_2$ and $1/b_j < 1/\bar{\beta}+\epsilon/2$ then $\arg(1-b_j z) < \theta-\pi$ and $\arg z < \theta$. %See Figure \ref{flucF3}. 
Hence, by (\ref{flucE134}),  
\begin{align}
v_{k}(z) &< (\theta-\pi) \frac{1}{n_k} \sum_{i=1}^{m_k} \one_{\{a_i > \bar{\alpha}-\epsilon/2\}} + \theta \bigg(\gamma_k+\frac{m_k}{n_k}\bigg) \quad \text{ for } z \in E_1 \label{flucE133} \\
v_{k}(z) &<  (\theta-\pi) \frac{1}{n_k} \sum_{j=1}^{n_k} \one_{\{1/b_j < 1/\bar{\beta}+\epsilon/2\}} + \theta \bigg(\gamma_k+\frac{m_k}{n_k}\bigg) \quad \text{ for } z \in E_2 \label{flucE135}
\end{align}
By the ergodicity of $\bfa, \bfb$ and Lemma \ref{flucL5}, the a.s.\ limits of (\ref{flucE133}) and (\ref{flucE135}) as $k \rightarrow \infty$, respectively, are 
\begin{align}
(\theta-\pi)r\alpha\{(\bar{\alpha}-\epsilon/2, 1)\} + \theta(\gamma + r) < 0\nonumber \\
(\theta-\pi)r\beta\bigg\{\bigg(\frac{1}{1/\bar{\beta}+\epsilon/2}, 1\bigg)\bigg\} + \theta(\gamma + r) < 0
\end{align}
where the inequalities hold provided that $\theta$ is small enough. Hence, once more restricting further to an a.s.\ event and choosing $k_0$ larger if necessary, we have 
\begin{align}
v_k(z) < 0 \quad \text{ for } z \in E_1 \cup E_2 \quad \text{ and } \quad k \ge k_0. 
\end{align}
On the other hand, $\Phi_k$ and $\Psi_k$ are stationary curves of $v_k$ by Lemma \ref{flucL1}. Therefore, $v_k(\Phi_k(t)) = v_k(\Psi_k(s)) = v_k(\zeta_k) = 0$ for $t \in [0, T_k)$ and $s \in [0, \infty)$, where the last equality can be seen from (\ref{flucE134}) using $a_i < \zeta_k < 1/b_j$ for $i \in [m_k]$ and $j \in [n_k]$. It follows that $\Phi_{k}([0, T_k)) \cup \Psi_{k}([0, \infty))$ does not intersect $E_1 \cup E_2$ for $k \ge k_0$. 
\end{proof}
In the preceding proof, the conclusion that $\Phi_k((0, T_k))$ does not intersect $\conj{\Disc}(1/\bar{\beta}, \epsilon)$ for $k \ge k_0$ also comes directly from $\Phi_k((0, T_k)) \subset \Disc(0, \zeta_k)$ for $k \in \bbN$ (Lemma \ref{flucL8}) and $\zeta_k < 1/\bar{\beta}-\epsilon$ for $k \ge k_0$. 

We next derive a uniform bound in $k$ for the length of $\Phi_{k}$ up to the first time $\Phi_{k}$ hits a fixed disk around $0$. One would expect a uniform bound for the full length of $\Phi_k$ but we could prove this only in the special case $\ubar{\alpha} > 0$ i.e. when $\liminf_{i \rightarrow \infty} a_i \stackrel{\text{a.s.}}{>} 0$. Let $\delta > 0$ and, for $k \in \bbN$, define    
\begin{align}
\tau_{k}'(\delta) &= \inf \{t \in [0, T_{k}): |\Phi_{k}(t)| = \delta\}. \label{flucE29}
\end{align}
We will drop $\delta$ from notation. It follows from Lemma \ref{flucL4}a that $\tau_{k}' < \infty$ provided that $\delta < \zeta_{k}$ for $k \in \bbN$. 
\begin{lem}
\label{flucL9}
Let $\delta \in (0, \zeta)$. There exists a deterministic $C > 0$ such that, for a.e. $(\bfa, \bfb)$, there exists $k_0 = k_0^{\bfa, \bfb} \in \bbN$ such that $\tau_{k}' < C$ for $k \ge k_0$. 
\end{lem}
\begin{proof}
Recall $\epsilon_1$ and $\theta_0$ from Lemma \ref{flucL3}. Let $0 < \epsilon < \min\{\epsilon_0/128, \epsilon_1\}$ and 
\begin{align}K = \bigg(\conj{\bbH} \cap \conj{\Disc}(0, \zeta+\epsilon)\bigg) \smallsetminus \bigg(E(\epsilon, \theta_0) \cup \Disc(\zeta, \epsilon/4) \cup \Disc(0, \delta)\bigg).\end{align}
By Lemma \ref{flucL5}, for a.e. $(\bfa, \bfb)$ there exists $k_0 = k_0^{\bfa, \bfb} \in \bbN$ such that $\zeta_k \in \Disc(\zeta, \epsilon/4)$ for $k \ge k_0$. We will work with the a.s. event on which $k_0$ exists. Then $\Disc(\zeta, \epsilon/4) \subset \Disc(\zeta_k, \epsilon/2)$ for $k \ge k_0$. Hence, by Lemma \ref{flucL2}, $\Phi_k(t) \not \in \Disc(\zeta, \epsilon/4)$ for $t \in [\tau_k, T_k)$ and $k \ge k_0$ after choosing $k_0$ larger if necessary. By Lemma \ref{flucL4}, $\Phi_k(t) \subset \bbH$ for $t \in [0, T_k)$ and $k \in \bbN$. Also, by Lemma \ref{flucL8}, $\Phi_k(t) \in \conj{\Disc}(0, \zeta_k) \subset \Disc(0, \zeta+\epsilon/4)$ for $t \in [0, T_k)$ and $k \ge k_0$. Finally, by Lemma \ref{flucL3}, $\Phi_k(t) \not \in E(\epsilon, \theta_0)$ for $t \in [0, T_k)$ and $k \ge k_0$ again after choosing $k_0$ larger if necessary. It follows that $\Phi_k(t) \subset K$ for $t \in [\tau_k, \tau_k']$ and $k \ge k_0$. 

Note that $K$ is compact. By Lemma \ref{flucL24}, $\nabla u_{k} = \conj{f_{k}'}$ is nonzero on $K$. Applying Proposition \ref{flucL14} yields 
\begin{align}\tau_{k}'-\tau_{k} \le \dfrac{2\max \limits_{K} |u_{k}|}{\min \limits_{K} |\nabla u_{k}|} \quad \text{ for } k \ge k_0.
\label{flucEE3}
\end{align}
Since $K$ is disjoint from $[0, \bar{\alpha}] \cup [1/\bar{\beta}, \infty)$, by Lemma \ref{flucL30}, $\lim_{k \rightarrow \infty} u_k(z) = u(z)$ and $\lim_{k \rightarrow \infty} \nabla u_k(z) = \nabla u(z)$ both uniformly in $z \in K$ a.s. It follows that 
\begin{align}
\lim_{k \rightarrow \infty} \max_{K} |u_k| &\stackrel{\text{a.s.}}{=} \max_{K} |u| \\
\lim_{k \rightarrow \infty} \min_{K} |\nabla u_k| &\stackrel{\text{a.s.}}{=} \min_{K}|\nabla u| > 0, 
\end{align}
where the inequality is due to Lemma \ref{flucL32}. Combining these with (\ref{flucEE3}), we conclude that, further restricting to an a.s. event and choosing $k_0$ larger if necessary, we have 
\begin{align}
\tau_k'-\tau_k \le \dfrac{2\max \limits_{K} |u_{k}|}{\min \limits_{K} |\nabla u_{k}|}+1 \quad \text{ for } k \ge k_0. 
\end{align}
Also, $\tau_k < 2\epsilon$ for $k \ge k_0$ due to Lemma \ref{flucL2}. Hence, the result. 
\end{proof}
%\noindent One would expect a uniform bound for the lengths of $\Phi_k([0, T_k))$ as well but we could prove this only when $\bfa$ is bounded from below by a positive number.  

\section{Integral asymptotics and estimates} 
\label{flucS6}
In this section, we compute the first order asymptotics (\ref{flucE89}) as $k \rightarrow \infty$ and obtain upper bounds uniform in $k$ and $s$. These results will be derived as corollaries of the more general calculations carried out in Appendix \ref{flucAp2}.  

Let $0 < \delta < \zeta$ and recall $\tau_{k}'(\delta)$ defined in (\ref{flucE29}). We can further deform $\Gamma_{k}$ in (\ref{flucE89}) into the contour that consists of the curves   
\begin{itemize}
\item $\Gamma_{k}': [-\tau_{k}', \tau_{k}'] \rightarrow \bbC$ given by $t \mapsto \Phi_{k}(t)$ for $t \in [0, \tau_{k}']$ and $t \mapsto \conj{\Phi}_{k}(t)$ for $t \in [-\tau_{k}', 0]$,
\item $\Gamma_{k}'': [0, 2\delta(\pi-\arg \Phi_{k}(\tau_{k}'))] \rightarrow \bbC$ given by $t \mapsto \delta \exp\{\ii (\arg \Phi_{k}(\tau_{k}') + \delta^{-1}t)\}$, which parametrizes with unit speed the arc of the circle $|z| = \delta$ that contains $-\delta$ and goes from  $\Phi_{k}(\tau_{k}')$ to $\conj{\Phi}_{k}(\tau_{k}')$. 
\end{itemize} 
Note that the new contour is also simple ($\Gamma_k'$ and $\Gamma_k''$ do not intersect because $|\Phi_k(t)| > \delta$ for $t \in [0, \tau_k')$ by definition (\ref{flucE29})), closed and encloses $\{a_1, \dotsc, a_{m_k}\}$. See Figure \ref{flucF5}. 

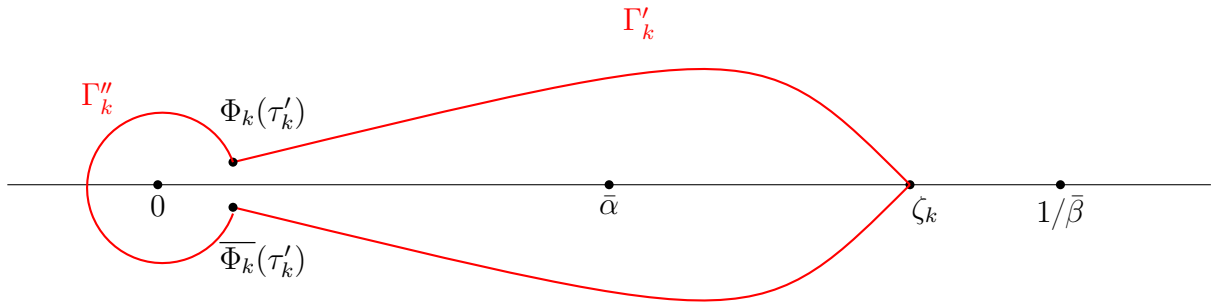
\begin{figure}[h]
\centering
\begin{tikzpicture}[scale = 2]
\draw[-](0, 0)--(1,0)node[below]{$0$}--(4, 0)node[below]{$\bar{\alpha}$}--(6.1, 0)node[below]{$\zeta_k$}--(7, 0)node[below]{$1/\bar{\beta}$}--(8, 0);
\fill(1,0)circle(0.03);
\fill(4,0)circle(0.03);
\fill(6,0)circle(0.03);
\fill(7,0)circle(0.03);
%\draw[dashed, thick](6, 0)..controls (5.7, 0.5)..(5.3, 0.83);
%\draw[dashed,  thick](6, 0)..controls (5.7, -0.5)..(5.3, -0.83);
%\draw[dashed,  thick] (5.3, 0.83)..controls (5, 1)..(1, 0);
%\draw[dashed, thick] (5.3, -0.83)..controls (5, -1)..(1, 0);
\draw[red, thick] (6, 0)..controls(5, 1)..(1.5, 0.15);
\draw[red, thick] (6, 0)..controls (5, -1)..(1.5, -0.15);
%\fill(5.3, 0.83)circle(0.03);
%\node at (5.3, 0.90)[above]{$\Phi_k(\tau_k)$};
%\node at (5.3, -0.90)[below]{$\conj{\Phi_k}(\tau_k)$};
\node at (1.7, 0.3)[above]{$\Phi_k(\tau_k')$};
\node at (1.7,- 0.3)[below]{$\conj{\Phi_k}(\tau_k')$};
\fill(1.5, 0.15)circle(0.03);
\fill(1.5, -0.15)circle(0.03);
%\fill(5.3, -0.83)circle(0.03);
%\draw[-, red, thick](6, 0)--(5.3, 0.83);
%\draw[-, red, thick](6, 0)--(5.3, -0.83);
%\draw[blue] (6.2, 0) arc (0: 120: 0.2);
%\node at (6.2, 0.1)[above]{{\color{blue}$2\pi/3$}}; 
%\draw[red] (6.5, 0) arc (0: 130: 0.5);
%\node at (6.5, 0.4)[above]{{\color{red}$2\pi/3+\phi$}}; 
%\node at (5.7, -0.7)[right]{{\color{blue}$\Gamma_1$}};
%\node at (5.7, 0.3)[left]{{\color{red}$\Gamma_{k}'$}};
\node at (4.2, 0.9)[above]{{\color{red}$\Gamma_{k}'$}};
\draw[red, thick] (1.5, 0.15) arc (20: 340: 0.5);
\node at (0.6, 0.4)[above]{{\color{red}$\Gamma_{k}''$}};
\end{tikzpicture}
\caption[The choice for the contour of integration]{\small{An illustration of the deformation of contour $\Gamma_k$ described in the text and the curves $\Gamma_{k}'$ and $\Gamma_{k}''$.}}\label{flucF5}
\end{figure}

The reason for this additional deformation is that we do not have a uniform bound in $k$ for the lengths of $\Phi_k$ inside $\Disc(0, \delta)$ (see the remarks preceding Lemma \ref{flucL9}), which is needed to uniformly control the contribution to the integral (\ref{flucE89}) from the parts of the contour outside a small disk around $\zeta_k$.   

Recall $P_k(s) = s_+^{3/2} \wedge (s_+ n_k^{1/3})$ defined for $k \in \bbN$ at (\ref{flucE143}). Also, we will abbreviate $m_k, n_k, p_k(s)$ in the subscripts as $k, s$ for $k \in \bbN$ and $s \in \bbR$.  
\begin{lem} 
\label{flucL7}
Let $s_0 \ge 0$. Then, 
\begin{enumerate}[(a)]
\item There exists a deterministic constant $C > 0$ such that, for a.e. $(\bfa, \bfb)$, there exists $k_0 = k_0^{\bfa, \bfb} \in \bbN$ such that 
\begin{align}
\bigg|\frac{\sigma_{k} n_{k}^{1/3} I_{k, s}}
{ \zeta_{k} F_{k, s}(\zeta_{k})} - \Ai(s)\bigg| \le \frac{C}{n_k^{1/3}} \qquad \text{ for } s \in [-s_0, s_0] \text{ and } k \ge k_0. \label{flucE163}
\end{align}
\item There exist deterministic constants $C, c > 0$ such that, for a.e. $(\bfa, \bfb)$, there exists $k_0 = k_0^{\bfa, \bfb} \in \bbN$ such that 
\begin{align}
|I_{k, s}| \le \frac{C}{n_k^{1/3}} F_{k, s}(\zeta_{k}) \exp\{-c P_k(s)\} \qquad \text{ for } s \ge -s_0 \text{ and } k \ge k_0.  \label{flucE164}
\end{align}
\end{enumerate}
\end{lem}

We begin with matching the notation in Appendix \ref{flucAp2} with the present setting. Recall from (\ref{flucE17}) and (\ref{flucEE4}) that, for $k \in \bbN$ and $s \in \bbR$,  
\begin{align}
\label{flucEE5}
p_k(s) = \lf n_k \gamma_k + n_k^{1/3} \sigma_k s \rf = n_k\gamma_k + n_k^{1/3}\sigma_k p_k'(s).  
\end{align}
\begin{longtable}[h!]{|c|c|}
\caption{Correspondence between Appendix \ref{flucAp2} and Section \ref{flucS6}} \label{flucTa1}\\
\hline
\textbf{Appendix \ref{flucAp2}} & \textbf{Section \ref{flucS6} } \\
\hline
\endfirsthead
\multicolumn{2}{c}%
{\tablename\ \thetable\ -- \textit{Continued from previous page}} \\
\hline
\textbf{Appendix \ref{flucAp2}} & \textbf{Section \ref{flucS6} }\\
\hline
\endhead
\hline \multicolumn{2}{c}{\textit{Continued on next page}} \\
\endfoot
\hline
\endlastfoot
$F$ & $F_{m_k, n_k, p_k(s)}$ \\
\hline 
$u$ & $n_k u_k$ \\
\hline
$\tilde{u}$ & $n_k^{1/3}\sigma_k p_k'(s)\log |\cdot |$ \\
\hline 
$a, r, n$ & $\zeta_k, \rho$ (see (\ref{flucE146})), $3$ \\
\hline 
$\Phi, T$ & $\Phi_k, \tau_k'(\delta)$ \\
\hline
$c, \xi, N$ & $-\dfrac{n_k \sigma_k^3}{3\zeta_k^3}$ (see (\ref{flucE18}) and (\ref{flucE137})), $e^{\ii \pi/3}$,  $\dfrac{n_k^{1/3} \sigma_k}{3^{1/3} \zeta_k}$\\
\hline
$k, d, M$ & $1$, $\dfrac{n_k^{1/3}\sigma_k p_k'(s)}{\zeta_k}$, $\dfrac{n_k^{1/3} \sigma_k |p_k'(s)|}{\zeta_k}$ if $p_k'(s) \neq 0$\\
\hline
$k, d, M$ & $0$, $0$, $0$ if $p_k'(s) = 0$\\
\hline 
$K_0, \epsilon_0$ & $K_{0, k}, \epsilon_{0, k}$, see (\ref{flucE173}) and (\ref{flucE174})\\
\hline
$\tilde{K}_0, \tilde{\epsilon}_0$ & $\tilde{K}_{0, k} = \dfrac{4\zeta_k}{\rho^{2}}\log|\zeta_k + \rho/2|$, $\tilde{\epsilon}_{0, k} = \dfrac{1}{4\tilde{K}_{0, k}}$ if $p_k'(s) \neq 0$\\
\hline 
$\tilde{K}_0, \tilde{\epsilon}_0$ & $0, \infty$ if $p_k'(s) = 0$\\
\hline
$\eta$ & $\eta_k$: the direction of $\Phi_k(\tau_k(\epsilon))-\zeta_k$ \\
\hline 
$\sI$ & $\sI_k = \displaystyle \int \limits_{[0, \infty \eta_k e^{\ii \pi/3}]} \exp \bigg\{u^3 + p_k'(s)e^{-\ii \pi/3}3^{1/3}u \bigg\} \dd u$ \\
\hline 
$I'$ & $I_k' = \displaystyle \int \limits_{\zeta_k}^{\Phi_k(\tau_k(\epsilon))} F_{m_k, n_k, p_k(s)}(z) \ \dd z$ \\
\hline 
$I''$ & $I_k'' = \displaystyle \int \limits_{\Phi_k(\tau_k(\epsilon))}^{\Phi_k(\tau_k'(\delta))} F_{m_k, n_k, p_k(s)}(z) \ \dd z$ \\
\hline
$x, p, u_1$ & $0$, $n_k \gamma_k + m_k$, $n_k u_k-(n_k \gamma_k + m_k) \log |\cdot|$\\ 
\hline 
$\tilde{p}, \tilde{u}_1$ & $n_k^{1/3}\sigma_k p_k'(s)$, $0$\\
\hline
$\Psi, S$ & $\Gamma_k''$, $\delta(\pi-\arg \Phi_k(\tau_k'))$\\
\hline
$I'''$ & $I_k''' = \displaystyle \int \limits_{\Gamma_k''(0)}^{\Gamma_k(\delta(\pi-\arg \Phi_k(\tau_k')))} F_{m_k, n_k, p_k(s)}(z) \ \dd z$
\end{longtable}
%\clearpage
Note from (\ref{flucE165}) that 
\begin{align}
\log |F_{m_k, n_k, p_k(s)}(z)| &= - \sum \limits_{i=1}^{m_k} \log |z-a_i| + \sum \limits_{j = 1}^{n_k} \log|1-b_j z| + (m_k + p_k(s)) \log |z| \\
&= n_k u_k(z) + n_k^{1/3}\sigma_k p_k'(s) \log |z| \qquad \text{ for } z \in \bbC \smallsetminus \sP_{k}. 
\end{align}
Hence, the choices of $F, u, \tilde{u}$ in Table \ref{flucTa1} are legitimate. By Lemma \ref{flucL5}, $\lim_{k \rightarrow \infty} \zeta_k \stackrel{\text{a.s.}}= \zeta$. Therefore, there exists a deterministic $\tilde{K}_0 > 0$ such that, after restricting to an a.s.\ event and working with larger $k_0^{\bfa, \bfb}$ if necessary, $\tilde{K}_0 > \sup_{k \ge k_0} \tilde{K}_{0, k}$. Also, set $\tilde{\epsilon}_0 = \dfrac{1}{\tilde{K}_0} < \inf_{k \ge k_0} \tilde{\epsilon}_{0, k}$. 
Let $\epsilon > 0$ satisfy 
\begin{align}
\epsilon < \min \bigg\{\frac{\epsilon_0}{1024}, \frac{\tilde{\epsilon_0}}{4}, \frac{1}{12}\bigg\}, 
\end{align}
where $\epsilon_0$ is given by (\ref{flucE158}). 

%where $\epsilon_1, \epsilon_2$ are as defined in (\ref{flucE157}) and (\ref{flucE158}), and the last bound comes from the hypothesis of Proposition \ref{flucP6}. Then the bounds on $\epsilon$ in Table \ref{flucTa1} hold for each $k \in \bbN$. 

We next connect $\sI_k$ with the Airy function. Changing variables with $v = - e^{-\ii \pi/3} 3^{1/3} u$, we obtain 
\begin{align}
\sI_k &= e^{\ii \pi/3} 3^{-1/3} \int \limits_{[\infty (-\eta_k), 0]} \exp \bigg\{\frac{v^3}{3} - p_k'(s) v \bigg\} \dd v,  
\end{align}
where $\eta_k \in \bbS(0, 1)$ such that $\Phi_k(\tau_k(\epsilon)) = \zeta_k + \epsilon \eta_k$ for $k \in \bbN$. Hence, 
\begin{align}
\frac{\Im \{3^{1/3}e^{-\ii \pi/3}\sI_k\}}{\pi} = \frac{1}{2\pi \ii} \int \limits_{[\infty (-\eta_k), 0] + [0, \infty (-\conj{\eta_k})]} \exp \bigg\{\frac{v^3}{3} - p_k'(s) v \bigg\} \dd v = \Ai(p_k'(s)). \label{flucE166}
\end{align}
The last equality holds for $k \ge k_0$. To justify it, refer to (\ref{flucE9}) and note that, by Proposition \ref{flucL10}b, 
\begin{align}|(-\conj{\eta}_k)-e^{\ii \pi/3}| = |\eta_k-e^{\ii 2\pi/3}| \le \frac{256\epsilon}{\epsilon_{0}} < \frac{1}{4} < 2-\sqrt{3} = |1-e^{\ii \pi/6}| \quad \text{ for } k \ge k_0, \end{align}
which implies that $\pi/6 < \arg\{-\conj{\eta}_k\} < \pi/2$ for $k \ge k_0$. 
Note from (\ref{flucEE5}) that 
\begin{align}
|p_k'(s)-s| \le \frac{1}{n_k^{1/3}\sigma_k} \quad \text{ for } k \in \bbN \text{ and } s \in \bbR.  \label{flucEE6}
\end{align}
Hence, we have the mean value bound
\begin{align}
|\Ai(s)-\Ai(p_k'(s))| \le \frac{1}{\sigma_k n_k^{1/3}} \max \limits_{|t| \le s+\sigma_k^{-1}n_k^{-1/3}} |\Ai'(t)| \quad \text{ for } k \in \bbN \text{ and } s \in \bbR.\label{flucE168}
\end{align}
We can similarly relate $I_k', I_k''$ and $I_k'''$ to $I_{k, s}$ by 
\begin{align}
\frac{\Im \{I_k' + I_k'' + I_k'''\}}{\pi} = I_{k, s} \quad \text{ for } k \in \bbN \text{ and } s \in \bbR. \label{flucE167}
\end{align}

\begin{proof}[Proof of Lemma \ref{flucL7}]
We first prove (\ref{flucE163}). By Lemma \ref{flucL5}, we can restrict to an a.s.\ event on which the inequalities 
\begin{align}\frac{1}{2} < \frac{\sigma_k}{\sigma} < 2 \quad \text{ and } \quad \frac{1}{2} < \frac{\zeta_k}{\zeta} < 2 \end{align}
hold for $k \ge k_0$ possibly after choosing a larger $k_0$. Then it follows from the assumption $|s| \le s_0$ and (\ref{flucEE6}) that $|p_k'(s)| \le s_0 + 2\sigma$ for $k \ge k_0$. 
Applying Proposition \ref{flucP2} and using the preceding inequalities, we obtain
\begin{align}
\bigg|\frac{e^{\ii \pi/3}\sigma_k n_k^{1/3}I_k'}{3^{1/3}\zeta_k F_{k, s}(\zeta_k)}-\sI_k\bigg| &\le 3072 (K_0 + 2\tilde{K}_0 |p_k'(s)| + 1) \exp\{32(|p_k'(s)|+1)^{3/2}\} \\
&\cdot \bigg(\frac{2\zeta_k}{\sigma_k n_k^{1/3}} + \exp\bigg\{-\frac{\epsilon^3 \sigma_k^3 n_k}{12 \zeta_k^3} \bigg\}\bigg) \nonumber \\
&\le \frac{C_1}{n_k^{1/3}} \label{flucE162}
\end{align}
for $k \ge k_0$ for some constant $C_1 > 0$. We emphasize that $C_1$ (and various other constants introduced below) 
can be chosen independent of $(\bfa, \bfb)$. 

Next, combine Lemmas \ref{flucL5} and \ref{flucL8} with $|\Phi_k(t)| \ge \delta$ for $t \in [0, \tau_k'(\delta)]$ (see (\ref{flucE29})). Then there are constants $c_1, c_2 > 0$ such that   
\begin{align}
p_k'(s) \log\bigg(\frac{|\Phi_k(t)|}{\zeta_k}\bigg) &\le \begin{cases}p_k'(s) \log \bigg(\sup \limits_{\tau_k(\epsilon) \le t \le T_k} \dfrac{|\Phi_k(t)|}{\zeta_k}\bigg) &\text{ if } p_k'(s) \ge 0 \\
p_k'(s) \log(\delta/\zeta_k) \quad &\text{ otherwise }  \end{cases} \\
&\le \begin{cases}-p_k'(s) c_1 &\text{ if } s_k \ge 0 \\
-p_k'(s) c_2 \quad &\text{ otherwise }  \end{cases} \label{flucE171}
\end{align}
for $t \in [\tau_k(\epsilon), \tau_k'(\delta)]$ and $k \ge k_0$. Also, by Lemma \ref{flucL9}, $\tau_k'(\delta) < C_2$ for $k \ge k_0$ for some constant $C_2 > 0$. Therefore, by Proposition \ref{flucP4},   
\begin{align}
|I_k''| &\le \tau_k'(\delta) F_{k, s}(\zeta_k) \exp\bigg\{-\frac{\epsilon^3 \sigma_k^3 n_k}{6\zeta_k^3} + \sigma_k n_k^{1/3} \sup_{t \in [\tau_k(\epsilon), \tau_k'(\delta)]} p_k'(s) \log \bigg(\frac{|\Phi_k(t)|}{\zeta_k}\bigg)\bigg\} \label{flucE170}\\
&\le C_2 F_{k, s}(\zeta_k)\exp\bigg\{-\frac{\epsilon^3 \sigma_k^3 n_k}{6\zeta_k^3} + (c_1 \vee c_2)(s_0 \sigma_k n_k^{1/3} + 1)\bigg\} \\
&\le C_3 F_{k, s}(\zeta_k) \exp(-c_3 n_k) \label{flucE161}
\end{align}
for $k \ge k_0$ for some constants $C_3, c_3 > 0$.  

Note that $|\Gamma_k''(t)-a_i| \ge |\Gamma_k''(0)-a_i| = |\Phi_k(\tau_k'(\delta))-a_i|$ and $1-\delta \le |1-b_j\Gamma_k''(t)| \le 1+\delta$ for $i \in [m_k]$, $j \in [n_k]$ and $t \in [0, \pi-\arg \Phi_k(\tau_k'(\delta))]$. Also, choose $\delta > 0$ sufficiently small such that 
\begin{align}
\log \bigg(\frac{1+\delta}{1-\delta}\bigg) < \frac{64}{6}\frac{\epsilon^3 \sigma^3}{\zeta^3}. 
\end{align}
Using these inequalities and Proposition \ref{flucP5}, we obtain  
\begin{align}
|I_k'''| &\le 2\pi\delta F_{k, s}(\zeta_k) \exp\bigg\{-\frac{\epsilon^3 \sigma_k^3 n_k}{6\zeta_k^3} + \sigma_k n_k^{1/3} \sup_{t \in [\tau_k(\epsilon), \tau_k'(\delta)]} p_k'(s) \log \bigg(\frac{|\Phi_k(t)|}{\zeta_k}\bigg)\bigg\} \nonumber\\ 
&\cdot \exp\bigg\{\sup_{t \in [0, \pi-\Phi_k(\tau_k'(\delta))]} \bigg\{- \sum \limits_{i=1}^{m_k} \log \bigg(\frac{|\Gamma_k''(t)-a_i|}{|\Gamma_k''(0)-a_i|}\bigg) + \sum \limits_{j=1}^{n_k} \log \bigg(\frac{|1-b_j\Gamma_k''(t)|}{|1-b_j\Gamma_k''(0)|}\bigg) \bigg\}\bigg\} \nonumber \\
&\le 2\pi\delta F_{k, s}(\zeta_k)\exp\bigg\{-\bigg(\frac{\epsilon^3 \sigma_k^3}{6\zeta_k^3}-\log \bigg(\frac{1+\delta}{1-\delta}\bigg)\bigg)n_k + (c_1 \vee c_2)(s_0 \sigma_k n_k^{1/3} + 1)\bigg\} \nonumber \\
&\le C_4 F_{k, s}(\zeta_k)\exp\{-c_4 n_k\} \label{flucE160}
\end{align}
for $k \ge k_0$ for some constants $C_4, c_4 > 0$. 

Now, putting together (\ref{flucE166}), (\ref{flucE167}), (\ref{flucE162}), (\ref{flucE161}) and (\ref{flucE160}), we obtain 
\begin{align}
\bigg|\frac{\sigma_k n_k^{1/3}I_{k, s}}{\zeta_k F_{k, s}(\zeta_k)}-\Ai(s_k)\bigg| &= \frac{1}{\pi}\bigg|\Im \bigg\{\frac{\sigma_k n_k^{1/3}}{\zeta_k F_{k, s}(\zeta_k)}(I_k'+I_k''+I_k''')-3^{1/3}e^{-\ii \pi/3}\sI_k\bigg\}\bigg| \\ 
&\le  \bigg|\frac{e^{\ii \pi/3}\sigma_k n_k^{1/3}}{3^{1/3}\zeta_k F_{k, s}(\zeta_k)}(I_k'+I_k''+I_k''')-\sI_k\bigg| \le \frac{C}{n_k^{1/3}}
\end{align}
for $k \ge k_0$ for some constant $C > 0$. This and (\ref{flucE168}) imply (\ref{flucE163}).  

We now turn to proving (\ref{flucE164}). Observe that, in the case $s < 2\sigma_k^{-1} n_k^{-1/3}$, (\ref{flucE164}) holds due to (\ref{flucE163}). Hence, we assume that $s \ge 2\sigma_k^{-1}n_k^{-1/3}$ here on. Then $p_k'(s) \ge \sigma_k^{-1}n_k^{-1/3}$ and the hypotheses of Proposition \ref{flucP6} hold. (The condition $d\Re\{e^{\ii \pi j/n} \conj{\xi}\} < 0$ in our case reduces to $p_k'(s) > 0$). Also, $p_k'(s)/s \ge 1-\sigma_k^{-1}n_k^{-1/3}s^{-1} \ge 1/2$. Hence, we obtain  
\begin{align}
|\Im I'_k| &\le \frac{64\zeta_k F_{k, s}(\zeta_k)}{\sigma_k n_k^{1/3}}\exp\bigg\{-\min \bigg\{\frac{\epsilon p_k'(s) n_k^{1/3}\sigma_k}{8 \zeta_k}, \frac{[p_k'(s)]^{3/2}}{8}\bigg\}\bigg\} \nonumber \\
&\le C_5 F_{k, s}(\zeta_k) \exp \{-c_5 \min \{sn_k^{1/3}, s^{3/2}\}\} \label{flucE169}
\end{align}
for $k \ge k_0$ and some constants $C_5, c_5 > 0$. Moreover, utilizing $p_k'(s) > 0$ and $p_k'(s)/s > 1/2$ again, we have 
\begin{align}
\sup_{t \in [\tau_k(\epsilon), \tau_k'(\delta)]} p_k'(s) \log \bigg(\frac{|\Phi_k(t)|}{\zeta_k}\bigg) \le -\frac{c_1}{2}s, 
\end{align}
which leads to the following refinement of the bounds (\ref{flucE170}) and (\ref{flucE171})
\begin{align}
|I_k''|+|I_k'''| \le C_6 F_{k, s}(\zeta_k) \exp\{-c_6(n_k + n_k^{1/3}s)\} \label{flucE172}
\end{align}
for $k \in k_0$ for some constants $C_6, c_6 > 0$. Combining (\ref{flucE169}) and (\ref{flucE172}) gives (\ref{flucE164}). 
\end{proof}

\section{Right tail bound} \label{flucSe6}

\begin{lem}
\label{flucL16}
Let $m, n \in \bbN$, $x, y \in \bbZ_+$, $s \in \bbR$ and  $z \in \bbC \smallsetminus \{a_i: i \in [m]\} \smallsetminus \{1/b_j: j \in [n]\}$. Then 
\begin{align}
g_{m, n}^{\bfa, \bfb}(z) &= \frac{m}{n} g_{n, m}^{\bfb, \bfa}(z^{-1}) \\
\gamma_{m, n}^{\bfa, \bfb} &= \frac{m}{n} \gamma_{n, m}^{\bfb, \bfa} \label{flucEq104}\\
\sigma_{m, n}^{\bfa, \bfb} &= \frac{m^{1/3}}{n^{1/3}} \sigma_{n, m}^{\bfb, \bfa} \label{flucEq103}\\
\zeta_{m, n}^{\bfa, \bfb}\zeta_{n, m}^{\bfb, \bfa} &= 1 \label{flucEq102}\\
%\zeta_{m, n, x}^{+, \bfa, \bfb} \zeta_{n, m, \frac{n}{m} x}^{-, \bfb, \bfa} &= 1 \\
p_{m, n}^{\bfa, \bfb}(s) &= p_{n, m}^{\bfb, \bfa}(s) \\
F_{m, n, x}^{\bfa, \bfb}(z)F_{n, m, y}^{\bfb, \bfa}(z^{-1}) &= z^{x-y}. \label{flucEq105}
\end{align}
\end{lem}
\begin{proof}
By (\ref{flucE7}),  
\begin{align}
g_{m, n}^{\bfa, \bfb}(z) &= \frac{m}{n} \bigg(\frac{1}{m} \sum \limits_{i=1}^m \frac{a_i z^{-1}}{1-a_i z^{-1}} + \frac{1}{m} \sum \limits_{j=1}^n \frac{b_j}{z^{-1}-b_j}\bigg) = \frac{m}{n} g_{n, m}^{\bfb, \bfa}(z^{-1}). \label{flucE187}
\end{align}
Taking infimum over $z \in [\bar{\alpha}, 1/\bar{\beta}]$ gives 
\begin{align}
\gamma_{m, n}^{\bfa, \bfb} = \frac{m}{n}  \inf \limits_{z \in [\bar{\alpha}, 1/\bar{\beta}]} g_{n, m}^{\bfa, \bfb}(z^{-1}) = \frac{m}{n} \inf \limits_{z \in [\bar{\beta}, 1/\bar{\alpha}]} g_{n, m}^{\bfa, \bfb}(z) = \frac{m}{n} \gamma_{n, m}^{\bfb, \bfa}, \label{flucE188}
\end{align}
by (\ref{flucE5}). Then (\ref{flucEq102}) follows from (\ref{flucE187}) and (\ref{flucE188}). Differentiating (\ref{flucE187}) with respect to $z$ twice gives 
\begin{align}
\partial_z^2 g_{m, n}^{\bfa, \bfb}(z) = \frac{m}{n} \bigg(\frac{\partial_z^2 g_{n, m}^{\bfb, \bfa}(z^{-1})}{z^4} + \frac{2 \partial_z g_{n, m}^{\bfa, \bfb}(z^{-1})}{z^3}\bigg). 
\end{align}
Set $z = \zeta_{m, n}^{\bfa, \bfb}$, and use (\ref{flucEq102}), $\partial_{z} g_{n, m}^{\bfb, \bfa}(\zeta_{n, m}^{\bfb, \bfa}) = 0$ and (\ref{flucE128}). Then, rearranging terms, we obtain (\ref{flucEq103}). 
Now it is immediate from (\ref{flucE17}), (\ref{flucEq104}) and (\ref{flucEq103}) that 
\begin{align}
p_{m, n}^{\bfa, \bfb}(s) = \lf n\gamma_{m, n}^{\bfa, \bfb} + n^{1/3} \sigma_{m, n}^{\bfa, \bfb}s \rf = \lf m\gamma_{n, n}^{\bfb, \bfa} + m^{1/3} \sigma_{n, m}^{\bfa, \bfb}s \rf = p_{n, m}^{\bfb, \bfa}(s).   
\end{align} 
Finally, (\ref{flucEq105}) can be seen from (\ref{flucE165}). 
\end{proof}

\begin{lem}
\label{flucL18}
Let $F: [0, \infty) \rightarrow [0, \infty)$ be continuous and nondecreasing. Furthermore, suppose that there exists $x_0 \in [0, \infty)$ such that $F$ is piecewise $\sC^1$ on $[x_0, \infty)$ and $F'(x) \ge 1$ for a.e. $x \ge x_0$. Then 
\begin{align*}
h \sum \limits_{l = 0}^\infty e^{-F(lh)} \le (1+h+x_0)F(0) \qquad \text{ for } h > 0. 
\end{align*} 
\end{lem}
\begin{proof}
\begin{align*}
h \sum \limits_{l = 0}^\infty e^{-F(lh)} &\le he^{-F(0)} + \int \limits_0^\infty e^{-F(x)} dx = he^{-F(0)} + \int \limits_0^{x_0} e^{-F(x)}\ dx + \int \limits_{x_0}^\infty e^{-F(x)} dx \\
&\le (h+x_0)e^{-F(0)} + \int \limits_{x_0}^\infty F'(x)e^{-F(x)}\ dx = (h+x_0)e^{-F(0)} + e^{-F(x_0)} \\ 
&\le (1+h+x_0) e^{-F(0)}\qedhere
\end{align*}
\end{proof}
%
%\begin{lem}
%\label{flucL18}
%Let $p \ge 1$, $c > 0$ and $s_0 \ge 0$. There exists $C > 0$ such that 
%\begin{align*}
%\int \limits_{s}^\infty \exp(-c x_+^p)\ dx \le C\exp(-cs_+^p) \qquad \text{ for } s \ge -s_0. 
%\end{align*} 
%\end{lem} 
%\begin{proof}
%We have  
%\begin{align*}
%\int \limits_{s}^\infty \exp(-c x^p)\ dx \le \int \limits_{s}^\infty x^{p-1} \exp(-cx^p) \ dx = \frac{1}{pc} \exp(-cs^p) \qquad \text{ for } s \ge 1. 
%\end{align*}
%Using this, we also obtain 
%\begin{align*}
%\int \limits_{s}^\infty \exp(-c x_+^p)\ dx \le (1+s_0)+\frac{e^{-c}}{pc} = Ce^{-c} \le C\exp(-cs_+^{p}) \qquad \text{ for } -s_0 \le s \le 1, 
%\end{align*}
%where $C = e^c(1+s_0) + \dfrac{1}{pc}$. With this value of $C$, the desired bound holds. 
%\end{proof}

\begin{proof}[Proof of Theorem \ref{flucT2}]
Note from (\ref{flucE17}) that 
\begin{align}
\label{flucE74}
p_{m_k, n_k}^{\bfa, \bfb}(s)+l = p_{m_k, n_k}^{\bfa, \bfb}\bigg(s+\frac{l}{n_k^{1/3} \sigma_{m_k, n_k}^{\bfa, \bfb}}\bigg) \text{ for } s \in \bbR \text{ and } l \in \bbZ. 
\end{align}
Hence, by Lemma \ref{flucL7}b, 
\begin{align}
\bigg|I_{m_k, n_k, p_{m_k, n_k}^{\bfa, \bfb}(s) + l}^{\bfa, \bfb}\bigg| \le \frac{C}{n_k^{1/3}}\exp\bigg\{-cP_k\bigg(s+\frac{l}{n_k^{1/3}\sigma_{m_k, n_k}^{\bfa, \bfb}}\bigg)\bigg\} F_{m_k, n_k, p_{m_k, n_k}^{\bfa, \bfb}(s)+l}^{\bfa, \bfb}(\zeta_{m_k, n_k}^{\bfa, \bfb}) 
\label{flucE70}
\end{align}
for $l \in \bbZ_+, s \ge -s_0$, sufficiently large $k \in \bbN$ a.s.\ for some constants $C, c > 0$. Since
\begin{align*}
c_1^{\beta, \alpha} = \frac{1}{c_2^{\alpha, \beta}} < \frac{1}{r} <\frac{1}{c_1^{\alpha, \beta}} = c_2^{\beta, \alpha}, 
\end{align*}
we can also apply Lemma \ref{flucL7}b after interchanging $\bfa$ with $\bfb$ and $m_k$ with $n_k$. Then, utilizing the identities in Lemma \ref{flucL16}, we obtain 
\begin{align}
\bigg|I_{n_k, m_k, p_{m_k, n_k}^{\bfa, \bfb}(t) + l}^{\bfb, \bfa}\bigg| &= \bigg|I_{n_k, m_k, p_{n_k, m_k}^{\bfb, \bfa}(t) + l}^{\bfb, \bfa}\bigg| \nonumber\\
&\le \frac{C}{m_k^{1/3}}\exp\bigg\{-cP_k\bigg(t+\frac{l}{\sigma_{n_k, m_k}^{\bfb, \bfa}m_k^{1/3}}\bigg)\bigg\} F_{n_k, m_k, p_{n_k, m_k}^{\bfb, \bfa}(t)+l}^{\bfb, \bfa}(\zeta_{n_k, m_k}^{\bfb, \bfa}) \nonumber\\
&= C\frac{n_k^{1/3}}{m_k^{1/3}}\frac{1}{n_k^{1/3}} \frac{\exp\bigg\{-cP_k\bigg(t+\dfrac{l}{\sigma_{m_k, n_k}^{\bfa, \bfb}n_k^{1/3}}\bigg)\bigg\}} {F_{m_k, n_k, p_{n_k, m_k}^{\bfa, \bfb}(t)+l}^{\bfa, \bfb}(\zeta_{m_k, n_k}^{\bfa, \bfb})} \label{flucE69}
\end{align} 
for $l \in \bbZ_+, t \ge -s_0$, sufficiently large $k$ a.s.\  for some constants $C, c > 0$. By (\ref{flucE70}) and (\ref{flucE69}) (and using $\lim_{k \rightarrow \infty} n_k/m_k = 1/r$),   
\begin{align}
\bigg|I_{m_k, n_k, p_k(s) + l}^{\bfa, \bfb} I_{n_k, m_k, p_k(t) + l}^{\bfb, \bfa}\bigg| \le \frac{C\zeta_{k}^{p_k(s)-p_k(t)}}{n_k^{2/3}} \exp\bigg\{-c\bigg[P_k\bigg(s+\frac{l}{\sigma_k n_k^{1/3}}\bigg)+P_k\bigg(t+\frac{l}{\sigma_{k}n_k^{1/3}}\bigg)\bigg]\bigg\} \label{flucE191}
\end{align}
for sufficiently large $k$ a.s.\ for some constants $C, c > 0$, where we switched to the abbreviated notation on the right-hand side. Summing over $l \in \bbZ_+$, and using (\ref{flucE71}), Lemma \ref{flucL5} and Lemma \ref{flucL18} (with $F(x) = cP_k(s+x)+cP_k(t+x)$, $x_0 = s_0+c^{-2}+1$ and $h = \sigma_k^{-1}n_k^{-1/3}$) lead to 
\begin{align}
|\KK_k(p_k(s), p_k(t))| \le \frac{C \zeta_{k}^{p_k(s)-p_k(t)}}{n_k^{1/3}} \exp\{-cP_k(s)-cP_k(t)\} \label{flucE73}
\end{align}
for sufficiently large $k$ a.s.\ for some constants $C, c > 0$. Now Hadamard's inequality 
\begin{align*}|\det_{i, j \in [l]} A| \le \prod_{i \in [l]} \bigg(\sum_{j \in [l]} A(i, j)^2\bigg)^{1/2}\end{align*} and (\ref{flucE73}) imply 
\begin{align}
\det \limits_{i, j \in [l]}  \left[K_{k}(p_k(s_i), p_k(s_j))\right] &= \det \limits_{i, j \in [l]}\left[\zeta_k^{p_k(s_j)-p_k(s_i)}K_{k}(p_k(s_i), p_k(s_j))\right] \label{flucE87}\\
&\le \frac{C^l} {n_k^{l/3}} \prod_{i=1}^l e^{-c P_k(s_i)} \left(\sum_{j=1}^l e^{-2cP_k(s_j)}\right)^{1/2} \nonumber\\
&\le \frac{C^l l^{l/2}}{n_k^{l/3}} \prod_{i=1}^l e^{-c P_k(s_i)} \nonumber
\end{align}
for $s_1, \dots, s_l \ge -s_0$ and sufficiently large $k$ a.s.
\end{proof}

\begin{lem}
\label{flucL17}
Let $s_0 \ge 0$. There exist deterministic constants $C, c > 0$ such that, for a.e $(\bfa, \bfb)$, there exist $k_0=k_0^{\bfa, \bfb} \in \bbN$ such that
\begin{align*}
\sum \limits_{\substack{q_1, \dotsc, q_l \in \bbN \\ q_i \ge p_k(s_i)}} \det \limits_{i, j \in [l]} \left[K_{k}(q_i, q_j)\right] &\le C^l l^{l/2} \exp\bigg(-c\sum_{i=1}^l P_k(s_i)\bigg) \\
&= C^l l^{l/2} \exp\bigg(-c\sum_{i=1}^l \bigg((s_i)_+^{3/2} \wedge \{(s_i)_+ n_k^{1/3}\}\bigg)\bigg)
\end{align*}
for $l \in \bbN$, $k \ge k_0$ and $s_1, \dotsc, s_l \ge -s_0$. 
\end{lem}
\begin{proof}
By (\ref{flucE74}) and Theorem \ref{flucT2}, there exist $C, c > 0$ such that, for a.e. $\bfa$ and $\bfb$, there exist $k_0 = k_0^{\bfa, \bfb} \in \bbN$ such that 
\begin{align}
\label{flucE75}
\det \limits_{i, j \in [l]} [K_{k}(q_i, q_j)] \le \frac{C^l l^{l/2}}{n_k^{l/3}} \exp\bigg\{-c\sum_{i=1}^l P_k\bigg(s_i + \frac{q_i-p_k(s_i)}{n_k^{1/3}\sigma_k}\bigg)\bigg\} 
\end{align}
whenever $k \ge k_0$, $s_1, \dotsc, s_l \ge -s_0$ and $q_i \ge p_k(s_i)$ for $i \in [l]$. For each $i \in [l]$, by Lemma \ref{flucL18} (set $F(x) = cP_k(s_i+x)$, $x_0 = s_0 + c^{-2}+1$ and $h = n_k^{-1/3}\sigma_k$), 
\begin{align}
\label{flucE76}
\frac{1}{n_k^{1/3}\sigma_k} \sum \limits_{q \ge p_k(s_i)} \exp \bigg\{- cP_k\bigg(s_i + \frac{q-p_k(s_i)}{n_k^{1/3}\sigma_k}\bigg)\bigg\} &\le C'e^{-c P_k(s_i)}
\end{align} 
for some constant $C' > 0$. Summing (\ref{flucE75}) over $q_1, \dotsc, q_l$ and using (\ref{flucE76}) yield the result. 
\end{proof}

\begin{proof}[Proof of Theorem \ref{flucT6}]
By (\ref{flucE105}) and Lemma \ref{flucL17} (also recall that the determinants below are nonnegative, see (\ref{flucEq98})), for a.e. $\bfa$ and $\bfb$, 
\begin{align}
\bfP(G(m_k, n_k) > n_k \gamma_k + n_k s) &\le \sum \limits_{l=1}^\infty \frac{1}{l!} \sum \limits_{\substack{q_1, \dotsc, q_l \in \bbN \\ q_i \ge p_k(n_k^{2/3} \sigma_k^{-1}s)}} \det \limits_{i, j \in [l]} \left[\KK_{m_k, n_k}(q_i, q_j)\right] \nonumber\\
&\le \sum \limits_{l=1}^\infty \frac{C^l l^{l/2}e^{-clP_k(n_k^{2/3}\sigma_k^{-1}s)}}{l!} \nonumber \\ 
&\le e^{-cP_k(n_k^{2/3}\sigma_k^{-1}s)}  \sum \limits_{l=1}^\infty \frac{C^l l^{l/2}}{l!},  \nonumber 
\end{align}
which completes the proof. 
%for $s \ge 0$, sufficiently large $k \in \bbN$ and some constants $C, c > 0$. Redefine $C$ as the value of the last series, which  converges by the root test. {\color{red}Since $\lim_{k \rightarrow \infty} \gamma_k = \gamma$, we have $|\gamma-\gamma_k| < s$ and 
%\begin{align*}
%|(s+\gamma-\gamma_k)^{3/2}-s^{3/2}| \le |\gamma-\gamma_k|\frac{3\sqrt{2s}}{2} \le 1
%\end{align*}
%for sufficiently large $k$. Then, 
%\begin{align*}
%\bfP(G(m_k, n_k) > n_k \gamma + n_k s) &\le Ce^{-c(s+\gamma-\gamma_k)^{3/2}} \le Ce^c e^{-cs^{3/2}}
%\end{align*} 
%for $s \ge 0$ and sufficiently large $k \in \bbN$. }
\end{proof}

\section{Scaling limit for the correlation kernel}

For the results proved in this and the next section, it suffices to work with weaker versions of the bounds established in Section \ref{flucSe6}. Note that, given $s_0 \ge 0$, there exist constants $C, c > 0$ such that 
\begin{align}
\exp\{-P_k(s)\} = \exp(-\{s_+^{3/2} \wedge (s_+n_k^{1/3})\}) \le C \exp\{-cs\} \quad \text{ for } k \in \bbN \text{ and } s \ge s_0. 
\end{align}
In fact, we can take $c = 1$ and $C = \exp\{1+s_0\}$. 

\begin{proof}[Proof of Theorem \ref{flucT5}]
Each entry of the matrix on the right-hand side of (\ref{flucE87}) is bounded uniformly in $s_1, \dotsc, s_l \in [-s_0, s_0]$ a.s.\ by (\ref{flucE73}). Therefore, it suffices to show that  
\begin{align}
\label{flucE36}
\lim \limits_{k \rightarrow \infty} \sigma_k n_k^{1/3} \zeta_k^{p_k(t)-p_k(s)} \KK_{m_k, n_k}(p_k(s), p_k(t)) = \AiK(s, t)
\end{align}
uniformly in $s, t \in [-s_0, s_0]$ a.s. Let $L > 0$ and $\epsilon > 0$. By Lemma \ref{flucL7}a, there exists $k_0 \in \bbN$ such that 
\begin{align}
\left|\frac{\sigma_k {n_k}^{1/3} I_{m_k, n_k, p_k(s)+l}^{\bfa, \bfb}}{\zeta_k F_{m_k, n_k, p_k(s)+l}^{\bfa, \bfb}(\zeta_k)} - \Ai(s+n_k^{-1/3}\sigma_k^{-1}l)\right| &< \frac{\epsilon}{L} \label{flucE33.1}\\
\left|\frac{\sigma_{k} n_k^{1/3} I_{n_k, m_k, p_k(t)+l}^{\bfb, \bfa}}{\zeta_{k}^{-1} F_{n_k, m_k, p_k(t)+l}^{\bfb, \bfa}(\zeta_{k}^{-1})} - \Ai(t+n_k^{-1/3}\sigma_{k}^{-1}l)\right| &< \frac{\epsilon}{L}, \label{flucE33.3}
\end{align}
for $s, t \in [-s_0, s_0]$, $0 \le l < L\sigma_k n_k^{1/3}$ and $k \ge k_0$. Using Lemma \ref{flucL7}b, boundedness of the Airy function on $[-s_0, L+s_0]$ and (\ref{flucEq105}), we can combine (\ref{flucE33.1}) and (\ref{flucE33.3}) via triangle inequality to obtain 
\begin{align}
&\left|\sigma_k n_k^{1/3} \zeta_k^{p_k(t)-p_k(s)} I_{m_k, n_k, p_k(s)+l}^{\bfa, \bfb} I_{n_k, m_k, p_k(t)+l}^{\bfb, \bfa} - \sigma_k^{-1} n_k^{-1/3} \Ai(s+n_k^{-1/3}\sigma_k^{-1}l) \Ai(t+n_k^{-1/3}\sigma_k^{-1}l)\right| \nonumber\\
&< \frac{C\epsilon}{Ln_k^{1/3}} \label{flucE86}
\end{align} 
a.s.\ for some constant $C > 0$. By uniform continuity of the Airy function on $[-s_0, L+s_0]$, we also have  
\begin{align}
|\Ai(s+n_k^{-1/3}\sigma_k^{-1}l) \Ai(t+n_k^{-1/3}\sigma_k^{-1}l)-\Ai(s+x)\Ai(t+x)| < \frac{\epsilon}{L}
\label{flucE85}
\end{align}
whenever $s, t \in [-s_0, s_0]$, $0 < l \le L\sigma_k n_k^{1/3}$, $x \in [n_k^{-1/3}\sigma_k^{-1}(l-1), n_k^{-1/3}\sigma_k^{-1}l]$ and $k \ge k_0$, by choosing $k_0$ larger if necessary. It follows from (\ref{flucE86}) and (\ref{flucE85}) that 
\begin{align*}
\left|\sigma_k n_k^{1/3} \zeta_k^{p_k(t)-p_k(s)} I_{m_k, n_k, p_k(s)+l}^{\bfa, \bfb} I_{n_k, m_k, p_k(t)+l}^{\bfb, \bfa}- \int \limits_{n_k^{-1/3}\sigma_k^{-1}(l-1)}^{n_k^{-1/3}\sigma_k^{-1}l}\Ai(s+x)\Ai(t+x) \dd x\right| < \frac{C\epsilon}{Ln_k^{1/3}}
\end{align*}
a.s.\ for some constant $C > 0$. We now sum over $0 < l \le L\sigma_k \eta_k^{1/3}$ and obtain  
\begin{align}
\label{flucE35.1}
\left|\sigma_k n_k^{1/3} \zeta_k^{p_k(t)-p_k(s)} \sum \limits_{0 < l \le L\sigma_k n_k^{1/3}}I_{m_k, n_k, p_k(s)+l}^{\bfa, \bfb} I_{n_k, m_k, p_k(t)+l}^{\bfb, \bfa}- \int \limits_{0}^{L}\Ai(s+x)\Ai(t+x) \dd x\right| < C\epsilon 
\end{align} 
a.s.\ for some constant $C > 0$. Moreover, choosing $L$ large enough, we have 
\begin{align}
\label{flucE35.2}
\left |\int_{L}^{\infty} \Ai(s+x)\Ai(t+x) \dd x \right| < \epsilon
\end{align}
Finally, summing (\ref{flucE191}) over $l > L\sigma_k n_k^{1/3}$ gives  
\begin{align}
\label{flucE35.3}
\left|\sigma_k n_k^{1/3} \zeta_k^{p_k(t)-p_k(s)} \sum \limits_{l > L\sigma_k n_k^{1/3}}I_{m_k, n_k, p_k(s)+l}^{\bfa, \bfb} I_{n_k, m_k, p_k(t)+l}^{\bfb, \bfa} \right| \le Ce^{-c(s+t+2L)}< \epsilon 
\end{align}
for $s, t \in [-s_0, s_0]$, $k \ge k_0$ a.s.\ for some constants $C, c > 0$ provided that $L$ is sufficiently large. 
Then, we conclude (\ref{flucE36}) from (\ref{flucE35.1}), (\ref{flucE35.2}) and (\ref{flucE35.3}). 
\end{proof}

\section{Tracy-Widom fluctuations}

In this section, we combine Theorem \ref{flucT5} with standard estimates on the Airy kernel to establish the second equality in 
\begin{align}
\lim_{k \rightarrow \infty} \bfP^{\bfa, \bfb}(G(m_k, n_k) \le n_k \gamma_k + n_k^{1/3}\sigma_k s) &= \lim_{k \rightarrow \infty} \sum \limits_{l=1}^\infty \frac{(-1)^l}{l!} \sum \limits_{\substack{q_1, \dotsc, q_l\in \bbN \\ q_i \ge p_k(s)}} \det \limits_{i, j \in [l]}\left[K_{m_k, n_k}^{\bfa, \bfb}(q_i, q_j)\right] \nonumber \\ 
&= \sum \limits_{l=1}^\infty \frac{(-1)^l}{l!} \int \limits_{[s, \infty)^l} \det \limits_{i, j \in [l]} [\AiK(x_i, x_j)] \dd x_1 \dotsc \dd x_l \nonumber \\ %\label{flucEq54}\\
&= F_{\GUE}(s) \nonumber
\end{align}
for $s \in \bbR$ a.s. The first and last equalities come from (\ref{flucE105}) and (\ref{flucEq91}), respectively. 

\begin{lem}
\label{flucLe8}
For any $s_0 > 0$, there exists a constant $C > 0$ such that
\begin{align*}
\int \limits_{\substack{x_1, \dotsc, x_l \in \bbR \\ x_i \ge s_i}} \det \limits_{i, j \in [l]} \left[\AiK(x_i, x_j)\right] \dd x_1 \dotsc \dd x_l 
\le C^l e^{-\sum_{i=1}^l s_i}
\end{align*}
for $l \in \bbN$ and $s_1, \dotsc, s_l \ge -s_0$. 
\end{lem}
\begin{proof}
It follows from (\ref{flucEq106}) that $A$ is symmetric and 
\begin{align*}
\sum_{i, j \in [l]} v_i \AiK(x_i, x_j) v_j &= \sum_{i, j \in [l]} \int_{0}^\infty v_i \Ai(x_i+t) \Ai(x_j+t) v_j dt \\
&= \int_{0}^\infty \sum_{i, j \in [l]} v_i \Ai(x_i+t) \Ai(x_j+t) v_j dt \\
&= \int_{0}^\infty \left(\sum_{i=1}^l \Ai(x_i+t) v_i\right)^2 dt \ge 0
\end{align*}
for any $v_1, \dotsc, v_l \in \bbR$. That is, $[\AiK(x_i, x_j)]_{i, j \in [l]}$ is a nonnegative-definite matrix. Therefore, by (\ref{flucEq89}) and Hadamard's inequality, there exists a constant $C > 0$ such that 
\begin{align*}
\det_{i, j \in [l]} [\AiK(x_i, x_j)] \le \prod_{i=1}^l \AiK(x_i, x_i) \le C^l e^{-\sum_{i=1}^l x_i}
\end{align*}
whenever $x_i \ge -s_0$ for $i \in [l]$. Integrating over $x_i \ge s_i$ for $i \in [l]$, where $s_i \ge -s_0$, completes the proof.  
\end{proof}

\begin{proof}[Proof of Theorem \ref{flucT1}]
Introduce $\epsilon > 0$ and $S > s$. By Lemma \ref{flucL17}, 
\begin{align*}
\sum \limits_{l=1}^\infty \frac{1}{l!} \sum \limits_{\substack{q_1, \dotsc, q_l \in \bbN \\ q_i \ge p_k(s) \\ \max q_i \ge p_k(S)}} \det\limits_{i, j \in [l]} \left[\KK_{m_k, n_k}(q_i, q_j)\right] 
&\le Ce^{-cS} \sum \limits_{l=1}^\infty \frac{l^{l/2} e^{-cs(l-1)}C^{l-1}}{(l-1)!} 
\end{align*}
for $k \ge k_0$ for some $k_0 \in \bbN$ a.s.\ for some constants $C, c > 0$. The right-hand side is finite by the root test and can be made less than $\epsilon$ choosing $S$ sufficiently large. For such $S$, we similarly obtain from Lemma \ref{flucLe8} that
\begin{align*}
\sum \limits_{l=1}^\infty \frac{1}{l!} \int \limits_{\substack{x_1, \dotsc, x_l \ge s \\ \max x_i \ge S}} \det \limits_{i, j \in [l]} [\AiK(x_i, x_j)] \dd x_1 \dotsc \dd x_l < \epsilon. 
\end{align*}
Hence, it suffices to prove
\begin{align}
&\lim_{k \rightarrow \infty} \sum \limits_{l=1}^\infty \frac{(-1)^l}{l!} \sum \limits_{\substack{q_1, \dotsc, q_l \in \bbN \\ p_k(s) \le q_i < p_k(S)}} \det \limits_{i, j \in [l]}\left[\KK_{m_k, n_k}(q_i, q_j)\right] \nonumber\\
&= \sum \limits_{l=1}^\infty \frac{(-1)^l}{l!} \int \limits_{[s, S)^l} \det_{i, j \in [l]} \limits [\AiK(x_i, x_j)] \dd x_1 \dotsc \dd x_l. 
\label{flucEq106}
\end{align}
for $s \in \bbR$ a.s. By Lemma \ref{flucL17} and finiteness of $\sum_{l=1}^\infty l^{l/2}C^le^{-csl}/l!$, we can conclude the last statement from dominated convergence if we show that  
\begin{align}
\label{flucEq37}
\lim \limits_{k \rightarrow \infty} \sum \limits_{\substack{q_1, \dotsc, q_l \in \bbN \\ p_k(s) \le q_i < p_k(S)}} \det \limits_{i, j \in [l]} \left[\KK_{m_k, n_k}(q_i, q_j)\right] = \int \limits_{[s, S]^l} \det \limits_{i, j \in [l]} \left[\AiK(x_i, x_j)\right] \dd x_1 \dotsc \dd x_l. 
\end{align}
for each $l \in \bbN$ for $s \in \bbR$ a.s. 

Fix $l \in \bbN$. For $k \in \bbN$, consider the partition of the interval $[s, \infty)$ into intervals of length $\sigma_k^{-1} n_k^{-1/3}$ with endpoints at 
\begin{align*}t(q, k) = s+(q-p_k(s))\sigma_k^{-1}n_k^{-1/3}\end{align*}
for $q \ge p_k(s)$. Observe that 
\begin{align}p_k(t(q, k)) = \lf n_k \gamma_k + n_k^{1/3}\sigma_k t(q, k) \rf = \lf n_k \gamma_k+n_k^{1/3}\sigma_k s\rf + q-p_k(s) = q.\end{align} 
Also, for $p_k(s) \le q_1, \dotsc, q_l < p_k(S)$, we have 
$s \le t(q_i, k) < s + (p_k(S)-p_k(s)) \sigma_k^{-1} n_k^{-1/3} \le S+1$ for each $i \in [l]$.  
Therefore, by Theorem \ref{flucT5}, there exists $k_0 \in \bbN$ such that
\begin{align}
\label{flucE38}
\left| \det\limits_{i, j \in [l]} \left[\KK_{m_k, n_k}(q_i, q_j)\right]  - \sigma_k^{-l} n_k^{-l/3}\det_{i, j \in [l]} \left[\AiK(t(q_i, k), t(q_j, k))\right]\right| < \epsilon \sigma_k^{-l} n_k^{-l/3}
\end{align}  
whenever $k \ge k_0$ and $p_k(s) \le q_i < p_k(S)$ for $i \in [l]$. By uniform continuity of $\det[\AiK(x_i, x_j)]_{i, j \in [l]}$ on $[s, S+1]^{l}$, choosing $k_0$ larger if necessary, we obtain   
\begin{align}
\label{flucE39.1}
\left| \det\limits_{i, j \in [l]} \left[\KK_{m_k, n_k}(q_i, q_j)\right]  - \int_{R_{q_1, \cdots, q_l, k}} \det \limits_{i, j \in [l]}[\AiK(x_i, x_j)]\  \dd x_1 \cdots \dd x_l  \right| < 2\epsilon \sigma_k^{-l} n_k^{-l/3}, 
\end{align}
where $R_{q_1, \dotsc, q_l, k}$ denotes the product of the intervals $[t(q_i, k), t(q_i+1, k)]$ for $i \in [l]$. The pairwise intersections of $\{R_{q_1, \dotsc, q_l, k}: p_k(s) \le q_i < p_k(S)\}$ are Lebesgue null-sets and their union is 
\begin{align}\label{flucE91}[t(p_k(s), k), t(p_k(S), k)]^l = [s, t(p_k(S), k)]^l. \end{align}
Hence, by the triangle inequality and (\ref{flucE39.1}), 
\begin{align}
\label{flucEq40}
&\left|\sum \limits_{\substack{q_1, \dotsc, q_l \in \bbN \\ p_k(s) \le q_i < p_k(S)}} \det\limits_{i, j \in [l]} \left[\KK_{m_k, n_k}(q_i, q_j)\right]  - \int_{[s, t(p_k(S), k)]^l} \det \limits_{i, j \in [l]} [\AiK(x_i, x_j)] \dd x_1 \cdots \dd x_l  \right| \\
&< 2\epsilon\sigma_k^{-l}n_k^{-l/3}(p_k(S)-p_k(s))^l \le 2\epsilon (S-s+1)^l \nonumber
\end{align}
The set in (\ref{flucE91}) differs from $[s, S]^l$ by a set of measure 
\begin{align}\label{flucE92}|t(p_k(S), k)^l-S^l| \le \sigma_k^{-1}n_k^{-1/3} l (S+1)^{l-1}, \end{align}
where the inequality follows from $|t(p_k(S), k)-S| \le \sigma_k^{-1} n_k^{-1/3}$ and the mean value theorem.  Because (\ref{flucE92}) can be made arbitrarily small and the Airy kernel is bounded on $[s, S+1]$,  we have
\begin{align}
\label{flucEq41}
\left|\int_{[s, t(p_k(S), k)]^l} \det \limits_{i, j \in [l]}[\AiK(x_i, x_j)] \dd x_1 \cdots \dd x_l   - \int_{[s, S]^l} \det \limits_{i, j \in [l]}[\AiK(x_i, x_j)] \dd x_1 \cdots \dd x_l  \right| < \epsilon, 
\end{align}
for $k \ge k_0$ by choosing $k_0$ large enough. Now (\ref{flucEq37}) follows from combining (\ref{flucEq40}) and (\ref{flucEq41}).  
\end{proof}

%% Start the appendices:
\appendix       % Chapters, sections are now appendix style
\chapter{Appendix}\label{flucappendix}
%\fixchapterheading
\setcounter{theorem}{0}

\section{Tracy-Widom GUE distribution}
\label{flucSe2}

%We recall the definitions of the Airy function, the Airy kernel and the Tracy-Widom GUE distribution, and record some standard facts about them that will be needed later. For brevity, we omit the proofs, which can be found, for example, in \cite[Chapter~4]{Seppalainen}. 

The Airy function can be defined as the contour integral   
\begin{equation}
\label{flucE9}
\Ai(s) = \frac{1}{2\pi\ii} \int \limits_{\sC} e^{z^3/3 - sz} \dd z \quad \text{ for } s \in \bbR, 
\end{equation}
where contour $\sC$ consists of the rays from $\infty e^{-{\ii}\theta}$ to $0$ and from $0$ to $\infty e^{\ii \theta}$ for any $\theta \in (\pi/6, \pi/2)$. This integral is absolutely and uniformly convergent on compact subsets of $\bbR$. Up to a constant factor, the Airy function is the unique solution of the ODE  
\begin{align*}
\frac{d^2 u}{ds^2} = su, \quad s \in \bbR, 
\end{align*}
known as the Airy equation, subject to the condition $u(s) \rightarrow 0$ as $s \rightarrow \infty$ \cite[Chapter~9]{Olver}. In the sequel, we will use continuity of the Airy function and the following bound. Given $s_0 > 0$, there exists a constant $C > 0$ such that 
\begin{align}
\label{flucEq89}
|\Ai(s)| \le Ce^{-s} \quad \text{ for } s \ge -s_0. 
\end{align}
These properties can be derived from (\ref{flucE9}); we suggest \cite[Chapter~4]{Seppalainen09} for details. 

One way to define the Airy kernel is by 
\begin{equation}
\AiK(s, t) = \int \limits_{0}^\infty \Ai(s+x) \Ai(t + x) \dd x \quad \text{ for } s, t \in \bbR, 
\end{equation}
where the absolute and the uniform convergence of the integral over compact subsets of $\bbR^2$ are ensured by (\ref{flucEq89}). Moreover, the Airy kernel is continuous and for each $s_0 > 0$ there is a constant $C > 0$ such that
\begin{align}
\label{flucEq90}
|\AiK(s, t)| \le Ce^{-s-t} \quad \text{ for } s, t \ge -s_0. 
\end{align}
For any $a \in \bbR$, we can view the Airy kernel as the kernel of the integral operator on $L^2((a, \infty))$ that maps $f$ to the function 
\begin{align*}s \mapsto \int_a^\infty \AiK(s, t) f(t) \dd t. \end{align*}
That this image is in $L^2((a, \infty))$ comes from (\ref{flucEq90}) and an application of the Cauchy-Schwarz inequality. 

The $n \times n$ Gaussian Unitary Ensemble (GUE) is the distribution of the random matrix $X = [X(i, j)]_{i, j \in [n]}$ with the following properties. 
\begin{enumerate}[(\romannumeral1)]
\item $X(i, i)$ is distributed as the real normal distribution with mean $0$ and variance $1$ for $i \in [n]$.  
\item $X(i, j)$ is distributed as the complex normal distribution with mean $0$ and variance $1$ for distinct $i, j \in [n]$. That is, $\Re \{X(i, j)\}$ and $\Im \{X(i, j)\}$ are distributed as independent real normal distributions with mean $0$ and variance $1/2$ for distinct $i, j \in [n]$. 
\item The entries $\{X(i, j): 1 \le i \le j \le n\}$ are independent. 
\item $X$ is Hermitian, that is, $X(i, j) = \conj{X(j, i)}$ for $i, j \in [n]$. 
\end{enumerate}
Being a Hermitian matrix, $X$ has $n$ real eigenvalues $\lambda_1 \ge \dotsc \ge \lambda_n$. The Tracy-Widom GUE distribution arises as the distributional limit of the rescaled largest eigenvalue
\begin{align*}n^{1/6}(\lambda_1-2n^{1/2})\end{align*}
as $n \rightarrow \infty$.  
Its cumulative distribution function is given by the following Fredholm determinant of the integral operator whose kernel is the Airy kernel \cite{Forrester}:
\begin{equation}
\label{flucEq91}
F_{\GUE}(s) = 1 + \sum \limits_{l=1}^\infty \frac{(-1)^l}{l!} \int \limits_{[s, \infty)^l} \det_{i, j \in [l]} [\AiK(x_i, x_j)] \dd x_1 \dotsc \dd x_l \quad \text{ for } s \in \bbR.
\end{equation}
The absolute convergence of the series above follows from (\ref{flucEq90}) and Hadamard's inequality, see Lemma \ref{flucLe8} below. 
Another characterization of the Tracy-Widom distribution is 
\begin{align*}
F_{\GUE}(s) = \exp\left(-\int_{s}^\infty (t-s) q(t)^2 \dd t \right) \quad \text{ for } s \in \bbR, 
\end{align*}
where $q$ is the unique solution of the Painlev\'{e} \RNum{2} equation
\begin{align*}
\frac{d^2u}{ds^2} = 2u^3 + su \quad \text{ for } s \in \bbR
\end{align*}
subject to the condition $u(s)/\Ai(s) \rightarrow 1$ as $s \rightarrow \infty$, see \cite{TracyWidom}. 

\section{Steepest-descent curves of harmonic functions}
\label{flucAp1}

In this section, we define precisely and discuss general properties of steepest-descent curves of harmonic functions. Our primary interest is in the class of harmonic functions that are real linear combinations of translated and dilated versions of $z \mapsto \log |z|$. A motivating example is given by (\ref{flucE4}). We develop lemmas that describe the global nature of and provide sufficiently strong control over steepest-descent curves. Some of the facts presented below 
%(especially until {\color{red} Lemma ...}) 
are routinely used in steepest-descent analysis. We phrase these as formal statements for ease of reference and include their proofs to keep the discussion more self-contained. 

\subsection{Definition and basic properties}
\label{S2Sub1}

We will appeal to the standard theory of ODEs through Theorem \ref{flucT4} below. 
Let $U$ be an open subset of $\bbR^d$ and $F: U \rightarrow \bbR^d$ be $\sC^1$. Let $t_0 \in \bbR$ and $x_0 \in U$. Consider the initial value problem 
\begin{align}
\label{flucE54}
x(t_0) = x_0 \qquad x'(t) = F(x(t)). 
\end{align}
Let $I$ be an open interval containing $t_0$ and $\varphi: I \rightarrow U$ be $\sC^1$. We call $\varphi$ a solution of (\ref{flucE54}) if $\varphi(t_0) = x_0$ and $\varphi'(t) = F(\varphi(t))$ for $t \in I$. 
\begin{thm}
\label{flucT4}\ 
\begin{enumerate}[(a)]
\item There exists a unique solution $\Phi: J \rightarrow U$ of (\ref{flucE54}) such that if $\varphi: I \rightarrow U$ is any other solution of (\ref{flucE54}) then $I \subset J$ and $\varphi$ is the restriction of $\Phi$ to $I$.  
\item Let $K \subset U$ be compact. Write $t^-$ and $t^+$ for the left and right endpoints of $J$, respectively. Then either $t^- = -\infty$ or $\Phi(t) \not \in K$ for some $t \in (t^-, t_0]$. Similarly, either $t^+= \infty$ or $\Phi(t) \not \in K$ for some $t \in [t_0, t^+)$. 
\end{enumerate}
\end{thm}
See, for example, \cite[Chapter~1]{Arnold} and \cite[Section~7.2]{HirschSmaleDevaney} for a proof. 

Let $u \in \sC^2(U)$ be real-valued. Let $I$ be an open interval and $\varphi: I \rightarrow U$ be a $\sC^1$ curve with unit speed i.e. $|\varphi'(t)| = 1$ for $t \in I$. 
Since $|\varphi'| = 1$, we have
\begin{align}
\label{flucE59}
t-s = \int \limits_s^t |\varphi'(\tau)| \dd \tau \ge \bigg|\int \limits_s^t \varphi'(\tau) \dd \tau\bigg| = |\varphi(t)-\varphi(s)| \ge ||\varphi(t)|-|\varphi(s)|| 
\end{align} 
for $s, t \in I$ with $s < t$.
 
\begin{defn}\label{flucD1} We call $\varphi$
a steepest-descent curve of $u$ if
\begin{align}\frac{d}{dt}u(\varphi(t)) = - |\nabla u(\varphi(t))| \quad \text{ for } t \in I, \label{flucE40}\end{align}
a steepest-ascent curve of $u$ if 
\begin{align}\frac{d}{dt}u(\varphi(t)) = |\nabla u(\varphi(t))| \quad \text{ for } t \in I, \label{flucE41}\end{align}
a stationary curve of $u$ if 
\begin{align}\frac{d}{dt}u(\varphi(t)) = 0 \quad \text{ for } t \in I. \label{flucE42}\end{align}
\end{defn}
It is clear from (\ref{flucE42}) that if $\varphi$ is a stationary curve of $u$ then the reversed curve $\tilde{\varphi}: t \mapsto \varphi(-t)$ defined for $-t \in I$ is also a stationary curve of $u$. Similarly, by (\ref{flucE40}) and (\ref{flucE41}), $\varphi$ is a steepest-descent curve of $u$ if and only if $\tilde{\varphi}$ is a steepest-ascent curve of $u$. 

Here on $d = 2$. Identify $(x, y) \in \bbR^2$ with $z = x + \bfi y \in \bbC$. Let $V = \{z \in U: \nabla u(z) \neq 0\}$. Assume that $u$ is nonconstant and, thus, $V$ is nonempty. Furthermore, assume that $\varphi(I) \subset V$. 
Then the chain rule $\frac{d}{dt}u(\varphi(t)) = \nabla u(\varphi(t)) \cdot \varphi'(t)$ implies that (\ref{flucE40}), (\ref{flucE41}) and (\ref{flucE42}) are equivalent to
\begin{align}
\varphi'(t) &= -\frac{\nabla u(\varphi(t))}{|\nabla u(\varphi(t)|} \quad \text{ for } t \in I, \label{flucE79}\\ 
\varphi'(t) &= \frac{\nabla u(\varphi(t))}{|\nabla u(\varphi(t)|} \quad \ \ \ \text{ for } t \in I, \label{flucE80}\\
\varphi'(t) &= \pm \ii \frac{\nabla u(\varphi(t))}{|\nabla u(\varphi(t)|} \quad \text{ for } t \in I, \label{flucE43}
\end{align} 
respectively. Since $\dfrac{\nabla u}{|\nabla u|} \in \sC^1(V)$, if $\varphi$ satisfies one of (\ref{flucE79}), (\ref{flucE80}) and (\ref{flucE43}) then, by Theorem \ref{flucT4}a, we can extend $\varphi$ as a solution of the ODE uniquely to a maximal domain.   

From this point on, when we refer to a steepest-descent or -ascent curve of $u$, it will be understood that the image of the curve is contained in $V$. There are two reasons for the special treatment of the zeros of $\nabla u$. First, there are multiple steepest-descent and -ascent curves emanating from a zero, which require some work to describe them precisely, see the paragraph of (\ref{flucE110}) below. Second,  as we will observe in Lemma \ref{flucL6}, when a steepest-descent curve $u$ goes through a zero of $\nabla u$ it can turn into a steepest-ascent curve, and vice versa. Such behavior naturally leads us to separate the analysis of the curve into two cases, namely, before and after the curve goes through the zero. 

Now suppose that $u$ is harmonic i.e. $\Delta u= 0$ where $\Delta = \partial_x^2 + \partial_y^2$ is the Laplacian. On any open set $U' \subset U$ that is simply connected, $u$ equals the real part of a holomorphic function, which is unique up to a purely imaginary constant. Let $a \in U$ and $r > 0$ such that $\Disc(a, r) \subset U$. Since disks are simply-connected, there exists $f \in \Hol(\Disc(a, r))$ such that $u = \Re f$ on $\Disc(a, r)$. Let $v = \Im f$. %Along a steepest-descent curve of The utility of the next lemma is that $f \circ \varphi$ is a strictly monotone real-valued plus a pure imaginary constant. 
\begin{lem}
\label{flucL1}
$\varphi$ is a steepest-descent or -ascent curve of $u$ if and only if $\varphi$ is a stationary curve of $v$. 
\end{lem}
\begin{proof}
The Cauchy-Riemann equations $\partial_x u = \partial_y v$ and $\partial_y u = - \partial_x v$ imply that, for $z \in \Disc(a, r)$,  
\begin{align}
\label{flucE44}
\nabla u(z) = (\partial_x u(z), -\partial_x v(z)) = \conj{f'(z)} \qquad \nabla v(z) = (\partial_x v(z), \partial_x u(z)) = \bfi \conj{f'(z)}. 
\end{align}
Now the conclusion is immediate from (\ref{flucE79}), (\ref{flucE80}) and (\ref{flucE43}). 
\end{proof}

\subsection{Steepest-descent and -ascent curves that emanate from a point}
\label{S2Sub2}

We will utilize the following basic estimate.  
\begin{lem}[Cauchy estimate]
\label{flucL20}
Let $m \in \bbZ_+$, $z_0 \in \bbC$, $0 < \rho_1 < \rho_2 < \rho$, and $F \in \Hol(\Disc(z_0, \rho))$. Then 
\begin{align}
\bigg|F(z)-\sum \limits_{i = 0}^m \frac{F^{(i)}(z_0)}{i!}(z-z_0)^i\bigg| \le \frac{\rho_1^{m+1}}{\rho_2^{m+1}} \frac{\sup \limits_{\Circle(z_0, \rho_2)} |F|}{1-\dfrac{\rho_1}{\rho_2}} \quad \text{ for } z \in \conj{\Disc}(z_0, \rho_1). \label{appE1}
\end{align}
\end{lem}
\begin{proof}
By Cauchy's integral formula, 
\begin{align}
F^{(i)}(z_0) = \frac{i!}{2\pi \ii} \oint \limits_{|z-z_0| = \rho_2} \frac{F(z) \dd z}{(z-z_0)^{i+1}}  \qquad \text{ for } i \in \bbZ_+. 
\end{align}
Use this with the power series representation of $F$ around $z_0$ to obtain that, for $z \in \conj{\Disc}(z_0, \rho_1)$, the left-hand side of (\ref{appE1}) does not exceed
\begin{align}
\sum \limits_{i = m+1}^\infty \frac{|F^{(i)}(z_0)|}{i!}|z-z_0|^i &\le \frac{1}{2\pi} \sum \limits_{i = m+1}^\infty \bigg| \oint \limits_{|z-z_0| = \rho_2} \frac{F(z) \dd z}{(z-z_0)^{i+1}} \bigg| |z-z_0|^i \\
&\le \sup \limits_{\Circle(z_0, \rho_2)} |F| \sum \limits_{i=m+1}^\infty \frac{\rho_1^i}{\rho_2^{i}}, 
\end{align}
which is the right-hand side of (\ref{appE1}). 
\end{proof}

We next describe locally the steepest-descent and -ascent curves of $u$. Since $u$ is assumed nonconstant, $f$ is nonconstant as well. Let $n$ be the smallest positive integer such that $f^{(n)}(a) \neq 0$. Since $f$ is uniquely determined up to constant, $n$ is well-defined given $u$ and $a$. Fix an $n$th root of $f^{(n)}(a)$ and let $\xi \in \Circle(0, 1)$ be its direction. 

Introduce constants 
\begin{align}
K_0 = \frac{2^{n+1}}{r^{n+1}} \frac{n!}{|f^{(n)}(a)|} \sup \limits_{\conj{\Disc}(a, r/2)}\{|f| + r|f'|\} \qquad \epsilon_0 = \frac{1}{16} \min \bigg\{r, \frac{1}{K_0}, 1\bigg\} \label{flucE112}
\end{align}
Also, define $\tilde{f} \in \Hol(\Disc(a, r))$ by 
\begin{align}
\label{flucE82}
f(z)  = f(a) + (z-a)^n \tilde{f}(z) \text{ for } z \in \Disc(a, r) \qquad \tilde{f}(a) = \dfrac{f^{(n)}(a)}{n!}. 
\end{align}
It follows from Lemma \ref{flucL20} that if $|z-a| < \epsilon_0$ then 
\begin{align*}
|\tilde{f}(z)-\tilde{f}(a)| &= \frac{1}{|z-a|^n} \bigg|f(z)-f(a)-\dfrac{f^{(n)}(a)}{n!}(z-a)^n\bigg| \le \frac{\epsilon_0}{(r/2)^{n+1}} \frac{1}{1- 2\epsilon_0/r}\sup_{\Circle(a, r/2)} |f| \\
&\le 2\epsilon_0 K_0 \dfrac{|f^{(n)}(a)|}{n!} \le \frac{|\tilde{f}(a)|}{2}. 
\end{align*}
Therefore, $|\tilde{f}(z)| \ge |\tilde{f}(a)|/2 > 0$ for $z \in \Disc(a, \epsilon_0)$. Hence, we can define $h \in \Hol(\Disc(a, \epsilon_0))$ by 
\begin{align}
\label{flucE23}
h(z) = \xi \bigg(\dfrac{|f^{(n)}(a)|}{n!}\bigg)^{1/n} (z-a) \exp\bigg(\frac{1}{n}\int_{a}^z \frac{\tilde{f}'(w)}{\tilde{f}(w)}\dd w\bigg) \qquad \text{ for } z \in \Disc(a, \epsilon_0), 
\end{align}
where the contour can be chosen as the line segment $[a, z]$. Note from (\ref{flucE82}) that 
\begin{align}
h(z)^n &= \dfrac{\xi^n|f^{(n)}(a)|}{n!}(z-a)^n \exp\bigg(\int_{a}^z \frac{\tilde{f}'(w)}{\tilde{f}(w)}\dd w\bigg) = \dfrac{f^{(n)}(a)}{n!}(z-a)^n \frac{\tilde{f}(z)}{\tilde{f}(a)} \nonumber \\
&= f(z)-f(a). \label{flucE81}
\end{align}
For the second equality above, observe that all derivatives of $\dfrac{\tilde{f}(z)}{\tilde{f}(a)}$ and $\exp\bigg(\displaystyle \int_{a}^z \dfrac{\tilde{f}'(w)}{\tilde{f}(w)}\dd w\bigg)$ are equal at $z = a$. 
Also, by (\ref{flucE23}), $h$ is nonzero on $\Disc(a, \epsilon_0) \smallsetminus \{a\}$ and 
\begin{align}
\label{flucE83}
h'(a) = \xi \bigg(\dfrac{|f^{(n)}(a)|}{n!}\bigg)^{1/n}. 
\end{align}
\begin{lem}
\label{flucL37}
$|h'(z)-h'(a)| \le \dfrac{16|h'(a)|}{\epsilon_0}|z-a|$ for $z \in \Disc(a, \epsilon_0/4)$. 
\end{lem}
\begin{proof}
Differentiating (\ref{flucE23}) and using (\ref{flucE83}) yield    
\begin{align}
h'(z) = h'(a) \exp\bigg(\frac{1}{n}\int_{a}^z \frac{\tilde{f}'(w)}{\tilde{f}(w)}\dd w\bigg) \bigg(1 + \frac{1}{n}\frac{\tilde{f}'(z)}{\tilde{f}(z)}(z-a) \bigg) \quad \text{ for } z \in \Disc(a, \epsilon_0). \label{flucE182}
\end{align}
Define $g \in \Hol(\Disc(a, r))$ by $g(z) = -n(f(z)-f(a))+f'(z)(z-a)$. Note from (\ref{flucE82}) that $\tilde{f}'(z) = (z-a)^{-n-1} g(z)$ for $z \in \Disc(a, r) \smallsetminus \{a\}$, and $g^{(k)}(a) = 0$ for $0 \le k \le n$. Therefore, by Lemma \ref{flucL20}, 
\begin{align*}
|\tilde{f}'(z)| &= \frac{|g(z)|}{|z-a|^{n+1}} \le \frac{2^{n+2}}{r^{n+1}} \sup \limits_{\Circle(a, r/2)} |g| \le \frac{2^{n+3}n}{r^{n+1}} \sup \limits_{\conj{\Disc}(a, r/2)} \{|f|+r|f'|\} = 4n K_0 |\tilde{f}(a)|
\end{align*}
for $z \in \Disc(a, r/4) \smallsetminus \{a\}$. It follows that $\dfrac{|\tilde{f}'(z)|}{|\tilde{f}(z)|} \le 8nK_0$ for $z \in \Disc(a, r/4)$ and, by (\ref{flucE182}), 
\begin{align*}
|h'(z)| \le |h'(a)| \exp\{8\epsilon_0K_0\} (1+8K_0) \le 4|h'(a)| \quad \text{ for } z \in \Disc(a, \epsilon_0). 
\end{align*}
Then another application of Lemma \ref{flucL20} gives 
\begin{align}
|h'(z)-h'(a)| \le \frac{4|z-a|}{\epsilon_0} \sup \limits_{\conj{\Disc}(a, \epsilon_0/2)}|h'| \le \dfrac{16|h'(a)|}{\epsilon_0}|z-a| \quad \text{ for } z \in \Disc(a, \epsilon_0/4). \qquad \qedhere
\end{align} 
\end{proof}

Choose $0 < \epsilon \le \epsilon_0/32$. Then, for $z \in \conj{\Disc}(a, \epsilon)$, we have $|h'(z)-h'(a)| \le |h'(a)|/2$, by Lemma \ref{flucL37}. In particular, $h'$ is nonzero on $\Disc(a, \epsilon)$.
 
For $1 \le j \le n$, let $\varphi_j$ denote the unique steepest-ascent curve $\varphi_j: I_j \rightarrow \Disc(a, \epsilon)$ of $\Re(e^{-\bfi \pi j/n} h)$ that passes through $a$, that is, 
\begin{align}
\label{flucE64}
\varphi_j(0) = a \qquad \varphi_j'(t) = e^{\bfi \pi j/n} \frac{\conj{h'}(\varphi_j(t))}{|h'(\varphi_j(t))|} \quad \text{ for } t \in I_j. 
\end{align}
In particular, note from (\ref{flucE183}) that 
\begin{align}
\label{appE4}
\varphi_j'(0) = \conj{\xi}e^{\ii \pi j/n} \quad \text{ for } j \in [n]. 
\end{align}
These curves can be related to the steepest-descent and -ascent curves of $u$ via the following proposition. 
\begin{prop}
\label{flucL6}For $1 \le j \le n$, 
\begin{align}
\label{flucE45}
\varphi_j'(t) = (-1)^j \sgn(t)^{n-1} \frac{\nabla u(\varphi_j(t))}{|\nabla u(\varphi_j(t))|} \quad \text{ for } t \in I_j \smallsetminus \{0\}. 
\end{align}
\end{prop}
\begin{proof}
By Lemma \ref{flucL1}, $\varphi_j$ is a stationary curve of $\Im(e^{-\bfi \pi j/n} h)$. This and $h(\varphi_j(0)) = h(a) = 0$ imply that $e^{-\bfi \pi j/n}h(\varphi_j(t))$ is real-valued, increasing for $t \in I_j$ and, therefore, has sign $\sgn(t)$. Also using $f' = nh^{n-1}h'$, we compute for $t \in I_j \smallsetminus \{0\}$
\begin{align*}
\frac{f'(\varphi(t))}{|f'(\varphi(t))|} &= (-1)^{j} \frac{[e^{-\bfi \pi j/n}h(\varphi_j(t))]^{n-1}}{|h(\varphi_j(t))|^{n-1}} e^{-\bfi \pi j/n}\frac{h'(\varphi(t))}{|h'(\varphi(t))|}\\
&= (-1)^{j} \sgn(t)^{n-1} e^{-\bfi \pi j/n} \frac{h'(\varphi(t))}{|h'(\varphi(t))|}\\ 
&= (-1)^j \sgn(t)^{n-1}\conj{\varphi_j'}(t), 
\end{align*}
which gives (\ref{flucE45}). 
%
%\begin{align*}
%\frac{d}{dt}(u(\varphi_j(t))) &= \nabla u(\varphi_j(t)) \cdot \varphi_j'(t) = \Re\{f'(\varphi_j(t))\varphi_j(t)\}\\
%&= \Re\bigg\{nh(\varphi_j(t))^{n-1}h'(\varphi_j(t)) e^{\bfi \pi j/n} \frac{\conj{h'}(\varphi_j(t))}{|h'(\varphi_j(t))|}\bigg\}\\
%&= (-1)^j n |h'(\varphi_j(t))| \Re\{[e^{-\bfi \pi j/n}h(\varphi_j(t))]^{n-1}\} \\
%&= (-1)^j \sgn(t)^{n-1} |nh'(\varphi_j(t))h(\varphi_j(t))^{n-1}| \\
%&= (-1)^j \sgn(t)^{n-1}|f'(\varphi_j(t))|, 
%\end{align*}
%which is (\ref{flucE45}).  
\end{proof}

For each $j \in [n]$, extend $\varphi_j$ through (\ref{flucE45}) to its maximal domain $(T_j^-, T_j^+) \supset I_j$, where possibly $T_j^+ = \infty$ and $T_j^- = -\infty$. To be clear, $\varphi_j$ is now defined on $(T_j^-, T_j^+) \smallsetminus \{0\}$ by (\ref{flucE45}) and on $I_j$ by (\ref{flucE64}). Proposition \ref{flucL6} shows that these ODEs coincide on $I_j \smallsetminus \{0\}$, hence, the extension is well-defined. Define curves $\varphi_j^+: [0, T_j^+) \rightarrow V \cup \{a\}$ and $\varphi_j^-: [0, -T_j^-) \rightarrow V \cup \{a\}$ by 
\begin{align}
\varphi_j^+(t) &= \varphi_j(t) \qquad \ \ \text{ for } 0 \le t < T_j^+ \label{flucE110}\\
\varphi_j^-(t) &= \varphi_j(-t) \qquad \text{ for } 0 \le t < -T_j^-, \label{appE5}
\end{align}
respectively. By Lemma \ref{flucL6}, $\varphi_j^+ \big|_{(0, T_j^+)}$ is a steepest-descent curve of $u$ for odd $j$ and a steepest-ascent curve of $u$ for even $j$. Similarly, $\varphi_j^- \big|_{(0, -T_j^-)}$ is a steepest-descent curve of $u$ for odd $j+n$ and a steepest-ascent curve of $u$ for even $j+n$. 

Our next purpose is to show that the curves $\varphi_j^{\pm}$ are, in fact, all steepest-descent and -ascent curves of $u$ \emph{emanating from} $a$. By Lemma \ref{flucL1},  $\varphi_j$ is a stationary curve of $v$ for $j \in [n]$; therefore, 
\begin{align}
\label{flucE88}
\bigcup \limits_{j \in [n]} \{\varphi_j(t): t \in (T_j^-, T_j^+)\} \cap \Disc(a, \delta) \subset \{z \in \Disc(a, \delta): v(z) = v(a)\} 
\end{align} 
for $0 < \delta < r$.  
\begin{lem}
\label{flucL12}
Let $0 < \delta < \epsilon$. If $\delta$ is sufficiently small then the inclusion in (\ref{flucE88}) is, in fact, equality. That is, 
\begin{align}
\bigcup \limits_{j \in [n]} \{\varphi_j(t): t \in (T_j^-, T_j^+)\} \cap \Disc(a, \delta) = \{z \in \Disc(a, \delta): v(z) = v(a)\} \label{appE2}
\end{align} 
\end{lem}
\begin{proof}
For $z \in \Disc(a, \epsilon)$, $v(z) = v(a)$ is equivalent to $\Im(h(z)^n) = 0$. The latter holds if and only if 
\begin{align}
\label{flucE62}
\Im \{e^{-\bfi \pi j/n}h(z)\} = 0
\end{align}
for some $j \in [n]$. Since $h'(a) \neq 0$, it follows from the implicit function theorem that there exist an open set 
$U_j$ containing $a$, an open interval $J_j$ containing $0$ and $\sC^1$ curve $\psi_j: J_j \rightarrow U_j$ with $\psi_j(0) = a$ and $|\psi'| = 1$ such that (\ref{flucE62}) holds with $z \in U_j$ if and only if $z = \psi_j(t)$ for some $t \in J_j$. 
Because $\psi_j$ is a stationary curve of $\Im\{e^{-\bfi \pi j/n}h(z)\}$, by Lemma \ref{flucL1}, 
\begin{align}
\label{flucE63}
\psi_j'(t) = \pm e^{\bfi \pi j/n} \frac{\conj{h'}(\psi_j(t))}{|h'(\psi_j(t))|} \quad \text{ for } t \in J_j. 
\end{align}
Working with the curve $t \mapsto \psi_j(-t)$ instead if necessary, we can fix the sign in (\ref{flucE63}) as plus. Then $\psi_j$ is also a solution of (\ref{flucE64}). Hence, by Theorem \ref{flucT4}a, $J_j \subset I_j$ and $\psi_j = \varphi_j$ on $J_j$.  To complete the proof, choose $\delta$ small so that $\Disc(a, \delta) \subset \bigcap_{j \in [n]} U_j$.
\end{proof}

%The next lemma shows that the left-hand side of (\ref{appE2}) is almost a disjoint union. 
\begin{lem}
\label{appL1}
For $j, k \in [n]$ with $k \neq j$, 
$\varphi_j((T_j^-, T_j^+)) \cap \varphi_{k}((T_k^-, T_k^+)) = \{a\}$. 
\end{lem}
\begin{proof}
Arguing by contradiction, suppose that $\varphi_j(t_0) = \varphi_k(s_0) = b$ for some $k \neq j$, $t_0 \in (T_j^-, T_j^+) \smallsetminus \{0\}$ and $s_0 \in (T_k^{-}, T_k^+) \smallsetminus \{0\}$. Assume that $u(b) < u(a)$; the case $u(b) > u(a)$ is treated similarly. Then both $t \mapsto \varphi_j(t_0-\sgn(t_0)t)$ and $t \mapsto \varphi_k(s_0 -\sgn(s_0)t)$ are steepest-descent ascent curves passing through $b$ at time $0$. Since $\nabla u(b) \neq 0$, these curves are restrictions of the unique maximal steepest-ascent curve $\Phi$ passing through $b$ at time $0$. Also, since $\nabla u(a) = 0$ (having distinct of $j$ and $k$ implies $n > 1$ and $\nabla u(a) = 0$) and $\varphi_j(0) = \varphi_k(0) = a$, the right endpoint $T$ of the domain of $\Phi$ must be equal to both $|t_0|$ and $|s_0|$. Now $t_0 = \sgn(t_0) T$ and $s_0 = \sgn(s_0) T$, and we have $\varphi_j(\sgn(t_0) (T-t)) = \varphi_k(\sgn(s_0) (T-t))$ for $0 \le t \le T$. Differentiate both sides, use (\ref{flucE64}) and set $t = T$. We obtain $\sgn(s_0) \sgn(t_0) = e^{\ii \pi (j-k)/n}$, which is not possible as $0 < |j-k| < n$. 
\end{proof}

\begin{prop}
\label{flucL25}
Let $T > 0$ and $\psi: [0, T) \rightarrow U$ be a $\sC^1$ ($\psi'$ is continuous up to $0$) curve with unit speed such that $\psi(0) = a$ and $\psi \big|_{(0, T)}$ is a steepest-descent or -ascent curve of $u$ with $\psi((0, T)) \subset V$. Then $\psi$ is a restriction of $\varphi_j^{\pm}$ for some unique $j \in [n]$ and choice of sign $\pm$. 
\end{prop}
\begin{proof}
For concreteness, suppose that $\psi$ is a steepest-descent curve. Choose $t_0 > 0$ small enough such that $z_0 = \psi(t_0) \in \Disc(a, \delta)$, where $\delta > 0$ is as in Lemma \ref{flucL12}. Therefore, there exists $j \in [n]$ and $s_0 \in (T_j^-, T_j^+) \smallsetminus \{0\}$ such that $z_0 = \varphi_j(s_0)$. In fact, $j$ and $s_0$ are unique by Lemma \ref{appL1} and the strict monotonicity of $u \circ \varphi_j$.  

Since $\nabla u(z_0) \neq 0$, the unique steepest-descent curve of $u$ that passes through $z_0$ (at time $0$) is $t \mapsto \psi(t_0+t)$ for time values $t \in (-t_0, -t_0+T)$. We can also generate other such curves from $\varphi_j$. Consider the case $s_0 > 0$. Then $t \mapsto \varphi_j(s_0+t)$ for $(-s_0, -s_0+T_j^{+})$ is also a steepest-descent curve of $u$ and passes through $z_0$ at time $0$. By uniqueness, these curves coincide on $(-(s_0 \wedge t_0), -t_0+T)$. Since $\varphi_j(0) = \psi(0) = a$, we have $a = \lim_{t \downarrow -(s_0 \wedge t_0)} \psi(t_0+t) = \psi((t_0-s_0)_+)$ and $a = \lim_{t \downarrow -(s_0 \wedge t_0)} \phi_j(s_0+t) = \phi_j((s_0-t_0)_+)$. It follows that $s_0 = t_0$, and $\psi = \varphi_j^+$. Next consider the case $s_0 < 0$. Then repeat the preceding argument with $t \mapsto \varphi_j(s_0-t)$ for $s_0 \le t \le T_j^{-}+s_0$, which is a steepest-descent curve and passes through $z_0$ at time $0$. We now obtain $a = \lim_{t \downarrow -((-s_0) \wedge t_0)} \psi(t_0+t) = \psi((t_0+s_0)_+)$ and $a = \lim_{t \downarrow -((-s_0) \wedge t_0)} \varphi_j(s_0-t) = \varphi_j(-(-s_0-t_0)_+)$. Hence, $-s_0 = t_0$ and $\psi = \varphi_j^-$.  
\end{proof}

\begin{prop}
\label{flucL19}
Suppose that $a \in \bbR$ and $\nabla u$ is real-valued on $U \cap \bbR$. Let 
\begin{align}
S^+ = \inf \{s > 0: s+a \not \in V\} \qquad S^- = \sup \{s < 0: s+a \not \in V\}. 
\end{align}
Then $S^+ > 0$ and $S^- < 0$. Define curves $\psi^+: [0, S^+) \rightarrow \bbR$ and $\psi^-: [0, -S^-) \rightarrow \bbR$ by 
\begin{align}
\psi^{+}(t) &= t+a \quad \ \ \text{ for } t \in [0, S^+) \\ 
\psi^{-}(t) &= -t+a \quad \text{ for } t \in [0, -S^-). 
\end{align}
$\psi^+$ is a steepest-descent curve of $u$ if $f^{(n)}(a) < 0$ and a steepest-ascent curve of $u$ if $f^{(n)}(a) > 0$. Similarly, $\psi^-$ is a steepest-descent curve of $u$ if $(-1)^{n} f^{(n)}(a) < 0$ and a steepest-ascent curve of $u$ if $(-1)^{n}f^{(n)}(a) > 0$. 
\end{prop}
\begin{proof}
We have $S^- < 0 < S^+$ because $\Disc(a, r) \subset U$ and the zeros of $\nabla u = \conj{f'}$ are isolated. By the assumption on $\nabla u$, all derivatives of $f$ are real-valued on $(a-r, a+r)$. Hence, $f(a+t)-f(a) \in \bbR$ for $|t| < r$. 
Then, if $|t|$ is sufficiently small, we have 
\begin{align*}
\sgn(t) \sgn(f'(a+t)) = \sgn\bigg[\int \limits_{a}^{a+t} f'(s)\ ds \bigg] = \sgn[f(a+t)-f(a)] = \sgn[f^{(n)}(a)] \sgn(t)^n, 
\end{align*}
where the first equality holds because $f'$ does not change sign between $a$ and $a+t$ and the last equality follows from Lemma \ref{flucL20}. Hence, 
\begin{align*}
\frac{\nabla u(a+t)}{|\nabla u(a+t)|} = \sgn(\nabla u(a+t)) &= \sgn[f^{(n)}(a)] \sgn(t)^{n-1} \qquad \text{ for } t \in (S^-, S^+) \smallsetminus \{0\}, 
\end{align*}
and we have 
\begin{align*}
\frac{d}{dt}\psi^+(t) &= 1 = \sgn[f^{(n)}(a)] \frac{\nabla u(a+t)}{|\nabla u(a+t)|} \quad \text{ for }  t \in (0, S^+)\\
\frac{d}{dt}\psi^-(t) &= -1 = (-1)^n \sgn[f^{(n)}(a)] \frac{\nabla u(a-t)}{|\nabla u(a-t)|} \quad \text{ for } t \in (0, -S^-). 
\end{align*}
Now the conclusions of the lemma follow from (\ref{flucE79}) and (\ref{flucE80}). 
\end{proof}

%\begin{lem}
%\label{flucL13}
%Let $\delta > 0$ and $b \in \Disc(a, \delta)$ with $v(b) = v(a)$. Let $\psi: (0, S) \rightarrow \Disc(a, \delta) \smallsetminus \{a\}$ be a (with maximal $S$) steepest-descent curve of $u$ emanating from $b$. If $\delta$ is sufficiently small then the reversed curve $t \mapsto \psi(S-t)$ for $t \in (0, S)$ is a steepest-ascent curve of $u$ emanating from $a$. In particular, $\lim_{t \uparrow S} \psi(t) = a$. 
%\end{lem}
%\begin{proof}
%\end{proof} 

\subsection{The exit time of a small disk}\label{appS2Sub3}

We next obtain some bounds to control $\varphi$ in a small disk around $a$. Fix $j$ now and drop it from the notation. By the symmetry of positive and negative time values remarked after (\ref{flucD1}), we may restrict attention to $\varphi(t)$ for $t > 0$. Define 
\begin{align}
\label{flucE99}
\tau = \tau(\epsilon) = \inf \{0 < t < T^+: |\varphi(t)-a| = \epsilon\}. 
\end{align}
We will use the following bounds to control $\varphi$ and $u(\varphi)$. 
\begin{prop}
\label{flucL10} 
Suppose that $\epsilon < \epsilon_0/128$ and put $C = 64/\epsilon_0$. Then 
\begin{enumerate}[(a)]
\item $\epsilon (1-C\epsilon) \le \tau \le \epsilon (1+2C\epsilon)$. 
\item $|\epsilon^{-1}(\varphi(\tau)-a)-\conj{\xi}e^{\bfi \pi j/n}| \le 4C\epsilon$. 
\item $|u(\varphi(t))-u(a)| \ge \dfrac{|f^{(n)}(a)|}{2n!}\epsilon^n$ for $\tau \le t < T^+$. 
\item $|\varphi(t)-a| \ge \epsilon/2$ for $\tau \le t < T^+$. 
\end{enumerate}
\end{prop}
\begin{proof}
By Lemma \ref{flucL37}, $|h'(z)-h'(a)| \le \dfrac{C}{4}|h'(a)| \epsilon \le \dfrac{|h'(a)|}{2}$ for $z \in \conj{\Disc}(a, \epsilon)$. Using this, $h'(a) = \xi |h(a)|$ and the triangle inequality, we obtain that 
\begin{align*}
\bigg|\frac{h'(z)}{|h'(z)|}-\xi\bigg| &= \frac{|h'(z)-h'(a)+\xi|h'(a)|-\xi|h'(z)||}{|h'(z)|}\\ 
&\le \frac{2|h'(z)-h'(a)+\xi|h'(a)|-\xi|h'(z)||}{|h'(a)|} \\ 
&\le \frac{2|h'(z)-h'(a)|+2||h'(a)|-|h'(z)||}{|h'(a)|}
\le C\epsilon 
\end{align*}
Since $\varphi(t) \in \conj{\Disc}(a, \epsilon)$ for $0 \le t \le \tau$, we have  
\begin{align*}
|\varphi(t)-a-t\conj{\xi} e^{\bfi \pi j/n}| &= \bigg| \int \limits_0^t \varphi'(s)-\conj{\xi} e^{\bfi \pi j/n} ds \bigg| = \bigg| \int_0^t \frac{\conj{h'}(\varphi(s))}{|h'(\varphi(s))|}-\conj{\xi} ds \bigg| \le  \int_0^t \bigg|\frac{h'(\varphi(s))}{|h'(\varphi(s))|}-\xi\bigg| ds\\ &\le C\epsilon t. 
\end{align*}
Setting $t = \tau$ and using the triangle inequality yields 
\begin{align*}
|\epsilon-\tau| = ||\varphi(\tau)-a|-|\tau \conj{\xi} e^{\bfi \pi j/n}|| \le C\epsilon\tau,  
\end{align*}
which implies $\epsilon(1+C\epsilon)^{-1} \le \tau \le \epsilon(1-C\epsilon)^{-1}$ and (a) follows. Then, we also have (b) because $\tau \le 2\epsilon$ by (a) and, thus, 
\begin{align*}
|\varphi(\tau)-a-\epsilon \conj{\xi} e^{\bfi \pi j/n}| \le |\varphi(\tau)-a-\tau \conj{\xi} e^{\bfi \pi j/n}| + |\epsilon-\tau| \le 2C\epsilon\tau \le 4C\epsilon^2. 
\end{align*}

By Lemma \ref{flucL20}, for $z \in \conj{\Disc}(a, r/2)$,  
\begin{align}
\bigg|f(z)-f(a)-\frac{f^{(n)}(a)}{n!}(z-a)^n\bigg| &\le \frac{2^{n+2}|z-a|^{n+1}}{r^{n+1}}\sup \limits_{\conj{\Disc}(a, r/2)}|f| \nonumber\\ 
&\le \frac{2|f^{(n)}(a)|K_0}{n!} |z-a|^{n+1}. \label{appE6}
\end{align}
Then, by the triangle inequality,
\begin{align}
|f(z)-f(a)| &\ge \frac{|f^{(n)}(a)|}{n!}(1-2K_0\epsilon)\epsilon^n. 
\end{align}
for $z \in \conj{\Disc}(a, \epsilon)$. Using this, we derive the following lower bound for $\tau \le t < T^+$ 
\begin{align}
|u(\varphi(t))-u(a)| &\ge |u(\varphi(\tau))-u(a)| = |f(\varphi(\tau))-f(a)| \ge \frac{|f^{(n)}(a)|}{n!}\bigg(1-2K_0\epsilon\bigg) \epsilon^n \nonumber \\
&\ge \frac{7}{8}\frac{|f^{(n)}(a)|}{n!}\epsilon^n \label{flucE48}
\end{align}
claimed in (c). Here, the first inequality comes because $t \ge \tau$ and $u \circ \varphi$ is decreasing, and the equality is a consequence of $\varphi$ being a stationary curve of $v$, see Lemma \ref{flucL1}. Also, for $z \in \conj{\Disc}(a, \epsilon/2)$, using (\ref{appE6}) again, 
\begin{align}
\label{flucE49}
|u(z)-u(a)| \le |f(z)-f(a)| \le \frac{|f^{(n)}(a)|}{n!}\frac{\epsilon^n}{2^n}\bigg(1+K_0\epsilon\bigg) \le \frac{5}{8} \frac{|f^{(n)}(a)|}{n!}\epsilon^n.  
\end{align}
Now (d) follows from (\ref{flucE48}) and (\ref{flucE49}).  
\end{proof}

\subsection{The time spent in a compact set}

The next proposition bounds the time $\varphi$ can spend on a compact set on which $\nabla u$ is nonzero. 
\begin{prop}
\label{flucL14}
Let $K \subset V$ be nonempty and compact. Suppose that $\{t \in I: \varphi(t) \in K\}$ is nonempty and decompose it uniquely as the disjoint union of closed (relative to $I$) intervals $J_k$ for $0 \le k < N$, where $N \in \bbN \cup \{\infty\}$. Then $J_k = [t_{2k}, t_{2k+1}]$ for some $t_{2k}, t_{2k+1} \in I$ for $0 \le k < N$ and 
\begin{align}
\label{flucE67}
\sum \limits_{0 \le k < N} (t_{2k+1}-t_{2k}) \le \frac{2\max \limits_{K} |u|}{\min \limits_{K} |\nabla u|}. 
\end{align}
\end{prop}
\begin{proof}
We may assume that $\varphi$ is a steepest-descent curve of $u$. To obtain a contradiction, suppose that the right endpoint of $J_k$ is $T^+$ for some $k$. Then, by Theorem \ref{flucT4}b, $T^+ = \infty$. Choose $s \in J_k$. Since $\varphi([s, \infty)) \subset K$ and $|\varphi'| = 1$, we obtain  
\begin{align}
- 2\max \limits_{K} |u| &\le u(\varphi(t))-u(\varphi(s)) = -\int \limits_{s}^t |\nabla u(\varphi(\sigma))|\dd\sigma \label{flucE66}\\ 
&\le -(t-s) \min \limits_{K} |\nabla u| \label{flucE65}
\end{align}
for $t \ge s$. Letting $t \rightarrow \infty$ makes (\ref{flucE65}) arbitrarily large negative, which is a contradiction. Similarly, $J_k$ cannot have $T^-$ as its left endpoint. 

Now set $s = t_{2l}$ and $t = t_{2m+1}$, where $t_{2l} \le t_{2m+1}$ for some $0 \le l, m < N$. Then $\varphi(s), \varphi(t) \in K$ and repeating (\ref{flucE66}) yields 
\begin{align}
- 2\max \limits_{K} |u| \le -\sum \limits_{\substack{0 \le k < N \\ t_{2l} \le t_{2k} \\ t_{2k+1} \le t_{2m+1}}} \int_{t_{2k}}^{t_{2k+1}} |\nabla u(\varphi(\sigma))|\dd\sigma \le - \min \limits_{K} |\nabla u| \sum \limits_{\substack{0 \le k < N \\ t_{2l} \le t_{2k} \\ t_{2k+1} \le t_{2m+1}}} (t_{2k+1}-t_{2k}). 
\end{align} 
Taking the infimum of the far right-hand side over $l, m$ and rearranging terms leads to (\ref{flucE67}).  
\end{proof}

\subsection{Limit behavior}

We now turn to the behavior of $\varphi(t)$ as $t \rightarrow T^+$. As the next proposition shows, there are only a few possibilities if there are finitely many singularities of $u$ and $|u(z)| \rightarrow \infty$ as $z$ approaches any of these singularities. 
\begin{prop}
\label{flucP1}
Let $k, l \in \bbZ_+$ with $k+l > 0$ and $P = \{p_1, \dotsc, p_{k+l}\} \subset \bbC$. Suppose that $u$ is a harmonic function on $\bbC \smallsetminus P$ such that 
\begin{align}
\label{appE3}
\lim \limits_{z \rightarrow p_i} u(z) = \begin{cases}-\infty &\text{ if } 1 \le i \le k \\ +\infty &\text{ if } k < i \le k+l. \end{cases}
\end{align}
Let $Z$ denote the set of zeros of $\nabla u$. Then either $T^+ = \infty$ and $\lim_{t \rightarrow \infty} |\varphi(t)| = \infty$, or $T^+ < \infty$ and $\lim_{t \rightarrow T^+} \varphi(t) = b$ for some $b \in P \cup Z \smallsetminus \{a\}$.  
\end{prop}

As an example of a harmonic function satisfying (\ref{appE3}), consider 
\begin{align}
\label{flucE50}
u(z) = \sum \limits_{i=1}^{k+l} c_i \log |z-p_i| + u_0(z) \quad \text{ for } z \in \bbC \smallsetminus P, 
\end{align}
where $u_0$ is a harmonic function on $\bbC$, and coefficients $c_i > 0$ for $1 \le i \le k$ and $c_i < 0$ for $k < i \le k+l$. It follows from B\^{o}cher's theorem characterizing the positive harmonic functions on a punctured disk \cite[Theorem 3.9]{AxlerBourdonRamey01} that all examples are, in fact, of the form (\ref{flucE50}). 

We now derive a string of auxiliary results that will culminate in the proof of Proposition \ref{flucP1}. We begin with the observation that $\varphi$ eventually stays outside of any annulus in $V$. 

\begin{prop}
\label{flucC1}
Let $b \in \bbC$, $0 < r_1 < r_2$ and suppose that the closed annulus $\conj{\Ann}(b, r_1, r_2) \subset V$. Then there exists $T_1 \in [0, T^+)$ such that either $|\varphi(t)-b| < r_1$ for $t \in [T_1, T^+)$ or $|\varphi(t)-b| > r_2$ for $t \in [T_1, T^+)$. 
\end{prop}
\begin{proof}
We first show that there exists $T_0 \in [0, T^+)$ such that either $|\varphi(t)-b| > r_1$ for $t \in [T_0, T^+)$ or $|\varphi(t)-b| < r_2$ for $t \in [T_0, T^+)$. Arguing by contradiction, suppose that there exists an increasing sequence of time points $(s_k)_{k \in \bbN}$ such that $|\varphi(s_{2k})-b| \le r_1$ and $|\varphi(s_{2k+1})-b| \ge r_2$ for $k \in \bbN$. Apply Proposition \ref{flucL14} with the compact set $\conj{A}(b, r_1, r_2)$. By continuity of $\varphi$, at least one of the intervals $J_{l_k} = [t_{2l_k}, t_{2l_k+1}]$ described in the proposition must satisfy $s_{2k} < t_{2l_k} \le t_{2l_k+1} < s_{2k+1}$, $|\varphi(t_{2l_k})| = r_1$ and $|\varphi(t_{2l_k+1})| = r_2$ for $k \in \bbN$. Then, by (\ref{flucE59}), $t_{2{l_k}+1}-t_{2l_k} \ge r_2-r_1 > 0$ for $k \in \bbN$.  Then, since $J_{l_k}$ are distinct for various $k \in \bbN$, the left hand-side of (\ref{flucE67}) is infinite, a contradiction. Hence, there exists $T_0 \in [0, T^+)$ with the claimed property. 

Consider the case when $|\varphi(t)-b| > r_1$ for $t \in [T_0, T^+)$. We can choose $\delta > 0$ small such that $\conj{A}(b, r_2, r_2+\delta) \subset V$. Repeating the argument of the preceding paragraph, we find $T_1 \in [0, T^+)$ such that either 
$|\varphi(t)-b| < r_2+\delta$ for $t \in [T_1, T^+)$ or $|\varphi(t)-b| > r_2$ for $t \in [T_1, T^+)$. The former would imply $\varphi(t) \in \conj{A}(b, r_1, r_2+\delta)$ for $t \in [T_0 \vee T_1, T^+)$, which is ruled out by Proposition \ref{flucL14}. Similarly, $|\varphi(t)-b| < r_2$ for $t \in [T_0, T^+)$ implies that $|\varphi(t)-b| < r_1$ for $t \in [T_1, T^+)$ for some $T_1 \in [0, T^+)$. 
\end{proof}

\begin{cor}
\label{flucC2}
Suppose that $\conj{\Disc}(b, \delta) \smallsetminus \{b\} \subset V$. Then either there exists $0 \le T_1 < T^+$ such that $|\varphi(t)-b| > \delta$ for $t \in [T_1, T^+)$, or $\lim_{t \rightarrow T^+} \varphi(t) = b$.  
\end{cor}
\begin{proof}
Let $0 < \delta_1 < \delta$. By Proposition \ref{flucC1}, there exists $T_1$ such that either $|\varphi(t)-b| > \delta$ for $t \ge T_1$ or $|\varphi(t)-b| < \delta_1$ for $t \ge T_1$. In the latter case, repeating the preceding argument with any $0 < \delta_1' < \delta$ yields $|\varphi(t)-b| < \delta_1'$ for $t \ge T_1'$ for some $T_1' \in [0, T^+)$. Hence, $\lim_{t \rightarrow T^+} \varphi(t) = b$. 
\end{proof}

We next show that if $\varphi(t)$ converges as $t \rightarrow T^+$ to a point in $U$ then the limit point is a zero of $\nabla u$ and the convergence occurs in finite time. The finiteness of $T^+$ can be reasoned informally as follows. If $\lim_{t \rightarrow T^+} \varphi(t) = b$ for some $b \in U$ then, in a small neighborhood of $b$, $\varphi$ is given by the reversal of one of the steepest-ascent curves that emanate from $b$. Therefore, in the case $\nabla u(b) = 0$, $\varphi$ reaches the boundary of $V$ in finite time. The proof of the next lemma makes the foregoing argument more precise. 

\begin{lem}
\label{flucL22}
Suppose that $\lim_{t \rightarrow T^+} \varphi(t) = b$ for some $b \in U$. Then $\nabla u(b) = 0$ and $T^+ < \infty$. 
\end{lem}
\begin{proof}
To show that a contradiction arises, suppose that $\nabla u(b) \neq 0$ i.e. $b \in V$. Choose $\delta > 0$ such that $\conj{\Disc}(b, \delta) \subset V$. Then, by Proposition \ref{flucL14}, we have $|\varphi(t_k)-b| > \delta$ for a sequence $(t_k)_{k \in \bbN}$ in $[0, T^+)$ with $\lim_{k \rightarrow \infty} t_k = T^+$, which contradicts $\lim_{t \rightarrow T^+} \varphi(t) = b$. Hence, $\nabla u(b) = 0$. 

Now choose $t_0 \in [0, T^+)$ such that $\varphi(t) \in \Disc(b, \delta)$ for $t \in [t_0, T^+)$. Since $\varphi$ is a stationary curve of $v$ and $v(b) = \lim_{t \rightarrow T^+} v(\varphi(t))$, we have $v(b) = v(\varphi(t))$ for $t \in [t_0, T^+)$. Therefore, by Lemma \ref{flucL12}, $\varphi(t_0) = \psi(s_0) = z_0$ for some stationary curve $\psi$ of $v$ passing through $b$. Since $\lim_{t \rightarrow T^+}\varphi(t) = b$ and $\varphi$ is a steepest-descent curve of $u$, we have $u(b) < u(z_0)$. We consider two cases now. Suppose that $s_0 > 0$. Then $\psi_+$ (defined in as in (\ref{flucE110}) but for $b$) is a steepest-ascent curve of $u$ emanating from $b$. Then, since $\nabla u(z_0) \neq 0$, the unique steepest-descent curve passing through $z_0$ at time $t_0$ is given by $\psi(-t+t_0+s_0)$ for time values $t_0 \le t \le t_0+s_0$. Then, since $\nabla u(b) = 0$ and $\psi(0) = b$, we have $T^+ = t_0 + s_0 < \infty$. Now, suppose that $s_0 < 0$. Then $\psi_-$ is a steepest-ascent curve emanating from $b$, and the unique steepest-descent curve passing through $z_0$ at time $t_0$ is given by $\psi(t-t_0+s_0)$ for time values $t_0 \le t \le t_0-s_0$. Hence, $T^+ = t_0-s_0 < \infty$.  
\end{proof}

\begin{lem}
\label{flucL23}
Suppose that $\conj{\Disc}(b, \delta) \smallsetminus \{b\} \subset V$ and $\lim_{t \rightarrow T^+} \varphi(t) = b$ and $\lim_{z \rightarrow b} |u(z)| = \infty$. Then $T^+ < \infty$.  
\end{lem}
\begin{proof}
If $\varphi$ is a steepest-descent curve then the hypotheses imply that $\lim_{z \rightarrow b} u(z) = -\infty$. It follows from B\^{o}cher's theorem characterizing the positive harmonic functions on a punctured disk \cite[Theorem 3.9]{AxlerBourdonRamey01} that 
\begin{align*}u(z) = c \log |z-b| + u_0(z) \quad \text{ for } z \in \Disc(b, \delta) \smallsetminus \{b\}\end{align*} 
for some $c > 0$ and harmonic function $u_0$ on $\Disc(b, \delta)$. Hence, 
\begin{align*}\nabla u(z) = \frac{c}{\conj{z}-\conj{b}} + \conj{f_0'}(z) \quad \text{ for } z \in \Disc(b, \delta) \smallsetminus \{b\},\end{align*}
where $f_0 \in \Hol(\Disc(b, \delta))$ with $\Re f_0 = u_0$. Then, for $z \in \Disc(b, \delta_1) \smallsetminus \{b\}$ with $0 < \delta_1 < \delta$,  
\begin{align}
\label{flucE61}
\frac{(z-b) \cdot \nabla u(z)}{|z-b| |\nabla u(z)|} = \frac{c + \Re\{(z-b) f_0'(z)\}}{|c + (z-b)f_0'(z)|} \ge \frac{1-\dfrac{\delta_1}{c}\max \limits_{\conj{\Disc}(b, \delta_1)}|f_0'|}{1+\dfrac{\delta_1}{c} \max \limits_{\conj{\Disc}(b, \delta_1)}|f_0'|} \ge \frac{1}{4}
\end{align}
provided that $\delta_1 > 0$ is chosen sufficiently small. It follows that if $\varphi(t) \in \Disc(b, \delta_1)$ then 
\begin{align}
\label{flucE58}
\frac{d}{dt}|\varphi(t)-b| = \frac{\varphi(t)-b}{|\varphi(t)-b|} \cdot \varphi'(t) = - \frac{(\varphi(t)-b) \cdot \nabla u(\varphi(t))}{|\varphi(t)-b| |\nabla u(\varphi(t))|} \le -\frac{1}{4}. 
\end{align}
Since $\lim_{t \rightarrow T^+} \varphi(t) = b$, we can pick $T_1 \in (0, T^+)$ with $\varphi(t) \in \Disc(b, \delta_1)$. Because the derivative of $|\varphi-b|$ is negative at $T_1$, there exists a maximal $T_2 \in (T_1, T^+]$ such that $|\varphi(t)-b|$ is decreasing for $t \in [T_1, T_2)$. If $T_2 < T^+$ then, by (\ref{flucE58}) and continuity, $|\varphi-b|$ has a negative derivative at $T_2$ as well, which contradicts the maximality of $T_2$. Hence, $T_2 = T^+$ and, consequently, $\varphi(t) \in \Disc(b, \delta_1)$ for $t \in [T_1, T^+)$. Therefore, by (\ref{flucE58}), 
\begin{align*}
|\varphi(t)-b| \le |\varphi(T_1)-b|-\frac{t-T_1}{4} < \delta_1-\frac{t-T_1}{4} \quad \text{ for } t \in [T_1, T^+), 
\end{align*}
which implies that $T^+ < 4\delta + T_1 < \infty$.  
\end{proof}

\begin{proof}[Proof of \ref{flucP1}]
In the present setting $\Disc(a, \epsilon) \subset \bbC \smallsetminus P$ and $\varphi: [0, T^+) \rightarrow \bbC \smallsetminus P$ is a steepest-descent or -ascent curve of $u$ with $\varphi(0) = a$ and $\varphi((0, T^+)) \subset \bbC \smallsetminus (P \cup Z)$. 

We may assume that $u$ satisfies (\ref{flucE50}). Note that 
\begin{align}
\label{flucE51}
\nabla u(z) = \sum \limits_{i=1}^{k+l} \frac{c_i}{\conj{z}-\conj{p_i}} + \nabla u_0(z) \qquad \text{ for } z \in \bbC \smallsetminus P.  
\end{align}
Let $Z$ denote the set of zeros of $\nabla u$. Since $Z$ is a set of isolated points (has no limit point in $\bbC$), we can enumerate its elements as $z_1, z_2, \dotsc, $ such that $|z_1| \le |z_2| \le \dotsc$. 

Suppose that $\sup_{0 \le t < T^+} |\varphi(t)| = \infty$. Then, by continuity, $T^+ = \infty$. There exists an increasing sequence of $(r_i)_{i \in \bbN}$ of positive real numbers such that $\conj{A_i}(0, r_{2i-1}, r_{2i})$ is disjoint from $Z \cup P$ for each $i \in \bbN$ and $\lim_{i \rightarrow \infty} r_i = \infty$. By Corollary \ref{flucC1} and unboundedness of $|\varphi|$, there exists $T_i \in [0, T^+)$ such that $|\varphi(t)| > r_{2i}$ for $t \in [T_i, T^+)$. Hence, $\lim_{t \rightarrow T^+} |\varphi(t)| = \infty$. 

Now the case $\sup_{0 \le t < T^+} |\varphi(t)| < \infty$. Choose $R > 0$ such that $\varphi(t) \in \Disc(0, R)$ for $t \in [0, T^+)$. Arguing by contradiction, suppose that $\varphi(t)$ does not converge to any point in $P \cup Z$ as $t \rightarrow T^+$. The intersection of $P \cup Z$ with $\conj{\Disc}(0, R)$ is finite. Hence, by Corollary \ref{flucC2} there exist $\delta > 0$ and $T_0 \in [0, T^+)$ such that $|\varphi(t)-b| \ge \delta$ for $t \ge T_0$ and $b \in P \cup Z$ with $|b| \le R$. It follows that $\varphi(t) \in K$ for $t \in [T_0, T^+)$, where 
\begin{align}K = \conj{\Disc}(0, R) \smallsetminus \bigcup \limits_{\substack{b \in P \cup Z \\ |b| \le R}} \Disc(b, \delta).\end{align} 
However, this is not possible by Proposition \ref{flucL14}. Hence, we must have $\lim_{t \rightarrow T^+} \varphi(t) = b$ for some $b \in P \cup Z$. Since $u \circ \varphi$ is decreasing, $b \neq a$. Also, $T^+ < \infty$ either by Lemma \ref{flucL22} or by Lemma \ref{flucL23} and (\ref{flucE50}).
\end{proof}

\section{Contour integrals over steepest-descent curves}
\label{flucAp2}

Utilizing the development in the preceding section, we now derive bounds for contour integrals of certain meromorphic functions on $\bbC$. Recall that a function $F$ is meromorphic on $\bbC$ if there exists a discrete set $P \subset \bbC$ such that $F$ is holomorphic on $\bbC \smallsetminus P$ and has poles at points of $P$ ($F$ has no essential singularities). 

We describe the integrand and the contour of integration precisely. Let $F$ be a nonconstant meromorphic function on $\bbC$. Write $P$ and $Z$ for the set of poles and zeros of $F$, respectively. Suppose that $u$ and $\tilde{u}$ are harmonic functions on $\bbC \smallsetminus (P \cup Z)$ such that  
\begin{align*}
\log |F(z)|  = u(z) + \tilde{u}(z) \quad \text{ for } z \in \bbC \smallsetminus (P \cup Z).
\end{align*}
Choose $a \in \bbC$ and $r > 0$ such that $\Disc(a, r) \subset \bbC \smallsetminus (P \cup Z)$. There exist $f, \tilde{f} \in \Hol(\Disc(a, r))$ such that $f = u + \ii v$ and $\tilde{f} = \tilde{u} + \ii \tilde{v}$ for some harmonic functions $v, \tilde{v}$ on $\Disc(a, r)$. Since $\Disc(a, r)$ is simply-connected and does not intersect $P \cup Z$, we can define $\log F(z)$ on $\Disc(a, r)$ by 
\begin{align}
\log F(z) = \int_a^z \frac{F'(w)}{F(w)} \dd w + z_0 \quad \text{ for } z \in \Disc(a, r),  
\end{align}
where the contour can be taken as the line segment $[a, z]$ and $z_0 \in \bbC$ satisfies $\exp(z_0) = F(a)$. 
Because the harmonic conjugate of $\log |F(z)|$ on $\Disc(a, r)$ is unique up to a constant,  the functions $v+\tilde{v}$ and $\Im \{\log F(z)\}$ differ by a constant. Hence, absorbing this constant into $v+\tilde{v}$, we have the identity
\begin{align}
F(z) = \exp\{f(z) + \tilde{f}(z)\} \quad \text{ for } z \in \Disc(a, r). \label{appE7}
\end{align}

Assume that $u$ is nonconstant. Hence, there exists a minimal $n \in \bbN$ such that $f^{(n)}(a) \neq 0$. Fix an odd $j \in [n]$ and let $\Phi: [0, T^+) \rightarrow \bbC$ denote the curve $\varphi_j^+$ defined in (\ref{flucE110}). By Propositions \ref{flucL6} and \ref{flucL25}, $\Phi$ is a (maximal) steepest-descent curve of $u$ emanating from $a$. Pick $T \in [0, T^+)$. We are interested in approximating the contour integral 
\begin{align}
\label{flucE95}
I = \int \limits_{a}^{\Phi(T)} F(z) \dd z, 
\end{align}
where the integration is from $a = \Phi(0)$ to $\Phi(T)$ along the curve $\Phi$. 

We need further notation to introduce the approximate value for (\ref{flucE95}).  
Put $c = \dfrac{f^{(n)}(a)}{n!}$, 
which is nonzero, and let $\xi$ and $N$ be the direction and absolute value of an $n$th root of $c$. The bounds below are expressed in terms of the parameter $N$ and improve as $N \rightarrow \infty$. If $\tilde{u}$ is nonconstant, let $k \in \bbN$ be minimal such that $\tilde{f}^{(k)}(a) \neq 0$ and set $d = \dfrac{\tilde{f}^{(k)}(a)}{k!}$ and $M = |d|^{1/k}$. Recall the constants from (\ref{flucE112}). 
In addition, define 
\begin{align}
\tilde{K}_0 = \frac{2^{k+1}}{r^{k+1}} \frac{\sup \limits_{z \in \conj{\Disc}(a, r/2)} |\tilde{f}|}{M^k} \qquad \tilde{\epsilon}_0 = \frac{1}{4\tilde{K}_0}. \label{flucE156}
\end{align}
If $\tilde{u}$ is constant, set $k = M = d = \tilde{K}_0 = 0$ and $\tilde{\epsilon}_0 = \infty$. We can interchange $u$ and $\tilde{u}$ if $k > n$ and work with $u+\tilde{u}, 0$ instead of $u, \tilde{u}$ if $k = n$. Hence, it suffices to consider the case $k < n$. 

Choose $\epsilon > 0$ such that
\begin{align}
\epsilon \le \min \bigg\{\frac{\epsilon_0}{128}, \frac{\tilde{\epsilon}_0}{4}, \frac{1}{4n}\bigg\}. \label{flucE149}
\end{align}
Let $\eta$ denote the direction of $\Phi(\tau(\epsilon))-a$, where $\tau(\epsilon)$ is the exit time of $\Phi$ from $\Disc(a, \epsilon)$, see (\ref{flucE99}). Define the contour integral 
\begin{align}
\label{flucE100}
\sI = \int \limits_{[0, \infty \eta \xi]} \exp\bigg\{u^n +dN^{-k}\conj{\xi}^ku^k\bigg\}\ du. 
\end{align}
The next lemma implies that the integral in (\ref{flucE100}) is absolutely convergent and, hence, $\sI$ is well-defined. 
\begin{lem}
\label{flucL29}
$\Re(\eta^n \xi^n) < -1/2$. 
\end{lem}
\begin{proof}
It follows from Proposition \ref{flucL10}b and the choice of that $|\eta-e^{\bfi \pi j/n} \conj{\xi}| < 2\epsilon$. Then, since $j$ is odd, $|\eta^n \xi^n +1| = |\eta^n-e^{\ii \pi j}\conj{\xi}^n| < 2n \epsilon < 1/2$, which implies the conclusion. 
\end{proof}

We will use the following basic bound. 
\begin{lem}
\label{flucL28}
Let $p > 0$, $q \ge 0$ and $x \ge 0$. Let $n, k \in \bbZ_+$ with $n > k$. Then 
\begin{align}
\label{flucE103}
\int \limits_{x}^\infty \exp\bigg\{-p t^n +qt^k\bigg\}\ \dd t \le \dfrac{8}{p^{1/n}} \exp\bigg\{q^{n/(n-k)}\bigg(\frac{4k}{pn}\bigg)^{k/(n-k)}-\frac{px^n}{2}\bigg\}. 
\end{align}
\end{lem}
\begin{proof}
The integral above is less than $\exp\{-px^n/2\} \int_0^\infty \exp\{-p t^n/2 +qt^k\} \dd t$. The maximum over $t \in [0, \infty)$ of the function $t \mapsto -pt^n/4+qt^k$ occurs at $t_0 = \sqrt[n-k]{\dfrac{4qk}{pn}}$. Hence, 
the last integral does not exceed
\begin{align*}
\exp(-pt_0^n/4 + qt_0^k) \int \limits_0^\infty \exp(-pt^n/4) \dd t &= \exp\bigg\{q^{n/(n-k)}\bigg(1-\frac{k}{n}\bigg)\bigg(\frac{4k}{pn}\bigg)^{k/(n-k)}\bigg\}\frac{4^{1/n}}{p^{1/n}} \\ 
&\cdot \int \limits_0^\infty \exp(-t^n) \dd t, 
\end{align*}
which implies the result. 
\end{proof}

Split (\ref{flucE95}) into two parts
\begin{align}
I' &= \int \limits_{a}^{\Phi(\tau(\epsilon))} F(z) \dd z \label{flucE96}\\
I'' &= \int \limits_{\Phi(\tau(\epsilon))}^{\Phi(T)} F(z) \dd z \label{flucE97}, 
\end{align}

\begin{prop}
\label{flucP2}\ 
\begin{align*}
\bigg|\frac{\xi N I'}{F(a)} -\sI\bigg| \le C\bigg(\frac{1}{N} + \exp\bigg\{-\frac{\epsilon^nN^n}{4}\bigg\} \bigg), 
\end{align*}
where 
\begin{align}
\label{flucE119}
C = 128(n+1)! (K_0 + \tilde{K}_0M^kN^{-k} + 1)\exp\{(2M^k N^{-k}+1)^{n/(n-k)} (16k/n)^{k/(n-k)}\}. 
\end{align}
\end{prop}
To see the utility of the preceding bound, imagine that $u$ and $\tilde{u}$ are rescaled such that $M, N \rightarrow \infty$ with bounded ratio $M/N$. Note also that $K_0, \tilde{K}_0$ and $\epsilon_0, \tilde{\epsilon}_0$ are invariant under rescaling of $u$ and $\tilde{u}$. Then the constant $C$ above remains bounded as $N \rightarrow \infty$ and the right-hand side decays like $1/N$. 
\begin{proof}[Proof of Proposition \ref{flucP2}]
Since $\Disc(a, r)$ is disjoint from $P \cup Z$, we can deform the contour in (\ref{flucE96}) into the line segment $[a, \Phi(\tau(\epsilon))] = [a, a+\epsilon \eta]$. Also, for $z \in \Disc(a, r)$, we have 
\begin{align*}
f(z) = f(a) + c(z-a)^n + f_1(z) \qquad \tilde{f}(z) = \tilde{f}(a) + d(z-a)^k + \tilde{f}_1(z), 
\end{align*}
where $f_1, \tilde{f}_1 \in \Hol(\Disc(a, r))$. By Lemma \ref{flucL20},  
\begin{align}
|f_1(z)| \le 2K_0N^n |z-a|^{n+1} \qquad |\tilde{f}_1(z)| \le 2\tilde{K}_0 M^k |z-a|^{k+1} \label{flucE101}
\end{align}
for $z \in \Disc(a, \epsilon)$. Change the variables in (\ref{flucE96}) by setting $z = a + \conj{\xi}N^{-1}u$. We obtain 
\begin{align}
I' &= \int \limits_{[a, a+\epsilon \eta]} F(z) \dd z \nonumber \\
&= F(a) \int \limits_{[a, a+\epsilon \eta]} \exp\bigg\{c(z-a)^n + f_1(z)+d(z-a)^k + \tilde{f}_1(z)\bigg\} \dd z \nonumber \\
&= \frac{F(a) \conj{\xi}}{N} \int \limits_{[0, \epsilon N \eta \xi]} \exp \bigg\{u^n +dN^{-k}\conj{\xi}^ku^k + E(u)\bigg\}\ \dd u, \label{flucE98}
\end{align}
where $E \in \Hol(\Disc(0, rN))$ is given by 
\begin{align}
E(u) = f_1\bigg(a + \conj{\xi}N^{-1}u\bigg) + \tilde{f}_1\bigg(a + \conj{\xi}N^{-1}u\bigg) \label{flucE102}
\end{align}
Bounds in (\ref{flucE101}) imply that 
\begin{align}
\label{flucE104}
|E(u)| \le \frac{2}{N} \bigg(K_0|u|^{n+1} + \tilde{K}_0M^{k}N^{-k}|u|^{k+1}\bigg) \qquad \text{ for } u \in \Disc(0, \epsilon |c|^{1/n}).  
\end{align}

To complete the proof, it suffices to bound suitably the difference of $\sI$ and the integral in (\ref{flucE98}). We accomplish this in two steps. First, by Lemmas \ref{flucL29} and \ref{flucL28}, 
\begin{align}
\bigg |\int \limits_{[\epsilon N \eta \xi, \infty \eta \xi]} \exp\bigg\{u^n +dN^{-k}\conj{\xi}^ku^k\bigg\}\ \dd u \bigg| &= \int \limits_{\epsilon N}^\infty \exp \{-t^n/2 + M^k N^{-k}t^k\} \dd t\\
&\le 16\exp\bigg\{\bigg(\frac{M}{N}\bigg)^{nk/(n-k)} \bigg(\frac{8k}{n}\bigg)^{k/(n-k)}-\frac{\epsilon^n N^n}{4}\bigg\}.\label{flucE108}
\end{align}
Second, using the bound $|e^z-1| \le |z|e^{|z|}$ for $z \in \bbC$, we have 
\begin{align}
&\bigg| \int \limits_{[0, \epsilon N \eta \xi]} \exp \bigg\{u^n +dN^{-k}\conj{\xi}^ku^k\bigg\} (\exp\{E(u)\}-1)\ \dd u\bigg| \label{flucE109} \\ 
&\le \int \limits_{0}^{\epsilon N} \exp \{\Re[\eta^n \xi^n]t^n +M^k N^{-k}t^k + |E(\eta \xi t)|\} |E(\eta \xi t)| \dd t. \label{flucE107}
\end{align}
For $0 \le t \le \epsilon N$, by (\ref{flucE149}) and (\ref{flucE104}), 
\begin{align}
|E(\eta \xi t)| &\le \frac{2}{N} \bigg(K_0t^{n+1} + \tilde{K}_0M^{k}N^{-k}t^{k+1}\bigg) \le  2\epsilon \bigg(K_0 t^n + \tilde{K}_0 M^k N^{-k} t^k\bigg) \nonumber\\ 
&\le \frac{1}{8}(t^n + M^k N^{-k} t^k). \label{flucE115}
\end{align}
From $t^{n+1}+t^{k+1} \le 2(n+1)!\exp(t^k)$ and (\ref{flucE115}), we also have 
\begin{align}
|E(\eta \xi t)| \le \frac{4(n+1)!}{N} (K_0+\tilde{K}_0M^{k}N^{-k}) \exp\{t^k\}. \label{flucE150}
\end{align}
Now combine (\ref{flucE115}), (\ref{flucE150}) with Lemmas \ref{flucL29} and \ref{flucL28} to bound (\ref{flucE107}) by  
\begin{align}
&\frac{4(n+1)!}{N} (K_0+\tilde{K}_0M^{k}N^{-k}) \int \limits_{0}^{\epsilon N} \exp \{-t^n/4 +(2M^k N^{-k}+1)t^k\} \dd t \\
&\le \frac{1}{N}2^7(n+1)! (K_0+\tilde{K}_0M^{k}N^{-k}) \exp \{(2M^k N^{-k}+1)^{n/(n-k)} (16k/n)^{k/(n-k)}\}\label{flucE118}
%&\le \frac{C}{|c|^{1/n}}\label{flucE109}. 
\end{align}
Observe that the claimed result is a consequence of (\ref{flucE108}) and (\ref{flucE118}).  
\end{proof}

\begin{prop}
\label{flucP6}
Let $k = 1$ (then $n > 1$) and $a \in \bbR$. Suppose that 
\begin{align}
\label{flucE116}
\conj{f}(z) = f(\conj{z}) \text{ and } \conj{\tilde{f}}(z) = \tilde{f}(\conj{z}) \quad \text{ for } z \in \Disc(a, r). 
\end{align} 
Furthermore, suppose that $d\Re\{e^{\ii \pi j/n}\conj{\xi}\} < 0$ and $\epsilon  < \dfrac{|\Re\{e^{\ii \pi j/n}\conj{\xi}\}|}{4(1+\tilde{K}_0)}$. Then 
\begin{align*}
|\Im I'| \le 32\frac{|F(a)|}{N} \exp \bigg[-\min \bigg\{\frac{\epsilon M |\Re\{e^{\ii \pi j/n}\conj{\xi}\}|}{4}, \frac{1}{12}\bigg(\frac{M |\Re\{e^{\ii \pi j/n}\conj{\xi}\}|}{N}\bigg)^{n/(n-1)}\bigg\}\bigg]. 
\end{align*}
\end{prop}
\begin{proof}
It follows from (\ref{flucE116}) that $f$ and $\tilde{f}$ are real-valued on $\Disc(a, r) \cap \bbR$. Hence, $c, d \in \bbR$ and, by (\ref{appE7}), $F$ is also real-valued on $\Disc(a, r) \cap \bbR$. Change variables in (\ref{flucE98}) via $u = \xi v$ and use $\xi^n = \sgn(c)$. We obtain 
\begin{align}
\label{flucE117}
I' = \frac{F(a)}{N} \int \limits_{[0, \epsilon N \eta]} \exp \bigg\{\sgn(c) v^n + dN^{-1}v + R(v)\bigg\} \ \dd v, 
\end{align}
where $R(v) = f_1(a+N^{-1}v) + \tilde{f}_1(a + N^{-1}v)$ for $v \in \Disc(0, \epsilon N)$, see (\ref{flucE102}). By (\ref{flucE109}),  
\begin{align}
\label{flucE148}
|R(v)| \le \frac{2}{N} \bigg(K_0 t^{n+1} + \tilde{K}_0 M N^{-1}t^{2}\bigg) \le 2\epsilon(K_0 t^n + \tilde{K}_0 MN^{-1} t)
\end{align}
By (\ref{flucE116}) and since $a \in \bbR$, we have $\conj{R}(v) = R(\conj{v})$. Using this, (\ref{flucE117}) and that $F(a) \in \bbR$ leads to 
\begin{align*}
\bar{I'} &= \frac{F(a)\conj{\eta}}{N} \int \limits_0^{\epsilon N} \exp \bigg\{\sgn(c) \conj{\eta}^nt^n +dN^{-1}\conj{\eta}t + R(\conj{\eta} t)\bigg\}\ \dd t \\
&= \frac{F(a)}{N} \int \limits_{[0, \epsilon N \conj{\eta}]} \exp \bigg\{\sgn(c)v^n +dN^{-1}v + R(v)\bigg\}\ \dd v, 
\end{align*}
which implies that 
\begin{align}
\label{flucE111}
\frac{2\ii N}{F(a)} \Im I' = \frac{N}{F(a)}(I'-\bar{I'}) = \bigg(\int \limits_{[0, \epsilon N \eta ]}+\int \limits_{[\epsilon N \conj{\eta}, 0]} \bigg) \exp \bigg\{\sgn(c) v^n + dN^{-1}v + R(v)\bigg\} \ \dd v. 
\end{align}

Let $y \in [0, \epsilon N]$. By Cauchy's theorem, the contour in (\ref{flucE111}) can be replaced with 
\begin{align*}
[\epsilon N \conj{\eta}, y\conj{\eta}] + [y \conj{\eta}, y \eta] + [y \eta, \epsilon N \eta]. 
\end{align*}
For $v = y \Re \eta + \ii t$ with $|t| \le y |\Im \eta|$, use bounds $|v| \le y$, $|\Re \eta - \Re\{e^{\ii \pi j/n}\conj{\xi}\}| \le 2\epsilon$ (Proposition \ref{flucL10}b) and (\ref{flucE148}). We obtain  
\begin{align}
\Re \bigg\{\sgn(c) v^n + dN^{-1}v + R(v)\bigg\} &\le y^n+dN^{-1}\Re \{\eta\} y + 2K_0 \epsilon y^n + 2\tilde{K}_0 M N^{-1} \epsilon y\nonumber\\
&\le (1+2K_0 \epsilon) y^n + (d\Re\{e^{\ii \pi j/n} \conj{\xi}\}N^{-1} + 2MN^{-1}\epsilon(1+\tilde{K}_0))y \nonumber\\
&\le 2 y^n -\frac{M |\Re\{e^{\ii \pi j/n} \conj{\xi}\}|} {2N}y. \label{flucE153}
\end{align}  
The condition on $\epsilon$ comes in for the last inequality. Since $[y\conj{\eta}, y \eta]$ has length less than $2y$ and $\log y \le y^n$, the last bound leads to 
\begin{align}
\bigg| \int \limits_{[y \conj{\eta}, y \eta]} \exp \bigg\{\sgn(c)u^n + dN^{-1}u + R(u)\bigg\} \ \dd u\bigg| \le 2\exp \bigg\{3y^n - \frac{M |\Re\{e^{\ii \pi j/n} \conj{\xi}\}|} {2N}y\bigg\}. \label{flucE121}
\end{align}
To optimize, set 
\begin{align}y = \min \bigg\{\epsilon N, \sqrt[n-1]{\frac{M|\Re\{e^{\ii \pi j/n} \conj{\xi}\}|}{6nN}}\bigg\}, \label{flucE154}\end{align}
and then the exponent of the right-hand side in (\ref{flucE121}) is bounded from above by 
\begin{align}
- \min \bigg\{\frac{\epsilon M |\Re\{e^{\ii \pi j/n} \conj{\xi}\}|}{4}, \frac{1}{12}\bigg(\frac{M |\Re\{e^{\ii \pi j/n}\conj{\xi}\}|}{Nn}\bigg)^{n/(n-1)}\bigg\}. \label{flucE155} 
\end{align}
We next estimate the contribution from the piece $[\epsilon N \conj{\eta}, y\conj{\eta}]+[y \eta, \epsilon N \eta]$. For $v = \eta t$ with $y \le t \le \epsilon N$, using $|v| \ge y$ and Lemma \ref{flucL29} yields   
\begin{align}
\Re \bigg\{\sgn(c) v^n + dN^{-1}v + R(v)\bigg\} &\le \Re \{\eta^n \xi^n\} t^n + dN^{-1}\Re \{\eta\} t + 2K_0 \epsilon t^n + 2\tilde{K}_0\epsilon MN^{-1}t \nonumber \\
&\le -\frac{1}{4}t^n - \frac{M |\Re\{e^{\ii \pi j/n} \conj{\xi}\}|} {2N}y, \label{flucE152}
\end{align}
where the coefficient of $y$ in (\ref{flucE152}) is obtained as in (\ref{flucE153}). Then, it follows from Lemma \ref{flucL28} and (\ref{flucE154}) that 
\begin{align}
&\bigg| \int \limits_{[\epsilon N \conj{\eta}, y\conj{\eta}]+[y \eta, \epsilon N \eta]} \exp \bigg\{\sgn(c)u^n + dN^{-1}u + R(u)\bigg\}\ \dd u\bigg| \\
&\le 2\exp \bigg\{-\frac{M |\Re\{e^{\ii \pi j/n} \conj{\xi}\}|} {2N}y\bigg\} \int \limits_0^\infty \exp\{-t^n/4\} \dd t \\ 
&\le 64 \exp \bigg[-\min \bigg\{\frac{\epsilon M |\Re\{e^{\ii \pi j/n \conj{\xi}}\}|}{2}, \frac{1}{12n^{1/(n-1)}}\bigg(\frac{M |\Re\{e^{\ii \pi j/n \conj{\xi}}\}|}{N}\bigg)^{n/(n-1)}\bigg\}\bigg]\label{flucE121}
\end{align}
Note that, up to constants, the last exponent matches (\ref{flucE155}). Hence, the claimed bound. 
\end{proof}

\begin{prop}
\label{flucP4}
\begin{align*}
|I''| \le T|F(a)|\exp\bigg\{-\frac{\epsilon^nN^n}{2} + \max \limits_{\tau(\epsilon) \le t \le T} \big\{\tilde{u}(\Phi(t))-\tilde{u}(a)\big\}\bigg\}
\end{align*}
\end{prop}
\begin{proof}
By Proposition \ref{flucL10}c, $u(\Phi(t))-u(a) \le -N^n\epsilon^n/2$ for $t \in [\tau(\epsilon), T]$. Hence,  
\begin{align*}
|I''| &\le \int \limits_{\tau(\epsilon)}^T |F(\Phi(t))| \ \dd t \le |F(a)| \int \limits_{\tau(\epsilon)}^T \exp\{u(\Phi(t))-u(a) + \tilde{u}(\Phi(t))-\tilde{u}(a)\}\ \dd t \\
&\le T|F(a)|\exp\bigg\{-\frac{\epsilon^nN^n}{2} + \max \limits_{\tau(\epsilon) \le t \le T} \big\{\tilde{u}(\Phi(t))-\tilde{u}(a)\big\}\bigg\}. \qedhere 
\end{align*}
\end{proof}

%{\color{red} Motivation for the  next theorem. Preceding proposition is not sufficient for our purposes because we cannot prove the full lengths of the curves are uniformly bounded.}
Now suppose that $\lim_{t \rightarrow T^+} \Phi(t) = x$ for some $x \in Z$. There exist unique $p, \tilde{p} \in \bbR$ such that the functions $u_1, \tilde{u}_1$ defined on $\bbC \smallsetminus \{P \cup Z\}$ by 
\begin{align}
u_1(z) = u(z)-p \log |z-x| \qquad \tilde{u}_1(z) = \tilde{u}(z)-\tilde{p} \log |z-x|
\end{align}
can be extended as harmonic functions to $(\bbC \smallsetminus \{P \cup Z\}) \cup \{x\}$ by setting $u_1(x) = \lim_{z \rightarrow x} u_1(z)$ and $\tilde{u}_1(x) = \lim_{z \rightarrow x} \tilde{u}_1(z)$. Put $b = \Phi(T)$ and let $\Psi: [0, S]$ be any $\sC^1$ curve with unit speed such that $\Psi(0) = b$, $\Psi([0, S])$ is disjoint from $P \cup Z$ and $|\Psi(t)-x| \le \delta$ for $t \in [0, S]$. Define 
the contour integral 
\begin{align}
I''' = \int \limits_{\Psi(0)}^{\Psi(S)} F(z)\ \dd z. 
\end{align}
\begin{prop}
\label{flucP5}
\begin{align*}
|I'''| &\le S |F(a)| \exp \bigg\{-\frac{\epsilon^nN^n}{2} + \max \limits_{\tau(\epsilon) \le t \le T} \big\{\tilde{u}(\Phi(t))-\tilde{u}(a)\big\}\bigg\} \\
&\cdot \exp \bigg\{\max \limits_{0 \le t \le S} \big\{u_1(\Psi(t))-u_1(b)+\tilde{u}_1(\Psi(t))-\tilde{u}_1(b)\big\}\bigg\}. 
\end{align*}
\end{prop}
\begin{proof}
Since $p+\tilde{p} \ge 0$ (as $x \in Z$), $\Phi(T) = \Psi(0) = b$ and $\log |b-x| = \delta$, we have  
\begin{align*}
u(\Psi(t)) + \tilde{u}(\Psi(t)) &= (p+\tilde{p})\log |\Psi(t)-x| + u_1(\Psi(t))+\tilde{u}_1(\Psi(t)) \\
&\le (p+\tilde{p})\log \delta + u_1(\Psi(t))+\tilde{u}_1(\Psi(t)) \\
&= (p+\tilde{p})\log |b-x|+ u_1(b)+\tilde{u}_1(b) \\
&+u_1(\Psi(t))-u_1(b)+\tilde{u}_1(\Psi(t))-\tilde{u}_1(b) \\
&\le u(b) + \tilde{u}(b) + \max \limits_{0 \le t \le S} \big\{u_1(\Psi(t))-u_1(b)+\tilde{u}_1(\Psi(t))-\tilde{u}_1(b)\big\} \\
&\le u(a)+\tilde{u}(a)-\frac{|c|\epsilon^n}{2} + \max \limits_{\tau(\epsilon) \le t \le T} \big\{\tilde{u}(\Phi(t))-\tilde{u}(a)\big\} \\
&+ \max \limits_{0 \le t \le S} \big\{u_1(\Psi(t))-u_1(b)+\tilde{u}_1(\Psi(t))-\tilde{u}_1(b)\big\}. 
\end{align*}
For the last inequality, we use Proposition \ref{flucL10}c. Then the result follows because  
\begin{align*}
|I'''| &\le \int \limits_{0}^S |F(\Psi(t)| \ \dd t = \int \limits_{0}^S \exp \{u(\Psi(t)) + \tilde{u}(\Psi(t))\} \ \dd t. \qedhere
\end{align*}
\end{proof}

\section{Uniform convergence lemmas}

We collect the auxiliary lemmas utilized to obtain the uniform results in Section \ref{flucS5}. 

\begin{lem}
\label{flucAL1}
Let $\{f_i: i \in \sI\}$ be a collection of $\sC^1$ functions on the interval $[a, b]$. Let $(X_j)_{j \in \bbN}$ be an ergodic sequence in $[a, b]$ with marginal distribution $\mu$. Suppose that 
\begin{align}
\label{flucE181}
\sup \limits_{i \in \sI} \sup \limits_{x \in [a, b]} |f'_i(x)| < \infty. 
\end{align}
Then  
\begin{align}
\lim \limits_{n \rightarrow \infty} \frac{1}{n} \sum \limits_{j = 1}^n f_i(X_j) = \int \limits_a^b f_i(x) \mu(\dd x) \quad \text{uniformly in } i \in \sI \text{ a.s.}
\end{align}
 
\end{lem}
\begin{proof}
Since $(X_j)_{j \in \bbN}$ is ergodic,  
\begin{equation}
\label{flucAEQ1}
\lim \limits_{n \rightarrow \infty} \frac{1}{n} \sum \limits_{j=1}^n \one_{\{t \le X_j\}} = \mu([t, b]) \quad \text{ for } t \in [a, b] \text{ a.s.}
\end{equation}
Pick an a.s.\ event $E$ on which (\ref{flucAEQ1}) holds for all $t$ in a countable dense subset $S$ of $[a, b]$ that contains the points of discontinuities of $\mu([\cdot, b])$. Then (\ref{flucAEQ1}) also holds on $E$ for $t \not \in S$ by the continuity of the function $\mu([\cdot, b])$ at $t$ and the monotonicity of $\one_{\{\cdot \le X_j\}}$. We have 
\begin{align}
\bigg|\frac{1}{n} \sum \limits_{j = 1}^n f_i(X_j)-\int \limits_a^b f_i(x) \mu(\dd x)\bigg| &= \bigg|\frac{1}{n} \sum \limits_{j = 1}^n \int_{a}^{X_j} f_i'(t)\dd t -\int \limits_a^b \int \limits_{a}^x f_i'(t)\dd t\mu(\dd x)\bigg| \nonumber\\
&= \bigg|\int \limits_a^b f_i'(t) \left(\frac{1}{n}\sum \limits_{j = 1}^n \one_{\{t \le X_j\}}- \mu([t, b])\right) \dd t \bigg|\nonumber\\
&\le C \int \limits_{a}^b \left|\frac{1}{n} \sum \limits_{j=1}^n \one_{\{t \le X_j\}} - \mu([t, b])\right| \dd t \label{flucE15}
\end{align}
for some constant $C > 0$ due to (\ref{flucE181}). Note that the integrand in (\ref{flucE15}) does not depend on $i$, is bounded by $2$ and converges to $0$ for $t \in [a, b]$ on $E$. Therefore, the result follows from the dominated convergence. 
\end{proof}

%\begin{lem}
%\label{flucL36}
%Let $\sI$ be a nonempty set and $f_{n, i} \in \Hol(\Disc(a, r))$ for $n \in \bbN$ and $i \in \sI$. Suppose that, for any $0 < \rho < r$,  
%\begin{align}
%f_i(z) = \lim_{n \rightarrow \infty} f_{n, i}(z) \quad \text{ uniformly in } z \in \conj{\Disc}(a, \rho) \text{ and } i \in \sI
%\end{align}
%for some functions $\{f_i: i \in \sI\}$ on $\Disc(a, r)$. Then $f_i \in \Hol(\Disc(a, r))$ and, for any $0 < \rho < r$ and $k \in \bbZ_+$,  
%\begin{align}
% f_i^{(k)}(z) = \lim_{n \rightarrow \infty} f_{n, i}^{(k)}(z) \quad \text{ uniformly in } z \in \conj{\Disc}(a, \rho) \text{ and } i \in \sI. 
%\end{align}
%\end{lem}
%\begin{proof}
%{\color{red}Cite the standard proof.}
%\end{proof}

%
% The choice of bibliography style is a major decision, jointly made
% by you, your thesis advisor and the thesis editor. Common choices are
% "siam", "acm", "amsplain", "plain", "chicago".
%
\bibliographystyle{habbrv}
\bibliography{thesis}

\end{document}